\documentclass[12pt,a4paper,headings=normal,twoside]{scrbook}
\usepackage[T1]{fontenc} 
\usepackage{mathtools}  
\usepackage[utf8]{inputenc} 
\usepackage{lmodern} 
\usepackage[ngerman,english]{babel} 
\usepackage{csquotes} 

\usepackage{ellipsis}
\usepackage{microtype} 

\usepackage{booktabs} 
\usepackage{mathrsfs} 
\usepackage[arrow, matrix, curve]{xy}
\usepackage{tensor}
\usepackage{amsmath} 
\usepackage{amsthm}
\usepackage{amsfonts} 
\usepackage{amssymb} 
\usepackage{amsxtra}
\usepackage{amscd}
\usepackage{amsbsy} 
\usepackage{trfsigns} 
\usepackage{upgreek} 
\usepackage{textcomp}  
\usepackage{siunitx} 
\usepackage[hang,flushmargin]{footmisc} 

\newcommand{\ket}[1]{{| #1 \rangle}}      
\newcommand{\br}[1]{{\langle #1 \rangle}}  

\usepackage{tikz} 
\usetikzlibrary{cd}
\tikzcdset{arrow style=tikz, diagrams={>=stealth}}
\numberwithin{equation}{section}
\usepackage[onehalfspacing]{setspace}
\usepackage[left=2.5cm,right=2.5cm,top=2.5cm,bottom=2.5cm]{geometry}
\usepackage{rotating} 

\usepackage{graphicx} 
\usepackage[labelfont=bf]{caption} 
\usepackage{subcaption} 
\usepackage{multirow} 
\usepackage{multicol} 

\usepackage{pdfpages} 

\usepackage{cite}

\let\tmp\newinsert 
\let\newinsert\newbox
\usepackage{floatrow}
\let\newinsert\tmp 
\usepackage{xcolor} 
\usepackage[hidelinks]{hyperref}
\hypersetup{
colorlinks,
linkcolor={red!00!black},
citecolor={green!80!black},
urlcolor={blue!80!black},
menucolor={red!100!black}
}

\usepackage[procnames]{listings} 

\usepackage{fancyhdr}
\newcommand{\End}{\operatorname{End}}
\newcommand{\Hom}{\operatorname{Hom}}

\numberwithin{equation}{section}
\newtheorem{thm}{Theorem}[subsection]
\newtheorem{prop}[thm]{Proposition}

\newtheorem{conj}[thm]{Conjecture}
\newtheorem{lem}[thm]{Lemma}
\newtheorem{rem}[thm]{Remark}
\newtheorem{rems}[thm]{Remarks}
\newtheorem{xmpl}[thm]{Example}
\newtheorem{cor}[thm]{Corollary}
\newtheorem{cors}[thm]{Corollaries}
\newtheorem{definition}[thm]{Definition}

\newtheorem{defprop}[thm]{Definition-proposition}

\newcommand{\recbinom}{\genfrac{[}{]}{0pt}{}}
\newcommand{\rightarrowdbl}{\rightarrow\mathrel{\mkern-14mu}\rightarrow}

\newcounter{saveenumi}

\newcommand{\id}{{\operatorname{id}}}
\newcommand{\tr}{{\operatorname{tr}}}

\addto{\captionsenglish}{}

\usepackage[title]{appendix}

\usepackage{enumitem}

\begin{document}
\selectlanguage{english}
\raggedbottom
\definecolor{keywords}{RGB}{0,51,255}
\definecolor{comments}{RGB}{0,0,113}
\definecolor{red}{RGB}{160,0,0}
\definecolor{green}{RGB}{0,150,0}
\lstset{frame=single,language=Python, breaklines=true, 
	basicstyle=\ttfamily\scriptsize, 
	keywordstyle=\color{keywords},
	commentstyle=\color{comments},
	stringstyle=\color{red},
	showstringspaces=false,
	identifierstyle=\color{green},
	procnamekeys={def,class},
	numbers=left}


\begin{titlepage}
	\large
	\centering
	\vspace{1cm}
	\textbf{On the structure of tensor products of the finite-dimensional representations of quantum affine $\boldsymbol{\mathfrak{sl}_n}$}\\
	\vspace{4cm}
	\textbf{Master thesis}\\
	\vspace{1cm}
	To obtain the degree\\
	Master of Science - mathematics\\
	submitted by\\
	Henrik \textsc{Jürgens}\\
	\vspace{2cm}
	{4. September 2024\\
	(edited 28. Januar 2025)}\\
	\vspace{2cm}
	\begin{figure}[H]
		\includegraphics[scale = 2]{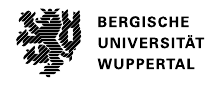}
	\end{figure}
	Fakultät für Mathematik und Naturwissenschaften\\
	\vspace{1cm}
	Erstgutachter: Prof. Dr. Matthias Wendt\\
	Zweitgutachter: Prof. Dr. Markus Reineke\\
\end{titlepage}
\tableofcontents
\thispagestyle{empty}

\newpage
\setcounter{page}{1}
\chapter{Introduction}
In mathematics, \textit{quantum affine algebras} (or \textit{affine quantum groups}), \textit{Yangians} and similar \textit{Hopf algebras} appear when considering (parameter dependent) braiding relations. In statistical mechanics, a branch of physics, these so-called '\textit{quantum groups}' come up as the local symmetries of integrable vertex models and their corresponding spin chains. Moreover, the braiding relations turn out to be the local integrability condition for the so-called $R$\textit{-matrix}, a matrix which is usually defined only to simplify the multiplication of the local (Boltzmann-)weights of the model. For these integrable vertex models, however, the local weights, collected in the $R$-matrix, depend on two spectral parameters (usually complex or real) and the $R$-matrix satisfies the so-called parameter dependent (quantum) \textit{Yang--Baxter equation}
\begin{equation*}
	R_{12}(u_1,u_2)R_{13}(u_1,u_3)R_{23}(u_2,u_3) = R_{23}(u_2,u_3)R_{13}(u_1,u_3)R_{12}(u_1,u_2).
\end{equation*}
In fact, from this equation, together with some regularity assumptions on the $R$-matrix, one obtains a one parameter family of commuting transfer matrices of the vertex model.
Equivalently, it is possible to construct an infinite number of independent conserved quantities (preserved symmetries of the model) for the corresponding spin chain and the parameter dependent transfer matrix of the vertex model can be interpreted as their generating function.

From an abstract point of view, the physical models can be build simply by considering representations of these quantum groups and applying them on an element $\mathcal{R}$, called the \textit{universal $R$-matrix}, which sits in the twofold tensor product of the quantum group (in some sense) and satisfies the braiding relations.\footnote{I.e. the (quantum) Yang--Baxter equation and some further (regularity) assumptions on the element $\mathcal{R}$ which will be explained in more detail later.}

In this thesis, the plan is to give a precise description of the finite-dimensional representation theory of quantum affine $\mathfrak{sl}_n$ and the cluster algebra structure of its \textit{Grothendieck ring}. The idea is to build a mathematical foundation, connect different levels of mathematical abstraction and explain how they can be used to answer questions about the structure of the representations.

Personally, my motivation to understand the structure of these representations, as a physicist, is based on an attempt to understand and calculate the correlation functions of the models, which correspond to fundamental representations of the Yangian of $\mathfrak{sl}_n$, exactly, i.e. rank $n-1$ versions of the Heisenberg XXX spin chain (see \cite{BJ} and \cite{HJue}). However, I ended up also learning about the structure of the finite-dimensional representations of quantum affine $\mathfrak{sl}_n$, which is known to be similar. Indeed, this makes sense from a physics point of view since the models corresponding to fundamental representations of quantum affine $\mathfrak{sl}_n$ are similar rank $n-1$ versions of the Heisenberg XXZ spin chain. The latter is known to limit in the XXX spin chain when sending the anisotropy parameter $\Delta$ to $1$ in a certain way.

However, there has been great success with calculating the correlation functions of the Heisenberg XXX and XXZ models (i.e. rank 1) by analysing the (tensor product) structure of their representations. This is explained in the paper \cite{BJMST} and consecutively in the famous series of papers called 'Hidden Grassmann Structure in the XXZ model' \cite{BJMST_HGS,BJMST_HGS2,JMS_HGS3}. Apart from my ansatz \cite{HJue}, which turned out to describe the correlation functions only partially, it is unclear if there is a way to generalize the whole construction for higher rank and if or how it is based on a more general mathematical structure.

Certainly, it is helpful to understand and explain the mathematics that I have applied. This is because on the one side I want to set a mathematical basis for continuing my research and on the other side I want to simplify the accessibility of this branch of mathematics for some of my friends and other mathematicians and mathematical physicists who may not be able to spend so much time studying in particular the books of R. Carter \cite{Carter}, V. Chari and A. Pressley \cite{CPBook}, and all the other papers I have been working through. Therefore, the ultimate goal of this thesis is to explain the finite-dimensional representation theory of quantum affine $\mathfrak{sl}_n$ from an almost elementary standpoint.
\section{Structure}
The thesis is organized in the following way.\\

In \textbf{Chapter} \ref{ch:math_backgr} I give an introduction to the mathematical structures we are interested in and outline the connection to simple Lie algebras.
These are the so-called 'quantum groups', most generally referred to as (non-cocommutative) Hopf algebras. As the chapter is mainly based on the books of Carter \cite{Carter} and Chari and Pressley \cite{CPBook}, it is called 'mathematical background'. In the concrete, I will mainly collect and reorganise some key facts from R. Carter's book \cite{Carter} on finite and affine Kac--Moody algebras, as well as from V. Chari and A. Pressley's book \cite{CPBook} on Quantum Groups, which are necessary to understand the main objects in this thesis in Section \ref{sect:Kac--Moody} and \ref{sect:quantum_groups}, respectively. These are the finite-dimensional representations of (untwisted) quantum affine algebras. The general idea is to provide an analogy between the (loop algebra) realisation of an untwisted affine Kac--Moody algebras given in the book of Carter \cite{Carter} in Chapter 18 and the 'second Drinfeld realization' of untwisted quantum affine algebras given in Chapter 12 in the book of Chari and Pressley. Building on this foundation, we proceed to discuss the finite-dimensional representations of the latter in Section \ref{sect:finite-dim_reps} and in Chapter \ref{ch:fin_dim_rep_quantum_aff} highlighting some of the analogies to the finite-dimensional representation theory of finite-dimensional simple Lie algebras over $\mathbb{C}$ discussed in Subsection \ref{subsect:Cartan_data}. The sections are organized as follows.\\

After providing a brief introduction to the concepts of Lie algebras and their universal enveloping algebras, we introduce Kac--Moody algebras in \textbf{Section} \ref{sect:Kac--Moody} as a generalisation of finite-dimensional simple Lie algebras over $\mathbb{C}$. We give a short summary of what is explained and proven in the book of Carter \cite{Carter} in the Chapters 14-18 and prepare the discussion of quantum affine algebras in Section \ref{sect:quantum_groups}. In particular, this is done to explain the importance of the so-called 'second Drinfeld realisation' for quantum affine algebras by drawing an analogy to the loop construction for affine Kac--Moody algebras.\\

The story goes as follows. After writing down the definition of a so-called 'generalized Cartan matrix' (GCM), denoted by $A$, we then provide an abstract definition of the Kac--Moody algebras $L(A)$ and $L(A)'$ in terms of $A$ in \textbf{Subsection} \ref{subsect:Kac--Moody_abstract}.\\

Given that the structure of a Kac--Moody algebra is fundamentally determined by its GCM, we briefly explore the trichotomy of possible GCM types in \textbf{Subsection} \ref{subsect:GCM_finite_affine}. We establish that GCMs of finite and affine types are symmetrizable, with finite-type GCMs being equivalent to Cartan matrices. In such cases, the Kac--Moody algebra $L(A)$ corresponds to the finite-dimensional complex simple Lie algebra with the Cartan matrix $A$.
Additionally, we note that affine Cartan matrices are classified by a list of affine connected Dynkin diagrams, paralleling the classification of Cartan matrices by connected Dynkin diagrams.\\

Since the Kac--Moody algebra $L(A)$ is generally infinite-dimensional, defining an invariant bilinear form, such as the Killing form for finite-dimensional Lie algebras, becomes problematic. To address this, a solution is provided by the standard invariant bilinear form, applicable when the GCM $A$ is symmetrisable. This is further explained in \textbf{Subsection} \ref{subsect:standard_inv_bil}.\\

Since this thesis focuses particularly on (untwisted) affine Kac--Moody algebras, we detail the concepts specific to these in \textbf{Subsection} \ref{subsect:aff_Kac--Moody}.\\

Given that the invariant bilinear form for a Kac--Moody algebra is constructed inductively on subalgebras $L(r)\subset L(A)$ where $L(0)= H$ and $\bigcup_{r\geq 0} L(r) = L(A)$, it is evident that Kac--Moody algebras are more complex than finite-dimensional simple Lie algebras. However, starting from a Cartan matrix $A^0$, there is a straightforward method for constructing an affine Cartan matrix $A$. These affine Cartan matrices are known as affine Cartan matrices of untwisted type. Fortunately, a similar construction -- referred to as the loop construction -- exists at the level of Lie algebras. We discuss this construction for untwisted affine Kac--Moody algebras in \textbf{Subsection} \ref{subsect:loop_constr_untw_Kac--Moody}. We conclude that $L(A)$ can be constructed from $L^0=L(A^0)$ by first considering the loop algebra $\mathfrak{L}(L^0)$, then constructing a one-dimensional central extension, and finally adjoining an element $d$ that acts as a derivation. Additionally, we identify the elements of the affine Kac--Moody algebra $L(A)$, as described in Subsection \ref{subsect:aff_Kac--Moody}, with their corresponding elements in $\hat{\mathfrak{L}}(L^0)=\tilde{\mathfrak{L}}(L^0)\oplus\mathbb{C}d =\mathfrak{L}(L^0)\oplus\mathbb{C}c\oplus\mathbb{C}d $, and provide an explicit description of the standard invariant bilinear form in $\hat{\mathfrak{L}}(L^0)$. This approach is much more concrete compared to the abstract definition of $L(A)$, since $\mathfrak{L}(L^0)$ can be viewed as the Lie algebra of Laurent polynomial maps from $\mathbb{C}^{\times}$ to $L^0$.\\

Before describing the structure of (affine) quantum groups and explaining how quantum affine algebras can be defined through quantisations of universal enveloping algebras of (affine) Kac--Moody algebras, we provide a short summary of the finite-dimensional representation theory of finite-dimensional semisimple Lie algebras over $\mathbb{C}$ in \textbf{Subsection} \ref{subsect:Cartan_data}. Moreover, we establish some concepts which will be generalised later when discussing the representation theory of quantized universal enveloping algebras and specifically quantum affine algebras in Section \ref{sect:finite-dim_reps} and Chapter \ref{ch:fin_dim_rep_quantum_aff}.\\

In \textbf{Section} \ref{sect:quantum_groups} we finally introduce the so-called 'quantum groups', most generally referred to as Hopf algebras.\footnote{ Sometimes also defined as Hopf algebras that are quasitriangular in some sense (see Subsection \ref{subsect:QUE_algebras}).} The focus of this section is to prepare for the discussion of the abelian monoidal structure of the finite-dimensional representations of quantized universal enveloping algebras of (affine) Kac--Moody algebras.\\

To set a starting point, we establish the general concept of a Hopf algebra using 'string diagrams' in \textbf{Subsection} \ref{subsect:Hopf_algebras}. We then provide the basic definition of a representation of a Hopf algebra and discuss some fundamental constructions and properties related to it.\\

With the basic concepts established, we introduce the notion of an almost cocommutative Hopf algebra in \textbf{Subsection} \ref{subsect_quasitriang_Hopf} and discuss the properties of its universal $R$-matrices. We show that these matrices naturally define an isomorphism $V\otimes W\to W\otimes V$, thereby inducing a braiding on the tensor product of representations. Furthermore, we demonstrate that the universal $R$-matrix satisfies the (quantum) Yang--Baxter equation in the special case where the Hopf algebra is quasitriangular.\\

In fact, the representations of a Hopf algebra and quasitriangular (or almost cocommutative) Hopf algebra can more generally be understood in terms of the concept of an abelian monoidal category and an abelian quasitensor category, respectively. This is discussed in \textbf{Subsection} \ref{subsect:Ab_mon_cat}. Moreover, we introduce the important concept of the Grothendieck ring for an abelian monoidal category and state its commutativity in the case when the category is quasitensor.\\

After discussing the general structure of representations of (quasitriangular) Hopf algebras, we return to examining interesting explicit examples in \textbf{Subsection} \ref{subsect:QUE_algebras}. These are among other things the quantizations of universal enveloping algebras of (affine) Kac--Moody algebras.\footnote{ We will call them simply quantized universal enveloping (or QUE) algebras.} These are, in fact, Hopf algebras and we discuss two definitions, $U_h(L(A)')$ over the ring $\mathbb{C}[[h]]$ of formal power series in an indeterminate $h$, and $U_q(L(A)')$ over the field $\mathbb{Q}(q)$ (or $\mathbb{C}(q)$ by base change) of rational functions of an indeterminate $q$. After providing some motivation, we introduce these algebras in terms of their Drinfeld--Jimbo generators,\footnote{ We also call them Chevalley generators in analogy to our discourse on Kac--Moody algebras in Section \ref{sect:Kac--Moody}.} and then deduce some of their basic properties. We also initiate the discussion of whether these specific Hopf algebras can be quasitriangular.\\

Finally, we introduce an alternative realisation in terms of loop algebra-like generators in the case when $L(A)'=\tilde{g}$ is an untwisted affine Kac--Moody algebra in \textbf{Subsection} \ref{subsect:quantum_affine}, drawing the analogy to the loop construction in Subsection \ref{subsect:loop_constr_untw_Kac--Moody}. This alternative is sometimes called the 'new realization' or 'second Drinfeld realization' and $U_h(\tilde{g})$ resp. $U_q(\tilde{g})$ are called (untwisted) quantum affine algebras. We also prepare the discussion of the finite-dimensional representations in Section \ref{sect:finite-dim_reps} and Chapter \ref{ch:fin_dim_rep_quantum_aff}, noting some of their important basic properties in addition to those already covered by the Drinfeld--Jimbo presentations in Subsection \ref{subsect:QUE_algebras}.\\

In \textbf{Section} \ref{sect:finite-dim_reps} we explore some of the important basic properties of the finite-dimensional representations of quantum affine algebras, primarily explained in the Chapters 10 and 12 of the book \cite{CPBook}. However, we will incorporate information from various related papers, definitions from my dissertation \cite{HJue}, and some of my own calculations to complete the proofs and prepare the comprehensive discussion of the finite-dimensional representation theory of untwisted quantum affine algebras in Chapter \ref{ch:fin_dim_rep_quantum_aff}.\\

We begin by drawing the analogy between the finite-dimensional (type $\boldsymbol{1}$) representations of $U_q(\mathfrak{g})$ and those of $\mathfrak{g}$ at the begin of the section (cf. \cite{CPBook} Chapter 10). Building on this analogy, we explore the general description of the finite-dimensional (type $\boldsymbol{1}$) representations of the quantum affine algebra $U_q(\tilde{\mathfrak{g}})$ via Drinfeld polynomials in \textbf{Subsection} \ref{subsect:drinfeld_poly}. Specifically, we establish a one-to-one correspondence (up to twisting) between irreducible finite-dimensional representations of $U_q(\tilde{\mathfrak{g}})$ and $I$-tuples of polynomials $P_i\in\mathbb{C}[u]$, $i\in I$, with constant term $1$. Additionally, we incorporate the definition of the loop-weight spaces, justify the term 'loop-weight' for finite-dimensional (type $\boldsymbol{1}$) representations, and then present our initial definition of the $q$-character based on the loop-weights.\\

In \textbf{Subsection} \ref{subsect:Jimbos_hom_und_ev_reps} we explore the consequences of the existence of an evaluation homomorphism for $\mathfrak{g}=\mathfrak{sl}_{l+1}$, $l\in\mathbb{N}$ (Jimbo's homomorphism). Specifically, we compute the Drinfeld polynomials and loop-weights for the evaluation representations of $U_q(\tilde{\mathfrak{sl}}_2)$. We then discuss further consequences established in \cite{CP1994a}. In addition, we obtain a complete description of the tensor product of irreducible (type $\boldsymbol{1}$) representations of $U_q(\tilde{\mathfrak{sl}}_2)$, implying in particular that $U_q(\tilde{\mathfrak{sl}}_2)$ can't be quasitriangular. Nevertheless, since the non-quasitriangularity depends on specific cases where the Drinfeld polynomials are in a 'special position', we can still establish the commutativity of the Grothendieck ring $\operatorname{Rep}U_q(\tilde{\mathfrak{sl}}_2)$.\\

In \textbf{Subsection} \ref{subsect:trigon_R_matrix_A_l}, we further discuss the properties of (untwisted) quantum affine algebras, highlighting a result by Drinfeld which shows that $U_h(\tilde{\mathfrak{g}})$ is 'pseudotriangular'. Precisely, this means that there exists an element $\tilde{\mathcal{R}}(\lambda)\in (U_h(\tilde{\mathfrak{g}})\otimes U_h(\tilde{\mathfrak{g}}))((\lambda))$, known as the 'pseudo-universal' $R$-matrix, which satisfies the (quantum) Yang Baxter equation with multiplicative spectral parameters. Fortunately, an explicit but rather technical formula for its computation was obtained by Khoroshkin and Tolstoy \cite{KT} in 1992. Moreover, the fundamental $R$-matrix $R(\lambda)=(\rho_f\otimes\rho_f)((\tau_\lambda\otimes\id)(\tilde{\mathcal{R}}))$ has been explicitly calculated for $\mathfrak{g}=\mathfrak{sl}_2$ and $\mathfrak{g}=\mathfrak{sl}_3$ in \cite{BGKNR}, and the general formula for $\mathfrak{g}=\mathfrak{sl}_{l+1}$ was conjectured based on direct calculations, identifying it as the unique (up to normalisation) intertwiner between fundamental representations. The Khoroshkin and Tolstoy formula, however, offers the advantage of providing a normalisation that preserves all the functional properties of the universal $R$-matrix. We conclude the subsection by presenting their result and stating the general properties of $\tilde{\mathcal{R}}(\lambda)$.\\

Having introduced the fundamental (or fundamental-fundamental) $R$-matrix in Subsection \ref{subsect:trigon_R_matrix_A_l}, we have also set the basis for defining the $\mathfrak{sl}_{l+1}$ fundamental Vertex models in statistical mechanics. We approach this topic by briefly introducing a graphical notation in \textbf{Subsection} \ref{subsect:graphical_notation},  followed by an explanation from a physics perspective in \textbf{Subsection} \ref{subsect:A_physics_pov}.\\

In \textbf{Chapter} \ref{ch:fin_dim_rep_quantum_aff}, we return to the study of the category $\mathrm{C}$ of finite-dimensional type $\boldsymbol{1}$ representations of $U_q(\tilde{\mathfrak{g}})$. To prepare the reader for the general discussion, we first summarize and reformulate the results for $\mathfrak{g}=\mathfrak{sl}_2$. We then examine the properties of the $q$-character and draw an analogy with the representation theory of $\mathfrak{g}$ by formulating the theorem of the highest loop-weight for $U_q(\tilde{\mathfrak{g}})$. Finally, in the case when $\mathfrak{g}=\mathfrak{sl}_{l+1}$, we provide a comprehensive explanation of the structure of the Grothendieck ring of $\mathrm{C}$ and present an explicit formula for calculating the $q$-characters of the (prime) snake modules, which turn out to be real (prime) simple objects of $\mathrm{C}$.\\

In \textbf{Section} \ref{sect:fin-dim_rep_quantum_aff_sl2} we prepare the general discussion of the category $\mathrm{C}$ by summarizing and partially reformulating the comprehensive results obtained in Subsections \ref{subsect:drinfeld_poly} and \ref{subsect:Jimbos_hom_und_ev_reps} for $\mathfrak{g}=\mathfrak{sl}_2$. Precisely, we explicitly write down the $q$-character of the evaluation representation $V^{(r)}(a)$ as a Laurent monomial in the fundamental loop-weights $Y_a$. By comparing the loop-weight spaces of different $\mathfrak{g}$-weight, we identify affine simple roots $A_a$ and observe a partial ordering of the loop-weights compatible with the partial ordering of the $\mathfrak{g}$-weights. This leads us to consider the formulation of a theorem of the highest loop-weight and the fact that the $q$-character defines an injective ring homomorphism $\chi_q:\operatorname{Rep}U_q(\tilde{\mathfrak{sl}}_2)\to \mathbb{Z}[Y_{a}^{\pm 1}]_{a\in\mathbb{C}^\times}$. We then elucidate that it is enough to only consider representations in a subcategory $\mathrm{C}_{\mathbb{Z}}$ which includes all special positions, highlighting the simplification obtained by effectively working 'modulo general positions'. Finally, we reformulate the product structure of the Grothendieck ring, as indicated by Proposition \ref{prop:special_pos_short_ex}, in terms of the so-called $\operatorname{T}$-system.\\

Moving on to the general description of the representation theory of $\mathrm{C}$ in \textbf{Section} \ref{sect:gen_untw_case}, we formulate the proof of Theorem \ref{thm:drinfeld-l_weights} which details the structure of the loop-weights. Hence, justifying the Definition \ref{def:q-character_drinfeld} of the $q$-character.
We simplify the notation by expressing everything in terms of the fundamental loop-weights $Y_{i,a}=1-au_i$, $i\in I$.
After briefly discussing a more direct but equivalent definition of the $Y_{i,a}$ provided in the papers \cite{FR} and \cite{FM}, we identify the affine simple roots and define the affine weight and root lattices, $\mathcal{P}$ and $\mathcal{Q}$, respectively. We introduce the concept of dominant loop-weights and observe a partial ordering of the loop-weights that is consistent with the partial ordering of $\mathfrak{g}$-weights.
We then formulate the theorem on the highest loop-weight and describe the general properties of the $q$-character, extending the results for $\mathfrak{g}=\mathfrak{sl}_2$ in a straightforward manner. In particular, we discuss the commutativity of the Grothendieck ring $\operatorname{Rep}U_q(\tilde{\mathfrak{g}})$, and that $\chi_q$ is an injective ring homomorphism $\operatorname{Rep}U_q(\tilde{\mathfrak{g}})\to \mathbb{Z}[\mathcal{P}]$, which restricts to the usual character homomorphism on $U_q(\mathfrak{g})\lhook\joinrel\rightarrow U_q(\tilde{\mathfrak{g}})$.
Finally, we remark that the $q$-character homomorphism $\chi_q$ is $\mathbb{C}^\times$-equivariant, thereby setting the stage for working 'modulo general positions' in the subsequent subsections.\\

Indeed, based on the characterization of special positions outlined in the paper \cite{C2002}, it is clear that limiting our consideration to a subcategory $\mathrm{C}_{\mathbb{Z}}$ is again sufficient. This is explained in \textbf{Subsection} \ref{subsect:q-char_and_snake_mod}, where we narrow our focus to type $A_l$, $l\in\mathbb{N}$, moving forward. We then define the so-called (prime) snake modules, specific irreducible elements of $\mathrm{C}_{\mathbb{Z}}$, introduced by Mukhin and Young in \cite{MY2} for types $A$ and $B$. These are characterised by being special, anti-special, thin and real (and, respectively, prime). We further claim that prime snake modules are exactly the prime irreducible elements of $\mathrm{C}_{\mathbb{Z}}$ and discuss a system of relations in the Grothendieck ring $\operatorname{Gr}(\mathrm{C}_{\mathbb{Z}})$ established in \cite{MY}. This so-called 'extended $\operatorname{T}$-system' is a system of relations satisfied by the prime snake modules that extends the usual $T$-system. It is reasonable to expect that our assertion can be proven with it. However, there already exists a proof in the paper \cite{BJY} by Duan, Li and Luo which uses a slightly different set of relations for the prime snake modules called the $S$-systems. We give the reference, briefly explain their findings, and offer a more detailed explanation in Subsection \ref{subsect:form_in_terms_of_clust_alg}.\\

Knowing about the importance of prime snake modules, we are interested in their $q$-characters. On one side, this is because it provides a way to explicitly write down the product in the Grothendieck ring, on the other side, because we can extract the $\mathfrak{g}$-character from it by restriction, and hence observe its decomposition into a direct sum of $\mathfrak{g}$-irreducible subspaces.
Fortunately, Mukhin and Young prove an explicit formula for the $q$-characters of snake modules shortly after introducing them in their paper \cite{MY2}. This so-called 'path formula' for the $q$-characters of snake modules is described in \textbf{Subsection} \ref{subsect:the_path_formula}. We conclude the subsection by suspecting an elegant solution to the problem of computing projectors onto prime snake modules.\\

Finally, in \textbf{Subsection} \ref{subsect:form_in_terms_of_clust_alg}, we return to a more detailed explanation of the results in the paper \cite{BJY}. As explained in Subsection \ref{subsect:q-char_and_snake_mod}, their results imply our claim that prime snake modules are exactly the prime irreducible elements of $\mathrm{C}_{\mathbb{Z}}$. Moreover, the authors prove the so-called Hernandez--Leclerc conjecture for types $A$ (and $B$), implying that the Grothendieck ring $\operatorname{Gr}(\mathrm{C}_{\mathbb{Z}})$ is isomorphic to an infinite cluster algebra, call it $\mathcal{A}_\infty$, such that the prime snake modules correspond exactly to the cluster variables of $\mathcal{A}_\infty$. In fact, the $S$-systems, satisfied by the prime snake modules, are defined such that they resemble the mutations of cluster variables in the cluster algebra $\mathcal{A}_\infty$. Therefore, prime snake modules are precisely the prime irreducible elements of $\mathrm{C}_{\mathbb{Z}}$. We briefly introduce the concept of cluster algebras, define the infinite cluster algebra $\mathcal{A}_\infty$ and finally state the Hernandez Leclerc Conjecture.

\newpage
\chapter{Mathematical background}
\label{ch:math_backgr}
In this chapter I want to make the reader familiar with the structures that I am interested in and give the connection to the basic concept of a Lie Algebra.
\begin{definition}[Lie algebra]~\\
	A \textbf{Lie algebra} is a vector space $\mathfrak{g}$ over a field $k$ on which a (Lie bracket) multiplication
	\begin{equation*}
		\mathfrak{g}\times \mathfrak{g} \to \mathfrak{g}\\
		(x,y)\mapsto [x,y]
	\end{equation*}
is defined satisfying the axioms
\begin{enumerate}[label=\normalfont(\roman*)]
	\item $[\cdot,\cdot]$ is bilinear,
	\item $[x,x]=0$ for all $x\in \mathfrak{g}$,
	\item $[[x,y],z]+[[y,z],x]+[[z,x],y]=0$ for all $x,y,z\in \mathfrak{g}$.\label{Lie_ax:3}
\end{enumerate}
Axiom \ref{Lie_ax:3} is called the Jacobi identity and it can easily be memorized by realizing its equivalence to the statement that the mapping $ad_x$, given by the left multiplication $[x,\cdot]:\mathfrak{g}\to \mathfrak{g}$, is a derivative for all $x\in \mathfrak{g}$, i.e. it fulfils the Leibniz (product) rule with respect to the Lie bracket. $\odot$
\end{definition}
Clearly, I do not intend going into the details about the classification of finite-dimensional Lie algebras\footnote{We assume $k=\mathbb{C}$ if not stated otherwise.} here and rather refer the interested reader to the first part of the book \cite{Carter}. On the other side, it can be nice to see how Kac--Moody algebras are a bigger class of Lie algebras that includes all finite-dimensional semisimple ones and I want to provide the connection. Moreover, the beauty of the classification of the finite-dimensional semisimple Lie algebras stresses the importance of the loop construction which will be presented here. Furthermore, we will see that this construction translates to the second Drinfeld realisation in the $q$-deformed case. However, I prefer to not require too much prior knowledge, give the most definitions from scratch and collect the required facts.\\

Before we start with some more advanced definitions, let us introduce another basic concept (cf. \cite{Carter} Section 9.1). The \textbf{universal enveloping algebra} $U(\mathfrak{g})$ of a Lie algebra $\mathfrak{g}$. It is an associative algebra $U(\mathfrak{g})$ such that it's representation theory is the same as the representation theory of the Lie algebra $\mathfrak{g}$. We can build it from the \textbf{tensor algebra} $T(\mathfrak{g})$ of $\mathfrak{g}$ by dividing out a 2-sided ideal $J$. Since the Lie algebra $\mathfrak{g}$ is in particular a vector space, one can construct the tensor algebra $T(\mathfrak{g})$ from it. It is the free algebra of all possible tensor products of all possible vectors in $\mathfrak{g}$, so
\begin{equation*}
	T(\mathfrak{g}) = k\;\oplus\;\mathfrak{g}\;\oplus\; (\mathfrak{g}\otimes \mathfrak{g}) \;\oplus\; (\mathfrak{g}\otimes \mathfrak{g}\otimes \mathfrak{g}) \;\oplus\; \cdots
\end{equation*}
where $T^0(\mathfrak{g})=k$, $T^1(\mathfrak{g})=\mathfrak{g}$, $T^2(\mathfrak{g})=\mathfrak{g}\otimes \mathfrak{g}$, etc..
\begin{definition}[universal enveloping algebra]
	Let $\mathfrak{g}$ be a Lie algebra over a field $k$ and let $J$ be the 2-sided ideal of the tensor algebra $T(\mathfrak{g})$ generated by all elements of the form
	\begin{equation*}
		x\otimes y - y\otimes x - [x,y]\qquad\text{for }x,y\in \mathfrak{g}.
	\end{equation*}
	Then $U(\mathfrak{g}) := T(\mathfrak{g})/J$ is an associative algebra over $k$ called the \textbf{universal enveloping algebra} of $\mathfrak{g}$. 
\end{definition}
It's name is justified by the following universal property.
\begin{prop}
	Let $A$ be any associative algebra with $1$ over $k$ and $[A]$ the corresponding Lie algebra. Then for any Lie algebra homomorphism $\theta:\mathfrak{g}\to [A]$ there exists a unique associative algebra homomorphism $\phi: U(\mathfrak{g})\to A$ such that $\phi\circ\sigma = \theta$, where $\sigma$ is the canonical linear map $\mathfrak{g}\to T^1(\mathfrak{g}) \to T(\mathfrak{g}) \to U(\mathfrak{g})$.
\begin{equation*}
	\begin{tikzcd}[row sep=huge, column sep=huge]
		U(\mathfrak{g}) \arrow[dashrightarrow]{r}{\exists !} & A \\
		\mathfrak{g} \arrow{u}{\sigma} \arrow{ur}{\theta}
	\end{tikzcd}
\end{equation*}
$\odot$
\end{prop}
So for any unital associative algebra $A$ the Lie algebra homomorphisms $\mathfrak{g}\to [A]$ are in bijection with the unital algebra homomorphisms $U(\mathfrak{g})\to A$ through $\sigma: \mathfrak{g} \to U(\mathfrak{g})$. It means that the functor $U:\mathfrak{g}\to U(\mathfrak{g})$ is the left adjoint of the functor $[\cdot]:A\to[A]$.\\
	
Using this universal property, we can relate representations of the Lie algebra $\mathfrak{g}$ to representations of the associative algebra $U(\mathfrak{g})$. If $V$ is a vector space over $\mathbb{C}$, the set $\End(V)$ is a unital associative algebra and $[\End(V)]$ the corresponding Lie algebra. A representation of $\mathfrak{g}$ is a Lie algebra homomorphism $\mathfrak{g}\to[\End(V)]$ and a representation of $U(\mathfrak{g})$ is an associative algebra homomorphism $U(\mathfrak{g})\to\End(V)$. Then, the correspondence is as follows (see \cite{Carter} Proposition 9.3 for the proof).
\begin{prop}
	There is a bijective correspondence between representations $\theta:\mathfrak{g}\to[\End(V)]$ and representations $\phi:U(\mathfrak{g})\to \End(V)$. Corresponding representations are related by the condition
	\begin{equation*}
		\phi(\sigma(x))=\theta(x)\qquad\text{for all }x\in\mathfrak{g}.\quad\odot
	\end{equation*}
\end{prop}

Let us also write down the Poincaré-Birkhoff-Witt (PBW) basis theorem and its fundamental consequences. They are proven in Section 9.2 in the book \cite{Carter}.
\begin{thm}[Poincaré-Birkhoff-Witt] Let $\mathfrak{g}$ be a Lie algebra with basis $\{x_i|\,i\in I\}$. Let $<$ be a total order on the index set I. Let $\sigma:\mathfrak{g}\to U(\mathfrak{g})$ be the natural linear map as above. Let $\sigma(x_i)=y_i$. Then the elements
	\begin{equation*}
		y_{i_1}^{r_1}y_{i_2}^{r_2}\cdots y_{i_n}^{r_n}
	\end{equation*}
for all $n\geq 0$, $r_i\geq 0$ and all $i_1,1_2,\dots,i_n\in I$ with $i_1<i_2<\dots<i_n$ form a basis for $U(\mathfrak{g})$. $\odot$
\end{thm}
The theorem has the following consequences.
\begin{cors}\hspace{1em}
	\begin{enumerate}
		\item The map $\sigma:\mathfrak{g}\to U(\mathfrak{g})$ is injective
		\item The subspace $\sigma(\mathfrak{g})$ is a Lie subalgebra of $[U(\mathfrak{g})]$ isomorphic to $\mathfrak{g}$. Therefore, $\sigma$ identifies $\mathfrak{g}$ with a Lie subalgebra of $[U(\mathfrak{g})]$.
		\item $U(\mathfrak{g})$ has no zero-divisors. $\odot$
	\end{enumerate}	
\end{cors}

\section{Kac--Moody algebras}\label{sect:Kac--Moody}
To give an abstract introduction to (affine) Kac--Moody algebras, the first part of this section is supposed to be a short summary of Chapter 14 in the book \cite{Carter}. Then, we collect the most important facts about the classification of (affine) generalized Cartan matrices in Chapter 15 of \cite{Carter}, define the standard invariant bilinear form and discuss some properties of affine Kac--Moody algebras in Chapter 16 and 17 of \cite{Carter}. Finally, the loop construction of affine Kac--Moody algebras (Chapter 18 in \cite{Carter}), which provides a direct connection to the well known finite-dimensional semisimple Lie algebras, is presented. We will see later, that this connection is one of the main reasons for their importance in physics. Here, the proofs are omitted as soon as they are not considered essential for the understanding and the reader is referred to the book \cite{Carter} otherwise.\\

V.G. Kac and R.V. Moody independently initiated the study of certain Lie algebras $L(A)$ associated with a so-called generalised Cartan matrix $A$ in 1967. An $n\times n$ matrix $A=(A_{ij})$ is called a \textbf{generalized Cartan matrix} (GCM) if it satisfies the conditions
\begin{equation*}
	\begin{split}
		&A_{ii} = 2\quad \text{for } i=1,\dots,n\\
		&A_{ij}\in \mathbb{Z}\quad\text{and}\quad A_{ij}\leq 0 \quad \text{if } i\neq j\\
		&A_{ij}=0 \quad\text{implies}\quad A_{ji}=0.
	\end{split}
\end{equation*}
In particular, the Cartan matrix of any finite-dimensional simple Lie algebra is a generalized Cartan matrix. In this special case, when $A$ is a Cartan matrix, we shall see that the Lie algebra $L(A)$ constructed by Kac and Moody coincides with the finite-dimensional simple Lie algebra with Cartan matrix $A$. Though, the Lie algebra $L(A)$ is generally infinite-dimensional. We will call the Lie algebra $L(A)$ associated to a GCM $A$ the Kac--Moody algebra associated to $A$.
\subsection{The Kac--Moody algebras $L(A)$ and $L(A)'$}\label{subsect:Kac--Moody_abstract}
Let $A$ be an $n\times n$ matrix over $\mathbb{C}$. A \textbf{realization} of a square matrix $A$ is a triple $(H,\Pi,\Pi^\vee)$ such that
\begin{enumerate}
	\item[] $H$ is a finite-dimensional vector space over $\mathbb{C}$
	\item[] $\Pi^\vee=\{h_1,\dots,h_n\}$ is a linearly independent subset of $H$
	\item[] $\Pi = \{\alpha_1,\dots,\alpha_n\}$ is a linearly independent subset of $H^*$ \footnote{With an upper star we denote the dual vector space.}
	\item[] $\alpha_j(h_i) = A_{ij}$ for all $i,j$.
\end{enumerate}
\begin{prop}
	If $(H,\Pi,\Pi^\vee)$ is a realization of $A$ then $\operatorname{dim}H\geq 2n-\operatorname{rank}A.$ $\odot$
\end{prop}
\begin{definition}
	A \textbf{minimal  realization} of $A$ is a realization in which
	\begin{equation*}
		\operatorname{dim}H = 2n-\operatorname{rank}A.\quad\odot
	\end{equation*}
\end{definition}
\begin{prop}
	Any $n\times n$ matrix over $\mathbb{C}$ has a minimal realization and any two minimal realizations are isomorphic. $\odot$
\end{prop}
Let $l$ be the rank of $A$. Let $(H,\Pi,\Pi^\vee)$ be a minimal realization of $A$. Then we have
\begin{enumerate}
	\item[] $\operatorname{dim}H=2n-l$
	\item[] $\Pi^\vee =\{h_1,\dots,h_n\}\subset H,\quad \Pi=\{\alpha_1,\dots,\alpha_n\}\subset H^*$
	\item[] $\alpha_j(h_i) = A_{ij}.$
\end{enumerate}
Then, we define the Lie algebra $\tilde{L}(A)$ by generators and relations as follows.
\begin{definition}[the Lie algebra $\tilde{L}(A)$]~\\
	Let $X=\{e_1,\dots,e_n,f_1,\dots,f_n,\tilde{x}\text{ for all } x\in H\}$ and let $R$ be the set of Lie words \footnote{A Lie word is a term $\mathfrak{w}$ with $\mathfrak{w}=0$.}
	\begin{enumerate}
		\item[] $\tilde{x}-\lambda\tilde{y}-\mu\tilde{z}$ for all $x,y,z\in H$, $\lambda,\mu \in \mathbb{C}$ with $x=\lambda y+\mu z$
		\item[] $[\tilde{x},\tilde{y}]$ for all $x,y\in H$
		\item[] $[e_i,f_i]-\tilde{h_i}$ for $i=1,\dots,n$
		\item[] $[e_i,f_j]$ for all $i\neq j$
		\item[] $[\tilde{x},e_i]-\alpha_i(x)e_i$ for all $x\in H$ and $i=1,\dots,n$
		\item[] $[\tilde{x},f_i]+\alpha_i(x)f_i$ for all $x\in H$ and $i=1,\dots,n$,
	\end{enumerate}
	then $\tilde{L}(A):= L(X;\,R)$ is defined as the Lie algebra generated by the elements $X$ subject to the relations $R$. $\odot$
\end{definition}
\begin{lem}
	If a different minimal realization of $A$ is chosen the Lie algebra $\tilde{L}(A)$ is the same up to isomorphism. $\odot$
\end{lem}
\begin{prop}
	We have $\tilde{L}(A) = \tilde{N}^-\oplus\tilde{H}\oplus\tilde{N}$ a direct sum of subspaces, where $\tilde{N}$ is generated by $e_1,\dots,e_n$ freely and $\tilde{N}^-$ is generated by $f_1,\dots,f_n$ freely. $\odot$
\end{prop}
Let $Q$ be the subgroup of $H^*$ given by $Q=\{\alpha=k_1\alpha_1+\cdots+k_n\alpha_n\,;\,k_1,\dots,k_n\in \mathbb{Z}\}$, let $Q^+=\{\alpha\neq 0\in Q\,;\,k_i\geq0\text{ for all }i\}$ and $Q^-=\{\alpha\neq 0\in Q\,;\,k_i\leq0\text{ for all }i\}$. We call $Q$ (resp. $Q^+$, $Q^-$) the (positive, negative) \textbf{root lattice}. For each $\alpha\in Q$ we define the \textbf{root space}
\begin{equation*}
	\tilde{L}_\alpha = \{y\in \tilde{L}(A)\,;\;[\tilde{x},y]=\alpha(x)y\quad\text{for all }x\in H\}.
\end{equation*}
\begin{prop}\hspace{1em}
	\begin{enumerate}[label=\normalfont(\roman*)]
		\item $\tilde{L}(A) =\oplus_{\alpha\in Q}\tilde{L}_\alpha$.
		\item $\operatorname{dim}\tilde{L}_\alpha$ is finite for all $\alpha\in Q$.
		\item $\tilde{L}_0 = \tilde{H}$.
		\item If $\alpha\neq 0$ then $\tilde{L}_\alpha=O$ unless $\alpha\in Q^+$ or $\alpha\in Q^-$.\footnote{ $O$ the trivial zero-dimensional Lie algebra $\{0\}$.}
		\item $[\tilde{L}_\alpha,\tilde{L}_\beta]\subset \tilde{L}_{\alpha+\beta}$ for all $\alpha,\beta\in Q$.
	\end{enumerate}\vspace{.5em}
	In particular $\operatorname{dim}\tilde{L}_{k\alpha_i} = \delta_{k,1}+\delta_{k,-1}$ for all $k\in\mathbb{Z}\backslash\{0\}$, $i=1,\dots,n$. $\odot$
\end{prop}
\begin{prop}
	The algebra $\tilde{L}(A)$ contains a unique ideal $I$ maximal with respect to $I\cap\tilde{H} = O$. $\odot$
\end{prop}
\begin{definition}[the Kac--Moody algebra $L(A)$]
	Let $A$ be a GCM. Then the Kac--Moody algebra $L(A)$ with GCM $A$ is defined as
	\begin{equation*}
		L(A)=\tilde{L}(A)/I.\;\, \odot
	\end{equation*}
\end{definition}
As before, using the natural homomorphism $\tilde{L}(A)\to L(A)$, we obtain the decomposition
\begin{equation*}
	L(A)= N^-\oplus H\oplus N.
\end{equation*}
Vice versa, on can check if a Lie algebra $\mathfrak{g}$ is isomorphic to $L(A)$ as follows.
\begin{prop}\label{prop:isoto_g(A)}
	Suppose $A$ is an $n\times n$ GCM. Let $\mathfrak{g}$ be a Lie algebra over $\mathbb{C}$ and $H$ a finite-dimensional abelian subalgebra of $\mathfrak{g}$ with linear independent subsets $\Pi^\vee\subset H$ and $\Pi\subset H^*$ such that $(H,\Pi,\Pi^\vee)$ is a minimal realization of $A$ as above. Suppose in addition that $e_1,\dots,e_n$, $f_1,\dots,f_n$ are elements of $\mathfrak{g}$ satisfying
	\begin{enumerate}
		\item[] $[e_i,f_j]=\delta_{i,j}h_i$,
		\item[] $[x,e_i] = \alpha_i(x)e_i$ for $x\in H$,
		\item[] $[x,f_i] = -\alpha_i(x)f_i$ for $x\in H$
	\end{enumerate}
	and that $\mathfrak{g}$ is generated by $H$, $e_1,\dots,e_n$, $f_1,\dots,f_n$. If $\mathfrak{g}$ has no non-zero ideal $J$ with $J\cap H=O$, then $\mathfrak{g}$ is isomorphic to the Kac--Moody algebra $L(A)$, $\mathfrak{g} \cong L(A)$. $\odot$
\end{prop}
\begin{cor}
	If $A$ is a Cartan matrix then $L(A)$ is the finite-dimensional semisimple Lie algebra with Cartan matrix $A$. $\odot$
\end{cor}
For each $\alpha\in Q$ we define the \textbf{root space}
\begin{equation*}
	L_\alpha = \{y\in L(A)\,;\;[x,y]=\alpha(x)y\quad\text{for all }x\in H\}
\end{equation*}
as above and analogously we can state the following.
\begin{prop}\hspace{1em}
	\begin{enumerate}[label=\normalfont(\roman*)]
		\item $L(A) =\oplus_{\alpha\in Q}L_\alpha$.
		\item $\operatorname{dim}L_\alpha$ is finite for all $\alpha\in Q$.
		\item $L_0 = \tilde{H}$.
		\item If $\alpha\neq 0$ then $L_\alpha=O$ unless $\alpha\in Q^+$ or $\alpha\in Q^-$.
		\item $[L_\alpha,L_\beta]\subset L_{\alpha+\beta}$ for all $\alpha,\beta\in Q$.
	\end{enumerate}\vspace{.5em}
	In particular $\operatorname{dim}L_{k\alpha_i} = \delta_{k,1}+\delta_{k,-1}$ for all $k\in\mathbb{Z}\backslash\{0\}$, $i=1,\dots,n$. $\odot$
\end{prop}
\begin{definition}\label{def:Cartan,neg_pos_fund_mult_roots}
	$H$ will be called a \textbf{Cartan subalgebra} of $L(A)$. An element $\alpha\in H^*$ is called a \textbf{root} of $L(A)$ if $\alpha\neq0$ and $L_\alpha\neq O$. Every root lies in $ Q^+$ or $ Q^-$. The roots in $ Q^+$ are called \textbf{positive roots} and those in $ Q^-$ \textbf{negative roots}. The dimension of $L_\alpha$ is called the \textbf{multiplicity} of $\alpha$ and the $\alpha_1,\dots,\alpha_n$ are called the \textbf{fundamental roots}.
\end{definition}
Before we discuss the different possible GCM, let us give a simple but important definition.
\begin{definition}[the Kac--Moody (sub)algebra $L(A)'$]
	We denote by $L(A)'$ the subalgebra of $L(A)$ generated by $e_1,\dots,e_n$, $f_1,\dots,f_n$. $\odot$
\end{definition}
\begin{prop}\hspace{1em}
	\begin{enumerate}[label=\normalfont(\roman*)]
		\item $L_\alpha$ lies in $L(A)'$ for each root $\alpha$ of $L(A)$.
		\item $L(A)'=(H\cap L(A)')\oplus\sum_{\alpha\neq0}L_\alpha$.
		\item $L(A)'=[L(A),L(A)]$. $\odot$
	\end{enumerate}
\end{prop}
\subsection{GCM of finite and affine type}\label{subsect:GCM_finite_affine}
The structure of the Kac--Moody algebra $L(A)$ depends crucially on the GCM $A$. One can proof that there are basically three different cases. The GCM can be of \textbf{finite}, \textbf{affine} or \textbf{indefinite type}. However, GCMs of finite and affine type will be most important for us. We will see, that GCMs are symmetrisable in these cases and that there is a full classification of all the finite and affine GCMs in terms of Dynkin diagrams. Let us collect the important facts about the classification of GCMs in Chapter 15 of \cite{Carter}. The proofs can be found in the book.\\

Two $n\times n$ GCMs $A$, $A'$ are called \textbf{equivalent} if there is a permutation $\sigma\in S_n$ \footnote{We define $S_n$ to be the group of permutations of $n$ elements $1,\dots,n$.} such that
\begin{equation*}
	A'_{ij} = A_{\sigma(i)\sigma(j)}\quad\text{for all }i,j.
\end{equation*}
A GCM is called \textbf{indecomposable} if it is not equivalent to a diagonal sum
\begin{equation*}
	\begin{pmatrix}
		A_1 & O \\
		O & A_2
	\end{pmatrix}
\end{equation*}
of smaller GCMS $A_1$, $A_2$.\footnote{ Here the symbol $O$ is used for a matrix with $0$'s in all its entries.} Obviously, if $A$ is a GCM so is its transpose $A^t$ and it is indecomposable only if $A^t$ is indecomposable.
Let us define three types of GCM. For a vector $v=(v_1,\dots,v_n)$ in $\mathbb{R}^n$ we write $v\geq 0$ if $v_i\geq 0$ for each $i$, and $v>0$ if $v_i>0$ for each $i$.
\begin{definition}
	A GCM has
	\begin{itemize}
		\item[] \textbf{finite type} if
		\begin{enumerate}[label=\normalfont(\roman*)]
			\item $\det A\neq 0$,
			\item there exists $u>0$ with $Au>0$,
			\item $Au\geq 0$ implies $u>0$ or $u=0$.
		\end{enumerate}
		\item[] \textbf{affine type} if
		\begin{enumerate}[label=\normalfont(\roman*)]
			\item $\operatorname{corank} A = 1$ (i.e. $\operatorname{rank} A = n-1$),
			\item there exists $u>0$ such that $Au =0$,
			\item $Au\geq0$ implies $Au=0$.
		\end{enumerate}
		\item[] \textbf{indefinite type} if
		\begin{enumerate}[label=\normalfont(\roman*)]
			\item there exists $u>0$ such that $Au<0$,
			\item $Au\geq0$ and $u\geq0$ imply $u=0$. $\odot$
		\end{enumerate}
	\end{itemize}
\end{definition}
In fact, it can be shown that these are the only possible cases for an indecomposable GCM $A$ and that the type of $A^t$ is the same as the type of $A$ (see \cite{Carter} Section 15.1).\\

Let $A=(A_{ij})$ be a GCM with $i,j\in\{1,\dots,n\}$ and let $J$ be a subset of $\{i,\dots,n\}$. Let $A_J=(A_{ij}),i,j\in J$. Then $A_J$ is also a GCM, called a \textbf{principal minor} of $A$.
\begin{definition}
	A GCM $A$ is called \textbf{symmetrisable} if there exists a non-singular diagonal matrix $\hat{D}$ and a symmetric matrix $\hat{B}$ such that $A=\hat{D}\hat{B}$. $\odot$
\end{definition}
The case when the GCM is symmetrisable provides a helpful simplification for the discussion of the finite and affine type because of the following Theorem.
\begin{thm}
	Let $A$ be an indecomposable GCM of finite or affine type. Then $A$ is symmetrisable. $\odot$
\end{thm}
\begin{lem}
	Let $A$ be a GCM. Then $A$ is symmetrisable iff
	\begin{equation*}
		A_{i_1i_2}A_{i_2i_3}\dots A_{i_ki_1}=A_{i_2i_1}A_{i_3i_2}\dots A_{i_1i_k}
	\end{equation*}
	for all $i_1,i_2,\dots,i_k\in\{1,\dots,n\}$. $\odot$
\end{lem}
\begin{cor}
	Let $A$ be a symmetrisable and indecomposable GCM. Then $A$ can be expressed in the form $A=\hat{D}\hat{B}$ where $\hat{D}=\operatorname{diag}(\hat{d}_1,\dots,\hat{d}_n)$, $B$ is symmetric, with $\hat{d}_1,\dots,\hat{d}_n >0$ in $\mathbb{Z}$ and $\hat{B}_{ij}\in\mathbb{Q}$. Moreover, $\hat{D}$ is determined up to a scalar multiple. $\odot$
\end{cor}
We obtain the following simplification for the possible cases if indecomposable GCM.
\begin{thm}
	Let $A$ be an indecomposable GCM. Then:
	\begin{enumerate}[label=\normalfont(\alph*)]
		\item $A$ has finite type iff all its principal minors have positive determinant.
		\item A has affine type iff $\det A=0$ and all proper principal minors have positive determinant.
		\item A has indefinite type otherwise. $\odot$
	\end{enumerate}
\end{thm}
In particular one can check that an indecomposable GCM $A$ has finite type iff $A$ is a Cartan matrix. Those are classified by the standard list of connected Dynkin diagrams in Section 6.4 in the book \cite{Carter}. Similarly, if the GCM $A$ is of affine type, there is a full classification in terms of the affine list of Dynkin diagrams in Section 15.3 in the book \cite{Carter}.
\subsection{The standard invariant bilinear form}\label{subsect:standard_inv_bil}
Let us discuss the problem of defining an invariant bilinear form (cf. \cite{Carter} Section 16.1). When $A$ is a Cartan matrix, the Kac--Moody algebra $L(A)$ is the corresponding finite-dimensional Lie Algebra $\mathfrak{g}$ with Cartan matrix $A$. Thus, it has a non degenerate symmetric bilinear form
\begin{equation*}
	\langle ,\rangle : L(A) \times L(A) \to \mathbb{C}
\end{equation*}
which is invariant, i.e.
\begin{equation*}
	\langle [x,y],z\rangle = \langle x,[y,z]\rangle\qquad \text{ for }x,y,z\in L(A).
\end{equation*}
The Killing form defined by
\begin{equation*}
	\langle x,y\rangle \coloneq \tr(\operatorname{ad}_x\circ\operatorname{ad}_y)
\end{equation*}
has these properties. Though, since $L(A)$ is infinite-dimensional whenever the GCM $A$ is not finite, one cannot use the definition of the Killing form in general. However, when $A$ is symmetrisable, there is another way to define it. So let's suppose that $A$ is symmetrisable, then $A=\hat{D}\hat{B}$ with $\hat{D}$ diagonal and $\hat{B}$ symmetric. Let $\hat{D}= \operatorname{diag}(\hat{d}_1,\dots,\hat{d}_n)$ and suppose that $(H,\Pi,\Pi^\vee)$ is a minimal realisation of $A$, i.e. $\Pi^\vee=\{h_1,\dots,h_n\}$ is a linearly independent subset of $H$, $\Pi = \{\alpha_1,\dots,\alpha_n\}$ is a linearly independent subset of $H^*$, $\alpha_j(h_i)=A_{ij}$ and $\dim H = 2n-l$ where $l=\operatorname{rank}A$. Furthermore, let $H'$ be the subspace of $H$ spanned by $h_1,\dots,h_n$ and let $H''$ be a complementary subspace of $H'$ in $H$. Then
\begin{equation*}
	H = H'\oplus H''\qquad \dim H' = n, \dim H' = n-l
\end{equation*}
and we can define a bilinear form $\langle,\rangle: H\times H \to \mathbb{C}$ on $H$ by the rules:
\begin{equation*}
	\begin{split}
		&\langle h_i,h_j\rangle = \hat{d}_i\hat{d}_j\hat{B}_{ij}\qquad i,j = 1,\dots,n\\
		&\langle h_i,x\rangle = \langle x,h_i\rangle = \hat{d}_i\alpha_i(x)\qquad \text{for }x\in H''\\
		&\langle x,y\rangle = 0\qquad \text{for } x,y\in H''.
	\end{split}
\end{equation*}
\begin{prop}
	This form on $H$ is non degenerate. $\odot$
\end{prop}
\begin{thm}\label{thm:standard_inv}
	Suppose $A$ is a symmetrisable GCM. Then the Kac--Moody algebra $L(A)$ has a non-degenerate symmetric invariant bilinear form. $\odot$
\end{thm}
The full proof of the theorem can be found in the book \cite{Carter}. The idea is as follows. We have $L(A)=\underset{\alpha\in  Q}{\bigoplus}L_\alpha$. For $\alpha= m_1\alpha_1+\cdots+m_n\alpha_n\in Q$ we define the height of $\alpha$ by $\operatorname{ht}\alpha = m_1+\cdots+m_n$. Then
\begin{equation*}
	L(A)=\bigoplus_{i\in\mathbb{Z}} L_i
\end{equation*}
where $L_i$ is the direct sum of all $L_\alpha$ with $\operatorname{ht}\alpha = i$. Since $[L_\alpha,L_\beta]\subset L_{\alpha+\beta}$ we have $[L_i,L_j]\subset L_{i+j}$. In this way, $L(A)$ may be considered as a $\mathbb{Z}$-graded Lie algebra. For each $r\geq 0$, define
\begin{equation*}
	L(r)=\bigoplus_{-r\leq i\leq r} L_i.
\end{equation*}
Then we have $\bigcup_{r\geq 0} L(r) = L(A)$ such that
\begin{equation*}
	H = L(0)\subset L(1) \subset L(2)\subset \cdots.
\end{equation*}
Now, since a symmetric bilinear form is already defined on $H=L(0)$, the idea is to extend this definition to $L(r)$ by induction on $r$ and thus defining it on $L(A)$.
Moreover, the proof of the theorem shows the following.
\begin{cor}\label{cor:unique_bil_restr_to_H}
	Any symmetric invariant bilinear form on $L(A)$ is uniquely determined by its restriction to $H$. $\odot$
\end{cor}
\begin{definition}[the standard invariant form]
	We call the form constructed in Theorem \ref{thm:standard_inv} the \textbf{standard invariant form} on $L(A)$.\footnote{ Note that it depends on the choice of $\hat{D}=\operatorname{diag}(\hat{d}_1,\dots,\hat{d}_n)$, which is unique up to a scalar multiple.} $\odot$
\end{definition}
Since the form $\langle,\rangle$ is non-degenerate on $H$ it determines a bijection $H^*\to H$ given by $\alpha\mapsto h'_\alpha$ such that
\begin{equation*}
	\langle h'_\alpha, h\rangle = \alpha(h)\qquad\text{for all }h\in H.
\end{equation*}
\begin{cor}\label{cor:prop_inv_bil_form}\hspace{1em}
	\begin{enumerate}[label=\normalfont(\roman*)]
		\item For each $i\in\mathbb{Z}$ the pairing $L_i\times L_{-i}\to\mathbb{C}$ given by $(x,y)\mapsto \langle x,y\rangle$ is non-degenerate.
		\item $\langle L_\alpha,L_\beta\rangle = 0$ unless $\alpha+\beta = 0$.
		\item The pairing $L_\alpha\times L_{-\alpha}\to \mathbb{C}$ given by $(x,y)\mapsto\langle x,y\rangle$ is non-degenerate.
		\item Suppose $x\in L_\alpha, y\in L_{-\alpha}$, then $[x,y] = \langle x,y\rangle h'_\alpha$.
		\item For each $x\in L_\alpha$ with $x\neq 0$ there exists $y\in L_{-\alpha}$ with $[x,y]\neq 0$. $\odot$
	\end{enumerate}	
\end{cor}
In fact a non-degenerate symmetric invariant bilinear form on $L(A)$ is uniquely determined on the subalgebra $L(A)'$ up to a non-zero constant as follows.
\begin{prop}
	Let $A$ be an indecomposable symmetrisable GCM and $\{,\}$ a non-degenerate symmetric invariant bilinear form on the Kac--Moody algerba $L(A)$. Then there exists a non-zero constant factor $\xi\in\mathbb{C}$ such that
	\begin{equation*}
		\{x,y\} = \xi \langle x,y\rangle\text{ for all }x,y\in L(A)'.\quad \odot
	\end{equation*}
\end{prop}
We now have the following corollaries.
\begin{cor}
	$L(A)'\cap H$ is the subspace of $H$ spanned by $h_1,\dots,h_n$. $\odot$
\end{cor}
\begin{cor}
	Any non-degenerate symmetric invariant bilinear form on a finite-dimensional simple Lie algebra is a constant multiple of the Killing form. $\odot$
\end{cor}
\subsection{Kac--Moody algebras of affine type}\label{subsect:aff_Kac--Moody}
Before we discuss the loop construction for untwisted affine Kac--Moody algebras, let us fix some notions which are used in the affine case (cf. \cite{Carter} Section 17.1). So let $A$ be an $n\times n$ GCM of affine type and consider the affine Lie algebra $L=L(A)$.\footnote{We will also say 'affine Cartan matrix' in the following.} Then $n=l+1$ with $l=\operatorname{rank}(A)$ and we shall number the rows and and columns of $A$ by the integers $0,1,\dots,l$. There exists unique vectors $(a_0,a_1,\dots,a_l)$, $(c_0,c_1,\dots,c_l)$ whose coordinates are positive integers with no common factor such that
\begin{equation*}
	A\begin{pmatrix}
		a_0	\\
		a_1	\\
		\vdots	\\
		a_l	
	\end{pmatrix}
	= 
	\begin{pmatrix}
		0	\\
		0	\\
		\vdots	\\
		0	
	\end{pmatrix} \quad\text{ and }\quad (c_0,c_1,\dots,c_l)A=0.\footnotemark
\end{equation*}\footnotetext{ In fact, the vector $(c_0,c_1,\dots,c_l)$ is the same as the vector $(a_0,a_1,\dots,a_l)$ for the transpose $A^t$.}
Let $(H,\Pi,\Pi^\vee)$ be a minimal realisation of $A$. Then $\dim H = 2n-l = l+2$. $\Pi^\vee = \{h_0,h_1,\dots,h_l\}$ is a linearly independent subset of $H$ and $\Pi = \{\alpha_0,\alpha_1,\dots,\alpha_l\}$ is a linearly independent subset of $H^*$. There exists a so-called \textbf{scaling element} $d\in H$ such that
\begin{equation*}
	\alpha_0(d)=1\qquad\alpha_i(d)=0\qquad\text{for }i=1,\dots,l.
\end{equation*}
\begin{prop}$h_0,h_1,\dots,h_l,d$ is a basis of $H$. $\odot$
\end{prop}
Then we can define an element $\gamma\in H^*$ uniquely determined by
\begin{equation*}
	\gamma(h_0)=1\qquad\gamma(h_i)=0\qquad\text{for } i=1,\dots,l\qquad\gamma(d)=0.
\end{equation*}
\begin{prop}$\alpha_0,\alpha_1,\dots,\alpha_l,\gamma$ is a basis of $H^*$. $\odot$
\end{prop}
\begin{prop}
	We have $A = \hat{D}\hat{B}$ where $\hat{D}=\operatorname{diag}(\hat{d}_0,\hat{d}_1,\dots,\hat{d}_l)$, $B$ is symmetric and $\hat{d}_i = a_i/c_i$. $\odot$
\end{prop}
The \textbf{standard invariant form} on $H$ with respect to $\hat{d}_i = a_i/c_i$ is then defined by
\begin{equation*}
	\begin{split}
		&\langle h_i,h_j\rangle = \hat{d}_i\hat{d}_jB_{ij} = a_jc_j^{-1}A_{ij}\qquad\text{for }i,j=0,1,\dots,l\\
		&\langle h_0,d\rangle = \hat{d}_0\alpha_0(d) = a_0\\
		&\langle h_i,d\rangle = 0\qquad\text{for }i=1,\dots,l\\
		&\langle d,d\rangle=0.
	\end{split}
\end{equation*}
It defines a bijection from $H^*$ to $H$ by $\alpha\mapsto h_\alpha'$ where $\alpha(h) = \langle h_\alpha',h\rangle$ for all $h\in H$.
\begin{prop}
	Using this bijection $H^*\to H$, $h_i\in H$ corresponds to $a_ic_i^{-1}\alpha_i\in H^*$ for $i=0,1,\dots,l$ and $d\in H$ corresponds to $a_0\gamma\in H^*$. $\odot$
\end{prop}
Moreover, we can use this bijection to transfer the standard bilinear form from $H$ to $H^*$. On $H^*$ it is given by
\begin{equation*}
	\begin{split}
		&\langle \alpha_i,\alpha_j\rangle = a_i^{-1}c_iA_{ij} =\hat{B}_{ij}\qquad\text{for }i,j=0,1,\dots,l\\
		&\langle \alpha_0,\gamma\rangle = a_0^{-1}\\
		&\langle \alpha_i,\gamma\rangle = 0\qquad\text{for }i=1,\dots,l\\
		&\langle \gamma,\gamma\rangle=0.
	\end{split}	
\end{equation*}
In particular, we have that
\begin{equation*}
	A_{ij}=\frac{2\langle\alpha_i,\alpha_j\rangle}{\langle\alpha_i,\alpha_i\rangle} \qquad \text{and}\qquad \frac{2\alpha_i}{\langle\alpha_i,\alpha_i\rangle} \mapsto h_i.\footnotemark
\end{equation*}
\footnotetext{$h_i\eqcolon h_{\alpha_i}$ is called the coroot of $\alpha_i$ and generally $h_\alpha := \frac{2h_\alpha'}{\langle\alpha,\alpha\rangle}$ the coroot of $\alpha\in H^*$.}
At last, we can define an element $c\in H$ by $c=\sum_{i=0}^{l}c_i h_i$ and, by making use of the bijection, the element $\delta$ through $c\mapsto \delta = \sum_{i=0}^{l}a_i \alpha_i$. $\delta$ is called the \textbf{basic imaginary root}.\footnote{ An imaginary root is a root which cannot be mapped onto any $\alpha_i\in \Pi$ by the action of the Weyl group (cf. \cite{Carter} Chapter 16.2-3). The roots $k\delta, k\in\mathbb{Z}$ are all the imaginary roots in the affine case.} The element $c$ has the following properties.
\begin{prop}
	The element $c$ lies in the centre of $L(A)$. Furthermore, the centre consists of all scalar multiples of $c$, i.e. it is one-dimensional. $\odot$
\end{prop}
$c$ is therefore called the \textbf{canonical central element} of $L(A)$.
Let us also write down the following important definition (cf. \cite{Carter} Section 20.1).
\begin{definition}
	The number $h=a_0+a_1+\dots+a_l$ is called the \textbf{Coxeter number} of $L$. The number $h^\vee=c_0+c_1+\dots+c_l$ is called the \textbf{dual Coxeter number} of $L$. $\odot$
\end{definition}
\subsection{The loop construction for untwisted affine Kac--Moody algebras}\label{subsect:loop_constr_untw_Kac--Moody}
We are now in position to explain a construction which describes the affine Kac--Moody algebras in terms of finite-dimensional simple Lie algebras. In general, there is such a construction for all affine Kac--Moody algebras up to twisting with certain outer automorphisms. The full description is given in Chapter 18 of the book \cite{Carter}. However, we will focus on the untwisted case which applies to our considerations (cf. \cite{Carter} Chapter 18.1-2).

Let $A^0 =(A^0_{ij})$, $i,j=1,\dots,l$, be an indecomposable Cartan matrix of finite type. Let $L^0 = L(A^0)$ be the finite-dimensional simple Lie algebra with Cartan matrix $A^0$. We may construct an $(l+1)\times(l+1)$ affine Cartan matrix $A$ from $A^0$ by adding an additional row and column, indexed by $0$ as follows. Let $\theta = \sum_{i=1}^{l}a_i\alpha_i$ be the highest root of $L^0$ and $h_\theta = \sum_{i=1}^{l}c_ih_i$ be the coroot of $\theta$. We define $A$ by
\begin{equation*}
	\begin{split}
		&A_{ij} = A^0_{ij}\qquad\text{if }i,j\in\{1,\dots,l\}\\
		&A_{i0} = -\sum_{j=1}^{l}a_jA^0_{ij}\qquad\text{if }i\in\{1,\dots,l\}\\
		&A_{0j} = -\sum_{i=1}^{l}c_iA^0_{ij}\qquad\text{if }j\in\{1,\dots,l\}\\
		&A_{00} = 2.
	\end{split}
\end{equation*}
Then we have $c_0=a_0 =1$ and $A$ is determined by $A^0$ and the relations
\begin{equation*}
	A\begin{pmatrix}
		a_0	\\
		a_1	\\
		\vdots	\\
		a_l	
	\end{pmatrix}
	= 
	\begin{pmatrix}
		0	\\
		0	\\
		\vdots	\\
		0	
	\end{pmatrix} \quad\text{ and }\quad (c_0,c_1,\dots,c_l)A=0.
\end{equation*}
\begin{prop}
	The type of the affine Cartan matrix $A$ is as follows.
	\begin{equation*}
		\begin{split}
			&\text{Type of }A^0:A_l,B_l,C_l,D_l,E_6,E_7,E_8,F_4,G_2.\\
			&\text{Type of }A\hspace{.4em}:\tilde{A}_l,\tilde{B}_l,\tilde{C}_l,\tilde{D}_l,\tilde{E}_6,\tilde{E}_7,\tilde{E}_8,\tilde{F}_4,\tilde{G}_2.\quad\odot
		\end{split}
	\end{equation*}
\end{prop}
\begin{definition}
	An affine Cartan matrix $A$ is called of \textbf{untwisted type} if it is one of
	\begin{equation*}
		\tilde{A}_l,\tilde{B}_l,\tilde{C}_l,\tilde{D}_l,\tilde{E}_6,\tilde{E}_7,\tilde{E}_8,\tilde{F}_4,\tilde{G}_2.\quad\odot
	\end{equation*}
\end{definition}
Since any affine Cartan matrix $A$ of untwisted type can be obtained as above by the addition of an extra row and column from a Cartan matrix $A^0$ of finite type, the question whether the affine Kac--Moody algebra $L(A)$ can be constructed in some way from the finite-dimensional simple Lie algebra $L^0 = L(A^0)$ may be natural. Indeed, there is a method of doing this which will be described in the following.

Let $\mathbb{C}[t,t^{-1}]$ be the ring of Laurent polynomials $\sum_{i\in\mathbb{Z}}\zeta_it^i$ for $\zeta_i\in\mathbb{C}$ with finitely many $\zeta_i\neq 0$. Let
\begin{equation*}
	\mathfrak{L}(L^0) = \mathbb{C}[t,t^{-1}]\otimes_\mathbb{C}L^0.
\end{equation*}
Then $\mathfrak{L}(L^0)$ can be provided with the structure of a Lie algebra in a unique way satisfying
\begin{equation*}
	[p\otimes x,q\otimes y]= pq\otimes[x,y]\qquad\text{for }p,q\in\mathbb{C}[t,t^{-1}],\,x,y\in L^0.
\end{equation*}
Then $\mathfrak{L}(L^0)$ is called the \textbf{loop algebra} of $L^0$.
We construct a one-dimensional central extension using the following Lemma.
\begin{lem}[the 2-cocycle condition]
	Let $\mathfrak{g}$ be a Lie algebra over $\mathbb{C}$ and $\tilde{\mathfrak{g}}$ be the set of elements $x+\lambda c$ with $x\in \mathfrak{g}$ and $\lambda \in \mathbb{C}$. Let $\kappa:\mathfrak{g}\times \mathfrak{g} \to \mathbb{C}$ be a bilinear map satisfying
	\begin{equation*}
		\begin{split}
			&\kappa(x,y) = -\kappa(y,x)\qquad\text{for }x,y\in \mathfrak{g}\\
			&\kappa([x,y],z)+\kappa([y,z],x)+\kappa([z,x],y) = 0\qquad\text{for }x,y,z\in \mathfrak{g}.
		\end{split}
	\end{equation*}
	($\kappa$ is called a 2-cocycle on $\mathfrak{g}$.) Then the Lie multiplication
	\begin{equation*}
		[x+\lambda c,y+\mu c] = [x,y]+\kappa(x,y)c
	\end{equation*}
	makes $\tilde{\mathfrak{g}}$ into a Lie algebra. $\odot$
\end{lem}
Note that $\tilde{\mathfrak{g}}$ is a one-dimensional central extension of $\mathfrak{g}$ in the sense that there is a surjective homomorphism $\theta:\tilde{\mathfrak{g}}\to \mathfrak{g}$ given by $\theta(x+\lambda c)=x$, such that $\dim(\ker \theta) =1$ and $\ker \theta$ lies in the centre of $\tilde{\mathfrak{g}}$.
We apply this idea to construct a one-dimensional central extension of $\mathfrak{L}(L^0)$ by taking a 2-cocycle on $\mathfrak{L}(L^0)$. Let $\langle,\rangle$ be the invariant bilinear form on $L^0$ satisfying $\langle h_\theta,h_\theta\rangle = 2$. Since an invariant bilinear form is determined up to a scalar multiple on $L^0$, this condition determines it uniquely. In fact, this form on $L^0$ is the restriction to $L^0$ of the standard invariant bilinear form on $L = L(A)$ as defined in Subsection \ref{subsect:aff_Kac--Moody}, since for the standard form we have $\langle\theta,\theta\rangle =2$, hence
\begin{equation*}
	\langle h_\theta,h_\theta\rangle = \langle \frac{2\theta}{\langle\theta,\theta\rangle},\frac{2\theta}{\langle\theta,\theta\rangle}\rangle= \frac{4}{\langle\theta,\theta\rangle} =2.\footnotemark
\end{equation*}
\footnotetext{This is explained in Proposition 17.18 in the book \cite{Carter}.}We define a bilinear form
\begin{equation*}
	\langle ,\rangle_t:\mathfrak{L}(L^0)\times\mathfrak{L}(L^0)\to\mathbb{C}[t,t^{-1}]
\end{equation*}
by $\langle p\otimes x,q\otimes y\rangle_t = pq\langle x,y\rangle$ and the residue function
\begin{equation*}
	\operatorname{Res} : \mathbb{C}[t,t^{-1}]\to \mathbb{C}
\end{equation*}
by $\operatorname{Res}(\sum\zeta_it^i) =\zeta_{-1}$. Then we construct a 2-cocycle as follows.
\begin{lem}
	The function $\kappa:\mathfrak{L}(L^0)\times\mathfrak{L}(L^0)\to \mathbb{C}$ defined by
	\begin{equation*}
		\kappa(a,b)=\operatorname{Res}\langle \frac{\operatorname{d}a}{\operatorname{d}t},b\rangle_t
	\end{equation*}
	is a 2-cocycle on $\mathfrak{L}(L^0)$. $\odot$
\end{lem}
We construct the one-dimensional central extension $\tilde{\mathfrak{L}}(L^0)$ of $\mathfrak{L}(L^0)$ given by
\begin{equation*}
	\tilde{\mathfrak{L}}(L^0) = \mathfrak{L}(L^0)\oplus\mathbb{C}c
\end{equation*}
whose Lie multiplication is given by
\begin{equation*}
	[a+\lambda c,b+\mu c]=[a,b]_0+\kappa(a,b)c
\end{equation*}
where $a,b\in\mathfrak{L}(L^0)$ and $[a,b]_0$ is the Lie product of $a$ with $b$ in $\mathfrak{L}(L^0)$.
Let us also adjoin to $\tilde{\mathfrak{L}}(L^0)$ an element $d$ which acts on $\tilde{\mathfrak{L}}(L^0)$ as a derivation.
\begin{lem}
	The map $\Delta:\tilde{\mathfrak{L}}(L^0)\to\tilde{\mathfrak{L}}(L^0)$ given by $\Delta(a+\lambda c)=t\frac{\operatorname{d}a}{\operatorname{d}t}$ for $a\in\mathfrak{L}(L^0), \lambda\in\mathbb{C}$, is a derivation. $\odot$
\end{lem}
Now, we define $\hat{\mathfrak{L}}(L^0)$ by
\begin{equation*}
	\hat{\mathfrak{L}}(L^0)=\tilde{\mathfrak{L}}(L^0)\oplus\mathbb{C}d
\end{equation*}
and provide it with the structure of a Lie algebra by defining
\begin{equation*}
	[a+\lambda d,b+\mu d]=[a,b]+\lambda\Delta(b)-\mu\Delta(a).
\end{equation*}
In particular we have
\begin{equation*}
	\begin{split}
		[(t^i\otimes x&)+\lambda  c+\mu d,(t^j\otimes y)+\lambda'c+\mu'd]\\
		=&(t^{i+j}\otimes [x,y])+\mu j(t^j\otimes y)-\mu'i(t^i\otimes x)+\delta_{i,-j}i\langle x,y\rangle c
	\end{split}	
\end{equation*}
for $x,y\in L^0$, $\lambda,\mu,\lambda',\mu'\in\mathbb{C}$.

The claim is now that $\hat{\mathfrak{L}}(L^0)$ is isomorphic to the affine Kac--Moody algebra $L(A)$. So $L(A)$ can be constructed from $L^0=L(A^0)$ by first considering the loop algebra $\mathfrak{L}(L^0)$, then constructing a one-dimensional central extension and finally adjoining an element $d$ which acts as a derivation. We can now write down the following Theorem.
\begin{thm}[the loop construction]
	Let $L^0 = L(A^0)$ be a finite-dimensional simple Lie algebra and let $A$ be the corresponding untwisted affine Cartan matrix obtained from $A^0$ as explained above. Then $L(A)$ is isomorphic to $\hat{\mathfrak{L}}(L^0)$. $\odot$
\end{thm}
The idea of the proof is to define elements $e_0,e_1,\dots,e_l;\,f_0,f_1,\dots,f_l;\,h_0,h_1,\dots,h_l$ in $\hat{\mathfrak{L}}(L^0)$ and use Proposition \ref{prop:isoto_g(A)} to show that it is isomorphic to $L(A)$. Instead of writing down the entire proof which can be found in Section 18.2 of the book \cite{Carter}, let us make a few comments on the isomorphism.

So let $E_1,\dots,E_l\,;\;F_1,\dots,F_l\,;\;H_1,\dots,H_l$ be the corresponding (Chevalley) generators of $L^0$.\footnote{ We use capital letters here to not confuse them with the Chevalley generators of $L(A)$.} We define
\begin{equation*}
	e_i\coloneq1\otimes E_i,\quad f_i\coloneq1\otimes F_i,\quad h_i\coloneq 1\otimes H_i
\end{equation*}
for $i=1,\dots,l$ such that $[e_i,f_i] = h_i$ for all $i$. To define $e_0,f_0,h_0\in \hat{\mathfrak{L}}(L^0)$ we consider the root spaces $L^0_\theta$ and $L^0_{-\theta}$ where $\theta$ is the highest root of $L^0$. These are both one-dimensional and the restriction of the standard invariant bilinear form of $L(A)$ to $L^0_\theta\times L^0_{-\theta}$ (as above) is non-degenerate. Let $\omega^0$ be the automorphism of $L^0$ satisfying $\omega^0(E_i)=-F_i$ and $\omega^0(F_i)=-E_i$. Then $\omega^0(L^0_\theta)=L^0_{-\theta}$ and we can choose elements $F_0\in L^0_\theta$ and $E_0\in L^0_{-\theta}$ such that $\omega^0(F_0)=-E_0$ and $\langle F_0,E_0\rangle = 1$. Now, we define the elements $e_0\coloneq t\otimes E_0$ and $f_0 \coloneq t^{-1}\otimes F_0$. Let $H^0$ be the subspace of $L^0$ spanned by $h_1,\dots,h_l$ and
\begin{equation*}
	H=(1\otimes H^0)\oplus\mathbb{C}c\oplus\mathbb{C}d.
\end{equation*}
Finally, we define $h_0\in H$ by
\begin{equation*}
	h_0 \coloneq (1\otimes(-H_\theta))+c.
\end{equation*}
We have
\begin{equation*}
	[e_0,f_0] = [t\otimes E_0,t^{-1}\otimes F_0]=(1\otimes [E_0,F_0])+\langle E_0, F_0\rangle c.
\end{equation*}
and
\begin{equation*}
	[E_0,F_0]=\langle E_0,F_0\rangle H'_{-\theta}= H_{-\theta}=-H_\theta
\end{equation*}
using Corollary \ref{cor:prop_inv_bil_form} and $\langle \theta,\theta\rangle =2$. Therefore
\begin{equation*}
	[e_0,f_0]=(1\otimes(-H_\theta))+c = h_0.
\end{equation*}
We also extend $\alpha_i\in (H^0)^*$, $i=1,\dots,l,$ to $H^*$ by demanding  $\alpha_i(c)=\alpha_i(d)=0$. Similarly we extend $\theta$ to $H^*$ by demanding $\theta(c)=\theta(d)=0$. Now let $\delta\in H^*$ be the element defined by
\begin{equation*}
	\delta(h_i)=0\quad\text{for }i=0,1,\dots,l,\quad \delta(d)=1.
\end{equation*}
Then we can define $\alpha_0\in H^*$ by $\alpha_0=-\theta+\delta$. The isomorphism is now obtained by checking that $(H,\Pi,\Pi^\vee)$ is indeed a realization of $A$, $\hat{\mathfrak{L}}(L^0)$ is generated by $H$, $e_0,e_1,\dots,e_l$, $f_0,f_1,\dots,f_l$, and $\hat{\mathfrak{L}}(L^0)$ has no non zero ideal $J$ with $J\cap H=0$. The last assertion is easily proven by checking that
\begin{equation*}
	\mathfrak{g}\coloneq\hat{\mathfrak{L}}(L^0)= H\oplus \sum_{(i,\alpha)\neq(0,0)}(t^i\otimes(L^0)_\alpha)
\end{equation*}
summed over $i\in\mathbb{Z}$ and $\alpha\in(H^0)^*$ is the root space decomposition of $\mathfrak{g}$ with respect to $H$ such that $H=\mathfrak{g}_0$ and $t^i\otimes(L^0)_\alpha = \mathfrak{g}_{\alpha+i\delta}$. Therefore, Proposition \ref{prop:isoto_g(A)} can be applied.
\begin{rem}[the standard invariant form on $\hat{\mathfrak{L}}(L^0)$]
	It is easily checked that the standard invariant bilinear form on $L(A)$ maps under this isomorphism to the form on $\hat{\mathfrak{L}}(L^0)$ given by
	\begin{equation*}
		\begin{split}
			&\langle t^i\otimes x,t^j\otimes y\rangle =0\qquad \text{if }j\neq i,\quad\text{for }x,y\in L^0\\
			&\langle t^i\otimes x,t^{-i}\otimes y\rangle=\langle x,y\rangle\\
			&\langle t^i\otimes x,c\rangle=0\\
			&\langle t^i\otimes x,d\rangle=0\\
			&\langle c,c\rangle =0\\
			&\langle d,d\rangle =0\\
			&\langle c,d\rangle =1.
		\end{split}
	\end{equation*}
	The form defined on $\hat{\mathfrak{L}}(L^0)$ in this way is clearly invariant. Moreover, this form on the subspace $(1\otimes H)\oplus \mathbb{C}c\oplus\mathbb{C}d$ of $\hat{\mathfrak{L}}(L^0)$ agrees with the standard invariant form on the subspace $H$ of $L(A)$ under our isomorphism between the subspaces. Since a symmetric invariant bilinear form on $L(A)$ is uniquely determined by its restriction to $H$ by Corollary \ref{cor:unique_bil_restr_to_H}, the above form on $\hat{\mathfrak{L}}(L^0)$ must correspond to the standard invariant bilinear form on $L(A)$. $\odot$
\end{rem}

We also observe that the element $c\in \hat{\mathfrak{L}}(L^0)$ corresponds to the canonical central element in $L(A)$ under the isomorphism. It is because
\begin{equation*}
	h_0 = (1\otimes (-H_\theta))+c
\end{equation*}
and therefore $\sum_{i=0}^{l} c_i h_i = c$ in  $\hat{\mathfrak{L}}(L^0)$ is mapped onto the canonical central element of $L(A)$.

Finally, since $\alpha_0(d)=1,\alpha_i(d)=0$ for $i=1,\dots,l$ the element $d\in\hat{\mathfrak{L}}(L^0)$ corresponds to an analogous scaling element $d$ for $L(A)$.
\subsection{Cartan data of finite semisimple Lie algebras}\label{subsect:Cartan_data}
To the end of this section, let us also say a few words about some notions that are quite standard in the (finite-dimensional) representation theory of (semi)simple finite-dimensional Lie algebras over $\mathbb{C}$. It will be helpful to recognize them when we formulate their 'loop analogues' in Chapter \ref{ch:fin_dim_rep_quantum_aff}. We shall refer the reader to the first part of the book \cite{Carter}, or, for example to the book of Humphreys \cite{Humphreys}, but we chose a the normalization for the invariant bilinear form which fits with the loop construction and the standard bilinear form for untwisted affine Kac--Moody algebras in the sense we explained in Subsection \ref{subsect:loop_constr_untw_Kac--Moody}.\\

Let $\mathfrak{g}$ be a semisimple Lie algebra over $\mathbb{C}$ and let $\mathfrak{h}$ be a \textbf{Cartan subalgebra}. In Definition \ref{def:Cartan,neg_pos_fund_mult_roots} which includes the finite case, i.e. all simple complex Lie algebras, we have already defined a particular Cartan subalgebra $H$ of $L(A)$ through its construction. More generally, a Cartan subalgebra of a semisimple Lie algebra $\mathfrak{g}$ over an algebraically closed field of characteristic zero can be defined as a maximal toral (or commuting) subalgebra.\footnote{ Of course, there is also the completely general definition which doesn't require any assumptions on the base field. Then, a Cartan subalgebra is defined to be a nilpotent and self-normalizing subalgebra (see \cite{Carter} Chapter 3).} Furthermore, one can show that any two Cartan subalgebras are conjugate under automorphisms of $\mathfrak{g}$. The common dimension of the Cartan subalgebras is called the rank of $\mathfrak{g}$.

Let now $\rho:\mathfrak{g}\to \End(V)$ be a representation of $\mathfrak{g}$ on a vector space $V$ over a field of characteristic zero, say $\mathbb{C}$, and let $\lambda:\mathfrak{h}\to\mathbb{C}$ be a linear functional, i.e. $\lambda\in\mathfrak{h}^*$. Then the \textbf{weight space} of $V$ of weight $\lambda$ is the subspace
\begin{equation*}
	V_\lambda\coloneq\{v\in V|\,\text{for all }h\in\mathfrak{h},\quad\rho(h).v=\lambda(h)v\}.
\end{equation*}
If $V_\lambda$ is not the zero space, then $\lambda$ is called a \textbf{weight} of $V$ and any non-zero vector $v\in V_\lambda$ is called a weight vector.\footnote{ Note that we will oftentimes refer to a representation $\rho:\mathfrak{g}\to \End(V)$ only by its vector space $V$ over which elements of the Lie algebra act.} "A weight vector is a simultaneous eigenvector for the action of the elements of the Cartan subalgebra and the corresponding eigenvalues are given by the weight $\lambda$." Furthermore, if $V$ is the direct sum of its weight spaces, i.e.
\begin{equation}\label{eqn:weight_space_dec}
	V=\bigoplus_{\lambda\in\mathfrak{h}^*}V_\lambda,
\end{equation}
then $V$ is called a weight module and equation (\ref{eqn:weight_space_dec}) is called its weight space decomposition. In fact, this corresponds to the existence of a common eigenbasis for $\mathfrak{h}$. Certainly, we know from linear algebra that this is the case when the representation $V$ is finite-dimensional over $\mathbb{C}$.
\begin{rem}\label{rem:roots_and_adj_rep_weights}
	We note that \textbf{roots} are exactly the weights of the \textbf{adjoint representation} $\operatorname{ad}:\mathfrak{g}\to \End(\mathfrak{g})$ of $\mathfrak{g}$. Hence, the \textbf{root space} to the root $\alpha\in\mathfrak{h}^*$ is defined by
	\begin{equation*}
		\mathfrak{g}_\alpha=\{x\in\mathfrak{g}\,|\,[h,x]=\alpha(h)x\,\text{for all}\,h\in\mathfrak{h}\}.\footnotemark
	\end{equation*}
	\footnotetext{ We will use this definition for any $\alpha\in \mathfrak{h}^*$ even if $\alpha$ is not a root.}
	Equation (\ref{eqn:weight_space_dec}) is then called the \textbf{Cartan decomposition} of $\mathfrak{g}$ and we have $\mathfrak{g}_0=\mathfrak{h}$ the \textbf{Cartan subalgebra}. Moreover, one can show that the set $\Phi$ of roots of $\mathfrak{g}$ span $\mathfrak{h}^*$. Explicitly, $\alpha\in\mathfrak{h}^*$ is called a root if $\alpha\neq 0$ and there exists a non-zero root vector $x\in\mathfrak{g}$ such that
	\begin{equation*}
		[h,x]=\alpha(h)x
	\end{equation*}
	for all $h\in\mathfrak{h}$. Further, it can be shown that the set of all roots $\Phi$ forms a root system which is important for the classification of $\mathfrak{g}$. For, if $\rho:\mathfrak{g}\to\End(V)$ is a representation of $\mathfrak{g}$, $v$ is a weight vector with weight $\lambda$ and $x$ a root vector to the root $\alpha$, then
	\begin{equation*}
		\rho(h).(\rho(x).v) = [(\lambda+\alpha)(h)](\rho(x).v)
	\end{equation*}
	for all $h\in\mathfrak{h}$, i.e. $\rho(x).v$ is either the zero vector or a weight vector with weight $\lambda+\alpha$. Hence, the action of the root vector $x$ maps the weight space with weight $\lambda$ into the weight space with weight $\lambda+\alpha$. $\odot$
\end{rem}

Let $\mathfrak{h}_0^*$ be the real subspace of $\mathfrak{h}^*$ generated by the roots. Since every real vector space has total orderings, we can chose a total ordering on $\mathfrak{h}_0^*$. Then, we have a \textbf{positive system} $\Phi^+\subset \Phi$ of roots $\alpha > 0$, the \textbf{positive roots}, and a the corresponding set $\Phi^-$ of \textbf{negative roots} $\alpha < 0$. We further define the \textbf{fundamental system} $\{\alpha_i|\,i=1,\dots,l \}\eqcolon\Pi\subset\Phi^+$ as the set of all roots $\alpha\in\Phi^+$ which cannot be written as the sum of two other positive roots. Then, one can show $\Phi = \Phi^+\amalg\Phi^-$, the fundamental system $\Pi$ forms a basis of $V=\mathfrak{h}_0^*$ and each $\alpha\in\Phi$ can be expressed in the form $\alpha=\sum n_i\alpha_i$ with $\alpha_i\in\Pi$, $n_i\in\mathbb{Z}$ and either $n_i\geq0$ for all $i$ or $n_i\leq 0$ for all $i$. We can therefore define the \textbf{height} of a root $\alpha=\sum_i n_i\alpha_i$ as $\operatorname{ht}(\alpha) = \sum_i n_i$.\footnote{ The correspondence to the construction of Kac--Moody algebras of finite type should be clear from here.}\\

We can now normalize the invariant inner product $\br{\cdot,\cdot}$ on $\mathfrak{g}$ such that the square length of the maximal root $\theta$ equals $2$.\footnote{ With respect to the induced inner product on $\mathfrak{h}^*$.} Furthermore, we identify $\mathfrak{h}^*$ and $\mathfrak{h}$ (resp. the subspaces $\mathfrak{h}_0^*$ and $\mathfrak{h}_0$) through $\alpha\mapsto h_\alpha'$ such that $\alpha=\br{h_\alpha',\cdot}$. Then, we define for $\alpha\in\Phi$ the coroot $\alpha^\vee$ by
\begin{equation*}
	\alpha^\vee=\frac{2\alpha}{\langle\alpha,\alpha\rangle}
\end{equation*}
and its corresponding element in $\mathfrak{h}$ under the isomorphism by $h_\alpha$. 
It can be shown that the coroots also form a root system $\Phi^\vee$ called the dual root system.

However, we call an element $\lambda\in\mathfrak{h}^*$ \textbf{algebraically integral}, if
\begin{equation*}
	\lambda(h_\alpha)=\langle\lambda,\alpha^\vee\rangle\in\mathbb{Z}
\end{equation*}
for all roots $\alpha\in\Phi$. The motivation for this condition is that the coroot $h_\alpha$ can be identified with the element $H$ of the Chevalley basis $E,F,H$ for an $\mathfrak{sl}_2$ subalgebra of $\mathfrak{g}$. For, it is known that, for $\mathfrak{sl}_2$, the eigenvalues of $H_\alpha$ in any finite-dimensional representation must be integer. Therefore, we conclude that the weight of any finite-dimensional representation of $\mathfrak{g}$ is algebraically integral.\footnote{ Note that we use the identification between $\mathfrak{h}^*$ and $\mathfrak{h}$ implicitly here.}

This motivates the following definition.
\begin{definition}[fundamental weights]
	The fundamental weights $\omega_1,\dots,\omega_n$ are defined as the basis dual to the fundamental system of coroots $\{\alpha_i^\vee|\,i=1,\dots,l \}\eqcolon\Pi^\vee$, i.e.
	\begin{equation*}
		\langle\omega_i,\alpha_j^\vee\rangle = \delta_{i,j}\quad\text{for all }i,j=1,\dots,l.\quad\odot
	\end{equation*}
\end{definition}

A weight $\lambda$ is then algebraically integral if and only if it is an integral combination of the fundamental weights. We call a weight $\lambda$ \textbf{dominant} if it is a non-negative combination of the fundamental weights. Equivalently, $\lambda$ is dominant if $\langle \lambda,\alpha\rangle\geq 0$ for each positive root $\alpha\in\Phi^+$. The set of all (not necessarily integral) $\lambda$ such that $\langle\lambda,\alpha\rangle\geq 0$ is called the \textbf{fundamental Weyl chamber} corresponding to $\Phi^+$. Moreover, the set of all $\mathfrak{g}$-integral weights is a lattice in $\mathfrak{h}_0^*$ (resp. $\mathfrak{h}_0$) called the \textbf{weight lattice} $P=P(\mathfrak{g})$.\footnote{ By the identification of $\mathfrak{h}^*$ and $\mathfrak{h}$ ($\alpha\mapsto h_\alpha'$).}

\begin{rem}[analytically integral]
	If $\mathfrak{g}$ is the Lie algebra of a Lie group $G$ we say that $\lambda\in\mathfrak{h}_0$ is \textbf{analytically integral} (or $G$-integral) if for each $w\in\mathfrak{h}_0$ with $\exp(w)=1\in G$ we have $\langle \lambda, w\rangle\in 2\pi i \mathbb{Z}$. This definition has the consequence that a representation of $\mathfrak{g}$ which is constructed from a representation of $G$ has $G$-integral weights. Additionally, if $G$ is semisimple, the set of all $G$-integral weights is a sublattice $P(G)\subset P(\mathfrak{g})$. If $G$ is also simply connected, then $P(G)=P(\mathfrak{g})$. Otherwise, $P(G)$ is a sublattice of $P(\mathfrak{g})$ and their quotient is isomorphic to the fundamental group of $G$. $\odot$
\end{rem}
Before we continue with the discussion on representations, let us briefly make a remark about another important concept, namely, the \textbf{Weyl group} (cf. \cite{CPBook} Appendix A 5-6).
\begin{rem}[the Weyl group]
	The \textbf{Weyl group} $W$ of $\mathfrak{g}$ is the subgroup of $GL(\mathfrak{h})$ generated by the so-called fundamental (or simple) reflections $s_i:\mathfrak{h}\to\mathfrak{h}$, defined by
	\begin{equation*}
		s_i(h)=h-\alpha_i(h)h_i=h-\langle h_{\alpha_i}',h\rangle h_i, \qquad(h\in\mathfrak{h}). \footnotemark
	\end{equation*}
	\footnotetext{Where we use the short hand notation $h_i=h_{\alpha_i}$ as above.} The action of $W$ preserves the bilinear form $\langle,\rangle$ on $\mathfrak{h}$.
	
	As an abstract group, the Weyl group $W$ is a Coxeter group with generators $s_1,\dots,s_l$ and defining relations
	\begin{equation*}
		s_i^2=\boldsymbol{1},\quad(s_is_j)^{m_{ij}}=\boldsymbol{1},\,\text{ if }i\neq j,
	\end{equation*}
	where the integers $m_{ij}$ are given by $2$, $3$ , $4$, $6$ (or $\infty$) when $a_{ij}a_{ji}$ is $0$, $1$, $2$, $3$ ( or $\geq 4$), respectively.\footnote{The case $m_{ij}=\infty$ means that the relation $(s_is_j)^{m_{ij}}=\boldsymbol{1}$ is omitted. However, this doesn't happen when $\mathfrak{g}$ is finite, rather when we allow the more general assumption that $\mathfrak{g}$ is an infinite-dimensional Kac--Moody algebra (or a direct sum of the latter).} The element $c=s_1s_2\cdots s_l$ is called a \textbf{Coxeter element} and its order in $W$ is called the \textbf{Coxeter number} $h$ of $\mathfrak{g}$.
	
	An expression $w=s_{i_1}\cdots s_{i_k}$ of an element $w\in W$ as a product of simple reflections is called a reduced decomposition if $k$ is the minimal number of simple reflections which appear in any decomposition of $w$. Then, $l(w)\coloneq k$ is called the length of $w$. If $w'\in W$ is another element, we say that $w'\leq w$ if $w'=s_{j_1}s_{j_2}\cdots s_{j_l}$ and the sequence $j_1,j_2,\dots,j_l$ is obtained from the sequence $i_1,\dots,i_k$ by (possibly) omitting some of its terms. This defines a partial ordering $\leq$ on $W$, called the Bruhat ordering.
	
	By applying the identification of $\mathfrak{h}$ with $\mathfrak{h}^*$ through the bilinear form of $\mathfrak{g}$ (carefully!), the Weyl group $W$ acts on $\mathfrak{h}^*$ by
	\begin{equation*}
		s_i(\alpha)=\alpha-\alpha(h_i)\alpha_i= \alpha-\langle\alpha,\alpha_i^\vee\rangle\alpha_i,\qquad(\alpha\in\mathfrak{h}^*).
	\end{equation*}
	We then have $W(\Phi) = \Phi$ and $\dim(g_\alpha)=\dim(g_{w(\alpha)})$ for all $w\in W$, $\alpha\in\Phi$.
	
	Moreover, there is a unique element $w_0\in W$ of maximal length $l(w_0) = N = |\Phi^+|$ and we have $w_0^2 =\boldsymbol{1}$. If $w_0 =s_{i_1}\cdots s_{i_N}$ is a reduced decomposition, then
	\begin{equation*}
		\Phi^+ = \{\alpha_{i_1},s_{i_1}(\alpha_{i_2}),\dots,s_{i_1}\cdots s_{i_{N-1}}(\alpha_{i_N})\},
	\end{equation*}
	each positive root occurring exactly once on the right hand side.
	
	Let us use the short hand notation $x_i^+$ and $x_i^-$ for the Chevalley generators $e_i$ and $f_i$, respectively. Using this, we can define automorphisms $T_1,\dots,T_n$ of $\mathfrak{g}$ such that
	\begin{equation*}
		\begin{split}
			&T_i(x_i^\pm)=-x_i^\mp,\quad T_i(h_j)=h_j-a_{ji}h_i,\\
			&T_i(x_j^+)=(-a_{ij})!^{-1}(\operatorname{ad}_{x_i^+})^{-a_{ij}}(x_j^+),\quad\text{if }i\neq j,\\
			&T_i(x_j^+)=(-1)^{a_{ij}}(-a_{ij})!^{-1}(\operatorname{ad}_{x_i^-})^{-a_{ij}}(x_j^-),\quad\text{if }i\neq j.
		\end{split}
	\end{equation*}
	They satisfy the defining relations of the braid group $\mathcal{B}_{\mathfrak{g}}$
	\begin{equation*}
		T_iT_j=T_jT_i,\quad T_iT_jT_i=T_jT_iT_j,\quad (T_iT_j)^2=(T_jT_i)^2,\quad (T_iT_j)^3=(T_jT_i)^3,
	\end{equation*}
	according as $a_{ij}a_{ji}=0$, $1$, $2$ or $3$, respectively. If $\mathfrak{g}$ is of type $A_n$, for instance, we have
	\begin{equation*}
		\begin{split}
			T_iT_{i+1}T_i &= T_{i+1}T_i T_{i+1}\\
			T_iT_j &= T_jT_i\quad\text{if }|i-j|>1.
		\end{split}
	\end{equation*}
	Further, each $T_i$ preserves $\mathfrak{h}$ and induces on it the action of the reflection $s_i$.
	
	If $\alpha=s_{i_1}s_{i_2}\cdots s_{i_{k-1}}(\alpha_{i_k})\in \Phi^+$, where $w_0 = s_{i_1}s_{i_2}\cdots s_{i_{N}}$ is a reduced decomposition, we define
	\begin{equation*}
		x_\alpha^\pm = T_{i_1}T_{i_2}\cdots T_{i_{k-1}} (x_{i_k}^\pm).
	\end{equation*}
	Then, $x_\alpha^\pm$ is a non-zero element of $\mathfrak{g}_{\pm\alpha}$ for all $\alpha\in\Phi^+$ and (up to a sign) independent of the choice of the reduced decomposition of $w_0$. $\odot$
\end{rem}
\begin{rem}
	The integer factor $h^\vee$ between the Killing form $(x,y)=\tr(\operatorname{ad}_x\circ\operatorname{ad}_y)$ of $\mathfrak{g}$ and $\langle,\rangle$ is sometimes called the \textbf{dual Coxeter number} of $\mathfrak{g}$. However, the term 'dual' cannot be meant in a group theoretic sense here, as the Weyl group of the dual root system is canonically isomorphic to the Weyl group itself (see for example \cite{Humphreys} Section 9.2). In fact, the (dual) Coxeter number of $\mathfrak{g}$ can be defined as the (dual) Coxeter number of the untwisted affine Kac--Moody algebra $\hat{\mathfrak{g}}\coloneq \hat{\mathfrak{L}}(\mathfrak{g})$ as was done at end of Subsection \ref{subsect:aff_Kac--Moody}, respectively, where the terminology is obvious.
\end{rem}
Coming back to the action of root vectors which we have discussed in Remark \ref{rem:roots_and_adj_rep_weights}, we can introduce a partial ordering on the set of weights which will enable us to formulate the \textbf{theorem of the highest weight} for representations of $\mathfrak{g}$.
\begin{definition}[partial ordering of weights]
	Let $\lambda$ and $\mu$ be weights of $\mathfrak{g}$. We define a partial ordering on the set of weights by saying that $\lambda$ is higher than $\mu$ ($\lambda\geq \mu$), if $\lambda-\mu$ is expressible as a linear combination of positive roots with non-negative real coefficients. Equivalently, $\mu$ is called lower than $\lambda$. $\odot$
\end{definition}
Note that this only defines a partial ordering since it can happen that $\lambda$ is neither higher nor lower than $\mu$.
A weight $\lambda$ of a representation $V$ of $\mathfrak{g}$ is called a \textbf{highest weight} if every other weight of $V$ is lower than $\lambda$.

\begin{thm}[theorem of the highest weight]
	Let $\mathfrak{g}$ be a complex semisimple Lie algebra, then
	\begin{enumerate}[label=\normalfont(\roman*)]
		\item every irreducible (finite-dimensional) representation has a highest weight,
		\item the highest weight is always dominant and algebraically integral,
		\item two irreducible representations with the same highest weight are isomorphic and
		\item \label{enum:irred_repr_to_dom_int}every dominant, algebraically integral weight $\lambda$ is the highest weight of an irreducible representation. $\odot$
	\end{enumerate}
\end{thm}
Therefore, finite-dimensional irreducible representations of $\mathfrak{g}$ are in correspondence with dominant integral weights, i.e. positive integral combinations of the fundamental weights $\omega_i$, $i=1,\dots,l$.
The representation in \ref{enum:irred_repr_to_dom_int} can be constructed using the so-called \textbf{Verma modules}. Let us formulate one more general remark.
\begin{rem}[highest-weight module]
	A representation $V$ of $\mathfrak{g}$ is called \textbf{highest-weight module} if it is generated by a weight vector $v\in V$ which is annihilated by the action of all positive root spaces $\mathfrak{g}_\alpha$, $\alpha\in\Phi^+$. Then, every irreducible representation with highest weight is necessarily a highest-weight module. However, when $V$ is infinite-dimensional it can be reducible. Moreover, to every $\lambda\in\mathfrak{h}^*$ there exists a unique (up to isomorphism) irreducible highest weight module $L(\lambda)$ with highest weight $\lambda$ which can be constructed as a quotient of the Verma module $M(\lambda)$. $\odot$
\end{rem}
Let us summarize our notions for a finite-dimensional semisimple complex Lie algebra $\mathfrak{g}$ as follows. We identify the Cartan subalgebra $\mathfrak{h}$ and its dual $\mathfrak{h}^*$ by means of the invariant inner product $\langle\cdot,\cdot\rangle$ on $\mathfrak{g}$ normalised such that the square of the maximal root equals $2$. Let $I=\{1,\dots,l\}$ and let $\{\alpha_i\}_{i\in I}$ be the set of simple roots with the corresponding simple coroots $\{\alpha_i^\vee\}_{i\in I}$ and fundamental weights $\{\omega_i\}_{i\in I}$. Let $A\coloneq(a_{ij})$ denote the Cartan matrix of $\mathfrak{g}$ and let $d_i$, $i=1,\dots,l$, be the relatively prime integers such that $B=(b_{ij})=(d_i a_{ij}) = DA$ is symmetric.\footnote{ $B$ is sometimes called the symmetrized Cartan matrix.} We have
\begin{align}
	2\br{\alpha_i,\alpha_j} = a_{ij}\br{\alpha_i,\alpha_i}, \quad 2\br{\alpha_i,\omega_j} = \delta_{i,j}\br{\alpha_i,\alpha_i}, \quad b_{ij} = d^\vee \br{\alpha_i,\alpha_j},
\end{align}
where $d^\vee$ is the maximal number of edges connecting two nodes in the Dynkin diagram of $\mathfrak{g}$, i.e. $d^\vee=1$ for simply-laced, $d^\vee =2$ for $B_l,C_l,F_4$, and $d^\vee=3$ for $G_2$.\footnote{ Hence, we have $d^\vee = \hat{d}_id_i$ for $i=1,\dots,l$, and $B=d^\vee \hat{B}$, where $\hat{d}_i$ and $\hat{B}$ are as in Subsection \ref{subsect:aff_Kac--Moody}.}
At last we denote the (positive) weight and root lattice by $P$ ($P^+$) and $Q$ ($Q^+$), respectively. We have that $P$ ($P^+$) and $Q$ ($Q^+$) are the $\mathbb{Z}$-span ($\mathbb{Z}_{\geq0}$-span) of the fundamental weights and simple roots, respectively. Then we have a partial order $\leq$ on $P$ in which $\lambda\leq\lambda'$ iff $\lambda'-\lambda\in Q^+$. If $\alpha\in Q$, we denote the root space of $\alpha$ by
\begin{align*}
	\mathfrak{g}_\alpha=\{x\in\mathfrak{g}\,|\,[h,x]=\alpha(h)x\,\text{for all}\,h\in\mathfrak{h}\}.
\end{align*}
We set $\Phi \coloneq \{\alpha\in Q\,|\,\alpha\neq0,\,\mathfrak{g}_\alpha\neq0\}$ the set of roots of $\mathfrak{g}$ and $\Phi^+\coloneq\Phi\cap Q^+$ (resp. $\Phi^-\coloneq-\Phi^+$) the set of positive (resp. negative roots). Then we have $\Phi = \Phi^+\amalg\Phi^-$.
\section{Quantum groups and quantum affine algebras}\label{sect:quantum_groups}
Having introduced the universal enveloping algebra and the notion for finite and affine Kac--Moody algebras in the last section, we are now in position to introduce the structures we are interested in. Those are generally called quantum groups, generally referred to as Hopf algebras or non-cocommutative Hopf algebras. Specifically, we are interested in the so-called quantum affine algebras, also known as affine quantum groups. These can be built on the basis of the finite and affine Kac--Moody algebras as we shall see. Later, we will explicitly consider the affine quantum groups of type $A_n$, which are important for the correlation functions discussed in my paper \cite{BJ} and arguably the simplest when it comes to the representation theory. However, let us begin by describing what a Hopf algebra is and explaining why it is so important. As we shall see in Section \ref{subsect:Hopf_algebras}, the comultiplication map $A\to A\otimes A$ of an arbitrary Hopf algebra $A$ allows one to define the tensor product $V\otimes W$ of two representations $V$ and $W$ of $A$. If $A$ is not cocommutative, $V\otimes W$ and $W\otimes V$ are not necessarily isomorphic as representations of $A$, although they are isomorphic for the special case of \textit{quasitriangular} Hopf algebras, discussed in Section \ref{subsect_quasitriang_Hopf}. Such Hopf algebras contain a distinguished invertible element $\mathcal{R}\in A\otimes A$, called the \textit{universal} $R$\textit{-matrix}, from which the isomorphism $V\otimes W\to W\otimes V$ is constructed. The affine quantum groups we consider in this thesis are at least in a weak sense quasitriangular (sometimes called \textit{pseudotriangular}) such that we can make use of their universal $R$-matrices. The crucial property of $\mathcal{R}$ is that it satisfies the so-called (\textit{quantum}) \textit{Yang--Baxter equation}
\begin{equation*}
	\mathcal{R}_{12}\mathcal{R}_{13}\mathcal{R}_{23}=\mathcal{R}_{23}\mathcal{R}_{13}\mathcal{R}_{12}
\end{equation*}
We will later introduce a rigorous graphical notation to visualise it. We refer the reader to Chapter 4-6, 9 and 12 in the book \cite{CPBook} for more details. Here, we discuss only a selection of aspects which is suitable for this thesis.
\subsection{Hopf algebras in terms of string diagrams}\label{subsect:Hopf_algebras}
The structure of a Hopf algebra is usually visualized in terms of commutative diagrams. Such a description can be found in the book \cite{CPBook} in Section 4.1 for instance. In this thesis, however, I prefer to give a slightly different description which illustrates the tensor product in a nice way. I should also give the reference to Jules Lamers who taught me this on the Les Houches summer school in 2018 when I was a young researcher. The diagrams I will draw are sometimes referred to as "string diagrams" and in our case we read them from bottom to top.\footnote{ Instead, one could use arrows to give them an independent orientation.} However, anything else can also be found in the book \cite{CPBook} in Section 4.1. In general, everything is defined over a commutative ring $k$. 
In our string diagrams $k$ is represented by a point. Moreover, we represent a $k$-module $A$ by a line and every morphism from $A$ to $A$ by a symbol, for instance a point, which intersects the line. Now, whenever we find a line $B$ next to a line $A$ such that the line $B$ is on the left of the line $A$, we identify them together with the tensor product $A\otimes B$ of the $k$-modules $A$ and $B$. In this way, two lines next to each other can be interpreted in terms of a single line for the tensor product of the corresponding modules. The commutative ring $k$, a $k$-module $A$, a morphism $\alpha:A\to A$ and the tensor product of $k$-modules $A$ and $B$ are described by the diagrams
\begin{figure}[H]
	\centering
	\begin{tikzpicture}[scale=1]
		\draw[thick] (0,0)node[anchor=north]{$A$} -- (0,1) node[anchor=west]{};
		\draw[thick] (0,1) -- (0,2) ;
		\draw[thick] (1,0)node[anchor=north]{$A$} -- (1,1) node[anchor=west]{};
		\draw[thick] (1,1) -- (1,2) ;
		\filldraw[black] (1,1) circle (2pt) node[anchor=west]{$\alpha$};
		\filldraw[black] (-1,1) circle (2pt) node[anchor=north]{$k$};
		\draw[thick] (3,0)node[anchor=north]{$B$} -- (3,1) node[anchor=west]{};
		\draw[thick] (3,1) -- (3,2) ;
		\draw[thick] (4,0)node[anchor=north]{$A$} -- (4,1) node[anchor=west]{};
		\draw[thick] (4,1) -- (4,2) ;
		\node at (5,1) {=};
		\draw[thick] (6,0)node[anchor=north]{$A\otimes B$} -- (6,1) node[anchor=west]{};
		\draw[thick] (6,1) -- (6,2) ;
		\node at (7,0) {.} ;
	\end{tikzpicture}
\end{figure}
Therefore, multiple tensor products of different modules can iteratively be built by drawing the respective lines next to each other in the corresponding order. Note that we will omit writing letters on the dots and lines when there is only one commutative ring $k$ and one $k$-module $A$ other than $k$ itself in the context. For the obvious $k$-module $k$, we will use a dashed line by default. The definitions are as follows (cf. \cite{CPBook} pp 101-105).
\begin{definition}
	A \textbf{unital associative algebra} over a commutative ring\footnote{ A commutative ring is always assumed to have a unit element $1$ if not stated otherwise.} $k$ is a $k$-module $A$ equipped with $k$-module maps called the multiplication $m:A\otimes A\to A$ and the unit $\eta:k\to A$, $\lambda\mapsto\lambda \boldsymbol{1}$ such that the multiplication is associative, i.e. $m\circ (m\otimes\id) = m\circ (\id\otimes m)$, and compatible with the unit in the sense that $m\circ (\eta\otimes\id) = \tilde{\id}= m\circ\,\sigma\circ(\eta\otimes\id)$ where $\tilde{\id}$ is the natural isomorphism\footnote{The isomorphism obtained from the universal property of the tensor product.} $:k\otimes A \overset{\sim}{\longrightarrow} A$ and $\sigma: V\otimes W \overset{\sim}{\longrightarrow} W\otimes V$ the flip map\footnote{ Defined for any two $k$-modules $V$ and $W$.}, i.e. the $k$-linear map such that $v\otimes w \mapsto w\otimes v$ for $v\in V$ and $w\in W$. In terms of string diagrams, the maps are visualised by
	\begin{figure}[H]
		\centering
		\begin{tikzpicture}[scale=1]
			\draw[thick] (-1,0) -- (-1,2);
			\filldraw[black] (-1,0) circle (2pt) node[anchor=north]{};
			\draw[thick] (-1,1) node[anchor=west]{$\eta$};
			\draw[thick] (1.5,0) arc[start angle=180,end angle=0,radius=.5];
			\draw[thick] (2,.5) -- (2,2);
			\draw[thick] (2,1) node[anchor=west]{$m$};
			\node[anchor=west] at (3,0) {,};
		\end{tikzpicture}
	\end{figure}
	for the compatibility condition we have
	\begin{figure}[H]
		\centering
		\begin{tikzpicture}[scale=1]
			\node (A) at (6,0) {};
			\node (B) at (7,1) {};
			\node (C) at (7,0) {};
			\node (D) at (6,1) {};
			\node (i) at (intersection of A--B and D--C) {}; 
			\draw[thick] (7,0) -- (6,1) arc[start angle=180,end angle=0,radius=.5] (7,1) -- (i);
			\draw[thick] (6,0) -- (i);
			\filldraw[black] (7,0) circle (2pt) node[anchor=north]{};
			\draw[thick] (6.5,1.5) -- (6.5,2.5);
			
			\node at (2,1) {=};
			\node at (5,1) {=};
			
			\draw[thick] (3.5,0) -- (3.5,2.5);
			\draw[dashed] (4,0) -- (4,1) arc[start angle=0,end angle=90,radius=.5];
			\filldraw[black] (4,0) circle (2pt) node[anchor=north]{};
			
			\draw[thick] (0.5,1.5) -- (0.5,2.5);
			\draw[thick] (0,0) -- (0,1) arc[start angle=180,end angle=0,radius=.5] (1,1) -- (1,0);
			\filldraw[black] (1,0) circle (2pt) node[anchor=north]{};
		\end{tikzpicture}
	\end{figure}
	and the associativity is given by
	\begin{figure}[H]
		\centering
		\begin{tikzpicture}[scale=1]
			\draw[thick] (1,1.5) -- (1,2);
			\draw[thick] (.5,0) -- (.5,1) arc[start angle=180,end angle=0,radius=.5] (1.5,1) -- (1.5,.5);
			\draw[thick] (1,0) arc[start angle=180,end angle=0,radius=.5];
			
			\node at (3,1) {=};
			
			\draw[thick] (5,1.5) -- (5,2);
			\draw[thick] (4.5,.5) -- (4.5,1) arc[start angle=180,end angle=0,radius=.5] (5.5,1) -- (5.5,0);
			\draw[thick] (4,0) arc[start angle=180,end angle=0,radius=.5];
			\node[anchor=west] at (6,0) {. $\odot$};
		\end{tikzpicture}
	\end{figure}
\end{definition}
The condition that an algebra is commutative ($m=m\circ\,\sigma$) is represented by
\begin{figure}[H]
	\centering
	\begin{tikzpicture}[scale=1]
		\node (A) at (3,0) {};
		\node (B) at (4,1) {};
		\node (C) at (4,0) {};
		\node (D) at (3,1) {};
		\node (i) at (intersection of A--B and D--C) {}; 
		\draw[thick] (4,0) -- (3,1) arc[start angle=180,end angle=0,radius=.5] (4,1) -- (i);
		\draw[thick] (3,0) -- (i);
		\draw[thick] (3.5,1.5) -- (3.5,2.5);
		
		\node at (2,1) {=};
		
		\draw[thick] (0.5,1.5) -- (0.5,2.5);
		\draw[thick] (0,0) -- (0,1) arc[start angle=180,end angle=0,radius=.5] (1,1) -- (1,0);
		\node[anchor=west] at (4.5,0) {.};
	\end{tikzpicture}
\end{figure}
Moreover, if we set $m_{\operatorname{op}}:=m\circ\,\sigma$, then $(A,\eta,m_{\operatorname{op}})$ is also a unital associative algebra, called the opposite of $A$ and denoted by $A_{\operatorname{op}}$.
\begin{definition}
	A \textbf{counital coassociative coalgebra} over a commutative ring $k$ is a $k$-module $A$ together with $k$-module maps $\Delta:A\to A\otimes A$, the comultiplication, and $\epsilon:A\to k$, the counit, such that the comultiplication is coassociative, i.e. $(\Delta\otimes \id)\circ\Delta  =  (\id \otimes \Delta)\circ\Delta$, and compatible with the counit in the sense that $(\epsilon\otimes\id)\circ\Delta = \hat{\id} = (\epsilon\otimes\id)\circ\sigma \circ\Delta$ where $\hat{\id}\coloneq(\tilde{\id})^{-1}$. In terms of string diagrams, the maps are visualised by
	\begin{figure}[H]
		\centering
		\begin{tikzpicture}[scale=1]
			\draw[thick] (0,0) -- (0,2);
			\filldraw[black] (0,2) circle (2pt) node[anchor=north]{};
			\draw[thick] (0,1) node[anchor=west]{$\epsilon$};
			
			\draw[thick] (2.5,2) arc[start angle=-180,end angle=0,radius=.5];
			\draw[thick] (3,1) node[anchor=west]{$\Delta$};
			\draw[thick] (3,0) -- (3,1.5);
			\node[anchor=west] at (4,0) {,};
		\end{tikzpicture}
	\end{figure}
	for the compatibility condition we have
	\begin{figure}[H]
		\centering
		\begin{tikzpicture}[scale=1]
			\node (A) at (6,1.5) {};
			\node (B) at (7,2.5) {};
			\node (C) at (7,1.5) {};
			\node (D) at (6,2.5) {};
			\node (i) at (intersection of A--B and D--C) {}; 
			\draw[thick] (6,2.5) -- (7,1.5) arc[start angle=0,end angle=-180,radius=.5] (6,1.5) -- (i);
			\draw[thick] (7,2.5) -- (i);
			\filldraw[black] (7,2.5) circle (2pt) node[anchor=north]{};
			\draw[thick] (6.5,0) -- (6.5,1);
			
			\node at (2,1) {=};
			\node at (5,1) {=};
			
			\draw[thick] (3.5,0) -- (3.5,2.5);
			\draw[dashed] (4,2.5) -- (4,1.5) arc[start angle=0,end angle=-90,radius=.5];
			\filldraw[black] (4,2.5) circle (2pt) node[anchor=north]{};
			
			\draw[thick] (0.5,0) -- (0.5,1);
			\draw[thick] (0,2.5) -- (0,1.5) arc[start angle=-180,end angle=0,radius=.5] (1,1.5) -- (1,2.5);
			\filldraw[black] (1,2.5) circle (2pt) node[anchor=north]{};
		\end{tikzpicture}
	\end{figure}
	and the coassociativity is given by
	\begin{figure}[H]
		\centering
		\begin{tikzpicture}[scale=1]
			\draw[thick] (1,.5) -- (1,0);
			\draw[thick] (.5,2) -- (.5,1) arc[start angle=-180,end angle=0,radius=.5] (1.5,1) -- (1.5,1.5);
			\draw[thick] (1,2) arc[start angle=-180,end angle=0,radius=.5];
			
			\node at (3,1) {=};
			
			\draw[thick] (5,.5) -- (5,0);
			\draw[thick] (4.5,1.5) -- (4.5,1) arc[start angle=-180,end angle=0,radius=.5] (5.5,1) -- (5.5,2);
			\draw[thick] (4,2) arc[start angle=-180,end angle=0,radius=.5];
			\node[anchor=west] at (6,0) {. $\odot$};
		\end{tikzpicture}
	\end{figure}
\end{definition}
The condition that a coalgebra is cocommutative ($\Delta=\sigma\circ\Delta$) is represented by
\begin{figure}[H]
	\centering
	\begin{tikzpicture}[scale=1]
		\node (A) at (3,1.5) {};
		\node (B) at (4,2.5) {};
		\node (C) at (4,1.5) {};
		\node (D) at (3,2.5) {};
		\node (i) at (intersection of A--B and D--C) {}; 
		\draw[thick] (3,2.5) -- (4,1.5) arc[start angle=0,end angle=-180,radius=.5] (3,1.5) -- (i);
		\draw[thick] (4,2.5) -- (i);
		\draw[thick] (3.5,0) -- (3.5,1);
		
		\node at (2,1) {=};
		
		\draw[thick] (0.5,0) -- (0.5,1);
		\draw[thick] (0,2.5) -- (0,1.5) arc[start angle=-180,end angle=0,radius=.5] (1,1.5) -- (1,2.5);
		\node[anchor=west] at (4.5,0) {,};
	\end{tikzpicture}
\end{figure}
i.e. $\Delta(A)$ is contained in the symmetric part of $A\otimes A$.
Moreover, if we set $\Delta^{\operatorname{op}}:=\sigma\circ\Delta$, then $(A,\epsilon,\Delta^{\operatorname{op}})$ is also a counital coassociative coalgebra, called the opposite of $A$ and denoted by $A^{\operatorname{op}}$. We note that the definition of a coalgebra is dual to the definition of an algebra. Moreover, it is easy to see that tensor products of algebras are algebras and tensor products of coalgebras are coalgebras by tensoring their defining maps and composing them appropriately with the permutation map $\sigma$ in the case of the multiplication and comultiplication. If $A$ and $B$ are coalgebras, a $k$-module map $\varphi:A\to B$ is a homomorphism of coalgebras if
\begin{equation*}
	(\varphi \otimes \varphi)\circ\Delta^A=\Delta^B\circ\varphi,\quad\epsilon^B\circ\varphi = \epsilon^A.
\end{equation*}
\begin{definition}[bialgebra]
	A \textbf{bialgebra} over a commutative ring $k$ is a $k$-module $A$ which is both a unital associative algebra and a counital coassociative coalgebra such that one of the following two equivalent conditions is satisfied.
	\begin{enumerate}[label=\normalfont(\roman*)]
		\item The comultiplication $\Delta: A\to A\otimes A$ and the counit $\epsilon:A\to k$ are homomorphisms of algebras.
		\item The multiplication $m: A\otimes A\to A$ and the unit $\eta:k\to A$ are homomorphisms of coalgebras.
	\end{enumerate}
	Writing them out, the conditions are $\Delta\circ m = (m\otimes m)\circ (\id\otimes\,\sigma\otimes\id)\circ(\Delta\otimes\Delta)$, $\epsilon\circ m= \tilde{\id}^k\circ\epsilon\otimes\epsilon$, $\Delta\circ\eta=\eta\otimes\eta\circ\hat{\id}^k$ and $\epsilon\circ\eta = \id^k$ where the upper index $k$ indicates that the maps are defined for the $k$-module $k$ instead of $A$.\footnote{ The dotted lines are sometimes omitted in all the diagrams and the natural isomorphisms $k\otimes A\cong A \cong A\otimes k$ are implicit. In this sense, we will just write $\id$ from now on and the natural isomorphisms $\tilde{\id}$, $\tilde{\id}\circ\sigma$, $\hat{\id}$, $\hat{\id}\circ\sigma$, etc. are understood.} In terms of string diagrams the conditions are illustrated by
	\begin{figure}[H]
		\centering
		\begin{tikzpicture}[scale=1]
			\draw[thick] (0,3) arc[start angle=-180,end angle=0,radius=.5];	
			\draw[thick] (.5,.5) -- (.5,2.5);
			\draw[thick] (0,0) arc[start angle=180,end angle=0,radius=.5];
			
			\node at (2,1.5) {=};
			
			\node (A) at (4,1) {};
			\node (B) at (5,2) {};
			\node (C) at (5,1) {};
			\node (D) at (4,2) {};
			\node (i) at (intersection of A--B and D--C) {};
			\draw[thick] (i) -- (5,2) arc[start angle=180,end angle=0,radius=.5] (6,2) -- (6,1) arc[start angle=0,end angle=-180,radius=.5] (5,1) -- (4,2) arc[start angle=0,end angle=180,radius=.5] (3,2) -- (3,1) arc[start angle=-180,end angle=0,radius=.5] (4,1) -- (i);
			\draw[thick] (3.5,0) -- (3.5,.5);
			\draw[thick] (5.5,0) -- (5.5,.5);
			\draw[thick] (3.5,2.5) -- (3.5,3);
			\draw[thick] (5.5,2.5) -- (5.5,3);
			\node at (7,0) {,};
		\end{tikzpicture}
	\end{figure}
	\begin{figure}[H]
		\centering
		\begin{tikzpicture}[scale=1]
			\draw[thick] (.5,1) -- (.5,2);
			\draw[thick] (0,0) -- (0,.5) arc[start angle=180,end angle=0,radius=.5] (1,.5) -- (1,0);
			\filldraw[black] (.5,2) circle (2pt) node[anchor=north]{};
			
			\node at (2,1) {=};
			
			\draw[thick] (3,0) -- (3,1);
			\filldraw[black] (3,1) circle (2pt) node[anchor=north]{};
			\draw[dashed] (3,1) -- (3,2);
			\filldraw[black] (3,2) circle (2pt) node[anchor=north]{};
			
			\draw[thick] (3.5,0) -- (3.5,1);
			\draw[dashed] (3.5,1) arc[start angle=0,end angle=90,radius=.5];
			\filldraw[black] (3.5,1) circle (2pt) node[anchor=north]{};
			
			\node at (4.5,0) {,};
			
			\draw[thick] (6,1) -- (6,0);
			\draw[thick] (5.5,2) -- (5.5,1.5) arc[start angle=-180,end angle=0,radius=.5] (6.5,1.5) -- (6.5,2);
			\filldraw[black] (6,0) circle (2pt) node[anchor=north]{};
			
			\node at (7.5,1) {=};
			
			\draw[thick] (8.5,2) -- (8.5,1);
			\filldraw[black] (8.5,1) circle (2pt) node[anchor=north]{};
			\draw[dashed] (8.5,1) -- (8.5,0);
			\filldraw[black] (8.5,0) circle (2pt) node[anchor=north]{};
			
			\draw[thick] (9,2) -- (9,1);
			\draw[dashed] (9,1) arc[start angle=0,end angle=-90,radius=.5];
			\filldraw[black] (9,1) circle (2pt) node[anchor=north]{};
			
			\node at (10,0) {,};
			
			\draw[thick] (11,0) -- (11,2);
			\filldraw[black] (11,2) circle (2pt) node[anchor=north]{};
			\filldraw[black] (11,0) circle (2pt) node[anchor=north]{};
			
			\node at (12,1) {=};
			
			\draw[dashed] (13,0) -- (13,2);
			\filldraw[black] (13,2) circle (2pt) node[anchor=north]{};
			\filldraw[black] (13,0) circle (2pt) node[anchor=north]{};
			
			\node at (14,0) {. $\odot$};
		\end{tikzpicture}
	\end{figure}
\end{definition}
If $A$ and  $B$ are bialgebras, a $k$-module map $\varphi:A\to B$ is a \textit{homomorphism of bialgebras} if it is a homomorphism of both the algebra and coalgebra structures of A. Similarly, the tensor product $A\otimes B$ of two bialgebras is a bialgebra. If $A$ is a bialgebra, there are three associated bialgebras $A_{\operatorname{op}}$, $A^{\operatorname{op}}$ and $A_{\operatorname{op}}^{\operatorname{op}}$ with algebra and coalgebra structures as defined above.
\begin{definition}[Hopf algebra]\label{def:Hopf_alg}
	A \textbf{Hopf algebra} over a commutative ring $k$ is a $k$-module $A$ which is a bialgebra over $k$ equipped with a bijective $k$-module map $S:A\to A$, called the antipode, such that the conditions $m\circ (S\otimes\id)\circ\Delta = \eta\circ\epsilon=m\circ (\id\otimes S)\circ\Delta$ are satisfied. These are illustrated by the string diagrams
	\begin{figure}[H]
		\centering
		\begin{tikzpicture}[scale=1]
			\draw[thick] (.5,0) -- (.5,1);
			\draw[thick] (.5,1) arc[start angle=-90,end angle=270,radius=.5];
			\draw[thick] (.5,2) -- (.5,3);
			\filldraw[black] (0,1.5) circle (2pt) node[anchor=north]{};
			\node[anchor=east] at (0,1.5) {$S$};
			
			\node[anchor=east] at (2,1.5) {$=$};
			
			\draw[thick] (3,0) -- (3,1);
			\filldraw[black] (3,1) circle (2pt) node[anchor=north]{};
			\draw[dashed] (3,1) -- (3,2);
			\filldraw[black] (3,2) circle (2pt) node[anchor=north]{};
			\draw[thick] (3,2) -- (3,3);
			
			\node[anchor=east] at (4,1.5) {$=$};
			
			\draw[thick] (5.5,0) -- (5.5,1);
			\draw[thick] (5.5,1) arc[start angle=-90,end angle=270,radius=.5];
			\draw[thick] (5.5,2) -- (5.5,3);
			\filldraw[black] (6,1.5) circle (2pt) node[anchor=north]{};
			\node[anchor=west] at (6,1.5) {$S$};
			
			\node[anchor=east] at (7.5,0) {. $\odot$};
		\end{tikzpicture}
	\end{figure}
\end{definition}
\textit{Homomorphisms of Hopf algebras} are the homomorphisms of the underlying bialgebra structure. A \textit{Hopf ideal} of a Hopf algebra $A$ over $k$ is a two-sided ideal $I$ of $A$ as an algebra such that
\begin{equation*}
	\Delta(I)\subseteq I\otimes A+A\otimes I,\quad \epsilon(I)=0,\quad S(I)\subseteq I.
\end{equation*}
Then, the Hopf algebra structure of $A$ is passed down to the quotient $k$-module $A/I$ in the obvious way. The kernel of any homomorphism of Hopf algebras is a Hopf ideal. The tensor product $A\otimes B$ of two Hopf algebra with the antipode
\begin{equation*}
	S^{A\otimes B}=S^A\otimes S^B.
\end{equation*}
The associated bialgebras $A_{\operatorname{op}}$, $A^{\operatorname{op}}$ and $A_{\operatorname{op}}^{\operatorname{op}}$ of a Hopf algebra $A$ are Hopf algebras and their antipodes are given by $S^{-1}$, $S^{-1}$ and $S$, respectively.
\begin{rems}\hspace{1em}
	\begin{enumerate}
		\item If a bialgebra has an antipode as in Definition \ref{def:Hopf_alg}, then it is unique. In the literature on Hopf algebras, the antipode is sometimes not required to be bijective. In this case, bijectivity can be shown if the Hopf algebra is either commutative or cocommutative.\label{rem:uniqueness_antipode}
		\item A homomorphism of Hopf algebras $\varphi:A\to B$ automatically satisfies $S^B\circ\varphi = \varphi\circ S^A$ (cf. \cite{Hazewinkel}, Proposition 37.1.10).
		\item The antipode $S$ is automatically an \textit{anti-automorphism} of $A$, i.e. $S:A\to A_{\operatorname{op}}^{\operatorname{op}}$ is an isomorphism of Hopf algebras. Thus, $S^2$ is an automorphism of $A$. It can be shown that, if $A$ is either commutative or cocommutative, then $S^2$ is the identity map (cf. \cite{Abe}, Theorem 2.1.4).
		\item The defining properties of the antipode $S$ in Definition \ref{def:Hopf_alg} can be interpreted in terms of the \textit{convolution} algebra structure on the set of $k$-module maps $A\to A$. If $f,g\in\End(A)$, their convolution is the map $f\bullet g:A\to A$ defined by
		\begin{equation*}
			f\bullet g = m\circ(f\otimes g)\circ\Delta.
		\end{equation*}
		Then it is easy to see that $\eta\circ\epsilon$ is the neutral element for the convolution product. Thus, the defining properties of $S$ assert that $S$ is the inverse of the identity map $\id:A\to A$ for the convolution product. The uniqueness of the antipode in Remark \ref{rem:uniqueness_antipode} is a direct consequence of this observation. $\odot$
	\end{enumerate}
\end{rems}
There are several variants of the definition for Hopf algebras we have given. In some situations, $A$ and $A\otimes A$ (as well as the ground ring $k$) are equipped with topologies for which the $k$-module structure maps of $A$ are continuous. In this case, on usually replaces the algebraic tensor product $A\otimes A$ by a suitable completion, and requires that the Hopf algebra structure maps are continuous. The most important example of such a situation for us is where the ground ring is $k[[h]]$, the ring of formal power series in an indeterminate $h$ over a field $k$. Every $k[[h]]$-module $V$ has the $h$\textit{-adic topology}, which is characterized by requiring that $\{h^nV\,|\,n\geq 0\}$ is a base of the neighbourhoods of $0$ in $V$, and that translations in $V$ are continuous. It is easy to see that, for modules equipped with this topology, every $k[[h]]$-module map is automatically continuous.
\begin{definition}
	A \textbf{topological Hopf algebra} over $k[[h]]$, the ring of formal power series in an indeterminate $h$ over a field $k$, is a complete $k[[h]]$-module $A$ equipped with $k[[h]]$-linear maps $\eta,m,\epsilon,\Delta$ and $S$ satisfying the axioms of Definition \ref{def:Hopf_alg}, but with the algebraic tensor products replaced by their completions in the $h$-adic topology. $\odot$
\end{definition}
Let us now give two important examples. The first one is based on a finite Group $G$. The second defines the Hopf algebra structure for $U(\mathfrak{g})$ (cf. \cite{CPBook} pp. 105-107).
\begin{xmpl}\label{xmpl:group_algebra}
	Let $G$ be a finite group. The group algebra $k[G]$ of $G$ over a commutative ring $k$ is the free $k$-module with basis $G$ and algebra structure obtained by extending linearly the product on $G$. Then, there is a Hopf algebra structure on $k[G]$ defined by extending linearly the formulas $\eta(1)=e$ (where $e$ is the identity element of G), $\epsilon(g)=1$, $\Delta(g)=g\otimes g$ and $S(g)=g^{-1}$ for all $g\in G$. Note that $k[G]$ is always cocommutative, but commutative only if $G$ is abelian. $\odot$
\end{xmpl}
\begin{xmpl}
	Let $\mathfrak{g}$ be a Lie algebra over a field $k$. The Poincaré--Birkhoff--Witt theorem (PBW) asserts that, if $\{x_i\}_{i\in I}$ is a basis of $\mathfrak{g}$, where the index set $I$ is totally ordered, the set of monomials $\{x_{i_1}x_{i_2}\cdots x_{i_k}\}$, where $k\geq 1$ and $i_1\leq i_2\leq\dots\leq i_k$, together with the unit element $\boldsymbol{1}\in U(\mathfrak{g})$ is a basis of $U(\mathfrak{g})$. It follows that the canonical linear map $\mathfrak{g}\to T^1(\mathfrak{g}) \to T(\mathfrak{g}) \to U(\mathfrak{g})$ gives an embedding of $\mathfrak{g}$ into $U(\mathfrak{g})$ and we identify $\mathfrak{g}$ with its image under this map. Since $\mathfrak{g}$ generates $U(\mathfrak{g})$ as an algebra, it is enough to give the structure maps on elements of $\mathfrak{g}$ to define the Hopf structure on $U(\mathfrak{g})$. We define
	\begin{equation*}
		\Delta(x) = x\otimes \boldsymbol{1} + \boldsymbol{1}\otimes x,\quad S(x) = -x,\quad\epsilon(x) = 0 \quad\text{for } x\in \mathfrak{g}\lhook\joinrel\xrightarrow{}U(\mathfrak{g}).
	\end{equation*}
	Note that $U(\mathfrak{g})$ is cocommutative, but commutative only if $\mathfrak{g}$ is abelian. $\odot$
\end{xmpl}
\begin{rem}[algebra-like]
	In fact, $\mathfrak{g}$ can be characterized as the set of \textbf{primitive} (or sometimes \textbf{algebra like}) elements of $U(\mathfrak{g})$, i.e. the set of elements $x\in U(\mathfrak{g})$ for which $\Delta(x)=x\otimes\boldsymbol{1}+\boldsymbol{1}\otimes x$. $\odot$
\end{rem}
\begin{rem}[group-like]
	Because of Example \ref{xmpl:group_algebra}, an element $a$ of an arbitrary Hopf algebra $A$ is called \textbf{group-like} if $\Delta(a)=a\otimes a$. $\odot$
\end{rem}

Now, after introducing Hopf algebras and discussing two important examples, let us explain the use of the "co-structure" and the antipode when it comes to discussing representations of Hopf algebras. In before, we intend to give the basic definition for the term "representation" in this thesis here and explain the possible further constructions in this general setting (cf. \cite{CPBook} pp. 108-111).
\begin{definition}[left $A$-module]
	Let $A$ be an algebra over a commutative ring $k$, a $k$-module $V$ is called a \textbf{left} $\boldsymbol{A}$\textbf{-module} if there is a $k$-module map $\lambda:A\otimes V\to V$ such that the following diagrams commute.
	\begin{equation*}
		\begin{tikzcd}[row sep=large, column sep=large]
			A\otimes A \otimes V \arrow{r}{m\otimes\id_V} \arrow[swap]{d}{\id_A\otimes\lambda} & A\otimes V \arrow{d}{\lambda} \\
			A\otimes V \arrow{r}{\lambda} & V
		\end{tikzcd}\qquad
		\begin{tikzcd}[row sep=large, column sep=large]
			k \otimes V \arrow{r}{\eta\otimes\id_V}  \arrow[swap]{d}{\cong} & A\otimes V \arrow{d}{\lambda} \\
			V \arrow{r}{\id_V} & V
		\end{tikzcd}
	\end{equation*}
	Equivalently, $a\mapsto \lambda(a\otimes \cdot)$ should be a homomorphism of algebras from $A$ into the endomorphsims of $V$. We shall sometimes write $\lambda(a\otimes v)$ as $a.v$ if the action is understood. Right $A$-modules are defined similarly. $\odot$
\end{definition}
\begin{definition}[Hopf algebra representation]
	A \textbf{representation} of a Hopf algebra $A$ over a commutative ring $k$ is a left $A$-module ($A$ being regarded simply as an algebra). The terms subrepresentation, irreducible representation, etc. are as usual. $\odot$
\end{definition}
\begin{rem}
	It is only necessary to consider left $A$-modules, for if $\rho:V\otimes A\to V$ is a right $A$-module, then $\lambda=\rho\circ(\id_V\otimes S)$ is a left $A$-module. $\odot$
\end{rem}
\begin{xmpl}\label{xmpl:trivial_rep}
	The \textbf{trivial representation} of a Hopf algebra $A$ on a $k$-module $V$ is given by
	\begin{equation*}
		\lambda(a\otimes v)=\epsilon(a)v.
	\end{equation*}
	Generally, an element $v$ of an arbitrary representation of $A$ is said to be \textbf{invariant} under the action of $A$ if the preceding equation holds for all $a\in A$. $\odot$
\end{xmpl}
\begin{xmpl}
	The \textbf{adjoint representation} of a Hopf algebra $A$ on itself is given by
	\begin{equation*}
		\operatorname{ad}(a\otimes a') = \sum_{i} a_ia'S(a^i),
	\end{equation*}
	where the comultiplication is given by $\Delta(a)=\sum_{i} a_i\otimes a^i$.\footnote{ Furthermore, this makes $A$ into a left $A$-module algebra (cf. \cite{CPBook} pp. 109-110).} This reduces, of course, in the case of a group algebra $k[G]$ or a universal enveloping algebra $U(\mathfrak{g})$, to the adjoint action $\operatorname{Ad}_G$ of a group $G$ on itself by conjugation or to the adjoint representation $\operatorname{ad}_{\mathfrak{g}}$ of a Lie algebra $\mathfrak{g}$, respectively.\footnote{ Moreover, this makes $A$ into a left $A$-module coalgebra (cf. \cite{CPBook} pp. 109-110).} $\odot$
\end{xmpl}
\begin{xmpl}
	The \textbf{regular representation} of a Hopf algebra $A$ on itself is the multiplication $m$ of $A$. $\odot$
\end{xmpl}

Let us now explain the role of the antipode and the coalgebra structure in representation theory.
In Example \ref{xmpl:trivial_rep} we have already seen that the counit allows one to define the trivial representation of $A$ and we have used the antipode to define the adjoint representation. Now, if $V$ and $W$ are representations of $A$, then $V\otimes W$ is naturally a representation of $A\otimes A$ by
\begin{equation*}
	(a_1\otimes a_2).(v\otimes w)=a_1.v\otimes a_2.w
\end{equation*}
and we can use the comultiplication to make the following definition.
\begin{definition}[the tensor product of representations]\label{def:tensor_prod_of_reps}
	Let $V$ and $W$ be representations of $A$, then we can make the \textbf{tensor product} $V\otimes W$ into a representation of $A$ by defining
	\begin{equation*}
		a.(v\otimes w)=\Delta(a).(v\otimes w).\quad\odot
	\end{equation*}
\end{definition}
Furthermore, we can use the antipode to make $\Hom_k(V,W)$ into a representation in two ways as follows.
\begin{definition}\label{def:Hom_reps}
	If $V$ and $W$ are representations of $A$, the antipode allows us to make $\Hom_k(V,W)$ into a representation in two ways by setting either
	\begin{equation*}
		a.f(v) = \sum_{i} a_i.f(S(a^i).v)\;\text{ or }\; a.f(v) = \sum_{i} a_i.f(S^{-1}(a^i).v),
	\end{equation*}
	for $f\in\Hom_k(V,W),v\in V$, where $\Delta(a) = \sum_{i}a_i\otimes a^i$ as before. $\odot$
\end{definition}
\begin{cor}[the dual representations]\label{cor:the_dual_reps}
	Choosing the trivial representation on $k$ for $W$ in Definition \ref{def:Hom_reps}, we turn the dual $V^\star=\Hom_k(V,k)$ into a representation of $A$ in two ways.\footnote{ Note that $V^\star$ is naturally a right $A$-module, as $V$ is a left $A$-module. Also, we use the symbol $\star$ here to avoid confusion with the notion for dual vector space and dual representation.} We call them the left and the right duals $V^*$ and $^*V$, respectively. Equivalently, the module operation on $V^*$ or $^*V$ can immediately be defined by setting
	\begin{equation*}
		\br{a.v|w}\eqcolon\br{v|S(a).w}\;\text{ or }\;\br{a.v|w}\eqcolon\br{v|S^{-1}(a).w}\quad v\in V^\star,\quad w\in V,
	\end{equation*}
	respectively, where $\br{\cdot|\cdot}$ is the dual pairing between $V^\star$ and $V$. $\odot$
\end{cor}
\subsection{Quasitriangular Hopf algebras}\label{subsect_quasitriang_Hopf}
So far, we have only seen examples of Hopf algebras which are either commutative or cocommutative. In this thesis, however, we are interested in Hopf algebras which are neither commutative nor cocommutative but somewhat close to being cocommutative. For, if a Hopf algebra $A$ is cocommutative ($\Delta = \Delta^{\operatorname{op}}$) and $V,W$ are representations of $A$, we have that the representations $V\otimes W$ and $W\otimes V$ are naturally isomorphic. In our case, almost cocommutative means that there exists an invertible element $\mathcal{R}$ in $A\otimes A$ or some appropriate extension of $A\otimes A$ which intertwines $\Delta$ and $\Delta^{\operatorname{op}}$, i.e. it can be used to identify $V\otimes W$ and $W\otimes V$. In physics, such an intertwiner of representations is called an $R$-matrix and it appeared naturally in the study of solutions of the so-called Yang--Baxter equation which is the local integrability condition for two-dimensional vertex models on a square lattice. Such a solution, however, is oftentimes found only on the level of (the model specific) representations and it is a priori unclear if it can be derived from an element $\mathcal{R}$ sitting in the tensor product $A\otimes A$ of a suitable Hopf algebra $A$. In mathematics, we call such a thing a \textit{universal} $R$\textit{-matrix} and its general properties shall be discussed in this subsection. Again, we refer the reader to Section 4.2 A - C in the book \cite{CPBook} for the proofs and further details.\\
\begin{definition}[almost cocommutative Hopf algebra]\label{def:almost_cocommutative_Hopf_alg}
	A Hopf algebra $A$ over a commutative ring $k$ is said to be almost cocommutative if there exists an invertible element $\mathcal{R}\in A\otimes A$ such that
	\begin{equation}
		\Delta^{\operatorname{op}}(a)=\mathcal{R}\Delta(a)\mathcal{R}^{-1}\label{eqn:almost_cocommutative}
	\end{equation}
	for all $a\in A$. $\odot$
\end{definition}
\begin{rems}\label{rems:almost_cocomm}\hspace{1em}
	\begin{enumerate}
		\item A Hopf algebra which is commutative and almost cocommutative is actually cocommutative.
		\item \label{enum:intertwiner_R_matr} If $\rho_V:A\to \End_k(V)$ and $\rho_W:A\to \End_k(W)$ are representations of an almost cocommutative Hopf algebra $(A,\mathcal{R})$, the tensor products $V\otimes W$ and $W\otimes V$ are isomorphic as representations of $A$. One easily checks that if $\sigma:V\otimes W \to W\otimes V$ is the flip map defined by $v\otimes w\mapsto w\otimes v$, $v\in V,w\in W$, then $\sigma\circ(\rho_V\otimes\rho_W)(\mathcal{R})$ is an isomorphism $V\otimes W\to W\otimes V$ which commutes with the action of $A$.
		\item If $A$ is a topological Hopf algebra over $k[[h]]$, one says that $A$ is \textbf{topologically almost cocommutative} if there exists an element $\mathcal{R}$ in the $h$-adic completion of $A\otimes A$ that satisfies (\ref{eqn:almost_cocommutative}). $\odot$
	\end{enumerate}
\end{rems}
If $A$ is almost cocommutative, then $\mathcal{R}$ is obviously unique up to right multiplication by an element in the centralizer $C$ of $\Delta(A)$ in $A\otimes A$. Moreover, it is clear that $Z\otimes Z \subseteq C$, where $Z$ is the centre of $A$. For the other direction, we can write down the following.
\begin{prop}
	Let $\varphi:A\otimes A \to A$ be the $k$-module map defined by $\varphi(a_1\otimes a_2) = a_1S(a_2)$. Then $\varphi(C)\subseteq Z$. $\odot$
\end{prop}
Now, applying the flip map to both sides of (\ref{eqn:almost_cocommutative}), we deduce that $\mathcal{R}_{21}\mathcal{R}\in C$, and hence that $\varphi(\mathcal{R}_{21}\mathcal{R})\in Z$, where $\mathcal{R}_{21}\coloneq \sigma(\mathcal{R})$. We can therefore find the following.
\begin{prop}\label{prop:inv_elem_u}
	Let $u=m(S\otimes \id)(\mathcal{R}_{21})$. Then, $u$ is an invertible element of $A$ and
	\begin{equation*}
		S^2(a) = uau^{-1}
	\end{equation*}
	for all $a\in A$. $\odot$
\end{prop}
Similarly, one can show that the elements
\begin{equation*}
	u_2=m(S\otimes\id)(\mathcal{R}^{-1}),\quad u_3=m(\id\otimes S^{-1})(\mathcal{R}_{21}),\quad u_4=m(\id\otimes S)(\mathcal{R}^{-1}),
\end{equation*}
have the same property. So let $u\coloneq u_1$, then $S^2(a)=u_1 a u_i^{-1}$ for all $a\in A$ and $i=1,2,3,4$.
\begin{cor}\label{cor:centr_elts}
	The elements $u_i\in A$ commute with each other and the elements $u_i u_j^{-1}$ are in the centre of $A$, for $i,j=1,2,3,4$. $\odot$
\end{cor}
As a straightforward consequence of Proposition \ref{prop:inv_elem_u} we have
\begin{cor}
	The left and the right duals of any representation of an almost cocommutative Hopf algebra are isomorphic. $\odot$
\end{cor}
Moreover, we will later see that under some additional assumptions the element $u$ is useful for defining a trace for representations of $A$ which commutes with the $A$-action.\\

An element $\mathcal{R}$ in (\ref{eqn:almost_cocommutative}) cannot be arbitrary since we know that $A^{\operatorname{op}}$ is a Hopf algebra. Writing down the coassociativity for $\Delta^{\operatorname{op}}$ and applying (\ref{eqn:almost_cocommutative}), we find that a sufficient condition for the coassociativity of $\Delta^{\operatorname{op}}$ is\footnotetext{ When we consider elements in higher tensor powers of $A$, we use indices to clarify how we embed $A^{\otimes k}$ into $A^{\otimes l}$, $0<k<l$. For instance $\mathcal{R}=\sum_{j}r_j\otimes r^j \in A\otimes A$, then $\mathcal{R}_{13} = \sum_{j}r_j\otimes \boldsymbol{1} \otimes r^j \in A^{\otimes3}$ and $\mathcal{R}_{31} = (\sigma(\mathcal{R}))_{13} = (R_{21})_{13}=\sum_{j}r^j\otimes \boldsymbol{1} \otimes r_j \in A^{\otimes3}$.}
\begin{equation}
	\mathcal{R}_{12}(\Delta\otimes\id)(\mathcal{R}) = \mathcal{R}_{23}(\id\otimes\Delta)(\mathcal{R}).\label{eqn:R_matrix_coassociativity_cond}\footnotemark
\end{equation}
Similarly, a sufficient condition for the compatibility $(\epsilon\otimes\id)\circ\Delta^{\operatorname{op}}=\id = (\id\otimes\epsilon)\circ\Delta^{\operatorname{op}}$ of the counit to hold is that $(\epsilon\otimes\id)(\mathcal{R}) = \boldsymbol{1} = (\id\otimes\epsilon)(\mathcal{R})$. We make a slightly stronger assumption which turns out to be convenient.
\begin{definition}[quasitriangular, coboundary and triangular]
	An almost cocommutative Hopf algebra $(A,\mathcal{R})$ is said to be
	\begin{enumerate}[label=\normalfont(\roman*)]
		\item \textbf{quasitriangular} if
		\begin{equation*}
			(\Delta\otimes\id)(\mathcal{R})=\mathcal{R}_{13}\mathcal{R}_{23},
		\end{equation*}
		\begin{equation*}
			(\id\otimes\Delta)(\mathcal{R})=\mathcal{R}_{13}\mathcal{R}_{12};
		\end{equation*}
		\item \textbf{coboundary} if $\mathcal{R}$ satisfies (\ref{eqn:R_matrix_coassociativity_cond}), $\mathcal{R}_{21}=\mathcal{R}^{-1}$ and $(\epsilon\otimes\epsilon)(\mathcal{R})=1$;
		\item \textbf{triangular} if $\mathcal{R}$ is quasitriangular and $\mathcal{R}_{21}=\mathcal{R}^{-1}$.
	\end{enumerate}
	If $A$ is quasitriangular, the element $\mathcal{R}$ is called the \textbf{universal $R$-matrix} of $(A,\mathcal{R})$. $\odot$
\end{definition}
\begin{rems}
	\begin{enumerate}
		\item If $(A,\mathcal{R})$ is quasitriangular, so is $(A,\mathcal{R}_{21}^{-1})$.
		\item If $(A,\mathcal{R})$ is quasitriangular, so is $(A_{\operatorname{op}},\mathcal{R}_{21})$, $(A^{\operatorname{op}},\mathcal{R}_{21})$ and $(A_{\operatorname{op}}^{\operatorname{op}},\mathcal{R})$. $\odot$
	\end{enumerate}
\end{rems}
If $A$ is a quasitriangular Hopf algebra, the central elements constructed in Corollary \ref{cor:centr_elts} reduce essentially to the single element $uS(u)$. In fact, it can be shown that $u_1=u_3 = u$ and $u_2=u_4=S(u)^{-1}$. Therefore, the central elements $u_iu_j^{-1}$ are either trivial or equal to $uS(u)$ or its inverse. Furthermore, Proposition \ref{prop:inv_elem_u} and the next proposition imply that
\begin{equation}\label{eqn:Delta(u)_property}
	\Delta(u) = (\mathcal{R}_{21}\mathcal{R}_{12})^{-1}(u\otimes u)=(u\otimes u)(\mathcal{R}_{21}\mathcal{R}_{12})^{-1}.
\end{equation}
\begin{prop}\label{prop:properties_quasitri}
	Let $(A,\mathcal{R})$ be a quasitriangular Hopf algebra. Then
	\begin{align}
		&\mathcal{R}_{12}\mathcal{R}_{13}\mathcal{R}_{23}=\mathcal{R}_{23}\mathcal{R}_{13}\mathcal{R}_{12},\label{eqn:Yang--Baxter}\\
		&(\epsilon\otimes\id)(\mathcal{R}) = \boldsymbol{1} = (\id\otimes\epsilon)(\mathcal{R}),\\
		&(S\otimes\id)(\mathcal{R}) = \mathcal{R}^{-1} = (\id\otimes S^{-1})(\mathcal{R}),\label{eqn:abstr_crossing}\\
		&(S\otimes S)(\mathcal{R}) = \mathcal{R}. \quad\odot
	\end{align}
\end{prop}
With Proposition \ref{prop:properties_quasitri} it is also clear that every triangular Hopf algebra is coboundary. Equation (\ref{eqn:Yang--Baxter}) is called the (quantum) \textbf{Yang--Baxter equation}.

\begin{rem}\label{rem:left_mult_intertwiner_R-matrix}
	If we denote left multiplication with $\mathcal{R}\in A\otimes A$ by $m_{\mathcal{R}}$ and define the string diagram for $\sigma\circ m_{\mathcal{R}}:A\otimes A\to A\otimes A$ by
	\begin{figure}[H]
		\centering
		\begin{tikzpicture}[scale=1]
			\draw[thick] (1,0) -- (0,1);
			\draw[thick] (0,0) -- (1,1);
			\node at (1.5,0) {,};
		\end{tikzpicture}
	\end{figure}
	then the string diagram for the (quantum) Yang--Baxter equation is given by
	\begin{figure}[H]
		\centering
		\begin{tikzpicture}[scale=1]
			\draw[thick] (2,0) -- (1,1);
			\draw[thick] (1,0) -- (2,1);
			
			\draw[thick] (2,1) -- (2,2);
			\draw[thick] (0,0) -- (0,1);
			\draw[thick] (0,2) -- (0,3);
			
			\draw[thick] (1,1) -- (0,2);
			\draw[thick] (0,1) -- (1,2);
			
			\draw[thick] (2,2) -- (1,3);
			\draw[thick] (1,2) -- (2,3);
			\node at (3,1.5) {=};
			
			\draw[thick] (4,0) -- (5,1);
			\draw[thick] (5,0) -- (4,1);
			
			\draw[thick] (4,1) -- (4,2);
			\draw[thick] (6,0) -- (6,1);
			\draw[thick] (6,2) -- (6,3);
			
			\draw[thick] (5,1) -- (6,2);
			\draw[thick] (6,1) -- (5,2);
			
			\draw[thick] (4,2) -- (5,3);
			\draw[thick] (5,2) -- (4,3);
			\node at (6.5,0) {.};
		\end{tikzpicture}
	\end{figure}
	In fact, by Equation (\ref{eqn:almost_cocommutative}), the operator $\sigma\circ m_{\mathcal{R}}$ is designed such that it commutes with the left multiplication $m_{\Delta(a)}:A\otimes A \to A\otimes A$ by any comultiplication $\Delta(a)$, for all $a\in A$. Thus, it commutes with taking tensor products of representations. Indeed, this explains the \ref{enum:intertwiner_R_matr}nd of the Remarks \ref{rems:almost_cocomm}. Furthermore we will clarify its categorical interpretation in the next subsection. $\odot$
\end{rem}

Combining Equation (\ref{eqn:abstr_crossing}) and the dual representations in Corollary \ref{cor:the_dual_reps}, we obtain equations for the corresponding $R$-matrices as follows.
\begin{cor}[crossing relations]
	Let $(V,\pi_V)$ and $(W,\pi_W)$ be two representations of a quasitriangular Hopf algebra $A$ and denote by $\;^t:\End(U)\to\End(U^*)$ the transposition map defined for any vector space $U$ and $M\in \End(U)$ by
	\begin{equation*}
		\langle M^t t|u\rangle = \langle v|Mw\rangle,\quad\text{for }t\in U^*, u\in U.
	\end{equation*}
	Then, by applying $(S\otimes\id)(\mathcal{R})=\mathcal{R}^{-1}$ we obtain the relation
	\begin{equation}
		(\pi_{V^*}\otimes\pi_W)(\mathcal{R}) = (\pi_{V}\otimes\pi_W)(\mathcal{R}^{-1})^{t_1},\label{eqn:crossing_1}
	\end{equation}
	and by applying $(\id\otimes S^{-1})(\mathcal{R})=\mathcal{R}^{-1}$ we obtain
	\begin{equation}
		(\pi_{V}\otimes\pi_{^*W})(\mathcal{R}) =\label{eqn:crossing_2} (\pi_{V}\otimes\pi_W)(\mathcal{R}^{-1})^{t_2}.\footnotemark
	\end{equation}\footnotetext{Where $t_1$ and $t_2$ denote $t\otimes 1$ and $1\otimes t$, respectively. I.e. transposition in the first resp. second component.}
	Equations (\ref{eqn:crossing_1}) and (\ref{eqn:crossing_2}) and combinations (or iterations) of them are sometimes called \textbf{crossing relations} for the $R$-matrix $R_{VW} = (\pi_{V}\otimes\pi_W)(\mathcal{R})$ (cf. \cite{NR}). $\odot$
\end{cor}
Let us make another refinement.
\begin{definition}[ribbon Hopf algebra]
	A \textbf{ribbon Hopf algebra} $(A,\mathcal{R},v)$ is a quasitriangular Hopf algebra $(A,\mathcal{R})$ equipped with an invertible central element $v$ such that
	\begin{equation*}
		\begin{split}
			v^2 = uS(u),\;S(v)=v,\;\epsilon(v)=1,\\
			\Delta(v)=(\mathcal{R}_{21}\mathcal{R}_{12})^{-1}(v\otimes v),
		\end{split}
	\end{equation*}
	where $u=m(S\otimes \id)(\mathcal{R}_{21})$.
\end{definition}
Note that, if $(A,\mathcal{R})$ is triangular, we have $S(u) = u^{-1}$ and together with (\ref{eqn:Delta(u)_property}) we conclude that $(A,\mathcal{R},1)$ is a ribbon Hopf algebra in this case.

Although, not every quasitriangular Hopf algebra is a ribbon Hopf algebra, every quasitriangular Hopf algebra $A$ can be enlarged to a ribbon Hopf algebra by formally adjoining a square root of $uS(u)$ (cf. \cite{CPBook} p. 126). However, having a ribbon Hopf algebra, we can make the following Definition.
\begin{definition}[quantum trace]
	Let $(A,\mathcal{R},v)$ be a ribbon Hopf algebra and let $\rho:A\to \End_k(V)$ be a representation of $A$ which is a free $k$-module of finite rank. If $f:V\to V$ is a $k$-linear map, its quantum trace is defined to be
	\begin{equation*}
		\operatorname{qtr}(f)=\operatorname{trace}(\rho(g)f),
	\end{equation*}
	where $g=v^{-1}u$ and $u=m(S\otimes\id)(\mathcal{R}_{21})$. In particular, the quantum dimension of $V$ is
	\begin{equation*}
		\operatorname{qdim}(V)=\operatorname{trace}(\rho(g)).\quad\odot
	\end{equation*}
\end{definition}
Note that a cocommutative Hopf algebra is always triangular with $\mathcal{R}=\boldsymbol{1}\otimes\boldsymbol{1}$. Then we can take $u=v=g=\boldsymbol{1}$ and the quantum trace becomes the usual trace.
\begin{rems}\hspace{1em}
	\begin{enumerate}
		\item The quantum trace is multiplicative, i.e. if $V$ and $W$ are representations of $A$ and $f\in\End_k(V)$, $g\in\End_k(W)$, then the quantum trace of $f\otimes g\in\End_k(V\otimes W)$ is $\operatorname{qtr}(f\otimes g)=\operatorname{qtr}(f)\operatorname{qtr}(g)$.
		\item The quantum trace $\operatorname{qtr}:\End_k(V)\to k$ is a homomorphism of representations of $A$, where $\End_k(V)$ is made into a representation of $A$ as in Definition \ref{def:Hom_reps} (via $S$) and $k$ is the trivial representation. In fact, $\operatorname{qtr}$ is the composite
		\begin{equation*}
			\End_k(V)\cong V\otimes V^*\cong V^{**}\otimes V^*\to k
		\end{equation*}
		where each map is a homomorphism of representations of $A$ and, except the isomorphism $V\to V^{**}$, canonical. The isomorphism $V\to V^{**}$ is defined by $v\mapsto v'$, where $\langle v'|\xi\rangle=\langle\xi|\rho(g)(v)\rangle$ for all $v\in V^*$. The fact that this commutes with the action of $A$ is equivalent to the statement that $S^2$ is given by the conjugation with $g$. $\odot$
	\end{enumerate}
\end{rems}

\subsection{Abelian monoidal categories}\label{subsect:Ab_mon_cat}
Finally, before we describe the kind of Hopf algebras we are interested in and draw the bow to the first section about Kac--Moody algebras, let us loose a few general words about what all this has to do with (rigid) abelian monoidal, quasitensor (or braided monoidal), and tensor categories and why they are so interesting for us. However, we leave it to the reader to look up the precise definitions and instead focus on explaining some useful consequences (cf. \cite{CPBook} Chapter 5).\\

Let us give a similar but shorter description of what is written in the book \cite{CPBook} on pages 135-136 as follows.
It is well known, that the properties of the operation of taking direct sums (of vector spaces, representations of groups, rings, Lie algebras, $\dots$) can be abstracted in the notion of an \textbf{abelian category}. Abstracting the properties of the tensor product operation leads to the notion of a \textbf{monoidal category}. Although, the basic theory of such categories was developed in the early 1960s by S. MacLane, it has only attracted a wider audience largely because of the discovery of interesting new examples. It is no surprise that many of these are related to \textbf{quantum groups}, which are non-trivial deformations (quantisations) of universal enveloping algebras of semisimple Lie algebras introduced independently in 1986 by V. Drinfeld and M. Jimbo. Indeed, we have just seen that, from an abstract point of view, a Hopf algebra may be regarded as an algebra $A$ equipped with the extra structure necessary to define tensor products (and duals) of representations of $A$.

In many examples of monoidal categories such as the triangular Hopf algebras, the tensor product operation is commutative up to involutive isomorphisms. Then, the monoidal category is also called a \textbf{tensor category}. However, if we weaken 'triangular' to 'quasitriangular', the tensor product is still commutative, but the commutativity isomorphisms don't have to be involutive anymore. Such a monoidal category is then called a \textbf{quasitensor} (or \textbf{braided monoidal}) \textbf{category}.

Considering the example of a semisimple abelian quasitensor category, the tensor product is determined by the decomposition into simple objects of the tensor product of any two simple objects and therefore by a collection of non-negative integers satisfying conditions which reflect the commutativity and associativity properties of the tensor product. Such combinatorial data was encountered in physics in the study of conformal field theories, where they are called \textbf{fusion rules}.

Moreover, Ribbon tangles can be organized into a quasitensor category and the invariants can be obtained by constructing a functor from this category into the category of representations of a particular type of quasitriangular Hopf algebra. In a special case one obtains the invariants of links in $\mathbb{R}^3$ and a method to construct the \textbf{Jones polynomial}.\\

As we have mentioned above, the notion of an abelian category $\mathcal{C}$ is meant to capture the properties of direct sums of vector spaces, representations of groups, Lie algebras, etc. To be more precise, an abelian category is defined such that we have the fundamental homomorphism theorem. Therefore, every morphism must have a \textbf{kernel} and a \textbf{cokernel}, every \textbf{monomorphism} is the kernel of its cokernel and every \textbf{epimorphism} is the cokernel of its kernel. Then, every morphism is expressible as the composite of an epimorphism followed by a monomorphism. Moreover, if a morphism in $\mathcal{C}$ is both a mono- and an epimorphism, it is an isomorphism.

Hence, we can generalize the notion of a short exact sequence. For, if
\begin{equation*}
	U\stackrel{f}{\to}W\stackrel{g}{\to}V
\end{equation*}
is a pair of morphisms such that $g\circ f=0$, we call it a \textbf{short exact sequence} if the morphisms $U\to \ker(g)$ and $\operatorname{coker}(f)\to V$, given by the defining properties of kernel and cokernel, are isomorphisms.

Furthermore, the concepts of an \textbf{irreducible}, or \textbf{completely reducible}, representation have analogues in any abelian category $\mathcal{C}$. An object $U$ in $\mathcal{C}$ is \textbf{simple} if every non-zero monomorphism $V\to U$ and every non-zero epimorphism $U\to W$ are isomorphisms. Any non-zero morphism between simple objects is an isomorphism. An object is \textbf{semisimple} if it is a direct sum of simple objects. The category $\mathcal{C}$ itself is called semisimple if every object of $\mathcal{C}$ is semisimple. In a semisimple abelian category, every short exact sequence splits.

The direct sum operation of an abelian category $\mathcal{C}$ is encoded in its \textbf{Grothendieck group} $\operatorname{Gr}(\mathcal{C})$. It is defined as follows.
\begin{definition}[Grothendieck group]
	Let $\mathcal{F}(\mathcal{C})$ be the free abelian group with generators the isomorphism classes of objects of $\mathcal{C}$, and let $[U]$ be the element of $\mathcal{F}(\mathcal{C})$ corresponding to an object $U$ of $\mathcal{C}$. Then, $\operatorname{Gr}(\mathcal{C})$ is the quotient of $\mathcal{F}(\mathcal{C})$ by the subgroup generated by the elements of the form $[W]-[U]-[V]$ for all possible short exact sequences
	\begin{equation*}
		U\to W\to V
	\end{equation*}
	in $\mathcal{C}$. $\odot$
\end{definition}
Note that, if $\mathcal{C}$ is semisimple, the defining relations of $\operatorname{Gr}(\mathcal{C})$ are simply $[U\oplus V] = [U]+[V]$. In this case, $\operatorname{Gr}(\mathcal{C})$ is generated as an abelian group by the simple objects of $\mathcal{C}$.

Coming back to the axiomatization of the properties of the tensor product, a monoidal category is defined as follows.
\begin{definition}[monoidal category]
	A \textbf{monoidal category} is a category $\mathcal{C}$ together with a functor $\otimes:\mathcal{C}\times\mathcal{C}\to\mathcal{C}$, written $(U,V)\mapsto U\otimes V$ for objects $U$ and $V$ of $\mathcal{C}$, which satisfies the conditions
	\begin{enumerate}[label=\normalfont(\roman*)]
		\item there are natural (associativity) isomorphisms $\alpha_{U,V,W}:U\otimes(V\otimes W)\to (U\otimes V)\otimes W$ such that
		\begin{equation*}
			\alpha_{U\otimes V,W,Z}\circ\alpha_{U,V,W\otimes Z} = (\alpha_{U,V,W}\otimes\id)\circ \alpha_{U,V\otimes W,Z}\circ (\id\otimes\alpha_{V,W,Z})\quad\text{for all }U,V,W,Z,
		\end{equation*}
		\item there is an identity object $\boldsymbol{1}$ in $\mathcal{C}$ and natural isomorphisms $\rho_U:U\otimes\boldsymbol{1}\to U$ and $\lambda_U:\boldsymbol{1}\otimes U\to U$ such that
		\begin{equation*}
			(\rho\otimes\id)\circ\alpha_{U,\boldsymbol{1},V} = \id\otimes(\id\otimes\lambda)\quad\text{for all }U,V.\quad \odot
		\end{equation*}
	\end{enumerate}
\end{definition}
With this, one can define an \textbf{abelian monoidal category} as a monoidal category which is also an abelian category such that $\otimes$ is a bi-additive functor. Functors between monoidal categories are defined such that the tensor product structure is preserved up to natural isomorphisms. We refer the reader to the book \cite{CPBook} for the precise definition.

Moreover, if $\mathcal{C}$ is an abelian monoidal category, its Grothendieck group $\operatorname{Gr}(\mathcal{C})$ is a unital ring with product given by
\begin{equation*}
	[U][V]=[U\otimes V],
\end{equation*}
and unit $[\boldsymbol{1}]$. The associativity of this product, and the basic property $[\boldsymbol{1}][U]=[U][\boldsymbol{1}] = [U]$ of the unit, follow from the existence of the isomorphisms $\alpha$, $\lambda$ and $\rho$. We will therefore call $\operatorname{Gr}(\mathcal{C})$ the \textbf{Grothendieck ring} of the abelian monoidal category $\mathcal{C}$.\\

If we incorporate the notion of a dual of a representation into the categorical framework, we have the notion of \textbf{rigidity} (cf. \cite{CPBook} Subsection 5.1C).

If $\mathcal{C}$ is a monoidal category, an object $U^*$ is said to be a (left) \textbf{dual} of an object $U$ in $\mathcal{C}$ if there are morphisms $\operatorname{ev}_U:U^*\otimes U\to \boldsymbol{1}$ and $\pi_U:\boldsymbol{1}\to U\otimes U^*$ such that $(\pi\otimes\id)\circ\id\circ(\id\otimes\operatorname{ev})=\alpha$ on $U\otimes (U^*\otimes U)$ and $(\operatorname{ev}\otimes\id)\circ\alpha\circ(\id\otimes\pi)=\id$ on $U^*$, where we have suppressed the isomorphisms $\lambda$ and $\rho$ for simplicity. If $U$ has a dual object, it is unique up to isomorphism.

Suppose now that every object in $\mathcal{C}$ has a dual object. If $\psi:U\to V$ is a morphism in $\mathcal{C}$, then the composite
\begin{equation*}
	V^*=V^*\otimes\boldsymbol{1}\xrightarrow{\id\otimes\pi}V^*\otimes U\otimes U^*\xrightarrow{\id\otimes\psi\otimes\id}V^*\otimes V\otimes U^*\xrightarrow{\operatorname{ev}\otimes\id}\boldsymbol{1}\otimes U^*=U^*
\end{equation*}
is a morphism $\psi^*:V^*\to U^*$ (we are now abusing notation by omitting the associativity maps). Thus, we obtain a contravariant functor $\mathcal{C}\to\mathcal{C}$, called the \textbf{dual object functor}, which takes $U$ to $U^*$ and $\psi$ to $\psi^*$.
\begin{definition}[rigid]
	A monoidal category $\mathcal{C}$ is \textbf{rigid} if every object in $\mathcal{C}$ has a dual object, and if the dual object functor $\mathcal{C}\to\mathcal{C}$ is an anti-equivalence of categories. $\odot$
\end{definition}
\begin{rems}\hspace{1em}
	\begin{enumerate}
		\item If $\mathcal{C}$ is a rigid monoidal category, it is easy to see that, for all objects $U$, $V$ and $W$ of $\mathcal{C}$, the map which assigns to a morphism $\psi:U\otimes V\to W$ the morphism $U\to W\otimes V^*$ given by
		\begin{equation*}
			U=U\otimes\boldsymbol{1}\xrightarrow{\id_U\otimes\pi_V}U\otimes V\otimes V^*\xrightarrow{\psi\otimes\id_{V^*}}W\otimes V^*
		\end{equation*}
		is bijective. In particular, there is a bijection between morphisms $U\otimes V\to \boldsymbol{1}$ and morphisms $U\to V^*$, and between morphisms $V\to W$ and morphisms $\boldsymbol{1}\to W\otimes V^*$.
		\item The dual object functor in a rigid monoidal category has a natural monoidal structure. In fact, the composite
		\begin{equation*}
			(V^*\otimes U^*)\otimes(U\otimes V) \xrightarrow{\id_{V^*}\otimes\operatorname{ev}_U\otimes\id_V}V^*\otimes\boldsymbol{1}\otimes V= V^*\otimes V\xrightarrow{\operatorname{ev}_V}\boldsymbol{1}
		\end{equation*}
		corresponds to a morphism $V^*\otimes U^*\to (U\otimes V)^*$, whereas
		\begin{equation*}
			\boldsymbol{1}\xrightarrow{\pi_U} U\otimes U^* = U\otimes\boldsymbol{1}\otimes U^* \xrightarrow{\id_U\otimes\pi_V\otimes\id_{U^*}}(U\otimes V)\otimes (V^*\otimes U^*)
		\end{equation*}
		corresponds to a morphism $W\to U\otimes V$, such that $W^*=V^*\otimes U^*$. It is easy to see that the dual of the latter morphism is the inverse of the former, thus $(U\otimes V)^*$ is naturally isomorphic to $V^*\otimes U^*$.
		\item Note also that if $U$ is an object of a rigid category $\mathcal{C}$, $U^{**}$ is not necessarily isomorphic to $U$. In the case where $U^{**}$ is isomorphic to $U$ for any $U$, the rigid category $\mathcal{C}$ is called reflexive.
	\end{enumerate}
\end{rems}
Let us give a few examples.
\begin{xmpl}\label{xmpl:k-modules_monoid_cat}
	The category of all modules over a commutative ring $k$ is an abelian monoidal category with the obvious associativity maps and the identity object $k$. The same is true for its full subcategory $\boldsymbol{\operatorname{mod}}_k$ consisting of the finitely-generated projective modules. However, only the latter category is rigid (even if $k$ is a field). The dual of a $k$-module $U$ is $U^*=\Hom_k(U,k)$ as usual, and the map $\operatorname{ev}_U$ is the obvious dual pairing which evaluates an element of $U^*$ on an element of $U$. To define $\pi_U$, we need the 'dual basis lemma' for finitely-generated projective modules, according to which, for any set $\{u_i\}_{i\in I}$ of generators of $U$, there exists a set of elements $\{\xi^i\}_{i\in I}$ in $U^*$ such that $u=\sum_{i\in I}u_i \xi^i(u)$ for all $u\in U$. Then, we set $\pi_U(\boldsymbol{1})=\sum_{i\in I}u_i\otimes\xi^i$.
	
	Two special cases are of particular interest here. If $k$ is a field, $\boldsymbol{\operatorname{mod}}_k$ is simply the category of finite-dimensional vector spaces over $k$ and the dual basis lemma gives the dual basis in the usual elementary sense. Whereas if $k=K[[h]]$ is the ring of formal power series in an indeterminate $h$ over a field $K$, then, since $K[[h]]$ is a local ring, $\boldsymbol{\operatorname{mod}}_{K[[h]]}$ consists of the free $K[[h]]$-modules of finite rank. We shall therefore call such $K[[h]]$-modules 'finite-dimensional'. $\odot$
\end{xmpl}
The following example is the main motivation of the discussion of (rigid) abelian monoidal categories in this subsection.
\begin{xmpl}\label{xmpl:Hopf_algebra_monoid_cat}
	If $A$ is a Hopf algebra over a commutative ring $k$, the category of all representations of $A$ is an abelian monoidal category with the same associativity maps as in Example \ref{xmpl:k-modules_monoid_cat}. The identity object is $k$, regarded as the trivial representation of $A$. The full subcategory $\boldsymbol{\operatorname{rep}}_A$ consisting of the representations of $A$ on finitely-generated projective $k$-modules is also rigid abelian monoidal. The dual of a representation $U$ of $A$ is the $k$-module dual $U^*$ defined in Example \ref{xmpl:k-modules_monoid_cat}, provided with the action of $A$ given by
	\begin{equation*}
		\langle a.\xi|u\rangle = \langle \xi|S(a).u\rangle,
	\end{equation*}
	for $a\in A$, $u\in U$ and $\xi\in U^*$, where $S$ is the antipode of $A$. The forgetful functor $\boldsymbol{\operatorname{rep}}_A\to\boldsymbol{\operatorname{mod}}_k$ is monoidal. $\odot$
\end{xmpl}
We are now in the position to give a more precise definition of tensor and quasitensor categories.
\begin{definition}[quasitensor category]\label{def:quasitensor_cat}
	A \textbf{quasitensor category} is a monoidal category $\mathcal{C}$ which is equipped with natural (braiding) isomorphisms $\sigma_{U,V}:U\otimes V\to V\otimes U$, for all objects $U$ and $V$ of $\mathcal{C}$, such that
	\begin{enumerate}[label=\normalfont(\roman*)]
		\item $\sigma$ and $\alpha$ are compatible, i.e.
		\begin{equation*}
			\begin{split}
				&\alpha_{W,U,V}\circ\sigma_{U\otimes V,W}\circ\alpha_{U,V,W} = (\sigma_{U,W}\otimes\id)\circ\alpha_{U,W,V}\circ(\id\otimes\sigma_{V,W})\quad\text{and}\\
				&(\id\otimes\sigma_{U,W})\circ(\alpha_{V,U,W})^{-1}\circ(\sigma_{U,V}\otimes\id)=(\alpha_{V,W,U})^{-1}\circ\sigma_{U, V\otimes W}\circ(\alpha_{U,V,W})^{-1} \quad\text{for all }U,V,W,
			\end{split}
		\end{equation*}
		\item $\sigma$ and $\rho,\lambda$ are compatible, i.e.
		\begin{equation*}
			\begin{split}
				&\rho_U \circ \sigma_{\boldsymbol{1}\otimes U} = \lambda_U\quad\text{and}\\
				& \lambda_U\circ\sigma_{U\otimes\boldsymbol{1}} =\rho_U \quad\text{for all }U.\quad\odot
			\end{split}
		\end{equation*}
	\end{enumerate}
\end{definition}
\begin{definition}[tensor category]
	A \textbf{tensor category} is a quasitensor category where the braiding isomorphisms satisfy the additional constraint
	\begin{equation*}
		\sigma_{V,U}\circ\sigma_{U,V}=\id_{U\otimes V} \quad\text{for all } U,V. \quad\odot
	\end{equation*}
\end{definition}
Note that in the case of a tensor category, the two equalities in Definition \ref{def:quasitensor_cat} (i) and (ii) are equivalent.
$\mathcal{C}$ is called an abelian (quasi)tensor category if $\mathcal{C}$ is an abelian monoidal category and $\sigma$ is bi-additive. Note that in this case, the Grothendieck Ring $\operatorname{Gr}(\mathcal{C})$ is commutative. Furthermore, a monoidal functor $\Phi:\mathcal{C}\to\mathcal{C}'$ between (quasi)tensor categories is a (quasi)tensor functor if it commutes with the action of the braiding isomorphisms up to the natural isomorphisms given in the definition of monoidal functor, i.e. $\Phi(\sigma_{U,V}) \cong \sigma'_{\Phi(U),\Phi(V)}$ for all $U,V$.

Finally, let us give an example which spans the bow between triangular (resp. quasitriangular) Hopf algebras and tensor (resp. quasitensor) categories.
\begin{xmpl}
	Of course, tensor (resp. quasitensor) categories are designed such that the category of representations of any triangular (resp. quasitriangular) Hopf algebra $(A,\mathcal{R})$ on free modules of finite rank is a tensor (resp. quasitensor) category. However, the braiding (or commutativity) maps $\sigma_{U,V}$ are not obvious. For, if $\rho_U:A\to\End(U)$ and $\rho_V:A\to\End(V)$ are representations of $A$, only the choice
	\begin{equation*}
		\sigma_{U,V} = \sigma\circ(\rho_U\otimes\rho_V)(\mathcal{R})
	\end{equation*}
	commutes with the action of $A$, where the $\sigma$ on the right side is the usual flip map as defined in Section \ref{subsect:Hopf_algebras} (cf. \cite{CPBook} p. 151).\footnote{ Instead of $\sigma_{U,V}$ it is convenient to write $\check{R}_{U,V}=\sigma\circ R_{U,V}$ in this case. This is because one usually defines the (physical) $R$-matrix corresponding to the representation $\rho_U\otimes\rho_V$ as $R_{U,V}\coloneq(\rho_U\otimes\rho_V)(\mathcal{R})$.} In fact, we have seen in the last section that the condition that $(A,\mathcal{R})$ is almost cocommutative is sufficient to obtain a quasitensor category. $\odot$
\end{xmpl}
\subsection{Quantized universal enveloping algebras}\label{subsect:QUE_algebras}
The most important examples of quantum groups are deformations of universal enveloping algebras and of algebras of functions on groups. In this section, however, we will focus on giving the first definition of a quantum Kac--Moody algebra and introduce their rational form (cf. \cite{CPBook} Section 6.5, Chapter 8 and Section 9.1) rather than going through the whole theory of deformations of Hopf algebras, described in Chapter 6 in the book \cite{CPBook}. Nevertheless, we will certainly loose a few words about the general idea.

Before we do this, however, let us cite the physics motivation (in the introduction of \cite{CPBook}) of the word 'quantum group', which explains what the term 'quantum' and 'quantization' has to do with our whole discussion about Hopf algebras, non-cocommutativity and why it is interesting to consider deformations of Hopf algebras.

\begin{quote}
	We shall now describe what a quantum group is, beginning by trying to explain the motivation for the use of the adjective 'quantum'.
	
	In classical mechanics, the phase space $M$ of a dynamical system is a \textit{Poisson manifold}. This means that the space $\mathcal{F}(M)$ of (differentiable) complex valued functions on $M$ is equipped with a Lie bracket $\{,\}:\mathcal{F}(M)\times\mathcal{F}(M)\to\mathcal{F}(M)$ (satisfying certain additional additional conditions), called the Poisson bracket. The dynamical equations defining the time evolution of the system are equivalent to the equations
	\begin{equation*}
		\frac{d}{dt}f(m(t))=\{\mathcal{H}_{\text{cl}},f\}(m(t))
	\end{equation*}
	for $f\in \mathcal{F}(M)$, where $\mathcal{H}_{\text{cl}}$ is a fixed function on $M$ called the (classical) hamiltonian, and $m(t)\in M$ is the 'state' of the system at time $t$. For example, for a single particle moving along the real line, $M$ is the cotangent bundle $T^*(\mathbb{R})$, and if $q$ is the coordinate on $\mathbb{R}$ ('position') and $p$ the coordinate in the fibre direction ('momentum'), the Poisson bracket is
	\begin{equation*}
		\{f_1,f_2\}=\frac{\partial f_1}{\partial p}\frac{\partial f_2}{\partial q}-\frac{\partial f_2}{\partial p}\frac{\partial f_1}{\partial q}.
	\end{equation*}
	in particular, the Poisson bracket of the coordinate functions is
	\begin{equation}
		\{p,q\} = 1.\label{eqn:pq}
	\end{equation}
	In quantum mechanics, the space $M$ is replaced by the set of rays in a complex Hilbert space $V$, and the space $\mathcal{F}(M)$ of functions on $M$ by the algebra $\operatorname{Op}(V)$ of (not necessarily bounded) operators on $V$. The time evolution of an operator $A$ is given by
	\begin{equation*}
		\frac{dA}{dt} = [\mathcal{H}_{\text{qu}},A]
	\end{equation*}
	for some operator $\mathcal{H}_{\text{qu}}\in \operatorname{Op}(V)$, called the (quantum) hamiltonian. For example, in the case of a single particle moving along the real line, $V$ is the space $\operatorname{L}^2(\mathbb{R})$ of square integrable functions of $q$, and the operators $P$ and $Q$ corresponding to the coordinate functions $p$ and $q$ are given by
	\begin{equation*}
		P = -\sqrt{-1}h\frac{\partial}{\partial q},\qquad Q=\text{mutiplication by }q,
	\end{equation*}
	where h is $i/2\pi$ times Planck's constant. Note that
	\begin{equation}
		[P,Q] = -\sqrt{-1}h\id_V.\label{eqn:PQ}
	\end{equation}
	The question is: how to pass from the classical to the quantum description of a system. This is the problem of \textit{quantization}. Ideally, one would like a map $\mathcal{Q}$ which assigns to each function $f\in\mathcal{F}(M)$ an operator $\mathcal{Q}(f)$ on $V$. Moreover, since time evolution in the classical and quantum descriptions is given by taking the Poisson bracket and the commutator with the hamiltonian, respectively, $\mathcal{Q}$ should satisfy the relation
	\begin{equation*}
		\mathcal{Q}\{f_1,f_2\}=\frac{[\mathcal{Q}(f_1),\mathcal{Q}(f_2)]}{-\sqrt{-1}h}
	\end{equation*}
	(the normalization comes from (\ref{eqn:pq}) and (\ref{eqn:PQ})). Unfortunately, it is known that, even for the simplest case of a single particle moving along the real line, no such map $\mathcal{Q}$ exists.\\
	
	There is, however, an alternative formulation of the quantization problem, introduced by J. E. Moyal in 1949. This begins by noting that the fundamental difference between the classical and quantum descriptions is that $\mathcal{F}(M)$ is a commutative algebra, whereas $\operatorname{Op}(V)$ is non-commutative (when $\operatorname{dim}(V)>1$). Moyal's idea is to try to reproduce the results of quantum mechanics by replacing the usual product on $\mathcal{F}(M)$ by a non-commutative product $*_h$, depending on a parameter $h$, such that $*_h$ becomes the usual product as $h\to 0$, just as 'quantum mechanics becomes classical mechanics as Planck's constant tends to zero', and such that
	\begin{equation}
		\lim_{h\to 0}\frac{f_1*_hf_2-f_2*_hf_1}{h}=\{f_1,f_2\}.\label{eqn:Moyal_deformation}
	\end{equation}
	If we think of $\mathcal{F}(M)$ with the Moyal product $*_h$ as a non-cocommutative algebra of functions $\mathcal{F}_h(M)$, we find ourselves in the realm of non-cocommutative geometry in the sense of A. Connes. The philosophy here is that any 'space' is determined by the algebra of functions on it (with the usual product). For example, every affine algebraic variety over $\mathbb{C}$ is determined (up to isomorphism) by the commutative algebra of regular functions on it, whereas every compact topological space is determined by its commutative $\operatorname{C}^*$-algebra of complex-valued continuous functions. More precisely, the category of 'spaces' in these examples is dual to the category of the corresponding algebras. Thus, a non-commutative algebra should be viewed as the space of functions on a 'non-cocommutative space', and we can say that Moyal's construction gives a deformation of the classical space $M$ to a family of non-cocommutative (or 'quantum') spaces $M_h$ such that $\mathcal{F}_h(M)$ is the algebra of functions on $M_h$.\\
	
	The category of quantum spaces, then, might be defined as the category dual to the category of associative, but not necessarily commutative, algebras. To define the notion of a quantum group, let us first return for a moment to the classical situation. If $G$ is a group, the multiplication $\mu:G\times G\to G$ of $G$ induces a homomorphism $\mu^*=\Delta:\mathcal{F}(G)\to \mathcal{F}(G\times G)$ of algebras of functions. Now, if we define the algebra $\mathcal{F}(G)$ and the tensor product appropriately, $\mathcal{F}(G\times G)$ will be isomorphic to $\mathcal{F}(G)\otimes\mathcal{F}(G)$ as an algebra. For example, if $G$ is an affine algebraic group over $\mathbb{C}$, and $\mathcal{F}(G)$ is the algebra of regular functions on $G$, the ordinary algebraic tensor product will do. Thus, we have a comultiplication $\Delta:\mathcal{F}(G)\to\mathcal{F}(G)\otimes\mathcal{F}(G)$. (The reason for this terminology is that the multiplication on $\mathcal{F}(G)$ can be viewed as a map $\mathcal{F}(G)\otimes\mathcal{F}(G)\to\mathcal{F}(G)$.) Similarly, the inverse map $\iota:G\to G$ induces a map $\iota^*=S:\mathcal{F}(G)\to\mathcal{F}(G)$, called the antipode, and evaluation at the identity element of $G$ is a homomorphism $\epsilon:\mathcal{F}(G)\to\mathbb{C}$, called the counit. The maps $\Delta$, $S$ and $\epsilon$ satisfy certain compatibility properties which reflect the defining properties of the inverse and the associativity of multiplication in $G$, and combine to give $\mathcal{F}(G)$ the structure of a \textit{Hopf algebra}. We might therefore \textit{define the category of quantum groups to be the category dual to the category of (not necessarily commutative) Hopf algebras.} (We said 'might' here, and in our tentative definition of a quantum space, because, to ensure that the categories of quantum spaces and quantum groups have reasonable properties, it would be necessary to impose some restrictions on the class of algebras which are acceptable as 'quantized algebras of functions'. Manin suggests that one should work with 'Koszul algebras', but we shall not discuss this point here.) As is common practice in the literature, we shall often abuse this terminology by referring to a Hopf algebra itself by a quantum group.
	
	As the preceding discussion suggests, one way to try to construct non-classical examples of quantum groups is to look for deformations, in the category of Hopf algebras, of classical algebras of functions $\mathcal{F}(G)$. Just as the classical Poisson bracket can be recovered as the 'first order part' of Moyal's deformation (see (\ref{eqn:Moyal_deformation})), so it turns out that the existence of a deformation $\mathcal{F}_h(G)$ of $\mathcal{F}(G)$ automatically endows the group $G$ itself with extra structure, namely that of a \textit{Poisson-Lie group}. This is a Poisson structure on $G$ which is compatible with the group structure in a certain sense. Conversely, to construct deformations of $\mathcal{F}(G)$, it is natural to begin by describing the possible Poisson-Lie group structures on $G$ an then to attempt to extend these 'first order deformations' to full deformations. This is the approach taken in this book. Poisson-Lie groups are also of interest in their own right, for they form the natural setting for the study of classical integrable systems with symmetry.
	
	There is another Hopf algebra associated to any Lie group $G$, namely the universal enveloping algebra $U(g)$ of its Lie algebra $\mathfrak{g}$. This is essentially the \textit{dual} of $\mathcal{F}(G)$ in the category of Hopf algebras. In general, the vector space dual $A^*$ of any finite-dimensional Hopf algebra $A$ is also a Hopf algebra: the multiplication $A^*\otimes A^*\to A$ is dual to the comultiplication $\Delta:A\to A\otimes A$ of $A$, and the comultiplication of $A^*$ is dual to the multiplication of $A$. Note that $A^*$ is commutative if and only if $A$ is cocommutative, i.e. if and only if $\Delta(A)$ is contained in the symmetric part of $A\otimes A$. If, as is usually the case in examples of interest, $A$ is infinite-dimensional, this duality often continues to hold provided the dual and tensor product are defined appropriately. To a deformation $\mathcal{F}_h(G)$ of $\mathcal{F}(G)$ through (not necessarily commutative) Hopf algebras therefore corresponds a deformation $U_h(\mathfrak{g})$ of $U(\mathfrak{g})$ through (not necessarily cocommutative) Hopf algebras.
	
	In fact, only non-cocommutative deformations of $U(\mathfrak{g})$ are of interest, since any deformation of $U(\mathfrak{g})$ through cocommutative Hopf algebras is necessarily of the form $U(\mathfrak{g}_h)$ for some deformation $\mathfrak{g}_h$ of $\mathfrak{g}$ through Lie algebras. However, many interesting Lie algebras have non-trivial deformations. This is the case, for example, if $\mathfrak{g}$ is a (finite-dimensional) complex semisimple Lie algebra, such as the Lie algebra $\mathfrak{sl}_2(\mathbb{C})$ of $2\times 2$ complex matrices of trace zero. This follows from the fact that the condition being semisimple is open, so that any small deformation of $\mathfrak{g}$ will still be semisimple, whereas semisimple Lie algebras are discretely parametrized (by their Dynkin diagrams, for example).\\
	(Chari and Pressley, \cite{CPBook}, pp. 1-4)
\end{quote}

As explained, we would like to refer the reader to Chapter 6 in the book \cite{CPBook} for a detailed explanation of the foundations of the deformation theory of Hopf algebras such as universal enveloping algebras and algebras of functions on groups. In fact, the obstructions to the existence and uniqueness of deformations of a given class of mathematical objects are generally described by some cohomology theory of the objects concerned. In the case of Hopf algebras, the appropriate theory is a combination of the Hochschild cohomology theory of associative algebras and the dual theory for coalgebras. Moreover, this theory implies, that, for semisimple algebraic groups and Lie algebras over fields of characteristic zero, every deformation of the algebra of functions has unchanged coalgebra structure an every deformation of the universal enveloping algebra has unchanged algebra structure.

Let us just write down the definition of a deformation of a Hopf algebra and the first few remarks (cf. \cite{CPBook} Section 6.1).
\begin{definition}[deformation]
	A \textbf{deformation} of a Hopf algebra $(A,\iota,m,\epsilon,\delta,S)$ over a field $k$ is a topological Hopf algebra $(A_h,\iota_h,m_h,\epsilon_h,\Delta_h,S_h)$ over the ring $k[[h]]$ of formal power series in an indeterminate $h$ over $k$, such that
	\begin{enumerate}[label=\normalfont(\roman*)]
		\item $A_h$ is isomorphic to $A[[h]]$ as a $k[[h]]$-module;
		\item $m_h\equiv m \;(\operatorname{mod} h)$, $\Delta_h\equiv\Delta \;(\operatorname{mod} h)$. $\odot$
	\end{enumerate}
\end{definition}
Two deformations $A_h$ and $A_{h'}$ are equivalent if there is an isomorphism $f_h:A_h\to A_{h'}$ of Hopf algebras over $k[[h]]$ which is the identity ($\operatorname{mod} h$).
\begin{rems}\hspace{1em}
	\begin{enumerate}
		\item The notation $A[[h]]$ means the algebra of formal power series in $h$ with coefficients in $A$
		\begin{equation*}
			a_h = a_0+a_1h+a_2h^2+\dots \quad(a_0,a_1,a_2,\dots\in A).
		\end{equation*}
		Note that, unless $A$ is finite-dimensional, $A[[h]]$ is 'bigger' than $A\otimes_kk[[h]]$, i.e. $A[[h]]$ is in general only the completion of $A\otimes_kk[[h]]$ in the $h$-adic topology.
		\item The definition does not mention the unit, counit and antipode of $A_h$. Indeed, it can be shown that any deformation is equivalent to one in which the unit and counit are obtained simply by extending the initial ones of $A$ $k[[h]]$-linearly. Further, any deformation of $A$ as a bialgebra (i.e. forgetting the antipode) is automatically a Hopf algebra.
		\item If $f_h:A_h\to A_{h'}$ is an equivalence, we write $A_{h'} = f_h*A_h$. We then have $g_h*(f_h*A_h) = (g_hf_h)*A_h$.
		\item The simplest deformation is called the null deformation. It is obtained simply by extending the structure maps of $A$ $k[[h]]$-linearly. Any deformation equivalent to the null deformation is said to be trivial.
	\end{enumerate}
\end{rems}
Before we can write down the different definitions of affine Kac--Moody algebras, it is helpful to introduce the notion of $q$\textit{-binomial coefficients}. So let $q$ be an indeterminate, and let $k\geq l\in \mathbb{N}$, then we define
\begin{align*}
	[k]_q\coloneq\frac{q^k-q^{-k}}{q-q^{-1}},\quad[k]_q!\coloneq[k]_q[k-1]_q\cdots[1]_q,\quad\recbinom{k}{l}_q\coloneq\frac{[k]_q!}{[k-l]_q![l]_q!}
\end{align*}
for the $q$-number, $q$-factorial and $q$-binomial, respectively. Moreover, these symbols are well-defined elements of the field $\mathbb{Q}(q)$ of rational functions of $q$ over $\mathbb{Q}$ and it is clear that $[k]_q\in\mathbb{Z}[q,q^{-1}]$. Further, taking $q=e^h$ we have that $[k]_{e^h}\equiv k\;(\operatorname{mod} h)$.
\begin{definition}[topologically generated]
	Let $k$ be a field and let $P=k\{X_j|\,j\in J\}$ be the algebra of non-commutative polynomials in the generators $X_j$, $j\in J$, let $\mathfrak{W}$ be a set of relations for the generators $X_j$, let $I$ be the two sided ideal of $P[[h]]$ generated by $\mathfrak{W}$ and let $\bar{I}$ be the closure of $I$ in the $h$-adic topology. We say that an algebra $A$ is \textbf{topologically generated} by the generators $X_j,j\in J$, subject to the relations $\mathfrak{W}$, if $A=P[[h]]/\bar{I}$ as an algebra over $k[[h]]$.
\end{definition}
With this, we can define the quantum Kac--Moody algebra over $\mathbb{C}[[h]]$ (cf. \cite{CPBook} Definition-proposition 6.5.1).
\begin{defprop}[quantum Kac--Moody algebra $U_h(\mathfrak{g})$]\label{defprop:U_h(g)_chevalley}
	Let $\mathfrak{g}=L(A)'$ be the Kac--Moody (sub)algebra associated to a symmetrisable GCM $A=(A_{ij})_{i,j=1,\dots,n}$.\footnote{As introduced in Section \ref{sect:Kac--Moody}.} Let $U_h(\mathfrak{g})$ be the algebra over $\mathbb{C}[[h]]$ topologically generated by the elements $H_i$ and $X_i^\pm$, $i=1,\dots,n$, and with defining relations
	\begin{equation*}
		\begin{split}
			[H_i&,H_j] = 0,\quad[H_i,X_j^\pm]=\pm a_{ij}X_j^\pm,\\
			[&X_i^+,X_j^-]=\delta_{i,j}\frac{e^{d_ihH_i}-e^{-d_ihH_i}}{e^{d_ih}-e^{-d_ih}},\\
			\sum_{k=0}^{1-a_{ij}}(-1)^{k}&\recbinom{1-a_{ij}}{k}_{e^{d_ih}}(X_i^\pm)^kX_j^\pm(X_i^\pm)^{1-a_{ij}-k}=0\quad\text{for }i\neq j.
		\end{split}
	\end{equation*}
	Then, $U_h(\mathfrak{g})$ is the topological Hopf algebra over $\mathbb{C}[[h]]$ with comultiplication defined by
	\begin{equation*}
		\begin{split}
			&\Delta_h(H_i) = H_i\otimes 1+ 1\otimes H_i,\\
			\Delta_h(X_i^+)=X_i^+\otimes &e^{d_ihH_i}+1\otimes X_i^+,\quad\Delta_h(X_i^-)= X_i^-\otimes 1+e^{-d_ihH_i}\otimes X_i^-,
		\end{split}
	\end{equation*}
	antipode defined by
	\begin{equation*}
		S_h(H_i)=-H_i,\quad S_h(X_i^+) =-X_i^+e^{-d_ihH_i},\quad S_h(X_i^-)=-e^{d_i h H_i}X_i^-,
	\end{equation*}
	and counit defined by
	\begin{equation*}
		\epsilon_h(H_i)=\epsilon_h(X_i^\pm) = 0.\quad \odot
	\end{equation*}
\end{defprop}
For the verification that the formulas for the comultiplication, antipode and counit extend to algebra homomorphisms, we refer the reader to the book.\\

Let us state some important theorems for $U_h(\mathfrak{g})$ before we move to a slightly different version of it (cf. \cite{CPBook} Chapter 8).
For any integer $r\geq 0$, and for $1\leq i\leq n$, we define the divided powers
\begin{equation*}
	(X_i^\pm)^{(r)}=\frac{(X_i^\pm)^{r}}{[r]_{q_i}!}.
\end{equation*}
\begin{thm}
	There is an action of the braid group $\mathcal{B}_{\mathfrak{g}}$ by algebra automorphisms of $U_h(\mathfrak{g})$ defined on the generators by
	\begin{equation*}
		\begin{split}
			&T_i(X_i^+) = -X_i^- e^{d_ihH_i},\quad T_i(X_i^-) = -e^{-d_ihH_i}X_i^+,\quad T_i(H_j)=H_j - a_{ij}H_i,\\
			&T_i(X_j^+)= \sum_{r=0}^{-a_{ij}}(-1)^{r-a_{ij}}e^{-rd_ih}(X_i^+)^{(-a_{ij}-r)}X_j^+(X_i^+)^{(r)}\quad\text{if }i\neq j,\\
			&T_i(X_j^-)= \sum_{r=0}^{-a_{ij}}(-1)^{r-a_{ij}}e^{rd_ih}(X_i^-)^{(r)}X_j^-(X_i^-)^{(-a_{ij}-r)}\quad\text{if }i\neq j. \quad\odot
		\end{split}
	\end{equation*}
\end{thm}
\begin{rems}\hspace{1em}
	\begin{enumerate}
		\item The action of the $T_i$ on the $H_j$ is the 'same' as that of the fundamental reflections $s_i\in W$ (Weyl group) on the corresponding generators of $\mathfrak{g}$.
		\item The inverses of the  automorphisms $T_i$ are given by
		\begin{equation*}
			\begin{split}
				&T_i^{-1}(X_i^+) =e^{-d_ihH_i} -X_i^- ,\quad T_i^{-1}(X_i^-) = -X_i^+e^{d_ihH_i},\quad T_i^{-1}(H_j)=H_j - a_{ji}H_i,\\
				&T_i^{-1}(X_j^+)= \sum_{r=0}^{-a_{ij}}(-1)^{r-a_{ij}}e^{-rd_ih}(X_i^+)^{(r)}X_j^+(X_i^+)^{(-a_{ij}-r)}\quad\text{if }i\neq j,\\
				&T_i^{-1}(X_j^-)= \sum_{r=0}^{-a_{ij}}(-1)^{r-a_{ij}}e^{rd_ih}(X_i^-)^{(-a_{ij}-r)}X_j^-(X_i^-)^{(r)}\quad\text{if }i\neq j.
			\end{split}
		\end{equation*}
		\item The $T_i$ are no Hopf algebra automorphisms of $U_h(\mathfrak{g})$.
		\item The $T_i$ can be expressed in terms of the Hopf algebra adjoint representation of $U_h(\mathfrak{g})$ and that of the opposite coalgebra $U_h^{\operatorname{op}}(\mathfrak{g})$
		\begin{equation*}
			T_i(X_j^+)= \operatorname{ad}^{\operatorname{op}}_{-(X_i^+)^{(-a_{ij})}}(X_j^+),\quad T_i(X_j^-)=\operatorname{ad}_{-(X_i^-)^{(-a_{ij})}}(X_j^-).
		\end{equation*}
		\item There is a $\mathbb{C}$-algebra anti-automorphism $\omega$ and a $\mathbb{C}$-algebra automorphism $\phi$ of $U_h(\mathfrak{g})$ compatible with the action of the braid group $\mathcal{B}_{\mathfrak{g}}$ in the sense that
		\begin{equation*}
			\omega\circ T_i = T_i\circ\omega,\quad \phi\circ T_i = T_i^{-1}\circ \phi.
		\end{equation*}
		On generators, they are defined in exactly the same way as
		\begin{equation*}
			X_i^+\mapsto X_i^-,\quad X_i^-\mapsto X_i^+,\quad H_i\mapsto H_i,\quad h\mapsto -h. \quad \odot
		\end{equation*}
	\end{enumerate}
\end{rems}
\begin{prop}
	Let $w=s_{i_1}s_{i_2}\cdots s_{i_{l(w)}}$ be a reduced decomposition of an element $w\in W$. Then, the automorphism $T_w=T_{i_1}T_{i_2}\cdots T_{i_{l(w)}}$ of $U_h(\mathfrak{g})$ depends only on $w$, and not on the choice of the reduced decomposition. $\odot$
\end{prop}
Using the braid group action on $U_h(\mathfrak{g})$ it is possible to establish a PBW type basis for $U_h(\mathfrak{g})$ and define the universal $\mathcal{R}$-matrix at least when $\mathfrak{g}$ is of finite type. More generally, we should refer to the reader to the papers \cite{Beck94_1,Beck94_2} and \cite{BeckCP99} where a PBW basis in the affine case is discussed. We will come back to this later when we come to the second Drinfeld realisation, which is the quantum analogue of the 'loop construction' (or 'loop realisation') for affine Kac--Moody algebras. 

In the case of a quantized universal enveloping (such as $U_h(\mathfrak{g})$) algebra it is generally not immediately clear how to obtain an analogue of the PBW theorem. This is because one doesn't have an underlying Lie algebra, neither a basis of it. The most natural basis of the classical enveloping algebra $U(\mathfrak{g})$ is the PBW basis associated to the basis of $\mathfrak{g}$ given by a basis of the Cartan subalgebra $\mathfrak{h}\subset \mathfrak{g}$ together with a root vector in each root space of $\mathfrak{g}$. In $U_h(\mathfrak{g})$, one has an analogue of the basis of $\mathfrak{h}$, namely $\{H_i\}_{i=1,\dots,n}$, but one is only given 'root vectors' $X_i^\pm$ associated to the simple (or fundamental) roots $\alpha_i$ (and their negatives). One must therefore define analogues of root vectors associated to non-simple root of $\mathfrak{g}$. We have seen in Subsection \ref{subsect:Cartan_data} that non-simple root vectors of $\mathfrak{g}$ can be defined by using the action of the braid group $\mathcal{B}_{\mathfrak{g}}$ (a finite covering of the Weyl group $W$) on $\mathfrak{g}$. This is based on the fact, that if 
\begin{equation}
	w_0=s_{i_1}s_{i_2}\cdots s_{i_N}\label{eqn:fixed_dec_longest_elem_W}
\end{equation}
is a reduced decomposition of the longest element of $W$, then every positive root occurs exactly once in the set
\begin{equation}
	\beta_1 = \alpha_{i_1},\;\beta_2= s_{i_1}(\alpha_{i_2}),\;\dots,\;\beta_N= s_{i_1}\cdots s_{i_{N-1}}(\alpha_{i_N}).\label{eqn:roots_beta_1_to_beta_N}
\end{equation}
Fortunately, we were already able to define the action of the braid group $\mathcal{B}_{\mathfrak{g}}$ based on the fundamental 'root vectors' in the quantum case.
We can therefore make the following definition.
\begin{definition}\label{def:affine_root_vectors_U_h}
	Let $\mathfrak{g}$ be a finite-dimensional complex simple Lie algebra, consider a fixed decomposition (\ref{eqn:fixed_dec_longest_elem_W}) of the longest element $w_0\in W$ and define $\beta_1,\dots,\beta_N$ as in (\ref{eqn:roots_beta_1_to_beta_N}). We then define elements $X_{\beta_r}^\pm\in U_h(\mathfrak{g})$ by
	\begin{equation*}
		X_{\beta_r}^\pm = T_{i_1}T_{i_2}\cdots T_{i_{r-1}}(X_{i_r}^\pm)
	\end{equation*}
	and call them the positive (resp. negative) root vectors of $U_h(\mathfrak{g})$. $\odot$
\end{definition}
If $\mathfrak{g}$, however, is an infinite-dimensional Kac--Moody algebra, it is no longer true that every root lies in the Weyl group orbit of a simple root. In this case, Definition \ref{def:affine_root_vectors_U_h} doesn't suffice for defining root vectors. Moreover, in contrast to the classical case, the definition of the root vectors strongly depends on the choice of the reduced decomposition (\ref{eqn:fixed_dec_longest_elem_W}). In fact, the root vectors corresponding to different decompositions are in general not even proportional to each other.\\

Let $U_h^+$, $U_h^0$ and $U_h^-$ be the $\mathbb{C}[[h]]$-subalgebras of $U_h=U_h(\mathfrak{g})$ topologically generated by the $X_i^+$, by the $H_i$ and by the $X_i^-$, respectively. Then we have
\begin{prop}\hspace{1em}
	\begin{enumerate}[label=\normalfont(\roman*)]
		\item Let $w\in W$ be such that $w(\alpha_i)\in\Phi^+$ for some simple root $\alpha_i$. Then $T_w(X_i^\pm)\in U_h^\pm$. If $w(\alpha_i)=\alpha_j$ is simple, then $T_w(X_i^\pm)=X_j^\pm$.
		\item In particular, for any choice of reduced decomposition of $w_0$, the positive (resp. negative) root vectors lie in $U_h^+$ (resp. $U_h^-$).
		\item The braid group action preserves $U_h^0$ and induces on it the classical action of the Weyl group.
	\end{enumerate}
\end{prop}
The meaning of the third part is that, if $\beta =\sum_i k_i\alpha_i\in Q$, and if we define $H_\beta=\sum_i k_i H_i$, then $T_w(H_\beta)= H_{w(\beta)}$ for all $w\in W$. Note that, if $\beta=\beta_r\in \Phi^+$ is a positive root (as in (\ref{eqn:roots_beta_1_to_beta_N})), the vectors $\{H_{\beta_r},X_{\beta_r}^\pm\}$ satisfy the defining relations of $U_{d_{i_r}h}(\mathfrak{sl}_2)$. This follows from the application of the automorphism $T_{i_1}T_{i_2}\cdots T_{i_{r-1}}$ on the relations satisfied by $\{H_{\beta_r},X_{\beta_r}^\pm\}$.

Using the quantum analogues of the root vectors, we can now construct a topological basis of $U_h$ as follows. Fix a choice of positive and negative root vectors as in Definition \ref{def:affine_root_vectors_U_h}. For $\boldsymbol{r}=(r_1,\dots,r_N)\in \mathbb{N}^N$, define
\begin{equation*}
	\begin{split}
		&(X^+)^{\boldsymbol{r}}=(X_{\beta_N}^+)^{r_N}(X_{\beta_{N-1}}^+)^{r_{N-1}}\cdots(X_{\beta_1}^+)^{r_1},\\
		&(X^-)^{\boldsymbol{r}}= \omega(X^+)^{\boldsymbol{r}}= (X_{\beta_1}^-)^{r_1}(X_{\beta_{2}}^-)^{r_{2}}\cdots(X_{\beta_N}^-)^{r_N},
	\end{split}
\end{equation*}
and for $\boldsymbol{s} = (s_1,\dots,s_n)\in \mathbb{N}^n$,
\begin{equation*}
	H^{\boldsymbol{s}} = H_1^{s_1}H_2^{s_2}\cdots H_n^{s_n}.
\end{equation*}
With this, we have
\begin{prop}\label{prop:pbw_U_h}
	The products $(X^-)^{\boldsymbol{r}}$, $H^{\boldsymbol{s}}$ and $(X^+)^{\boldsymbol{t}}$, for $\boldsymbol{r}$, $\boldsymbol{t}\in \mathbb{N}^N$, $\boldsymbol{s}\in \mathbb{N}^n$ form topological bases of $U_h^-$, $U_h^0$ and $U_h^+$, respectively, and the products $(X^-)^{\boldsymbol{r}}H^{\boldsymbol{s}}(X^+)^{\boldsymbol{t}}$ form a topological basis of $U_h$. $\odot$
\end{prop}
with the immediate consequence
\begin{cor}\label{cor:pbw_U_h}
	Multiplication defines an isomorphism of $\mathbb{C}[[h]]$-modules
	\begin{equation*}
		U_h^-\otimes U_h^0\otimes U_h^+ \to U_h
	\end{equation*}
	(completed tensor products). $\odot$
\end{cor}
(cf. \cite{CPBook} Subsection 8.1B)\\

Let now $X_{\beta_1}^\pm,X_{\beta_2}^\pm,\dots,X_{\beta_N}^\pm$ be the root vectors in Definition \ref{def:affine_root_vectors_U_h} and define the $q$-exponential
\begin{equation*}
	\exp_q(x)=\sum_{k=0}^{\infty}q^{\frac{1}{2}k(k+1)}\frac{x^k}{[k]_q!}.
\end{equation*}
Let also $q_\beta=e^{d_\beta h}$ and $d_\beta =d_i$ if the positive root $\beta$ is Weyl group conjugate to the simple root $\alpha_i$.
Then, one can show
\begin{thm}\label{thm:finite_univ_R_matr_U_h}
	For any finite-dimensional complex simple Lie algebra $\mathfrak{g}$, $U_h(\mathfrak{g})$ is topologically quasitriangular with universal $R$-matrix
	\begin{equation*}
		\mathcal{R}_h=\exp\left[h\sum_{i,j}(B^{-1})_{ij}H_i\otimes H_j\right]\prod_\beta \exp_{q_\beta}[(1-q_\beta^{-2})X_\beta^+\otimes X_\beta^-],
	\end{equation*}
	where the product is over all the positive roots of $\mathfrak{g}$, and the order of the terms is such that the $\beta_r$-term appears to the left of the $\beta_s$-term if $r>s$.
\end{thm}
The construction of the universal $R$-matrix and therefore the proof of Theorem \ref{thm:finite_univ_R_matr_U_h} is due to the observation that $U_h=U_h(\mathfrak{g})$ is 'almost' a so-called 'quantum double'. It is explained in Subsection 4.2D in the book \cite{CPBook} that quantum doubles naturally have a universal $R$-matrix. Moreover, it is explained in Subsection 8.3D \cite{CPBook} that the universal $R$-matrix in Theorem \ref{thm:finite_univ_R_matr_U_h} is independent of the choice of reduced decomposition of the longest element of the Weyl group of $\mathfrak{g}$. In fact, let $u_h =m_h(S_h\otimes\id)(\mathcal{R}_h)_{21}$ and let $\rho^*$ be the unique element of $\mathfrak{h}\subset U_h(\mathfrak{h})$ such that $\alpha_i(\rho^*)=2d_i$ for $i=1,\dots,n$.\footnote{If $\sum k_i \alpha_i$ is the sum of the positive roots of $\mathfrak{g}$, then $\rho^*=\sum k_i d_i H_i$.} Then $(U_h(\mathfrak{g}),v_h)$ is a ribbon Hopf algebra with $v_h = e^{-h\rho^*}u_h$ (see \cite{CPBook} Subsection 8.3F).\\

Let us now introduce the rational form of the quantum Kac--Moody algebra. Although one can think of $U_h(\mathfrak{g})$ as a family of Hopf algebras (over $\mathbb{C}$) depending on a parameter $h$, it should be clear that an algebra defined over the formal power series ring $\mathbb{C}[[h]]$ cannot be specialised to any value of $h$ except $h=0$. In fact, one can introduce an algebra $U_q(\mathfrak{g})$, defined over the field of rational functions of an indeterminate $q$, which can be viewed as the rational counterpart of the formal object $U_h(\mathfrak{g})$. The main advantage is that one can specialise $q$ to any transcendental complex number. Moreover, the definition and properties closely resemble those of $U_h(\mathfrak{g})$. In addition, on can construct an 'integral' form of the algebra, defined over $\mathbb{Z}[q,q^{-1}]$, which enables one to specialise $q$ to any non zero complex number. Furthermore, its definition is somewhat less technical, as it doesn't require the consideration of $h$-adic topologies, completed tensor products, etc. (see \cite{CPBook} pp. 279-280 for more details).
We introduce $U_q(\mathfrak{g})$ as follows (cf. \cite{CPBook} Section 9.1 A-B).
\begin{defprop}[quantum Kac--Moody algebra $U_q(\mathfrak{g})$]\label{defprop:U_q(g)_chevalley}
	Let $\mathfrak{g}=L(A)'$ be the Kac--Moody (sub)algebra associated to a symmetrisable GCM $A=(A_{ij})_{i,j=1,\dots,n}$ and let $d_i$ be the coprime positive integers such that $(d_ia_{ij})$ is symmetric. Let $q_i=q^{d_i}$, then $U_q(\mathfrak{g})$ is the associative algebra over $\mathbb{Q}(q)$ with generators $K_i$, $K_i^{-1}$ and $X_i^\pm$, $i=1,\dots,n$, and with defining relations
	\begin{equation*}
		\begin{split}
			[K_i,K_j] &= 0,\quad K_iK_i^{-1}=K_i^{-1}K_i=1,\\
			&K_iX_j^\pm K_i^{-1}=q_i^{\pm a_{ij}}X_j^\pm,\\
			&[X_i^+,X_j^-]=\delta_{i,j}\frac{K_i-K_i^{-1}}{q_i-q_i^{-1}},\\
			\sum_{k=0}^{1-a_{ij}}(-1)^{k}&\recbinom{1-a_{ij}}{k}_{q_i}(X_i^\pm)^kX_j^\pm(X_i^\pm)^{1-a_{ij}-k}=0\quad(i\neq j).
		\end{split}
	\end{equation*}
	Then, $U_q(\mathfrak{g})$ has the Hopf algebra structure with comultiplication $\Delta_q$, antipode $S_q$ and counit $\epsilon_q$ defined on generators by
	\begin{equation*}
		\begin{split}
			\Delta_q(K_i) = K_i&\otimes K_i,\\
			\Delta_q(X_i^+)=X_i^+\otimes K_i+1\otimes X_i^+,\quad&\Delta_q(X_i^-)= X_i^-\otimes 1+K_i^{-1}\otimes X_i^-,\\
			S_q(K_i)=K_i^{-1},\quad S_q(X_i^+) =-&X_i^+K_i^{-1},\quad S_q(X_i^-)=-K_iX_i^-,\\
			\epsilon_q(K_i)=1,\quad\epsilon_q&(X_i^\pm) = 0.\quad \odot
		\end{split}
	\end{equation*}
\end{defprop}
The verification that this defines the structure of a Hopf algebra is analogous to $U_h(\mathfrak{g})$. We shall also make the following remarks.
\begin{rems}\hspace{1em}
	\begin{enumerate}
		\item We omit the subscripts on the structure maps whenever the Hopf algebra is clear from the context.
		\item If we let $q_i = e^{d_i h}$ and $K_i=e^{d_ihH_i}$, the relations \ref{defprop:U_q(g)_chevalley} are consequences of the relations in \ref{defprop:U_h(g)_chevalley}.
		\item There are other rational forms of $U_h(\mathfrak{g})$, which have the same relation to $U_q(\mathfrak{g})$ as the various complex Lie groups with Lie algebra $\mathfrak{g}$ have to the adjoint group of $\mathfrak{g}$. The largest of these, the \textit{simply-connected} rational form, is obtained by adjoining to $U_q(\mathfrak{g})$ invertible elements $L_i$, $i=1,\dots,n$, such that $K_i=\prod_j L_j^{a_{ji}}$ and the relations are
		\begin{equation*}
			L_iX_j^\pm L_i^{-1} = q_i^{\pm\delta_{i,j}}X_j^\pm.
		\end{equation*}
		More generally, if $M$ is any lattice such that $Q\subseteq M\subseteq P$, there is a rational form $U_q^M(\mathfrak{g})$ obtained by adjoining to $U_q(\mathfrak{g})$ the elements $K_\beta=\prod_j L_j^{m_{ij}}$ for every $\beta=\sum_jm_{ij}\omega_j$, where $\omega_1,\dots,\omega_n$ are the fundamental weights of $\mathfrak{g}$. In this sense, $U_q(\mathfrak{g})$ is the adjoint form $U_q^Q(\mathfrak{g})$. For instance if $\mathfrak{g}=\mathfrak{sl}_2$, $P/Q=\mathbb{Z}_2$, so there are two rational forms, the adjoint form and the simply-connected form. The latter is obtained from the former by adjoining a square root of $K=K_1$. $\odot$
	\end{enumerate}
\end{rems}
Let us also recover some properties of $U_h(\mathfrak{g})$ for $U_q=U_q(\mathfrak{g})$ (cf. \cite{CPBook} Subsection 9.1B). The first is the braid group action for $U_q$.
\begin{prop}
	The braid group $\mathcal{B}_{\mathfrak{g}}$ of $\mathfrak{g}$ acts by $\mathbb{Q}(q)$-algebra automorphisms on $U_q$ given by
	\begin{equation*}
		\begin{split}
			&T_i(X_i^+) = -X_i^- K_i,\quad T_i(X_i^-) = -K_i^{-1}X_i^+,\quad T_i(K_j)=K_jK_i^{ - a_{ij}},\\
			&T_i(X_j^+)= \sum_{r=0}^{-a_{ij}}(-1)^{r-a_{ij}}q_i^{-r}(X_i^+)^{(-a_{ij}-r)}X_j^+(X_i^+)^{(r)}\quad\text{if }i\neq j,\\
			&T_i(X_j^-)= \sum_{r=0}^{-a_{ij}}(-1)^{r-a_{ij}}q_i^{r}(X_i^-)^{(r)}X_j^-(X_i^-)^{(-a_{ij}-r)}\quad\text{if }i\neq j. \quad\odot
		\end{split}
	\end{equation*}
\end{prop}
Similarly, root vectors can be defined in exactly as in Definition \ref{def:affine_root_vectors_U_h}. For $\beta=\sum_i k_i \alpha_i\in Q$, we define $K_\beta =\prod_i K_i^{k_i}$. Then, $T_w(K_\beta) = K_{w(\beta)}$ for all $w\in W$. Further, let $U_q^+$, $U_q^0$ and $U_q^-$ be the $\mathbb{Q}(q)$-subalgebras of $U_q$ generated by the $X_i^+$, the $K_i^\pm $ and the $X_i^-$, respectively. We can conclude the analogues to Proposition \ref{prop:pbw_U_h} and Corollary \ref{cor:pbw_U_h} in
\begin{prop}\label{prop:PBW_U_q}
	Multiplication defines an isomorphism of vector spaces over $\mathbb{Q}(q)$
	\begin{equation*}
		U_q^-\otimes U_q^0\otimes U_q^+ \to U_q.
	\end{equation*}
	Moreover, the products
	\begin{enumerate}[label=\normalfont(\roman*)]
		\item $(X^-)^{\boldsymbol{r}}=(X_{\beta_1}^-)^{r_1}(X_{\beta_{2}}^-)^{r_{2}}\cdots(X_{\beta_N}^-)^{r_N}$, for $\boldsymbol{r}\in \mathbb{N}^N$, form a basis of $U_q^-$.
		\item $K^{\boldsymbol{s}}=K_1^{s_1}K_2^{s_2}\cdots K_n^{s_n}$, for $\boldsymbol{s}\in \mathbb{Z}^n$, form a basis of $U_q^0$.
		\item $(X^+)^{\boldsymbol{t}}=(X_{\beta_N}^+)^{r_N}(X_{\beta_{N-1}}^+)^{r_{N-1}}\cdots(X_{\beta_1}^+)^{r_1}$, for $\boldsymbol{t}\in \mathbb{N}^N$, form a basis of $U_q^+$. $\odot$
	\end{enumerate}
\end{prop}
However, looking at the Definition of the universal $R$-matrix for $U_h$ (for $\mathfrak{g}$ finite), it is clear that it can't be well defined in $U_q(\mathfrak{g})$. This is because the $q$-exponentials in Theorem \ref{thm:finite_univ_R_matr_U_h} involve infinite sums and therefore can't be elements of $U_q(\mathfrak{g})\otimes U_q(\mathfrak{g})$. Furthermore, there is no element in $U_q(\mathfrak{g})\otimes U_q(\mathfrak{g})$ corresponding to the first factor
\begin{equation*}
	\exp\left[h\sum_{i,j}(B^{-1})_{ij}H_i\otimes H_j\right]
\end{equation*}
in $\mathcal{R}$. However, even though the universal $R$-matrix doesn't exist as an element of $U_q(\mathfrak{g})\otimes U_q(\mathfrak{g})$, it can be shown that it nevertheless acts on any tensor product of finite-dimensional representations (see \cite{CPBook} Subsection 10.1D).\\

Finally, let us also introduce the following algebra automorphisms.
\begin{prop}\label{prop:twisting}
	For each sequence of signs $(\sigma_1,\dots,\sigma_n)\in\{\pm1\}^n$, there is an algebra automorphism of $U_q$ given by
	\begin{equation*}
		K_i\mapsto \sigma_i K_i,\quad X_i^+\mapsto \sigma_i X_i^+,\quad X_i^-\mapsto X_i^-.\quad\odot
	\end{equation*}
\end{prop}
Note that, except for the identity, these automorphisms are not compatible with the coalgebra structure of $U_q$. Equivalently, one can formulate Proposition \ref{prop:twisting} more invariantly as follows. Let $Q_2^*$ be the group of homomorphisms $Q\to \mathbb{Z}_2=\{\pm1\}$ with product given by pointwise multiplication. Then, $Q_2^*$ acts as a group of automorphisms on $U_q$ given by
\begin{equation*}
	\sigma. K_\beta = \sigma(\beta) K_\beta,\quad\sigma.X_\alpha^+ = \sigma(\alpha)X_\alpha^+,\quad\sigma.X_\alpha^- = X_\alpha^-
\end{equation*}
for all $\beta\in Q$, $\alpha\in \Phi^+$, $\sigma\in Q_2^*$. Note that $W$ acts on $Q_2^*$ by
\begin{equation*}
	(w.\sigma)(\beta) = \sigma(w^{-1}(\beta)).
\end{equation*}
Hence, the action of $Q_2^*$ on $U_q^0$ extends to an action of the semidirect product $W \ltimes Q_2^*$.

Up to applying these 'twisting' automorphisms in Proposition \ref{prop:twisting} one can show that, if $\mathfrak{g}$ is finite, there is a one to one correspondence between the finite-dimensional representations of $U_q(\mathfrak{g})$ and those of $\mathfrak{g}$. We will come back to the role of these 'twistings' when we discuss the finite-dimensional representations of $U_q(\mathfrak{g})$ when $\mathfrak{g}$ is an affine Kac--Moody algebra im Section \ref{sect:finite-dim_reps}.
\subsection{Quantum affine algebras}\label{subsect:quantum_affine}
Let us now discuss a different realisation of $U_q(\mathfrak{g})$ in the special case when $\mathfrak{g}$ is an affine Kac--Moody algebra. Even more specifically, let us focus on the case when $\mathfrak{g}$ is an untwisted affine Kac--Moody algebra. We have seen in Section \ref{sect:Kac--Moody} that these can be understood as a one-dimensional central extension (and by adjoining a derivation) of the Loop algebra $\mathfrak{L}(\mathfrak{g})=\mathfrak{g}[t,t^{-1}]=\mathbb{C}[t,t^{-1}]\otimes_{\mathbb{C}}\mathfrak{g}$ of a finite-dimensional complex simple Lie algebra $\mathfrak{g}$. So let $\mathfrak{g}$ be from now on a finite-dimensional complex simple Lie algebra with Cartan matrix $A^0$.\footnote{The space of Laurent polynomial maps $\mathbb{C}^\times\to \mathfrak{g}$.} Then we denote by $\tilde{\mathfrak{g}}=\tilde{\mathfrak{L}}(\mathfrak{g})=\mathfrak{L}(\mathfrak{g})\oplus\mathbb{C}c$ (resp. $\hat{\mathfrak{g}}=\hat{\mathfrak{L}}(\mathfrak{g})=\tilde{\mathfrak{L}}(\mathfrak{g})\oplus\mathbb{C}d$) the untwisted affine Kac--Moody (sub)algebra $L(A)'$ (resp. $L(A)$) where $A$ is build from $A^0$ as explained in Subsection \ref{subsect:loop_constr_untw_Kac--Moody}. Thus we have $\hat{\mathfrak{g}}=\tilde{\mathfrak{g}}\oplus\mathbb{C}d=\mathfrak{L}(\mathfrak{g})\oplus\mathbb{C}c\oplus\mathbb{C}d$. The difference is that the simple roots in the Cartan subalgebra of $\tilde{\mathfrak{g}}$ are linear dependent, whereas $d$ removes the degeneracy for $\hat{\mathfrak{g}}$.

Luckily, as is the case for affine Kac--Moody algebras, Drinfeld (1988 \cite{D}) found another realisation of $U_h(\tilde{\mathfrak{g}})$ on a set of loop algebra-like generators which is, although still in terms of generators and relations, much closer to the description of affine Kac--Moody algebras as a space of maps. It is therefore called the 'new realization' or 'second Drinfeld realization'.
Let us just cite the comment of Beck on Drinfeld's findings (1993 \cite{Beck94_1}, p.1), "Drinfel'd found that the study of finite-dimensional representations of $U_q(\hat{\mathfrak{g}})$ is made much easier by the use of a ``new realization'' on a set of loop algebra-like generators over $\mathbb{C}[[h]]$. He gives (the proof is unpublished) an isomorphism to the usual presentation, although by his construction there is no explicit correspondence between the two sets of generators."

However, as discussed in Subsection \ref{subsect:QUE_algebras} it is more convenient for us to work with $U_q(\tilde{\mathfrak{g}})$ than with $U_h(\tilde{\mathfrak{g}})$. Moreover, we will refer to $U_q(\tilde{\mathfrak{g}})$, or any of its specialisations, as a \textbf{quantum affine algebra}.
The following presentation of $U_q(\tilde{\mathfrak{g}})$ is proved in the paper \cite{Beck94_1} by Beck 1993.

\begin{thm}[second Drinfeld realization of $U_q(\tilde{\mathfrak{g}})$]
	\label{thm:second_drin_quantum_aff}
	Let $\tilde{\mathfrak{g}}$ be the untwisted affine Lie algebra associated to a finite-dimensional complex simple Lie algebra $\mathfrak{g}$.\footnote{ I.e. the untwisted affine Kac--Moody subalgebra $L(A)'$, where $A$ is the GCM of $\tilde{\mathfrak{g}}$.} The standard quantisation $U_q(\tilde{\mathfrak{g}})$ of $\mathfrak{g}$ defined in \ref{defprop:U_q(g)_chevalley} is isomorphic, an an algebra over $\mathbb{C}(q)$, to the algebra $\mathcal{A}_q$ with generators $\mathcal{C}^{\pm1/2}$, $\mathcal{K}_i^{\pm1}$, $\mathcal{H}_{i,r}$, $\mathcal{X}^\pm_{i,s}$, $i\in I = \{1,\dots,l\}$, $r\in\mathbb{Z}\backslash\{0\}$, $s\in\mathbb{Z}$ and defining relations
	\begin{equation*}
		\begin{split}
			\mathcal{K}_i\mathcal{K}_i^{-1} =  \mathcal{K}_i^{-1}\mathcal{K}_i &= 1,\; \mathcal{C}^{1/2}\mathcal{C}^{-1/2}=1,\\
			\mathcal{C}^{\pm1/2}&\text{ are}\text{ central},\\
			[\mathcal{K}_i,\mathcal{K}_j]&=[\mathcal{K}_i,\mathcal{H}_{j,r}]=0,\\
			[\mathcal{H}_{i,r},\mathcal{H}_{j,s}]&=\delta_{r,-s}\frac{1}{r}[ra_{ij}]_{q_i}\frac{\mathcal{C}^r-\mathcal{C}^{-r}}{q_j-q_j^{-1}},\\
			\mathcal{K}_i\mathcal{X}^\pm_{j,r}\mathcal{K}_i^{-1}&=q_i^{\pm a_{ij}}\mathcal{X}_{j,r}^\pm,\\
			[\mathcal{H}_{i,r},\mathcal{X}^\pm_{j,s}]&=\pm \frac{1}{r}[ra_{ij}]_{q_i}\mathcal{C}^{\mp|r|/2}\mathcal{X}^\pm_{j,r+s},
		\end{split}
	\end{equation*}
	\begin{equation*}
		\begin{split}
			\mathcal{X}_{i,r+1}^\pm\mathcal{X}_{j,s}^\pm-q_i^{\pm a_{ij}}\mathcal{X}_{j,s}^\pm\mathcal{X}_{i,r+1}^\pm&=q_i^{\pm a_{ij}}\mathcal{X}_{i,r}^\pm\mathcal{X}_{j,s+1}^\pm-\mathcal{X}_{j,s+1}^\pm\mathcal{X}_{i,r}^\pm\\
			[\mathcal{X}_{i,r}^+,\mathcal{X}_{j,s}^-]&=\delta_{i,j}\frac{\mathcal{C}^{(r-s)/2}\varPhi_{i,r+s}^+-\mathcal{C}^{-(r-s)/2}\varPhi_{i,r+s}^-}{q_i-q_i^{-1}}\\
			\sum_{\sigma\in S_m}\sum_{k=0}^{m}(-1)\recbinom{m}{k}_{q_i}
			\mathcal{X}^\pm_{i,r_{\sigma(1)}}&\cdots\mathcal{X}^\pm_{i,r_{\sigma(k)}}\mathcal{X}_{j,s}^\pm\mathcal{X}_{i,r_{\sigma(k+1)}}^\pm\cdots\mathcal{X}^\pm_{i,r_{\sigma(m)}}=0,\, i\neq j,
		\end{split}
	\end{equation*}
	for all sequences of non-negative integers $r_1,\dots,r_m$, where $m=1-a_{ij}$, $q_i=q^{d_i}$ and the elements $\varPhi_{i,r}^\pm$ are determined by equating coefficients of powers of $u$ in the formal power series
	\begin{equation*}
		\boldsymbol{\varPhi}^\pm_i(u) = \sum_{r=0}^{\infty}\varPhi_{i,\pm r}^\pm u^{\pm r}= \mathcal{K}_i^{\pm1}\exp\left(\pm(q_i-q_i^{-1})\sum_{s=1}^{\infty}\mathcal{H}_{i,\pm s}u^{\pm s}\right).\footnotemark
	\end{equation*}
	\footnotetext{ We shall remind the reader that $(a_{ij})$ denotes the Cartan matrix of $\mathfrak{g}$, a complex semisimple Lie algebra of rank $l$ (see Subsection \ref{subsect:Cartan_data}).}
	Let $\theta = \sum_{i=1}^{l}a_i\alpha_i$ be the highest root of $\mathfrak{g}$, set $q_\theta = q_i$ if $\theta$ is Weyl group conjugate to $\alpha_i$, and set $\mathcal{K}_\theta\coloneq\prod_{i=1}^{l}\mathcal{K}_i^{a_i}$. Suppose that the root vector $x_\theta^+$ of $\mathfrak{g}$ is expressed in terms of the simple root vectors as
	\begin{align*}
		x_\theta^+=\lambda [x_{i_1}^+,[x_{i_2}^+,\dots,[x_{i_k}^+,x_j^+]\cdots]]
	\end{align*}
	for some $\lambda\in\mathbb{C}$.\footnote{ We use small letters to denote elements of $\mathfrak{g}$ here.} Define maps $w_i^\pm:U_q(\tilde{g})\to U_q(\tilde{g})$ by
	\begin{align*}
		w_i^\pm(a) \coloneq \mathcal{X}_{i,0}^\pm a - \mathcal{K}_i^{\pm1}a\mathcal{K}_i^{\mp1}\mathcal{X}_{i,0}^\pm.
	\end{align*}
	Then, there exists an isomorphism of $\mathbb{C}(q)$-algebras $f:U_q(\tilde{\mathfrak{g}})\to \mathcal{A}_q$ defined on generators by
	\begin{equation*}
		\begin{split}
			f(K_0)&\coloneq \mathcal{C}\mathcal{K}_\theta^{-1},\; f(K_i)\coloneq\mathcal{K}_i,\; f(X_i^\pm)\coloneq\mathcal{X}_{i,0}^\pm,\; i=1,\dots,l,\\
			f(X_0^+)&\coloneq\mu w_{i_1}^-\cdots w_{i_k}^-(\mathcal{X}_{j,1}^-)\mathcal{K}_\theta^{-1},\; f(X_0^-)\coloneq\lambda \mathcal{K}_\theta w_{i_1}^+\cdots w_{i_k}^+(\mathcal{X}_{j,-1}^+),
		\end{split}
	\end{equation*}
	where $\mu\in\mathbb{C}(q)$ is determined by the condition (cf. \cite{CPBook} p. 393)
	\begin{align*}
		[X_0^+,X_0^-]=\frac{K_0-K_0^{-1}}{q_\theta-q_\theta^{-1}}.\quad \odot
	\end{align*}
\end{thm}
\begin{rems}\label{rems:quant_aff}\hspace{1em}
	\begin{enumerate}
		\item The reader may wonder why we didn't write down the comultiplication $\Delta_q$ in terms of the (new) Drinfeld generators. The reason is that explicit formulas are only known in terms of the Chevalley (or Drinfeld--Jimbo) generators $X_i^+,\,X_i^-,\,K_i$ and $K_i^{-1}$, $i=0,\dots,l$, of $U_q(\tilde{\mathfrak{g}})$ as in \ref{defprop:U_q(g)_chevalley}.
		Therefore, the Hopf algebra structure for the (new) Drinfeld generators is defined only implicitly by the isomorphism $f$.
		\item Here, the description of the isomorphism $f$ is as in the book (cf. \cite{CPBook} p. 393). Though, it seems like it is not completely correct, which can easily be verified by computing the antipode $S_q(\mathcal{X}_{j,\mp1}^\pm)$ via $S_q(f(X_0^\mp)) = f(S_q(X_0^\mp))$ for $\mathfrak{g}=\mathfrak{sl}_3$. In this case, one can change the sign in the exponent of the $\mathcal{K}_i$ for the maps $w_i^\pm$ to make it work.\footnote{ Unfortunately, this doesn't provide the fix for $\mathfrak{sl}_{l+1}$ with $l>2$ (cf. \cite{CP1994a}).}
		
		In general, however, there must be a different fix and we shall we shall refer the reader to the paper of Beck \cite{Beck94_1} for the precise result. It seems like $S_q(\mathcal{X}_{i,1}^-) = -q^2  \mathcal{X}_{i,1}^- \frac{\mathcal{K}_i}{K_\theta^2}$ is the correct description for the action of the antipode on the $\mathcal{X}_{i,1}^-$.\footnote{ And then $S_q(\mathcal{X}_{i,-1}^+) = -q^{-2} \frac{K_\theta^2}{\mathcal{K}_i} \mathcal{X}_{i,-1}^+$.}
		\item The result for $U_h(\mathfrak{g})$ by Drinfeld \cite{D} 1988 is obtained by replacing the generators $\mathcal{K}_i$ and $\mathcal{C}^{1/2}$ by new generators $\mathcal{H}_{i,0}$ and $c/2$. The element $c$ is central and the third and fifth set of relations in Theorem \ref{thm:second_drin_quantum_aff} are replaced by
		\begin{equation*}
			[\mathcal{H}_{i,0},\mathcal{H}_{j,r}]=0,.\quad[\mathcal{H}_{i,0},\mathcal{X}_{j,r}^\pm]=\pm d_i a_{ij}\mathcal{X}_{j,r}^\pm.
		\end{equation*}
		In addition, $q_i$ is replaced by $e^{d_ih}$ and $\mathcal{C}^{1/2}$ by $e^{hc/2}$.\label{enum:second_Drin_h}
		\item To obtain a presentation of $U_q(\hat{\mathfrak{g}})$ from $U_q(\tilde{\mathfrak{g}})$, one introduces an additional set of generators $\mathcal{D}^{\pm1}$ and relations
		\begin{equation*}
			\begin{split}
				\mathcal{D}\mathcal{D}^{-1}=\mathcal{D}^{-1}\mathcal{D}=&1\\
				\mathcal{D}\mathcal{H}_{i,r}\mathcal{D}^{-1}=q^r\mathcal{H}_{i,r},\,[\mathcal{D},\mathcal{K}_i]&=[\mathcal{D},\mathcal{C}]=0,\\
				\mathcal{D}\mathcal{X}_{i,r}^\pm\mathcal{D}^{-1}=q^r&\mathcal{X}_{i,r}^\pm.
			\end{split}
		\end{equation*}
		\item One has a \textbf{one-parameter group of Hopf algebra automorphisms} $\tau_a$, $a\in\mathbb{C}^\times$, of $U_q(\tilde{\mathfrak{g}})$ defined on generators by 
		\begin{equation}
			\begin{split}
				&\tau_a(\mathcal{X}_{i,r}^\pm)=a^r\mathcal{X}_{i,r}^\pm,\quad\tau_a(\mathcal{H}_{i,s})=a^s\mathcal{H}_{i,s},\\
				&\tau_a(\mathcal{K}_i^{\pm1})=\mathcal{K}_i^{\pm1},\quad\tau_a(\mathcal{C}^{\pm1/2})=\mathcal{C}^{\pm1/2}.\footnotemark
				\label{eqn:tau_a_q-deformed}
			\end{split}
		\end{equation}
		\footnotetext{ The proof is a simple verification.}
		It is one of the main reasons for the importance of $U_q(\tilde{\mathfrak{g}})$, as it allows one to define for any given $U_q(\tilde{\mathfrak{g}})$-module a one parameter family of modules by adjoining different parameters $a\in\mathbb{C}^\times$ through $\tau_a$, i.e., pulling back by $\tau_a$. In fact, if $V$ is any representation of $U_q(\tilde{\mathfrak{g}})$, it is convenient to write $V(a)$ for the pull-back of $V$ by
		$\tau_a$, and, of course, $V=V(1)$ (cf. \cite{CP1991}). It is also clear from the definition that $(V(a))(b)=V(ab)$. Therefore, for any given representation $V$, $\tau$ defines an \textbf{action of the multiplicative group} of $\mathbb{C}$ on the family of representations $\{V(a)\}_{a\in\mathbb{C}^\times}$.
		\item In fact, $\tau$ and the square of the antipode $S_q^2$ are closely related. For, we have $S_q^2 = \tau_{q^c}\circ \operatorname{Ad}_{q^x}$, where $c=2d^\vee h^\vee$ is the eigenvalue of the Casimir element of $\mathfrak{g}$ in the adjoint representation,\footnote{ We remind the reader that $h^\vee$ is the dual Coxeter number $h^\vee=c_0+c_1+\dots+c_l$, and that we refer to the (quadratic) Casimir element defined in terms of the standard invariant bilinear form of $\tilde{g}$ (restricted to $\mathfrak{g}$) as defined at the end of Subsection \ref{subsect:aff_Kac--Moody}, and \ref{subsect:Cartan_data}, respectively.} $x=2d^\vee h'_\rho$ and $\rho = \sum_{i=1}^{n}\omega_i$ such that $\rho\mapsto h'_\rho$ by the isomorphism from $\mathfrak{h}^*$ to $\mathfrak{h}$ given by the standard invariant bilinear form of $\tilde{\mathfrak{g}}$ restricted to $\mathfrak{g}$.\footnote{ We used the definition $\mathcal{K}_i=q_i^{h_i}$ here by identifying $d_ih_i$ and $\mathcal{H}_{i,0}$ in Part \ref{enum:second_Drin_h}. We also remind the reader that $d^\vee$ is the maximal number of edges connecting two nodes in the Dynkin diagram of $\mathfrak{g}$, i.e. $d^\vee=1$ for simply-laced, $d^\vee =2$ for $B_l,C_l,F_4$, and $d^\vee=3$ for $G_2$.}\label{enum:almost_ribbon} Let us also emphasize that the signs here depend strongly on the choices of the comultiplication and the antipode for $U_q$, i.e. $\Delta_q$ or $\Delta_q^{\operatorname{op}}$ and then $S_q$ or $S_q^{-1}$ and that there is no strict convention in the literature. We stick to the Definition-proposition \ref{defprop:U_q(g)_chevalley} as given in the book \cite{CPBook} in Section 9.1.  (cf. \cite{NR} Equation (2.59))
		\item Setting \label{enum:grading_quant_aff}
		\begin{equation*}
			\deg(\mathcal{K}_i)=\deg(\mathcal{C})=0,\quad\deg(\mathcal{H}_{i,r})=\deg(\mathcal{X}_{i,r}^\pm)=r,\quad\deg(q)=0
		\end{equation*}
		provides $U_q(\tilde{\mathfrak{g}})$ with the structure of a graded algebra. In fact, $\mathcal{H}_{i,r}$ and $\mathcal{X}_{i,r}$ are the quantum analogues of the elements $H_i t^r$ and $X_i^\pm t^r$ in the loop presentation of $\tilde{g}$ (cf. \cite{D}).
		\item The elements $\varPhi_{i,r}^\pm$ are given in terms of the $\mathcal{H}_{j,s}$ by formulas of the form
		\begin{equation*}
			\varPhi_{i,r}^\pm = \pm \mathcal{K}_i^{\pm 1}((q_i-q_i^{-1})\mathcal{H}_{i,r}+f_{i,r}^\pm(\mathcal{H}_{i,\pm1},\dots,\mathcal{H}_{i,\pm(r-1)})),
		\end{equation*}
		where the $f_{i,r}^\pm$ are polynomials with coefficients in $\mathbb{C}(q)$, homogeneous of degree $r$ with respect to the grading in Part \ref{enum:grading_quant_aff}. $\odot$
	\end{enumerate}
\end{rems}
We note that Part \ref{enum:almost_ribbon} of the preceding Remarks \ref{rems:quant_aff} shows that $U_q(\tilde{\mathfrak{g}})$ is up to a rescaling with $\tau$ somewhat close to being a ribbon Hopf algebra provided that one can define a universal $R$-matrix in some way. In fact, we will see that this is possible on the level of finite-dimensional representations in Subsection \ref{subsect:trigon_R_matrix_A_l}. This justifies the notion of a quantum trace for finite-dimensional representations of quantum affine algebras. Clearly, if $V$ is any finite-dimensional representation of $U_q(\tilde{\mathfrak{g}})$, we have $V^{**}(\lambda)\cong V(\lambda q^c)$ (and $^{**}V(\lambda) \cong V(\lambda q^{-c})$) where the isomorphism is given by conjugation with the element $g\coloneq q^x$ (and $g\coloneq q^{-x}$). This motivates the following definition (cf. \cite{HJue} Definition 3.10).
\begin{definition}[the shifted dual representation]\label{def:the_shifted_dual}
	Let $V$ be a finite-dimensional representation of $U_q(\tilde{\mathfrak{g}})$. Then we denote by $V^\circledast \cong V^*(q^{-d^\vee h^\vee}) \cong {}^*V(q^{d^\vee h^\vee})$ the unique (up to isomorphism) dual of $V$ considered as a representation of $U_q(\mathfrak{g})\lhook\joinrel\xrightarrow{\iota} U_q(\tilde{\mathfrak{g}})$ such that $V^{\circledast\circledast}\cong V$ as a representation of $U_q(\tilde{\mathfrak{g}})$. $\odot$
\end{definition}
Let us also state a weak version of the PBW theorem for the specialisation with $q\in\mathbb{C}^\times$ transcendental (cf. \cite{CPBook} Prop 12.2.2). We refer the reader to the same paper of Beck \cite{Beck94_1} for the more general result.
\begin{prop}
	Let $q\in\mathbb{C}^\times$ be transcendental. Let $U_q^\pm$ and $U_q^0$ be the subalgebras of $U_q(\tilde{\mathfrak{g}})$ generated by the elements $\mathcal{X}_{i,r}^\pm$ for all $i,r$ and $\mathcal{C}^{\pm 1/2}, \mathcal{K}_i^{\pm 1}, \mathcal{H}_{i,r}$ for all $i,r$, respectively. Then
	\begin{equation*}
		U_q(\tilde{\mathfrak{g}})= U_q^-.U_q^0.U_q^+.\quad\odot
	\end{equation*}
\end{prop}
Finally, we note that, by Part \ref{enum:grading_quant_aff} of the preceding Remarks \ref{rems:quant_aff}, $U_q^0$ is also generated by the $\mathcal{C}^{\pm 1/2},K_i^{\pm 1}$ and $\varPhi_{i,r}^\pm$.
\section{Finite-dimensional representations}\label{sect:finite-dim_reps}
In this section, we prepare the discussion of finite-dimensional representations of (untwisted) quantum affine algebras. In fact, we will explain that this essentially reduces (up to twisting) to the consideration of finite-dimensional (type $\boldsymbol{1}$) representations of the quantum loop algebra $U_q(\mathfrak{L}(\mathfrak{g}))$ of a finite-dimensional complex simple Lie algebra $\mathfrak{g}$. Moreover, in Subsection \ref{subsect:drinfeld_poly} we will discuss the so-called evaluation representations for type $A_l$ and connect them with a notion of highest weights given by the so-called 'Drinfeld polynomials'. This is mostly explained in Subsection 12.2B in the book \cite{CPBook} and the paper \cite{CP1991}. Then, in Subsection \ref{subsect:trigon_R_matrix_A_l} we will briefly introduce the trigonometric $R$-matrix as the type $A_l$ fundamental representation '$(\rho_f\otimes\rho_f)(\tilde{\mathcal{R}})$' of the pseudo-universal $R$-matrix of $U_q(\tilde{\mathfrak{g}})$ as is done in the paper \cite{BGKNR} by applying the famous Khoroshkin and Tolstoy formula \cite{KT}. When $l=1$, this is the $R$-matrix of the so-called six-vertex model, a model which can be interpreted as a realisation of planar ice, but is also directly connected to the anisotropic Heisenberg XXZ spin chain. Finally, in Subsection \ref{subsect:graphical_notation}, we give a short introduction to a graphical notation which can be used to visualize the braiding of representations, and, in particular, the Yang-Baxter equation.\\

However, before we go into the details about Drinfeld polynomials and finite-dimensional representations of quantum loop algebras, we should say a few words about the analogy of the (finite-dimensional) representation theory of $U_q(\mathfrak{g})$ and $\mathfrak{g}$ when $\mathfrak{g}$ is a finite-dimensional complex simple Lie algebra. We refer the reader to Subsection 10.1A in the book \cite{CPBook} for the details. The analogy goes as follows.

We have seen in Subsection \ref{subsect:Cartan_data} that the finite-dimensional irreducible representations of $\mathfrak{g}$ are parametrised by their highest weight (theorem of the highest weight). The main result for $U_q(\mathfrak{g})$ (as defined in \ref{defprop:U_q(g)_chevalley}) is that, up to a rather trivial twisting, this is still the case. 

We remind the reader that $d^\vee$ denotes the maximal number of edges connecting two nodes in the Dynkin diagram of $\mathfrak{g}$, i.e. $d^\vee=1$ for simply-laced, $d^\vee =2$ for $B_l,C_l,F_4$, and $d^\vee=3$ for $G_2$.
Then, by a \textbf{weight} for $U_q(\mathfrak{g})$ we mean an $l$-tuple $\boldsymbol{\xi}=(\xi_1,\dots,\xi_l)\in(\mathbb{Q}(q)^\times)^l$. Moreover, we introduce a partial ordering of weights by writing $\boldsymbol{\xi}'\leq\boldsymbol{\xi}$, $\boldsymbol{\xi},\boldsymbol{\xi}'\in\mathbb{Q}(q)$, if $\xi_i'^{-1}\xi_i=q^{d^\vee \langle\alpha_i,\beta\rangle}$ for some $\beta\in Q^+$ and all $i\in I =\{1,\dots,l\}$. If $V$ is a (left) $U_q$-module, its \textbf{weight spaces} are all non-zero $\mathbb{Q}(q)$-linear subspaces of $V$ of the form
\begin{equation*}
	V_{\boldsymbol{\xi}} = \{v\in V\,|\,K_i.v=\xi_i v,\,i\in I\},
\end{equation*}
where $\boldsymbol{\xi}$ is a weight. A \textbf{primitive} vector in $V$ is a non-zero vector $v\in V$ such that $X_i^+.v=0$ for all $i\in I$, and $v\in V_{\boldsymbol{\xi}}$ for some $\boldsymbol{\xi}$. A \textbf{highest weight $U_q$-module} is a $U_q$-module $V$ which contains a primitive vector $v$ such that $V=U_q.v$. Then, by applying Proposition \ref{prop:PBW_U_q} it is clear that, if $v\in V_{\boldsymbol{\xi}}$, we have $\dim_{\mathbb{Q}(q)}(V_{\boldsymbol{\xi}})=1$ and that
\begin{equation*}
	V=\bigoplus_{\boldsymbol{\xi}'\leq\boldsymbol{\xi}}V_{\boldsymbol{\xi}'}.
\end{equation*}
In this case, $\boldsymbol{\xi}$ is called the \textbf{highest weight} of $V$ and $v$ is called a \textbf{highest weight vector}. Vice versa, for any weight $\boldsymbol{\xi}$ define the \textbf{Verma module} $M_q(\boldsymbol{\xi})$ to be the quotient of $U_q$ by the left ideal generated by the $X_i^+$ and $K_i-\xi_i1$, $i\in I$. Then, $M_q(\boldsymbol{\xi})$ is clearly a highest weight module with highest weight $\boldsymbol{\xi}$ and canonical highest weight vector $v_{\boldsymbol{\xi}}$ equal to the image of $\boldsymbol{1}\in U_q$. Moreover, every highest weight module with highest weight $\boldsymbol{\xi}$ is isomorphic to a quotient of $M_q(\boldsymbol{\xi})$. As in the classical case, it follows from Proposition \ref{prop:PBW_U_q} that $M_q(\boldsymbol{\xi})$ is a free $U_q^-$-module generated by $v_{\boldsymbol{\xi}}$. Since $\dim(M_q(\boldsymbol{\xi})_{\boldsymbol{\xi}})=1$, $M_q(\boldsymbol{\xi})$ similarly has a unique irreducible quotient $V_q(\boldsymbol{\xi})$ and every highest weight module is isomorphic to some $V_q(\boldsymbol{\xi})$.

Let us also introduce the notion of integrability. We may consider modules with highest weight $\boldsymbol{\xi}_{\sigma,\lambda}$ of the form
\begin{equation}
	\xi_i = \sigma(\alpha_i)q^{d^\vee\langle\alpha_i,\lambda\rangle},\label{eqn:integral_U_q}
\end{equation}
where $\lambda\in P$ and $\sigma\in Q_2^*$ ($=\{Q\to\{\pm1\}\}$ as explained after Proposition \ref{prop:twisting}). We note that the partial ordering induces on the weights $\boldsymbol{\xi}_{\sigma,\lambda}$ of fixed \textbf{type} $\sigma$ the usual partial ordering on $P$, i.e. $\boldsymbol{\xi}_{\sigma,\lambda'}\leq \boldsymbol{\xi}_{\sigma,\lambda}$ iff $\lambda-\lambda'\in Q^+$. In particular, the weights of a $U_q$-module with highest weight $\boldsymbol{\xi}_{\sigma,\lambda}$ are all of the form $\boldsymbol{\xi}_{\sigma,\mu}$ with $\mu\leq\lambda$. Then, we say that a $U_q$-module is \textbf{integrable} if it is the direct sum of its weight spaces and the $X_i^\pm$ act locally nilpotent on $V$, i.e. for any $v\in V$ there exists an $r\in \mathbb{N}$ such that $(X_i^\pm)^r.v=0$. In fact, we have the following propositions.
\begin{prop}
	The irreducible $U_q$-module $V_q(\boldsymbol{\xi})$ with highest weight $\boldsymbol{\xi}$ is integrable iff $\boldsymbol{\xi}=\boldsymbol{\xi}_{\sigma,\lambda}$ for some $\sigma\in Q_2^*$, $\lambda\in P^+$. $\odot$
\end{prop}
The reason for the importance of integrable modules is stated in
\begin{prop}\label{prop:U_q_integrable_highest_weight}
	Every finite-dimensional irreducible $U_q$-module is integrable and highest weight. $\odot$
\end{prop}
In analogy to the classical case we may therefore call weights of the form (\ref{eqn:integral_U_q}) \textbf{integral} and, if $\lambda\in P^+$, \textbf{dominant}.

Further, we note that the module $V_q(\boldsymbol{\xi}_{\sigma,0})$ is one-dimensional, for we have
\begin{equation*}
	X_i^\pm.v_{\boldsymbol{\xi}_{\sigma,0}}=0,\quad K_i.v_{\boldsymbol{\xi}_{\sigma,0}}=\sigma(\alpha_i)v_{\boldsymbol{\xi}_{\sigma,0}}.
\end{equation*}
Hence
\begin{equation*}
	V_q(\boldsymbol{\xi}_{\sigma,\lambda})\cong V_q(\boldsymbol{\xi}_{\sigma,0})\otimes V_q(\boldsymbol{\xi}_{\boldsymbol{1},\lambda})
\end{equation*}
as $U_q$-modules, where $\boldsymbol{1}(\alpha_i)=1$ for all $i\in I$. We conclude that it is enough to consider modules of type $\boldsymbol{1}$ and, by abuse of notation, refer to a weight $\boldsymbol{\xi}_{\boldsymbol{1},\lambda}$ simply as $\lambda$, and $V_q(\boldsymbol{\xi}_{\boldsymbol{1},\lambda})$ as $V_q(\lambda)$.

If $W$ is any highest weight $U_q$-module with highest weight $\lambda\in P$, its character $\chi(W)\in\mathbb{Z}[P]$ is defined, as in the classical case, by
\begin{equation*}
	\chi(W)=\sum_{\mu\in P}\dim(W_\mu)e^\mu,
\end{equation*}
where $e^\mu$ denotes the basis element of the group algebra $\mathbb{Z}[P]$ of $P$ corresponding to $\mu\in P$.

We have the following important result about the representations of $U_q(\mathfrak{sl}_2)$.
\begin{xmpl}\label{xmpl:spin_k/2_U_q}
	There are exactly two irreducible representations of $U_q(\mathfrak{sl}_2)$ of each (finite) dimension $k+1$, $k\in\mathbb{N}$. These are the modules $V_{q}(\boldsymbol{\xi}_{\sigma,k})$, $\sigma=\pm 1$, with basis $\{v_0^{(k)},\dots,v_k^{(k)}\}$ on which the action of the generators of $U_q$ is given by
	\begin{equation*}
		K_1.v_r^{(k)}=\sigma q^{\lambda-2r}v_r^{(k)},\quad X_1^+.v_r^{(k)}=\sigma[k-r+1]_q v_{r-1}^{(k)},\quad X_1^-.v_r^{(k)}=[r+1]_q v_{r+1}^{(k)}.
	\end{equation*}
	The vector $v_0^{(k)}$ is a highest weight vector of $V_q(\boldsymbol{\xi}_{\sigma,k})$.  If $\sigma=\boldsymbol{1}$, these are the counterparts of the $k+1$-dimensional representation for $U_h(\mathfrak{sl}_2)$ if one puts $q=e^h$ and $K_1=e^{hH_1}$. Those with $\sigma=-1$ have no analogues for $U_h(\mathfrak{sl}_2)$ (see \cite{CPBook} Subsection 9.1B).\footnote{ Reducing $\sigma_i e^{d_i h H_i}=e^{d_i h \eta_i}$ $(\operatorname{mod}h)$ for some $\eta_i$ only allows for $\sigma_i=1$.} $\odot$
\end{xmpl}
We note that, in the limit $q\to 1$, the formulas in Example \ref{xmpl:spin_k/2_U_q} describe the $(k+1)$-dimensional irreducible representation of $\mathfrak{sl}_2$ (by interpreting '$K_1=q^H$'). More generally, it can be shown (by using an 'integral' form of $U_q$ and specialising $q$ to $1$) that the structure of the $U_q(\mathfrak{g})$-modules $V_q(\lambda)$ with $\lambda\in P^+$ is exactly parallel to that of the corresponding highest weight $\mathfrak{g}$-modules (see \cite{CPBook} Proposition 10.1.4 for the details).
This is summarised in
\begin{prop}
	Let $V_q(\lambda)$ be the irreducible highest weight $U_q(\mathfrak{g})$-module with highest weight $\lambda\in P^+$. Then, there exists an irreducible $U(\mathfrak{g})$-module $\overline{V_q(\lambda)}$ over $\mathbb{Q}$ with highest weight $\lambda$ such that
	\begin{equation*}
		\dim_{\mathbb{Q}(q)}(V_q(\lambda)_\mu) = \dim_{\mathbb{Q}}(\overline{V_q(\lambda)}_\mu)
	\end{equation*}
	for all $\mu\in P$. $\odot$
\end{prop}
And we have
\begin{cor}\label{cor:finite_dim_irred_and_highest_weight}
	Every finite-dimensional highest weight $U_q(\mathfrak{g})$-module is irreducible and has highest weight in $P^+$.
\end{cor}
Aswell as
\begin{thm}
	Every finite-dimensional $U_q(\mathfrak{g})$-module is completely reducible. $\odot$
\end{thm}
This concludes our discussion of the analogy between finite-dimensional representations of $U_q(\mathfrak{g})$ and $\mathfrak{g}$.
\subsection{Drinfeld polynomials and $q$-characters}\label{subsect:drinfeld_poly}
Having introduced the second Drinfeld (or loop) realisation of (untwisted) quantum affine algebras in Subsection \ref{subsect:quantum_affine}, we are now in the position to explain its importance for the finite-dimensional representation theory of $U_q(\tilde{\mathfrak{g}})$. In fact, it turns out that the quantum loop description gives rise to the definition of the so-called 'Drinfeld polynomials', which describe the irreducible representations of $U_q(\tilde{\mathfrak{g}})$ in analogy to the description of irreducible representations of finite-dimensional complex simple Lie algebras in terms of highest weights. In fact, the term 'Drinfeld polynomial' goes back to an analogous classification result for Yangians obtained earlier by Drinfeld \cite{D2} which we are not going to discuss here (see for examples \cite{CPBook} Section 12.1). Here, our goal is to give an introduction to the notion of Drinfeld polynomials and write down the first consequences. Later, in Chapter \ref{ch:fin_dim_rep_quantum_aff}, we will introduce a more convenient notation, put it into the bigger context of $q$-characters, loop-weights and snake modules and discuss the cluster algebra structure of the Grothendieck ring of the category of (type $\boldsymbol{1}$) finite-dimensional representations of $U_q(\tilde{\mathfrak{g}})$ when $q$ is not a root of unity. Again, the introduction given here will be mostly in parallel to Subsection 12.2B in the book \cite{CPBook}, but we will also use some additional information provided in the papers \cite{CP1991} and \cite{CP1994}.\\

Let us assume in the following that $q\in\mathbb{C}^\times$ is transcendental.\footnote{In principle, the assumption that $q$ in not a root of unity is sufficient for everything we will discuss. However, the discussion of specialisations is not the focus of this thesis.}
\begin{definition}[type $\boldsymbol{1}$]
	\label{def:type_1}
	We say that a representation $V$ of $U_q(\tilde{\mathfrak{g}})$ is of \textbf{type} $\boldsymbol{1}$, if the $K_i$, $i=0,1,\dots,l$, act semisimply on $V$ with eigenvalues which are integer powers of $q$ and $\mathcal{C}^{1/2}=\boldsymbol{1}$ on $V$, i.e., $\mathcal{C}^{1/2}$ acts as the identity on $V$ and $V$ is of type $\boldsymbol{1}$ as a representation of $U_q(\mathfrak{g})$. $\odot$
\end{definition}
In fact, these are representations of the quotient of $U_q(\tilde{\mathfrak{g}})$ by the ideal generated by $\mathcal{C}^{1/2}-1$, which is the quantum loop algebra $U_q(\mathfrak{L}(\mathfrak{g}))$.

We have seen in Proposition \ref{prop:twisting} that there are $2^{l+1}$ $\mathbb{C}(q)$-algebra automorphisms of $U_q(\tilde{\mathfrak{g}})$ given on generators by
\begin{equation}
	K_i\mapsto \sigma_i K_i,\quad X_i^+\mapsto \sigma_i X_i^+,\quad X_i^-\mapsto X_i^-,\label{eqn:twisting_sigma}
\end{equation}
for any set of signs $\sigma_i\in\{\pm 1\}$, $i=0,1,\dots,l$. In addition, there is an automorphism given in terms of the second Drinfeld (or loop algebra-like) generators by
\begin{equation}\label{eqn:twisting_loop_grading}
	\begin{split}
		\mathcal{C}^{1/2}\mapsto -\mathcal{C}^{1/2},&\quad\mathcal{X}_{i,r}^\pm\mapsto (-1)^r\mathcal{X}_{i,r}^\pm,\\
		\mathcal{K}_i\mapsto\mathcal{K}_i,&\quad\mathcal{H}_{i,r}\mapsto\mathcal{H}_{i,r}.
	\end{split}
\end{equation}
Using these, we can state
\begin{prop}\label{prop:type_1_twisting_loop}\hspace{1em}
	\begin{enumerate}[label=\normalfont(\roman*)]
		\item\label{enum:type_1_twisting} Every finite-dimensional irreducible representation $V$ of $U_q(\tilde{\mathfrak{g}})$ can be obtained from a type $\boldsymbol{1}$ representation by twisting with a product of some of the automorphisms (\ref{eqn:twisting_sigma}) and (\ref{eqn:twisting_loop_grading}).
		\item\label{enum:part_2_requirements} Every finite-dimensional (non-zero) type $\boldsymbol{1}$ representation $V$ of $U_q(\tilde{\mathfrak{g}})$ contains a non-zero vector $v$ which is annihilated by the $\mathcal{X}_{i,r}^+$ for all $i$, $r$, and is a simultaneous eigenvector for the elements of $U_q^0$. $\odot$
	\end{enumerate}
\end{prop}
As this proposition is very important for the rest of our consideration, let us also write down the proof and comment on some of the steps to clarify what happens.
\begin{proof}[Proof of Proposition \ref{prop:type_1_twisting_loop}]
	Let $V$ be a finte-dimensional representation of $U_q(\mathfrak{g})$, and let
	\begin{equation*}
		V^0=\{v\in V\,|\,\mathcal{X}_{i,r}^+.v=0\text{ for all } i,r\}.
	\end{equation*}
	We assume for a contradiction that $V^0=0$. Then, let $w$ be any non-zero simultaneous eigenvector of the $K_i$, $i=1,\dots,l$. By the assumption, there is an infinite sequence of pairs $(i_k,r_k)$, $k\in\mathbb{N}$, such that the vectors $v, \mathcal{X}_{i_1,r_1}.w, \mathcal{X}_{i_2,r_2}\mathcal{X}_{i_1,r_1}.w, \dots$ are all non-zero. Since these vectors have different weights for the action of $U_q(\mathfrak{g})$, they are linearly independent and therefore leading to $V$ being infinite-dimensional, a contradiction.
	
	So, let $0\neq v\in V^0$. Then, by considering the $U_{q_i}(\mathfrak{sl}_2)$-subalgebra of $U_q(\tilde{\mathfrak{g}})$ generated by $X_i^\pm$, $K_i^\pm$, and using Corollary \ref{cor:finite_dim_irred_and_highest_weight} it follows that
	\begin{equation*}
		K_i.v = \sigma_i q^{s_i}v,\quad i=0,1,\dots,l,
	\end{equation*}
	for some $\sigma_i =\pm 1$ and $s_i\in \mathbb{Z}$ where $s_0\leq 0$ and $s_1,\dots,s_l\geq 0$.\\
	
	Let us stop here for a moment. Of course, one can verify that $s_0\leq 0$ and $s_1,\dots,s_l\geq 0$ indeed follows from a direct calculation. However, we have already seen in Subsection \ref{subsect:loop_constr_untw_Kac--Moody} how $\tilde{\mathfrak{g}}$ can be constructed from $\mathfrak{g}$ and therefore how '$\mathfrak{g}$ sits in $\tilde{\mathfrak{g}}$'. Thus, it is clear from the explanation of the representation theory for $U_q(\mathfrak{g})$ at the begin of this Section, that the $\sigma_i q^{s_i}$, $i\in I=\{1,\dots,l\}$, correspond to a dominant integral highest weight $\boldsymbol{\omega}_{\sigma,\lambda}$ such that $\sigma$ is given by the $\sigma_i$ and the $s_i$ determine $\lambda\in P^+$ by the condition $d^\vee\langle\alpha_i,\lambda\rangle = s_i$, $i\in I$. Therefore all the $s_i$, $i\in I$ are positive. Now, one only has to determine the action of $K_0$ on $v$. Using the isomorphism $f$ given in the second Drinfeld realization Theorem \ref{thm:second_drin_quantum_aff} and the definition that $v$ is annihilated by the $\mathcal{X}_{i,r}^+$, it follows immediately that $v$ is annihilated by $X_0^-$. Hence, $v$ is a lowest weight vector for $K_0$ and therefore $s_0\leq 0$ must follow. Moreover, we have seen in Subsection \ref{subsect:loop_constr_untw_Kac--Moody} that $\alpha_0=-\theta+\delta$ and it is therefore clear that $s_0 = d^\vee\langle \alpha_0,\lambda\rangle = -d^\vee\langle \theta,\lambda\rangle$, which is obviously smaller or equal than zero.\footnote{ Because $\lambda\in P^+$ dominant integral and $\theta\in Q^+$ is the highest root of $\mathfrak{g}$. To be precise, we should clarify that the weight $\lambda$ in the first bracket is understood by the trivial extension such that $\lambda(c)=\lambda(d)=0$ as in Subsection \ref{subsect:loop_constr_untw_Kac--Moody}.} We should also note that this calculation for $U_q(\tilde{\mathfrak{sl}}_2)$ is completely elementary as the isomorphism $f$ is given on generators by
	\begin{equation*}
		\begin{split}
			&K_0\mapsto \mathcal{C}\mathcal{K}^{-1}_1,\quad K_1\mapsto\mathcal{K}_1,\quad X_1^\pm\mapsto \mathcal{X}_{1,0}^\pm,\\
			& X_0^+\mapsto \mathcal{X}_{1,1}^-\mathcal{K}_1^{-1},\quad X_0^-\mapsto\mathcal{C}^{-1}\mathcal{K}_1\mathcal{X}_{1,-1}^+.
		\end{split}
	\end{equation*}\\
	
	However, let us come back to the actual proof. We remind the reader that the highest root $\theta$ may be written in the form $\theta =\sum_{i=1}^{n}a_i\alpha_i$. Then, by multiplying the $K_i^{a_i}$, $i\in I$, with $K_0$ we can determine the action of $\mathcal{C}$ as
	\begin{equation*}
		\mathcal{C}.v=\sigma_0 \prod_{i=1}^{n}\sigma_i^{a_i}q^{s_0+\sum_{i=1}^{n}a_i s_i}v.
	\end{equation*}
	On the other hand, considering $V$ as a representation of $U_q(\mathfrak{sl}_2)$ via the homomorphism $U_{q_i}(\mathfrak{sl}_2)\to U_q(\tilde{\mathfrak{g}})$ given by
	\begin{equation*}
		X_1^\pm \mapsto \mathcal{X}_{i,\pm r}^\pm,\quad K_1\mapsto 	\mathcal{C}^r\mathcal{K}_i,\quad i\in I,\, r\in\mathbb{Z},
	\end{equation*}
	we must have
	\begin{equation*}
		r\left(s_0+\sum_{i=1}^{n}a_i s_i\right)+s_i\geq 0
	\end{equation*}
	for all $r\in \mathbb{Z}$. Hence, $s_0+\sum_{i}a_i s_i=0$ and
	\begin{equation*}
		\mathcal{C}.v = \sigma_0\prod_{i=1}^{n}\sigma_i^{a_i}v.
	\end{equation*}\\
	
	Indeed, this was already clear from our first comment when we insert $\theta$ in the form $\theta=\sum_{i=1}^{n}a_i \alpha_i$ into our assertion that $s_0 = -d^\vee\langle\theta,\lambda\rangle$. However, this is a direct proof of this analogy.\\
	
	Now, by twisting with an automorphism (\ref{eqn:twisting_sigma}), we can assume that $\sigma_i=1$ for $i=0,1,\dots,l$, and then $\mathcal{C}.v=v$. If $V$ is irreducible, $\mathcal{C}$ acts as $\boldsymbol{1}$ on $V$. Hence, either $\mathcal{C}^{1/2}=1$ on $V$ or $\mathcal{C}^{1/2}=-1$ on $V$. The second case can be transformed into the first by twisting with the automorphism (\ref{eqn:twisting_loop_grading}). This proves \ref{enum:type_1_twisting}.
	
	If $V$ is of type $\boldsymbol{1}$, it follows from the fourth set of relations in Theorem \ref{thm:second_drin_quantum_aff} that the elements of $U_q^0$ act on $V$ as commuting operators, and from the fifth and sixth relations that these operators preserve $V^0$. Therefore, any simultaneous eigenvector $v\in V^0$ for the action of $U_q^0$ satisfies the requirements in part \ref{enum:part_2_requirements}.
\end{proof}

Thus, studying the finite-dimensional irreducible representations of $U_q(\tilde{\mathfrak{g}})$ reduces to studying the representations of the \textbf{quantum loop algebra} $U_q(\mathfrak{L}(\mathfrak{g}))$, the quotient of $U_q(\tilde{g})$ by the ideal generated by $C^{1/2}-1$. Therefore, we will only consider type $\boldsymbol{1}$ representations.
\begin{rem}[loop-weight spaces]\label{rem:loop_weight_spaces}
	Since the $\varPhi_{i,r}^\pm$ mutually commute with each other, we can decompose any finite-dimensional type $\boldsymbol{1}$ representation $V$ of $U_q(\tilde{\mathfrak{g}})$ into a direct sum
	\begin{equation*}
		V = \bigoplus_{\boldsymbol{\psi}=(\psi_{i,r}^\pm)_{i\in I,r\in\mathbb{Z}}} V_{\boldsymbol{\psi}}
	\end{equation*}
	of generalised eigenspaces
	\begin{equation*}
		V_{\boldsymbol{\psi}} = \{v\in V|\,\text{there exists a }p,\text{ such that } (\varPhi_{i,r}^\pm-\psi_{i,r}^\pm)^p.v=0,\text{ for all }i\in I,n\in \mathbb{Z}\}
	\end{equation*}
	
	Moreover, since $\varPhi_{i,0}^\pm= \mathcal{K}_i^{\pm1}$, all vectors in $V_{\boldsymbol{\psi}}$ have the same ($U_q(\mathfrak{g})$-)weight. Thus, the decomposition of $V$ into a direct sum of subspaces $V_{\boldsymbol{\psi}}$ is a refinement of its weight decomposition. We will later call the spaces $V_{\boldsymbol{\psi}}$ \textbf{loop-weight spaces}. (cf. \cite{FM} Subsection 2.2) $\odot$
\end{rem}
\begin{definition}
	Given a collection $(\psi_{i,r}^\pm)_{i\in I,r\in\mathbb{Z}}$ of generalised eigenvalues, we form the generating functions
	\begin{equation*}
		\Psi_i^\pm(u)=\sum_{r= 0}^{\infty}\psi^\pm_{i,\pm r}u^{\pm r}.
	\end{equation*}
	We will refer to each collection $\{\Psi_i^\pm(u)\}_{i\in I}$ occurring on a given representation $V$ as the \textbf{common (generalised) eigenvalues} of $\{\boldsymbol{\varPhi}_i^\pm(u)\}_{i\in I}$ on $V$, and to $\dim(V_{\boldsymbol{\psi}})$ as the multiplicity of this eigenvalue. (cf. \cite{FM} Subsection 2.2) $\odot$
\end{definition}
\begin{rem}
	Let $\mathfrak{B}_V$ be a Jordan basis of $\varPhi_{i,r}^\pm$, $i\in I$, $r\in\mathbb{Z}$ and consider the module $V(a)=\tau^*_a(V)$ (see formula (\ref{eqn:tau_a_q-deformed})). Then $V(a)=V$ as a vector space. Furthermore, the decomposition of $V$ into the direct sum of generalised eigenspaces for the $\varPhi_{i,r}^\pm$ does not depend on $a$. This is because the action of $\varPhi_{i,r}^\pm$ on $V$ and on $V(a)$ differs only by scalar factors $a^n$. In particular, $\mathfrak{B}_V$ is a Jordan basis for $\varPhi_{i,r}^\pm$ acting on $V(a)$ for all $a\in\mathbb{C}^\times$. If $v\in \mathfrak{B}_V$ is a generalised eigenvector with common eigenvalues $\{\Psi_i^\pm(u)\}_{i\in I}$, then the corresponding common eigenvalues on $v$ in $V(a)$ are $\{\Psi_i^\pm(au)\}_{i\in I}$. $\odot$
\end{rem}
With this, the definition of the so-called 'Drinfeld polynomials' goes back to the following suggestion.
\begin{definition}[highest weight]\label{def:pseudo-highest_weight}
	A type $\boldsymbol{1}$ representation $V$ of $U_q(\tilde{\mathfrak{g}})$ is called highest weight if it is generated by a vector $v_0$ which is annihilated by the $\mathcal{X}_{i,r}^+$ and a simultaneous eigenvector for the $\mathcal{K}_i$ and the $\mathcal{H}_{i,r}$, $i\in I=\{1,\dots,l\}$, $r\neq 0$. If $\varPhi_{i,r}^\pm.v_0=\varphi_{i,r}^\pm v_0$, the collection of complex numbers $\varphi_{i,r}^\pm$, denoted by $\boldsymbol{\varphi}$, is called the (pseudo-)highest weight of $V$.\footnote{ We refer the reader to the remark after Definition 12.2.4 in the book of Chari and Pressley \cite{CPBook} for an explanation of the term 'pseudo'. For us, the term will sometimes be useful to distinguish the $U_q(\tilde{\mathfrak{g}})$-weights from the usual weights of $U_q(\mathfrak{g})$, but we will leave it when it is clear.} $\odot$
\end{definition}
Together with the proof of Proposition \ref{prop:type_1_twisting_loop}, we conclude
\begin{cor}
	Every finite-dimensional irreducible type $\boldsymbol{1}$ representation $V$ of $U_q(\tilde{\mathfrak{g}})$ is highest weight. $\odot$
\end{cor}
We note that $\varphi_{i,0}^+\varphi_{i,0}^-=1$, $\varphi_{i,r}^+=0$ for $r<0$ and $\varphi_{i,r}^-=0$ for $r>0$. Moreover, one proves that there is a unique irreducible $U_q(\mathfrak{L}(\mathfrak{g}))$-module $V(\boldsymbol{\varphi})$ with highest weight $\boldsymbol{\varphi}$ for any collection of scalars $\boldsymbol{\varphi}=\{\varphi_{i,r}^\pm\}_{i\in I, r\in\mathbb{Z}}$ satisfying these conditions. This is done, as in the classical case, by an analogous consideration of the Verma module $M(\boldsymbol{\varphi})$. Then, $V(\boldsymbol{\varphi})$ is the unique irreducible quotient of highest weight $\boldsymbol{\varphi}$.

It only remains to determine the condition for which the $V(\boldsymbol{\varphi})$ are finite-dimensional. To state it, let us make the following basic
\begin{definition}\label{def:affine_positive_weight_lattice}
	We define $\mathcal{P}^+$ to be the set of all $I$-tuples $\boldsymbol{P} = (P_i)_{i\in I}$ of polynomials $P_i\in\mathbb{C}[u]$ with constant term $1$.
\end{definition}
With this definition, the observation goes as follows.
\begin{thm}[Drinfeld polynomials]\label{thm:drinfeld_polynomials}
	let $V$ be a finite-dimensional irreducible $U_q(\tilde{\mathfrak{g}})$-module of type $\boldsymbol{1}$ with highest weight $\boldsymbol{\varphi}=\{\varphi_{i,r}^\pm\}_{i\in I, r\in\mathbb{Z}}$. Then, there exists a unique element $\boldsymbol{P}_V = (P_{i,V})_{i\in I}\in\mathcal{P}^+$, such that
	\begin{equation*}
		\sum_{r=0}^{\infty}\varphi_{i,r}^+ u^r = q_i^{\deg(P_{i,V})}\frac{P_{i,V}(q_i^{-2}u)}{P_{i,V}(u)} = \sum_{r=0}^{\infty}\varphi_{i,-r}^- u^{-r},
	\end{equation*}
	in the sense that the left- and the right-hand sides are the Laurent expansions of the middle term about $0$ and $\infty$, respectively. Moreover, every $I$-tuple $\boldsymbol{P}=(P_i)_{i\in I}\in\mathcal{P}^+$ arises from a finite-dimensional irreducible $U_q(\tilde{\mathfrak{g}})$-module of type $\boldsymbol{1}$ in this way.
	
	Furthermore, if $V$ and $W$ are finite-dimensional type $\boldsymbol{1}$ representations of $U_q(\tilde{\mathfrak{g}})$ such that $V\otimes W$ is irreducible, then
	\begin{equation*}
		P_{i,V\otimes W} = P_{i,V}P_{i,W}.
	\end{equation*}
	In particular, $V\otimes W\cong W\otimes V$. $\odot$
\end{thm}
It should be clear that the proof is one of the main steps towards understanding the finite-dimensional type $\boldsymbol{1}$ representations of $U_q(\mathfrak{L}(\mathfrak{g}))$. Therefore, as for the last proposition, we shall comment on it and possibly clarify minor typos in \cite{CPBook} or \cite{CP1991}. However, we will still omit some technical details and refer the reader to Theorem 3.4 in the paper \cite{CP1991} and Theorem 12.2.6 in the book \cite{CPBook} (and the corresponding proofs after them).
\begin{proof}[Proof of Theorem \ref{thm:drinfeld_polynomials}]
	The proof is based on the following lemma. Let $U_qX^k_\pm\coloneq\sum_{i_k,r_k}\dots\sum_{i_2,r_2}\sum_{i_1,r_1} U_q(\mathfrak{L}(\mathfrak{g})).\mathcal{X}_{i_1,r_1}^\pm\mathcal{X}_{i_2,r_2}^\pm\cdots\mathcal{X}_{i_k,r_k}^\pm$, then we have
	\begin{lem}\label{lem:Newtons_formulae}
		Define elements $\mathcal{P}_{i,\pm r}^\pm\in U_q^0$, for $i\in I$, $r\geq0$, inductively by setting $\mathcal{P}_{i,0}^\pm=1$ and, for $r>0$,
		\begin{equation}
			\mathcal{P}_{i,\pm r}^\pm =\frac{\mp q_i^{\pm r}}{q_i^r-q_i^{-r}}\sum_{s=0}^{r-1}\varPhi^\pm_{i,\pm(s+1)}\mathcal{P}^\pm_{i,\pm(r-s-1)}\mathcal{K}_i^{\mp 1}.\label{eqn:dinfeld_poly_coeff}
		\end{equation}
		Then,
		\begin{equation}\label{eqn:coeff_finite_poly}
			\begin{split}
				&\mathcal{P}_{i,r}^+\equiv (-1)^r q_i^{r^2}\frac{(\mathcal{X}_{i,0}^+)^r(\mathcal{X}_{i,1}^-)^r}{([r]_{q_i})^2},\\
				&\mathcal{P}_{i,-r}^-\equiv (-1)^r q_i^{-r^2}\frac{(\mathcal{X}_{i,-1}^+)^r(\mathcal{X}_{i,0}^-)^r}{([r]_{q_i})^2},
			\end{split}
		\end{equation}
		\begin{equation}\label{eqn:coeff_inverse_poly}
			\begin{split}
				&(-1)^rq_i^{r(r-1)}\frac{(\mathcal{X}_{i,0}^+)^{r-1}(\mathcal{X}_{i,1}^-)^r}{[r-1]_{q_i}[r]_{q_i}}\equiv-\sum_{s=0}^{r-1}\mathcal{X}_{i,s+1}^-\mathcal{P}^+_{i,r-s-1}\mathcal{K}_i^{r-1},\\
				&(-1)^rq_i^{-r(r-1)}\frac{(\mathcal{X}_{i,-1}^+)^{r-1}(\mathcal{X}_{i,0}^-)^r}{[r-1]_{q_i}[r]_{q_i}}\equiv-\sum_{s=0}^{r-1}\mathcal{X}_{i,-s}^-\mathcal{P}^-_{i,-r+s+1}\mathcal{K}_i^{-r+1},
			\end{split}
		\end{equation}
		the congruences being $(\operatorname{mod} U_qX_+)$. $\odot$
	\end{lem}
	The proof goes by induction on $r$ and can be found in the paper of Chari and Pressley \cite{CP1991} (only the $\mathfrak{sl}_2$ case is proven there, but the proof in the general case is the same). However, let us reformulate Remark 3.5 in \cite{CP1991} as follows.
	
	\begin{rem}
		In the classical limit $q\to 1$, the formulae in Lemma \ref{lem:Newtons_formulae} appear in \cite{Chari1986} (see Equation (4.5) in \cite{Chari1986}, for example). In the classical case, the $\mathcal{P}_{i,r}^+$ are interpreted as the coefficients of a certain polynomial and the classical limits of the $(q_i-q_i^{-1})^{-1}(P_{i,r}^+)$ as the sum of the $r^{th}$ powers of its roots. Thus, Equation (\ref{eqn:dinfeld_poly_coeff}) may be interpreted as a $q$-analogue of Newton's formulae relating the elementary symmetric functions and the power sums. (cf. Chari and Pressley \cite{CP1991} pp. 267-268) $\odot$
	\end{rem}
	
	The definition given by Equation (\ref{eqn:dinfeld_poly_coeff}) can be reformulated by introducing the formal power series
	\begin{equation*}
		\boldsymbol{\mathcal{P}}_i^\pm(u) = \sum_{r=0}^{\infty}\mathcal{P}_{i,\pm r}^\pm u^{\pm  r},\quad\boldsymbol{\varPhi}_i^\pm (u) = \sum_{r=0}^{\infty}\varPhi_{i,\pm r}^\pm u^{\pm r}
	\end{equation*}
	in $U_q^0[[u^{\pm 1}]]$. Hence, (\ref{eqn:dinfeld_poly_coeff}) can be rewritten in the simple from
	\begin{equation}
		\boldsymbol{\varPhi}_i^\pm(u)=\mathcal{K}_i^{\pm1}\frac{\boldsymbol{\mathcal{P}}_i^\pm(q_i^{\mp 2}u)}{\boldsymbol{\mathcal{P}}_i^\pm(u)}.\label{eqn:form_power_ser_drinfeld_rewr}
	\end{equation}
	Suppose now that $V(\boldsymbol{\varphi})$ is finite-dimensional with highest weight $\boldsymbol{\varphi}$ and let $v_0$ be a pseudo highest weight vector. Then
	\begin{equation*}
		\mathcal{K}_i.v_0 = q_i^{r_i}v_0
	\end{equation*}
	for some $r_i\in \mathbb{N}$. Moreover, it is clear from the discussion of the representation theory of $U_q(\mathfrak{g})$ at the begin of this section that the $U_q(\mathfrak{g})$-submodule of $V(\boldsymbol{\varphi})$ generated by $v_0$ is isomorphic to the irreducible module $V_q(\lambda)$, where $\lambda=\sum_{i}r_i\omega_i$.\footnote{ We remind the reader that the $\omega_i$ are the fundamental weights defined by the condition $\langle\alpha_i,\omega_j\rangle =\delta_{i,j}.$} In particular we have $(\mathcal{X}_{i,0}^-)^{r_i+1}.v_0=0$.\footnote{ And then for weight reasons, i.e. the $5$th relation in Theorem \ref{thm:second_drin_quantum_aff}, we must obviously have that the action of any product of the $\mathcal{X}_{i,k}^-$, $k\in\mathbb{Z}$, of degree $r_i+1$ on $v_0$ equals zero.} From the first Equation in (\ref{eqn:coeff_finite_poly}), we obtain
	\begin{equation*}
		\mathcal{P}_i^+(u).v_0 = P_i(u)v_0
	\end{equation*}
	for some polynomial $P_i(u)=\sum_{r=0}^{r_i}p_{i,r}u^r$ of degree $r_i$. Therefore, the first equality in Theorem \ref{thm:drinfeld_polynomials} follows immediately from (\ref{eqn:form_power_ser_drinfeld_rewr}).
	
	To prove the second equation, apply both sides of the first Equation in (\ref{eqn:coeff_inverse_poly}), with $r=r_i+1$, to $v_0$. Since $(\mathcal{X}_{i,1}^-)^{r_i+1}=0$ for weight reasons, we clearly get
	\begin{equation*}
		\sum_{s=0}^{r_i}\mathcal{X}_{i,s+1}^-\mathcal{P}_{i,r_i-s}^+\mathcal{K}_i^{r_i}v_0=0.
	\end{equation*}
	Then, by applying $\mathcal{X}_{i,-s-1}^+$ for $s\geq 0$ and using the relation
	\begin{equation*}
		[\mathcal{X}_{i,r}^+,\mathcal{X}_{i,s}^-]=\frac{\varPhi^+_{i,r+s}-\varPhi^-_{i,r+s}}{q_i-q_i^{-1}},
	\end{equation*}
	we find that
	\begin{equation}
		\sum_{t=0}^{s}\varphi^-_{i,-t}p_{i,r_i-s-t} =\label{eqn:s<r_i} \sum_{t=0}^{r_i-s}\varphi^+_{i,t}p_{i,r_i-s-t}
	\end{equation}
	for $0\leq s\leq r_i$, and that
	\begin{equation}
		\sum_{t=s-r_i}^{s}\varphi^-_{i,-t}p_{i,r_i-s+t}=0\label{eqn:s>r_i}
	\end{equation}
	for $s>r_i$. Applying now Equation (\ref{eqn:dinfeld_poly_coeff}), with $r$ replaced by $r_i-s$, the right-hand side of Equation (\ref{eqn:s<r_i}) is equal to $q_i^{r_i}q_i^{-2(r_i-s)}p_{i,r_i-s}$. Then, by multiplying the $s$th equation with $u^{r_i-s}$ and summing from $s=0$ to $\infty$, we obtain
	\begin{equation*}
		\left(\sum_{t=0}^{\infty}\varphi_{i,-t}^-u^{-t}\right)P_i(u) = q_i^{r_i}P_i(q_i^{-2}u)
	\end{equation*}
	as required.\\
	
	Let us stop here for a moment and actually do the calculation for the right- and the left-hand side. It goes as follows. For the right hand-side, we write down $q_i^{r_i}q_i^{-2(r_i-s)}p_{i,r_i-s}$ and check the assertion by applying Equation \ref{eqn:dinfeld_poly_coeff}
	\begin{equation*}
		\begin{split}
			q_i^{r_i}q_i^{-2(r_i-s)}p_{i,r_i-s} \stackrel{Eq. (\ref{eqn:dinfeld_poly_coeff})}{=}& \frac{-q_i^{r_i}}{q_i^{2(r_i-s)}-1}\sum_{l=0}^{r_i-s-1}\varphi^+_{i,l+1}p^+_{r_i-s-l-1}q_i^{-r_i}\\
			=\quad&\frac{-1}{q_i^{2(r_i-s)}-1}\sum_{l=0}^{r_i-s-1}\varphi^+_{i,l+1}p^+_{r_i-s-l-1}\\
			\stackrel{\varphi_{i,0}^+=q_i^{r_i}}{=}& \frac{-1}{q_i^{2(r_i-s)}-1}\left(\sum_{t=0}^{r_i-s}\varphi^+_{i,t}p_{i,r_i-s-t}-q^{r_i}p^+_{i,r_i-s}\right).
		\end{split}
	\end{equation*}
	Solving this equation for $\sum_{t=0}^{r_i-s}\varphi^+_{i,t}p_{i,r_i-s-t}$ we find the desired equality.
	
	For the left-hand side, we simply write down the sum and manipulate it in a few steps
	\begin{equation*}
		\begin{split}
			&\sum_{s=0}^{r_i}\sum_{t=0}^{s}\varphi^-_{i,-t}p_{i,r_i-s+t}u^{r_i-s}+\sum_{s=r_i+1}^{\infty}\sum_{t=s-r_i}^{s}\varphi^-_{i,-t}p_{i,r_i-s+t}u^{r_i-s}\\
			&\stackrel{\varphi^-_{i,l}=0\text{ for }l>0}{=} \sum_{s=0}^{\infty}\sum_{t=s-r_i}^{s}\varphi^-_{i,-t}p_{i,r_i-s+t}u^{r_i-s}\\
			&\stackrel{t'=r_i-s+t}{=}\sum_{s=0}^{\infty}\sum_{t'=0}^{r_i}\varphi^-_{i,-t'+r_i-s}p_{i,t'}u^{r_i-s}\\
			&\stackrel{s'=t'-r_i+s}{=}\sum_{t'=0}^{r_i}\sum_{s'=t'-r_i}^{\infty}\varphi^-_{i,-s'}p_{i,t'}u^{t'-s'}\\
			&\stackrel{\varphi^-_{i,l}=0\text{ for }l>0}{=}\sum_{s'=0}^{\infty}\sum_{t'=0}^{r_i}\varphi^-_{i,-s'}p_{i,t'}u^{t'-s'}\\
			&=\left(\sum_{s'=0}^{\infty}\varphi^-_{i,-s'}u^{-s'}\right)P_i(u),
		\end{split}
	\end{equation*}
	and find what we were looking for. Hence, the second equality in Theorem \ref{thm:drinfeld_polynomials} follows. \\
	
	The proof that every $n$-tuple $(P_i)_{i\in I}\in\mathcal{P}^+$ of monic polynomials with non-zero constant term occurs is similar to the classical proof that the irreducible $\mathfrak{g}$-module $V(\lambda)$ of highest weight $\lambda\in P^+$ is finite-dimensional. We may therefore refer the reader to the book of Humphreys \cite{Humphreys}, or even the paper \cite{CP1994}. However, in type $A_l$, i.e. $\mathfrak{g}=\mathfrak{sl}_{l+1}$, we will see how to construct such representations using Jimbo's (evaluation) homomorphism in Subsection \ref{subsect:Jimbos_hom_und_ev_reps}.
	
	The multiplicative property of the polynomials $P_i$ in the second paragraph of Theorem \ref{thm:drinfeld_polynomials} can be proven using the partial description of the comultiplication $\Delta_q$ by
	\begin{prop}\label{prop:comult_partial}
		The comultiplication $\Delta_q$ of $U_q(\mathfrak{L}(\mathfrak{g}))$ satisfies
		\begin{enumerate}[label=\normalfont(\roman*)]
			\item modulo $U_qX^2_+\otimes U_qX_-$,\label{enum:comult_partial_i}
			\begin{equation*}
				\begin{split}
					&\Delta_q(\mathcal{X}_{i,k}^+)\equiv \mathcal{X}_{i,k}^+\otimes \mathcal{K}_i+1\otimes\mathcal{X}_{i,k}^++\sum_{j=1}^{k}\mathcal{X}_{i,k-j}^+\otimes\varPhi^+_{i,j}\quad \text{for }k\geq 0,\\
					&\Delta_q(\mathcal{X}_{i,-k}^+)\equiv \mathcal{X}_{i,-k}^+\otimes \mathcal{K}^{-1}_i+1\otimes\mathcal{X}_{i,-k}^++\sum_{j=1}^{k-1}\mathcal{X}_{i,-k+j}^+\otimes\varPhi^-_{i,-j}\quad\text{for }k> 0,
				\end{split}
			\end{equation*}
			\item modulo $U_qX_+\otimes U_qX^2_-$,\label{enum:comult_partial_ii}
			\begin{equation*}
				\begin{split}
					&\Delta_q(\mathcal{X}_{i,k}^-)\equiv \mathcal{X}_{i,k}^-\otimes 1+\mathcal{K}_i\otimes\mathcal{X}_{i,k}^-+\sum_{j=1}^{k-1}\varPhi^+_{i,j}\otimes\mathcal{X}_{i,k-j}^-\quad\text{for }k> 0,\\
					&\Delta_q(\mathcal{X}_{i,-k}^-)\equiv \mathcal{X}_{i,-k}^-\otimes 1+\mathcal{K}_i^{-1}\otimes\mathcal{X}_{i,-k}^-+\sum_{j=1}^{k}\varPhi^-_{i,-j}\otimes\mathcal{X}_{i,-k+j}^-\quad\text{for }k\geq 0,
				\end{split}
			\end{equation*}
			\item modulo $U_qX_+\otimes U_qX_-+U_qX_-\otimes U_qX_+$,\label{enum:comult_partial_iii}
			\begin{equation*}
				\begin{split}
					\Delta_q(\varPhi^\pm_{i,\pm k})\equiv \sum_{j=0}^{k} \varPhi^\pm_{i,j}\otimes\varPhi^\pm_{i,k-j}\quad\text{for }k\geq 0.\quad\odot
				\end{split}
			\end{equation*}
		\end{enumerate}
	\end{prop}
	In fact, the multiplicative property of the Drinfeld polynomials is an obvious consequence. For, \ref{prop:comult_partial} \ref{enum:comult_partial_i} implies that the tensor product of highest weight vectors in $V$ and $W$ is a highest weight vector in $V\otimes W$. Then, the assertion follows immediately from the group-like property \ref{prop:comult_partial} \ref{enum:comult_partial_iii} of $\varPhi^\pm_{i,\pm k}$.
	\begin{rem}
		We didn't use the $\boldsymbol{\mathcal{P}}_i^-$ in this proof. In fact, similar arguments show that $\boldsymbol{\mathcal{P}}_i^-.v_0=P_{i,V}^-(u^{-1})v_0$ for some polynomials $P^-_{i,V}$, and
		\begin{equation*}
			\sum_{r=0}^{\infty}\varphi_{i,-r}^{-}u^{-r} = q_i^{-r_i}\frac{P^-_{i,V}(q_i^2u^{-1})}{P^-_{i,V}(u^{-1})}.
		\end{equation*}
		It follows $P_{i,V}^-(u) = u^{r_i}P_{i,V}(u^{-1})$, up to a scalar multiple, and that the action of the $\varPhi^-_{i,-r}$ on the highest weight vector $v_0$ is determined by the action of the $\varPhi^+_{i,r}$, $i\in I = \{1,\dots,l\}$, $r\geq0$, on it. $\odot$
	\end{rem}
\end{proof}
The proof of Proposition \ref{prop:comult_partial} goes similar to the $\mathfrak{sl}_2$ case explained in Proposition 4.4 in the paper \cite{CP1991}. However, one verifies in a few lines that the comultiplication indeed acts on $\mathcal{X}_{i,-1}^+$ and $\mathcal{X}_{i,1}^-$ in the described way.\footnote{ Just use the isomorphism $f$ in Theorem \ref{thm:second_drin_quantum_aff} and apply the multiplicativity of the comultiplication, i.e. $\Delta$ is a (Lie) algebra homomorphism.} Let us nevertheless write it down to complete the proof of Theorem \ref{thm:drinfeld_polynomials}.\footnote{At least together with the construction in type A we will describe below, and some technicalities.}
\begin{proof}[Proof of Proposition \ref{prop:comult_partial}]
	The formulae are proved by induction on $k$. The initial case of each of the five formulae follows from the action of $\Delta_q$ on the Chevalley generators given in \ref{defprop:U_q(g)_chevalley}, and the isomorphism $f$ in Theorem \ref{thm:second_drin_quantum_aff}. One obtains
	\begin{equation*}
		\begin{split}
			\Delta_q(\mathcal{H}_{i,\pm1}) &= \pm(q_i-q_i^{-1})^{-1}\Delta_q(\mathcal{K}_i^{\mp1}\varPhi^\pm_{i,\pm1})\\
			&=\pm(\mathcal{K}_i^{\mp1}\otimes \mathcal{K}_i^{\mp1})[\Delta_q(\mathcal{X}_{i,0}^\pm),\Delta_q(\mathcal{X}_{i,\pm1}^\mp)]\\
			&= \mathcal{H}_{i,\pm 1}\otimes1+1\otimes\mathcal{H}_{i,\pm1}-(q^2-q^{-2})\mathcal{X}^+_{i,-\frac{1\mp1}{2}}\otimes \mathcal{X}^-_{i,\frac{1\pm1}{2}}.
		\end{split}
	\end{equation*}
	Using the relations
	\begin{equation*}
		[\mathcal{H}_{i,s},\mathcal{X}^\pm_{i,k}]=\pm(q_i+q_i^{-1})\mathcal{X}^\pm_{i,k+s}\quad\text{for }s=\pm1,
	\end{equation*}
	the formulae in part \ref{enum:comult_partial_i} and \ref{enum:comult_partial_ii} follow. Finally, part \ref{enum:comult_partial_iii} follows from \ref{enum:comult_partial_i} and \ref{enum:comult_partial_ii} using for example the relation
	\begin{equation*}
		(q_i-q_i^{-1})[\mathcal{X}_{i,\pm k}^+,\mathcal{X}_{i,0}^-]=\pm\varPhi^\pm_{i,\pm k}\quad\text{for }k>0.
	\end{equation*}
\end{proof}
Generally, Proposition \ref{prop:comult_partial} can be formulated in the following simple way.
\begin{rem}\label{rem:comult_partial_rem}
	Define the elements
	\begin{equation*}
		\begin{split}
			&\boldsymbol{\varPhi}^\pm_i(u) = \sum_{k=0}^{\infty}\varPhi^\pm_{i,k}u^{\pm k},\\
			&\boldsymbol{\mathcal{X}}^+_{i,\geq 0}(u) = \sum_{k\geq 0}\mathcal{X}^+_{i,k}u^k,\quad\boldsymbol{\mathcal{X}}^+_{i,<0}(u)=\sum_{k<0}\mathcal{X}^+_{i,k}u^k,\\
			&\boldsymbol{\mathcal{X}}^-_{i,> 0}(u) = \sum_{k> 0}\mathcal{X}^-_{i,k}u^k,\quad\boldsymbol{\mathcal{X}}^-_{i,\leq0}(u)=\sum_{k\leq0}\mathcal{X}^-_{i,k}u^k,
		\end{split}
	\end{equation*}
	of the Hopf algebras of formal power series $U_q(\mathfrak{L}(\mathfrak{g}))\otimes \mathbb{C}[[u]]$ and $U_q(\mathfrak{L}(\mathfrak{g}))\otimes\mathbb{C}[[u^{-1}]]$. Then, parts \ref{enum:comult_partial_i} - \ref{enum:comult_partial_iii} in Proposition \ref{prop:comult_partial} are equivalent to the statements
	\begin{equation*}
		\Delta_q(\boldsymbol{\mathcal{X}}^+_{i,\geq 0})\equiv \boldsymbol{\mathcal{X}}^+_{i,\geq 0}\otimes\boldsymbol{\varPhi}^+_i+1\otimes\boldsymbol{\mathcal{X}}^+_{i,\geq 0},\quad
		\Delta_q(\boldsymbol{\mathcal{X}}^+_{i,<0})\equiv \boldsymbol{\mathcal{X}}^+_{i,< 0}\otimes\boldsymbol{\varPhi}^-_i+1\otimes\boldsymbol{\mathcal{X}}^+_{i,<0}
	\end{equation*}
	modulo $(U_qX^2_+\otimes U_qX_-)[[u]]$, respectively $(U_qX^2_+\otimes U_qX_-)[[u^{-1}]]$,
	\begin{equation*}
		\Delta_q(\boldsymbol{\mathcal{X}}^-_{i,> 0})\equiv 	\boldsymbol{\mathcal{X}}^-_{i,> 0}\otimes1+\boldsymbol{\varPhi}^+_i\otimes\boldsymbol{\mathcal{X}}^-_{i,>0},\quad
		\Delta_q(\boldsymbol{\mathcal{X}}^-_{i,\leq0})\equiv 	\boldsymbol{\mathcal{X}}^+_{i,\leq 0}\otimes1+\boldsymbol{\varPhi}^-_i\otimes\boldsymbol{\mathcal{X}}^+_{i,\leq0}
	\end{equation*}
	modulo $(U_qX_+\otimes U_qX^2_-)[[u]]$, respectively $(U_qX_+\otimes U_qX^2_-)[[u^{-1}]]$, and that the element $\boldsymbol{\varPhi}_i^\pm$ is group-like
	\begin{equation*}
		\Delta_q(\boldsymbol{\varPhi}_i^\pm) \equiv \boldsymbol{\varPhi}_i^\pm\otimes\boldsymbol{\varPhi}_i^\pm 
	\end{equation*}
	modulo $(U_qX_+\otimes U_qX_-+U_qX_-\otimes U_qX_+)[[u^{\pm1}]]$. $\odot$
\end{rem}
Finally, let us write down the immediate consequences of Theorem \ref{thm:drinfeld_polynomials}.
\begin{cor}\label{cor:notation_V(P)}
	Assigning to $V= V(\boldsymbol{\varphi})$ (as in Theorem \ref{thm:drinfeld_polynomials}) the $I$-tuple $\boldsymbol{P}_V=(P_{i,V})_{i\in I}\in\mathcal{P}^+$ defines a bijection between $\mathcal{P}^+$ and the set of isomorphism classes of finite-dimensional irreducible representations of $U_q(\tilde{\mathfrak{g}})$ of type $\boldsymbol{1}$.
	We can therefore write equivalently $V= V(\boldsymbol{P})$ in this case. $\odot$	
\end{cor}
\begin{cor}\label{cor:pull-back_by_tau_a}
	If $\boldsymbol{P}=(P_i)_{i\in I}\in\mathcal{P}^+$, $a\in \mathbb{C}^\times$, then we have for the pull-back $\tau_a^*(V(\boldsymbol{P}))=V(\boldsymbol{P})(a)$ of $V(\boldsymbol{P})$ by the Hopf algebra automorphism $\tau_a$ that $V(\boldsymbol{P})(a)\cong V(\boldsymbol{P}^a)$ as representations of $U_q(\tilde{\mathfrak{g}})$, where $\boldsymbol{P}^a=(P_i^a)_{i\in I}$ and $P_i^a(u)=P_i(au)$. $\odot$
\end{cor}
The multiplicative property of the Drinfeld polynomials in the second paragraph of Theorem \ref{thm:drinfeld_polynomials} also suggests to formulate the slightly more general
\begin{cor}\label{cor:tensor_prod_drin_poly}
	For $\boldsymbol{P}$, $\boldsymbol{Q}\in\mathcal{P}$ denote by $\boldsymbol{P}\otimes\boldsymbol{Q}$ the $I$-tuple $(P_iQ_i)_{i\in I}\in \mathcal{P}^+$. Then $V(\boldsymbol{P}\otimes\boldsymbol{Q})$ is isomorphic to a quotient of the subrepresentation of $V(\boldsymbol{P})\otimes V(\boldsymbol{Q})$ generated by the tensor product of highest weight vectors. $\odot$
\end{cor}
In fact, Corollary \ref{cor:notation_V(P)} is just a reformulation, Corollary \ref{cor:pull-back_by_tau_a} is easily checked by applying the definitions, and Corollary \ref{cor:tensor_prod_drin_poly} is just a restatement of the group-like property \ref{prop:comult_partial} \ref{enum:comult_partial_iii} of $\varPhi^\pm_{i,\pm k}$ and the fact that the tensor product of highest weight vectors is a highest weight vector.
\begin{definition}[fundamental modules]\label{def:fund_modules_quant_aff}
	A finite-dimensional irreducible $U_q(\mathfrak{L}(\mathfrak{g}))$-module of type $\boldsymbol{1}$ is called \textbf{fundamental} if its associated Drinfeld polynomials are given by
	\begin{equation*}
		P_j(u) =
		\begin{cases}
			1 & \text{if }j\neq i\\
			1-au & \text{if }j=i
		\end{cases}
	\end{equation*}
	for some $i=1,\dots,l$, $a\in\mathbb{C}$ and we denote the corresponding representation by $V_{\omega_i}(a)$. $\odot$
\end{definition}
Indeed, we have seen that $V_{\omega_i}(a)$, as a $U_q(\mathfrak{g})$-module, has $\omega_i$ as its unique maximal weight. Thus, it contains the irreducible $U_q(\mathfrak{g})$-module $V_q(\omega_i)$ with multiplicity one. Moreover, in view of Corollary \ref{cor:pull-back_by_tau_a}, our notation is consistent with the definition $V_{\omega_i}\coloneq V_{\omega_i}(1)$. Then $\tau_a^*(V_{\omega_i})=V_{\omega_i}(a)=V_{\omega_i}(1a)$.  Combining Definition \ref{def:fund_modules_quant_aff} with Theorem \ref{thm:drinfeld_polynomials}, we conclude
\begin{cor}\label{cor:tensor_product_of_fund}
	Every finite-dimensional irreducible $U_q(\mathfrak{L}(\mathfrak{g}))$-module is isomorphic to a subquotient of the tensor product $V_{\omega_{i_1}}(a_1)\otimes \cdots\otimes V_{\omega_{i_k}}(a_k)$, for some $i_1,\dots,i_k\in I = \{1,\dots,l\}$ and $k\in \mathbb{N}$, generated by the tensor product of the highest weight vectors. The parameters $(\omega_{i_1},a_1),\dots,(\omega_{i_k},a_k)$ are, up to permutation, uniquely determined by this representation. $\odot$
\end{cor}
\begin{rem}
	The last paragraph in Theorem \ref{thm:drinfeld_polynomials} together with Corollary \ref{cor:tensor_product_of_fund} is a first hint that the Grothendieck ring of the abelian monoidal category $\mathit{C}$ of finite-dimensional type $\boldsymbol{1}$ representations of $U_q(\tilde{\mathfrak{g}})$ is the commutative ring in the isomorphism classes of fundamental modules $V_{\omega_i}(a)$ of $\mathit{C}$. In fact, this is proven in the paper of Frenkel and Reshetikhin \cite{FR} with Corollary 2. We will come back to this in the next chapter.\footnote{ The $\mathfrak{sl}_2$ case is discussed in the next subsection, whereas a general explanation is given in Section \ref{sect:gen_untw_case}.} $\odot$
\end{rem}
However, the notion of $\boldsymbol{q}$\textbf{-characters} can be introduced with the following generalisation of Theorem \ref{thm:drinfeld_polynomials}.
\begin{thm}\label{thm:drinfeld-l_weights}
	The eigenvalues $\Psi_i^\pm(u)$ of $\boldsymbol{\varPhi}_i^\pm(u)$ on any finite-dimensional type $\boldsymbol{1}$ representation of $U_q(\tilde{\mathfrak{g}})$ have the form
	\begin{equation*}
		\Psi_i^\pm(u) = q_i^{\deg(Q_{i})-deg(R_{i})}\frac{Q_{i,V}(q_i^{-2}u)R_{i,V}(u)}{Q_{i,V}(u)R_{i,V}(q_i^{-2}u)},
	\end{equation*}
	as elements of $\mathbb{C}[[u^{\pm 1}]]$, respectively, where $Q_i(u)$ and $R_i(u)$ are polynomials with constant term 1, i.e. $Q_i(u),R_i(u)\in\mathcal{P}^+$. $\odot$
\end{thm}
We can therefore write equivalently $V_{\boldsymbol{\psi}}=V_{\boldsymbol{Q}\boldsymbol{R}^{-1}}$, for the generalised eigenspace $V_{\boldsymbol{\psi}}$ with generalised eigenvalues $\boldsymbol{\psi}=(\psi_{i,r})_{i\in I, r\in \mathbb{Z}}$, where $\boldsymbol{Q}\boldsymbol{R}^{-1}=(Q_iR_i^{-1})_{i\in I}$. Moreover, we may call $\boldsymbol{Q}\boldsymbol{R}^{-1}$ the \textbf{loop-weight} of the loop-weight space $V_{\boldsymbol{Q}\boldsymbol{R}^{-1}}$, explaining the terminology earlier in Remark \ref{rem:loop_weight_spaces}.
\begin{definition}[$q$-character]\label{def:q-character_drinfeld}
	We define the $\boldsymbol{q}$\textbf{-character} $\chi_q(V)$ of a finite-dimensional (type $\boldsymbol{1}$) representation of $U_q(\tilde{\mathfrak{g}})$ as the sum over all (finitely many) generalised eigenvalues $\boldsymbol{Q}\boldsymbol{R}^{-1}$ (loop-weights) of $\dim(V_{\boldsymbol{Q}\boldsymbol{R}^{-1}})$ times the formal product $\prod_{i\in I}Q_i(u_i)R_i^{-1}(u_i)$, i.e.
	\begin{equation*}
		\chi_q(V) = \sum_{\boldsymbol{Q}\boldsymbol{R}^{-1}} \dim(V_{\boldsymbol{Q}\boldsymbol{R}^{-1}}) \prod_{i\in I}Q_i(u_i)R_i^{-1}(u_i).\quad\odot
	\end{equation*}
\end{definition}
Anyhow, the original definition of $q$-characters goes back to the paper \cite{FR} of Frenkel and Reshetikhin and the equivalence to our Definition \ref{def:q-character_drinfeld} is stated (and proven thereafter) in Proposition 2.4 in the paper \cite{FM} of Frenkel and Mukhin. Again, we will come back to this and the general proof of Theorem \ref{thm:drinfeld-l_weights} in the next chapter.
\subsection{Jimbo's homomorphism and evaluation representations}\label{subsect:Jimbos_hom_und_ev_reps}
Let us now discuss Jimbo's (evaluation) homomorphism and construct explicit representations of $U_q(\tilde{\mathfrak{sl}}_{l+1})$ with it. We also intend to explain how to obtain the Drinfeld polynomial of the spin $r/2$ evaluation representation of $U_q(\mathfrak{sl}_2)$. We refer the reader to Subsection 12.2C in \cite{CPBook} where the following explanation can be found.\footnote{ We remind the reader that the finite-dimensional type $\boldsymbol{1}$ representations of $U_q(\tilde{\mathfrak{g}})$ and $U_q(\mathfrak{L}(\mathfrak{g}))$ are the same. We emphasize this, because we use both notions interchangeably here and write what fits to the context.}

We recall that the generators $\mathcal{X}^\pm_{i,r}$ have classical limits $x_i^\pm t^r$. In this sense, one might hope that, for any $a\in\mathbb{C}$, there is a homomorphism of algebras $U_q(\mathfrak{L}(\mathfrak{g}))\to U_q(\mathfrak{g})$ such that $\mathcal{X}_{i,r}^\pm\mapsto a^r X_i^\pm$ for all $i$. Unfortunately, this does not work in general. Nevertheless, Jimbo defined a quantum analogue of the latter homomorphism when $\mathfrak{g}=\mathfrak{sl}_{l+1}$. However, when $l>1$, it takes values in an 'enlargement' of $U_q(\mathfrak{sl}_{l+1})$. Fix a square root $q^{1/2}$ of $q$ and, for any elements $u$, $v$ of an algebra over $\mathbb{C}$, set
\begin{equation*}
	[u,v]_{q^{1/2}} = q^{1/2}uv-q^{-1/2}vu.
\end{equation*}
\begin{defprop}\label{defprop:U_q(gl_l+1)}
	$U_q(\mathfrak{gl}_{l+1})$ is the associative algebra over $\mathbb{C}$ with generators $x_i^\pm$, $i=1,\dots,l$, $t_r^{\pm1}$, $r=0,\dots,l$ and the defining relations
	\begin{equation*}
		\begin{split}
			t_rt_r^{-1}&=1=t_r^{-1}t_r,\\
			t_rt_s&= t_st_r,\\
			t_rx_i^\pm t_r^{-1} &= q^{\delta_{r,i-1}-\delta_{r,i}}x_i^\pm,\\
			[x_i^\pm,[x_j^\pm,x_i^\pm]_{q^{1/2}}]_{q^{1/2}}&=0\quad\text{if }|i-j|=1,\\
			[x_i^\pm,x_j^\pm] &= 0,\quad\text{if }|i-j|>1,\\
			[x_i^+,x_j^-] &=\delta_{i,j} \frac{k_i-k_i^{-1}}{q-q^{-1}},
		\end{split}
	\end{equation*}
	where $k_i=t_{i-1}t_i^{-1}$.\\
	The Hopf algebra structure on $U_q(\mathfrak{gl}_{l+1})$ is given by the same formulas as in \ref{defprop:U_q(g)_chevalley}, together with
	\begin{equation*}
		\Delta_q(t_r)=t_r\otimes t_r,\quad S_q(t_r) = t_r^{-1},\quad \epsilon(t_r) = 1.\quad\odot
	\end{equation*}
\end{defprop}
We also note that there is a homomorphism of Hopf algebras $U_q(\mathfrak{sl}_{l+1})\to U_q(\mathfrak{gl}_{l+1})$ mapping $X_i^\pm$ to $x_i^\pm$ and $K_i$ to $k_i$.
\begin{prop}[Jimbo's homomorphism]\label{prop:Jimbos_hom}
	For any $a\in\mathbb{C}$, there exist unique homomorphisms of algebras $\operatorname{ev}_a$ and $\operatorname{ev}^a$ from $U_q(\mathfrak{L}(\mathfrak{sl}_{l+1}))$ to $U_q(\mathfrak{gl}_{l+1})$ such that
	\begin{equation*}
		\begin{split}
			\operatorname{ev}_a(X_i^\pm)&=x_i^\pm=\operatorname{ev}^a(X_i^\pm),\quad i=1,\dots,l,\\
			\operatorname{ev}_a(K_0)&=(k_1\cdots k_l)^{-1}=\operatorname{ev}^a(K_0),\\
			\operatorname{ev}_a(X_0^\pm)&= (\pm1)^{l-1}q^{\mp\frac{l+1}{2}}a^{\pm1}[x_l^\mp,[x_{l-1}^\mp,\dots,[x_2^\mp,x_1^\mp]_{q^{1/2}}\cdots]_{q^{1/2}}]_{q^{1/2}}(t_0t_l)^{\pm1},\\
			\operatorname{ev}^a(X_0^\pm)&= (\pm1)^{l-1}q^{\mp\frac{l+1}{2}}a^{\pm1}[x_l^\mp,[x_{l-1}^\mp,\dots,[x_2^\mp,x_1^\mp]_{q^{1/2}}\cdots]_{q^{1/2}}]_{q^{1/2}}(t_0t_l)^{\mp1}.\quad\odot
		\end{split}
	\end{equation*}
\end{prop}
The proof is via direct verification. We shall make some remarks.
\begin{rems}\label{rems:ev_homs}\hspace{1em}
	\begin{enumerate}
		\item The homomorphism $\operatorname{ev}_a$ is due to Jimbo 1986 \cite{JimboQG2}. The homomorphism $\operatorname{ev}^a=\sigma\circ\operatorname{ev}_a\circ\tilde{\sigma}$, where $\sigma$ and $\tilde{\sigma}$ are the automorphisms of $U_q(\mathfrak{gl}_{l+1})$ and $U_q(\mathfrak{L}(\mathfrak{sl}_{l+1}))$, respectively, given on generators by
		\begin{equation*}
			\sigma(x_i^\pm)=x^\pm_{l-i+1},\quad\sigma(t_r^{\pm1})=t_{l-r}^{\mp1},
		\end{equation*} 
		and
		\begin{equation*}
			\tilde{\sigma}(\mathcal{X}_{i,r}^\pm)=\mathcal{X}^\pm_{l-i+1,r},\quad\tilde{\sigma}(\mathcal{H}_{i,r})=\mathcal{H}_{l-i+1,r},\quad\tilde{\sigma}(\mathcal{K}_i^{\pm1})=\mathcal{K}^{\pm 1}_{l-i+1}.
		\end{equation*}
		Alternatively, $\operatorname{ev}^a$ may be obtained from $\operatorname{ev}_a$ essentially by twisting with the antipodes $S$ and $\tilde{S}$ of $U_q(\mathfrak{gl}_{l+1})$ and $U_q(\mathfrak{L}(\mathfrak{sl}_{l+1}))$, respectively. Checking this on the generators $X_0^\pm$, one finds
		\begin{equation*}
			S\circ \operatorname{ev}_a\circ\tilde{S} = \operatorname{ev}^b,
		\end{equation*}
		where $b=(-1)^{n-1}aq^{-l+1}$.\label{enum:comp_betw_ev_homs}
		\item Neither $\operatorname{ev}_a$ nor $\operatorname{ev}^a$ is a homomorphism of Hopf algebras.
		\item If $l=1$, $\operatorname{ev}_a = \operatorname{ev}^a$ for all $a\in\mathbb{C}^\times$. $\odot$
	\end{enumerate}
\end{rems}
\begin{definition}
	Fix an $(l+1)$th root $q^{1/(l+1)}$ of $q$. We say that a representation $V$ of $U_q(\mathfrak{gl}_{l+1})$ is of type $\boldsymbol{1}$ if
	\begin{enumerate}[label=(\alph*)]
		\item $V$ is of type $\boldsymbol{1}$ regarded as a representation of $U_q(\mathfrak{sl}_{l+1})$,
		\item the $t_r$ act semisimply on $V$ with eigenvalues which are integer powers of $q^{1/(l+1)}$,
		\item $t_0t_1\cdots t_l$ acts as the identity on $V$.
	\end{enumerate}
\end{definition}

Moreover, the restriction to $U_q(\mathfrak{sl}_{l+1})$ is an equivalence of the categories of finite-dimensional type $\boldsymbol{1}$ representations of $U_q(\mathfrak{gl}_{l+1})$ and $U_q(\mathfrak{sl}_{l+1})$. Hence, if $V$ is a finite-dimensional type $\boldsymbol{1}$ representation of $U_q(\mathfrak{sl}_{l+1})$, we may regard it as a type $\boldsymbol{1}$ representation of $U_q(\mathfrak{gl}_{l+1})$. We denote by $V_a$ and $V^a$ the pull-back of $V$ by $\operatorname{ev}_a$ and $\operatorname{ev}^a$, respectively. Both $V_a$ and $V^a$ are called \textbf{evaluation representations}.

If $V$ is irreducible as a representation of $U_q(\mathfrak{sl}_{l+1})$, then clearly $V_a$ and $V^a$ are irreducible representations of $U_q(\mathfrak{L}(\mathfrak{sl}_{l+1}))$, for all $a\in \mathbb{C}^\times$. Using the formulas in \ref{defprop:U_q(gl_l+1)} and the decomposition \ref{prop:PBW_U_q}, we can compute the (Drinfeld) polynomials associated to $V_a$ and $V^a$. Let us do this for $l=1$, i.e. the $\mathfrak{sl}_2$ case in the following example.
\begin{xmpl}\label{xmpl:spin_r/2_evaluation_rep}
	Let $a\in\mathbb{C}^\times$, $r\in\mathbb{N}$, and let $V= V_q(r)$ be the spin $r/2$ (or $(r+1)$-dimensional) irreducible type $\boldsymbol{1}$ representation of $U_q(\mathfrak{sl}_2)$. In the usual basis $\{v_0,v_1,\dots,v_r\}$, the action of $U_q(\mathfrak{gl}_2)$ on $V_q(r)$ is given by
	\begin{equation*}
		\begin{split}
			t_0.v_k = q^{(r-2k)/2}v_k,\quad &t_1.v_k = q^{-(r-2k)/2}v_k,\\
			x_1^+.v_k =[r-k+1]_qv_{k-1},\quad&x_1^-.v_k = [k+1]_q v_{k+1}.
		\end{split}
	\end{equation*}
	We then use $\operatorname{ev}_a$ which is defined in terms of Chevalley generators in Proposition \ref{prop:Jimbos_hom} to obtain a representation of $U_q(\mathfrak{L}(\mathfrak{sl}_{2}))$ which we call $V^{(r)}(a)$, i.e.
	\begin{equation*}
		\operatorname{ev}_a(X_0^\pm) = q^{\mp1}a^{\pm1}x_1^\mp,\quad \operatorname{ev}_a(X_1^\pm) = x_1^\pm,\quad \operatorname{ev}_a(K_0)=k_1^{-1} \quad\operatorname{ev}_a(K_1) = k_1.
	\end{equation*}
	
	Now, to compute the action of the Drinfeld generators $\mathcal{X}_{1,p}$, we must apply the isomorphism $f$ given in Theorem \ref{thm:second_drin_quantum_aff}.\footnote{ Luckily, we have already written it down explicitly in our comment in the proof of Proposition \ref{prop:type_1_twisting_loop}.} We find
	\begin{equation*}
		\begin{split}
			\operatorname{ev}_a(C)=1,\quad\operatorname{ev}_a(K_1)=k_1,& \quad\operatorname{ev}_a(\mathcal{X}_{1,0}^\pm)=x_1^\pm,\\ \quad\operatorname{ev}_a(\mathcal{X}_{1,-1}^+)=qa^{-1}k_1^{-1}x_1^+,& \quad\operatorname{ev}_a(\mathcal{X}_{1,1}^-)=q^{-1}ax_1^-k_1,
		\end{split}
	\end{equation*}
	and using 
	\begin{equation*}
		\pm[\mathcal{X}_{1,0}^\pm,\mathcal{X}_{1,\pm1}^\mp]=\mathcal{H}_{1,\pm1}\mathcal{K}_1^{\pm1},
	\end{equation*}
	that
	\begin{equation*}
		\operatorname{ev}_a(\mathcal{H}_{1,\pm1})=a^{\pm1}\left(\pm q^{\mp1}\frac{k_1-k_1^{-1}}{q-q^{-1}}-(q-q^{-1})x_1^-x_1^+\right).
	\end{equation*}
	Then, we can use
	\begin{equation*}
		[\mathcal{H}_{1,s},\mathcal{X}_{1,p}^\pm] = \pm(q+q^{-1})\mathcal{X}_{1,p+s}^\pm\quad\text{for }s=\pm1
	\end{equation*}
	to compute $\operatorname{ev}_a(\mathcal{X}_{1,p})$ inductively. We find
	\begin{equation*}
		\begin{split}
			&\operatorname{ev}_a(\mathcal{X}_{1,p}^+)=q^{-p}a^pk_1^px_1^+,\\
			&\operatorname{ev}_a(\mathcal{X}_{1,p}^-)=q^{-p}a^px_1^-k_1^p,
		\end{split}
	\end{equation*}
	for all $s\in\mathbb{Z}$. Therefore, the action of the $\mathcal{X}_{1,p}^\pm$ on $V^{(r)}(a)$ is given by
	\begin{equation}\label{eqn:ev_action_of_the_X_1p}
		\begin{split}
			&\mathcal{X}_{1,p}^+.v_k = a^pq^{p(r-2k+1)}[r-k+1]_qv_{k-1},\\
			&\mathcal{X}_{1,p}^-.v_k = a^pq^{p(r-2k-1)}[k+1]_qv_{k+1}.
		\end{split}
	\end{equation}
	Finally, we can compute the eigenvalues $\varphi_{1,p}^+$ of the $\varPhi_{1,p}^+$ on the highest weight vector $v_0$ by using for instance the relation
	\begin{equation*}
		\varPhi_{1,p}^+ = (q-q^{-1})[\mathcal{X}_{1,p}^+,\mathcal{X}_{1,0}^-]
	\end{equation*}
	for $p>0$ and $\varPhi_{1,0}^+=\mathcal{K}_1$. The eigenvalues $\varphi_{1,p}^+$ are therefore
	\begin{equation*}
		\varphi_{1,p}^+= (q-q^{-1})(aq^{r-1})^p[r]_q
	\end{equation*}
	for $p>0$ and $\varphi_{1,0}^+=q^r$.
	Hence, we obtain
	\begin{equation*}
		\begin{split}
			\sum_{p=0}^{\infty}\varphi_{1,p}^+u^p &= q^r+(q^r-q^{-r})\frac{aq^{r-1}u}{1-aq^{r-1}u}\\
			&=q^r\left(\frac{1-q^{-r-1}au}{1-q^{r-1}au}\right).
		\end{split}
	\end{equation*}
	We conclude that the Drinfeld polynomial must be
	\begin{equation*}
		P_{a}^{(r)}(u) = \prod_{k=1}^{r}(1-q^{r-2k+1}au).\quad \odot
	\end{equation*}
\end{xmpl}
Note that we could have chosen any other vector $v_k$ in the last example to compute all the eigenvalues $\psi_{1,p}^+$ of the $\varPhi_{1,p}^+$ on $v_k$, we just didn't do it to keep the derivation of the formula for the Drinfeld polynomial simple. However, let us do this now to proof Theorem \ref{thm:drinfeld-l_weights} for the spin $r/2$ evaluation representations of $U_q(\tilde{\mathfrak{sl}}_2)$.
\begin{cor}\label{cor:loop_weights_ev_reps_sl2}
	Theorem \ref{thm:drinfeld-l_weights} is true for the spin $r/2$ evaluation representations $V^{(r)}(a)$ of $U_q(\tilde{\mathfrak{sl}}_2)$. Precisely, the vectors $v_k$, $k=0,1,\dots,r$, span a loop-weight space of loop-weight $\boldsymbol{Q_{k,a}R_{k,a}^{-1}}$, where the polynomials $Q_{k,a}$ and $R_{k,a}$ with constant term $1$ are given by
	\begin{equation*}
		Q_{k,a} = \prod_{m=k+1}^{r}(1-q^{r-2m+1}au),\quad R_{k,a}=\prod_{m=1}^{k}(1-q^{r-2m+3}au).\quad \odot
	\end{equation*}
\end{cor}
\begin{proof}[Proof of Corollary \ref{cor:loop_weights_ev_reps_sl2}]
	As in Example \ref{xmpl:spin_r/2_evaluation_rep}, we use the action of the $\mathcal{X}_{1,p}^\pm$ given by Equation (\ref{eqn:ev_action_of_the_X_1p}) and for instance the relation
	\begin{equation*}
		\varPhi_{1,p}^+ = (q-q^{-1})[\mathcal{X}_{1,p}^+,\mathcal{X}_{1,0}^-]
	\end{equation*}
	for $p>0$ and $\varPhi_{1,0}^+=\mathcal{K}_1$ to compute the eigenvalues $\psi_{1,p}^+$ of $\varPhi_{1,p}^+$ on the vector $v_k$ of the spin $r/2$ evaluation representation.
	
	We immediately get
	\begin{equation*}
		\psi_{1,p}^+ = (q-q^{-1})([k+1]_q[r-k]_q(aq^{r-2k-1})^p-[k]_q[r-k+1]_q(aq^{r-2k+1})^p)
	\end{equation*}
	for $p>0$ and $\psi_{1,0}^+= q^{r-2k}$. Hence, we obtain
	\begin{equation*}
		\begin{split}
			\sum_{p=0}^{\infty}\psi_{1,p}^+ &= q^{r-2k}+(q-q^{-1})\left([k+1]_q[r-k]_q \frac{aq^{r-2k-1}u}{1-aq^{r-2k-1}u}-[k]_q[r-k+1]_q\frac{aq^{r-2k+1}u}{1-aq^{r-2k+1}u}\right)\\
			&= q^{r-2k}\left(\frac{1-aq^{-r-1}u}{1-aq^{r-2k-1}u}\right)\left(\frac{1-aq^{r+1}u}{1-aq^{r-2k+1}u}\right)
		\end{split}
	\end{equation*}
	which shows that $v_k$ spans a loop-weight space $V_{ \boldsymbol{Q_{k,a}R_{k,a}^{-1}}}$ of loop-weight $\boldsymbol{Q_{k,a}R_{k,a}^{-1}}$ as asserted. In particular, every loop-weight space is one-dimensional and the loop-weight for $v_0$ is exactly the Drinfeld polynomial.\footnote{ Note that we have calculated only $\Psi_i^+(u)$ here. However, the calculation of $\Psi_i^-(u)$ is as usual completely analogous and we therefore decided to spare the details.}
\end{proof}

Note that our notation '$V^{(r)}(a)$' is again compatible with the action of the multiplicative group $\mathbb{C}^\times$ through the Hopf algebra automorphism $\tau$ given by Equation (\ref{eqn:tau_a_q-deformed}), i.e. $V^{(r)}(a)=(V^{(r)}(1))(a)$ and we can therefore set $V^{(r)}\coloneq V^{(r)}(1)$. The example suggests the following
\begin{definition}\label{def:q_segment}
	Let $a\in\mathbb{C}^\times$, $r\in\mathbb{N}$. The $q$-segment $\Sigma_r(a)$ with centre $a$ and length $r$ is the $r$-tuple $(aq^{-r+1},aq^{-r+3},\dots,aq^{r-1})$.
\end{definition}
It should be clear that we can generalise Example \ref{xmpl:spin_r/2_evaluation_rep} for any $U_q(\mathfrak{L}(\mathfrak{sl}_{2}))$ subalgebra of $U_q(\mathfrak{L}(\mathfrak{sl}_{l+1}))$. Moreover, the generalisation for any representation $V$ of $U_q(\mathfrak{sl}_{l+1})$ with highest weight $\lambda=\sum_i r_i\omega_i$ is proven in the paper \cite{CP1994a}. It is given in
\begin{prop}\label{prop:center_drin_poly_ev}
	Let $V$ be the finite-dimensional irreducible type $\boldsymbol{1}$ representation of $U_q(\mathfrak{sl}_{n+1})$ with highest weight $\lambda=\sum_i r_i\omega_i$ and let $a\in\mathbb{C}^\times$. Then,
	\begin{enumerate}[label=\normalfont(\roman*)]
		\item if $r_i=0$, we have $P_{i,V_a}=P_{i,V^a}=1$,
		\item if $r_i>0$, the roots of $P_{i,V_a}$ and $P_{i,V^a}$ form, when suitably ordered, the $q$-segment of length $r_i$ with centre
		\begin{equation*}
			c_i = a^{-1}q^{r_i+i-1-\frac{2}{l+1}\sum_{j=1}^{l}(l-j+1)r_j+2\sum_{j=1}^{i-1}r_j}
		\end{equation*}
		and, respectively, with centre
		\begin{equation*}
			c^i = a^{-1}q^{r_i+l-i-\frac{2}{l+1}\sum_{j=1}^{l}jr_j+2\sum_{j=i+1}^{l}r_j}.\quad\odot
		\end{equation*}
	\end{enumerate}
\end{prop}
It follows immediately from this result that every fundamental representation of $U_q(\mathfrak{L}(\mathfrak{sl}_{l+1}))$ is isomorphic to an evaluation representation. By Lemma \ref{lem:Newtons_formulae}, we have
\begin{cor}\label{cor:iso_to_subquo}
	Every finite-dimensional irreducible type $\boldsymbol{1}$ representation of $U_q(\mathfrak{L}(\mathfrak{sl}_{n+1}))$ is isomorphic to a subquotient of a tensor product of evaluation representations. $\odot$
\end{cor}
As mentioned earlier, this part proves the converse of the first part of Theorem \ref{thm:drinfeld_polynomials} in the case $\mathfrak{g}=\mathfrak{sl}_{l+1}$.

Let us also make an important remark to clarify the meaning of the shifted dual representation in Definition \ref{def:the_shifted_dual}.
\begin{rem}[Drinfeld polynomials of the shifted duals]\label{rem:shifted_dual}
	Let $S$ and $\tilde{S}$ 
	be the antipodes 
	of $U_q(\mathfrak{gl}_{l+1})$ and $U_q(\mathfrak{L}(\mathfrak{sl}_{l+1}))$ as in Part \ref{enum:comp_betw_ev_homs} of the Remarks \ref{rems:ev_homs} and let $V(a)\coloneq V_a=\operatorname{ev}_a^*(V)$ be an evaluation representation. We can use the statement $S\circ\operatorname{ev}_a\circ\tilde{S} = \operatorname{ev}^b$ to compute the Drinfeld polynomials of the (left) dual of $V(a)$ by comparing the action of the maps
	\begin{equation*}
		\operatorname{ev}_a\circ\tilde{S} = S^{-1}\circ \operatorname{ev}^b
	\end{equation*}
	and then applying Proposition \ref{prop:center_drin_poly_ev}.\footnote{ Where $b=(-1)^{n-1}aq^{-n+1}$ as above.} In fact, we only have to apply them on the generators $X_0^\pm$ since the action on the other generators obviously coincides. After a short computation it is clear that if $V(a)$ has the Drinfeld polynomials $P_{i,V}^a$, then $V^*(a)$ must have the Drinfeld polynomials $P_{i,V^*}^a = P_{l+1-i,V}^{aq^{l+1}}$ (and similarly $^*V(a)$ the Drinfeld polynomials $P_{i,^*V}^a = P_{l+1-i,V}^{aq^{-(l+1)}}$). Therefore, the shifted dual $V^\circledast$ has the Drinfeld polynomials $P_{i,V^\circledast}^a = P_{l+1-i,V}^{a}$ which are simply the Drinfeld polynomials of $V(a)$, but in reversed order.	$\odot$
\end{rem}

In the $\mathfrak{sl}_2$ case, Corollary \ref{cor:iso_to_subquo} can be made much more precise as follows.
\begin{prop}\hspace{1em}
	\begin{enumerate}[label=\normalfont(\roman*)]\label{prop:iso_to_tens_of_ev}
		\item Every finite-dimensional irreducible type $\boldsymbol{1}$ representation of $U_q(\mathfrak{L}(\mathfrak{sl}_2))$ is isomorphic to a tensor product
		\begin{equation}
			V^{(r_1)}(a_1)\otimes\cdots\otimes V^{(r_k)}(a_k)\label{eqn:tensor_of_ev_sl2}
		\end{equation}
		of evaluation representations.
		\item Two such irreducible tensor products are isomorphic if and only if one is obtained from the other by a permutation of the factors.
		\item An arbitrary tensor product (\ref{eqn:tensor_of_ev_sl2}) is reducible if and only if at least one pair of $q$-segments $\Sigma_{r_i}(a_i)$, $\Sigma_{r_j}(a_j)$ is in \textbf{special position}, in the sense that their union is a $q$ segment which properly contains them both. $\odot$
	\end{enumerate}
\end{prop}
The proposition suggests the following
\begin{definition}
	Two $q$-segments $\Sigma_{r_i}(a_i)$ and $\Sigma_{r_j}(a_j)$ are said to be in \textbf{special position}, if their union is a $q$ segment which properly contains them both. Otherwise, we say that $\Sigma_{r_i}(a_i)$ and $\Sigma_{r_j}(a_j)$ are in \textbf{general position}.
\end{definition}
Moreover, the tensor product of two evaluation representations is described precisely as follows.
\begin{prop}[special position]\label{prop:special_pos_short_ex}
	The tensor product $V= V^{(k)}(a)\otimes V^{(l)}(b)$ of evaluation representations of $U_q(\mathfrak{L}(\mathfrak{sl}_2))$ has a unique proper subrepresentation $W$ if and only if there is a $0\leq p < \min\{k,l\}$ such that $b/a=q^{\pm(k+l-2p)}$. In this case the $q$-segments $\Sigma_{k}(a)$ and $\Sigma_{l}(b)$ are in \textbf{special position} and we have the short exact sequence
	\begin{align*}
		W\lhook\joinrel\rightarrow V^{(k)}(a)\otimes V^{(l)}(b)\rightarrowdbl V/W.
	\end{align*}
	Furthermore, the composition factors $W$ and $V/W$ are irreducible and described as follows.
	\begin{enumerate}
		\item If $b/a=q^{k+l-2p}$, we have
		\begin{align}
			W&\cong V^{(k-p-1)}(q^{-p-1}a)\otimes V^{(l-p-1)}(q^{p+1}b),\notag\\
			V/W&\cong V^{(p)}(q^{k-p}a)\otimes V^{(k+l-p)}(q^{-(k-p)}b).
		\end{align}
		As a representation of $\mathfrak{sl}_2$,
		\begin{align*}
			W\cong V^{(k+l-2p)}\oplus V^{(k+l-2p-2)}\oplus\dots\oplus V^{(|m-n|)}.
		\end{align*}
		\item If $b/a=q^{-(k+l-2p)}$, we have
		\begin{align}
			W&\cong V^{(p)}(q^{-(k-p)}a)\otimes V^{(k+l-p)}(q^{k-p}b),\notag\\
			V/W&\cong V^{(k-p-1)}(q^{p+1}a)\otimes V^{(l-p-1)}(q^{-p-1}b).
		\end{align}
		As a representation of $\mathfrak{sl}_2$,
		\begin{align*}
			W\cong V^{(k+l)}\oplus V^{(k+l-2)}\oplus\dots\oplus V^{(k+l-2p+2)}.\quad\odot
		\end{align*}
	\end{enumerate}
\end{prop}
In fact, this proposition is very important for the basic understanding of what happens in the more general setting discussed in \ref{sect:gen_untw_case}.
Therefore, before we conclude this subsection, let us also give an outline of the proof of Proposition \ref{prop:special_pos_short_ex} as described in the paper \cite{CP1991} in Subsection 4.8 and the book \cite{CPBook} for the Yangian $Y(\mathfrak{sl}_2)$ in Subsection 12.1E.
\begin{proof}[Outline of the proof of Proposition \ref{prop:special_pos_short_ex}]
	As an $U_q(\mathfrak{sl}_2)$ module,
	\begin{equation*}
		V^{(k)}(a)\otimes V^{(l)}(b) = \bigoplus_{j=0}^{\min\{k,l\}} V_q(k+l-2j).
	\end{equation*}
	The highest weight vector $w_j$ in the component $V_{k+l-2j}$ of $V^{(k)}(a)\otimes V^{(l)}(b)$ is given by
	\begin{equation*}
		w_j = \sum_{i=0}^{j}(-1)^iq^{i(l-i+1)}[k-j+i]_q![l-i]_q!\, v_{j-i}\otimes v_i.
	\end{equation*}
	Indeed, one verifies this in a few lines by checking that $w_j$ has the right weight and is annihilated by $x_1^+$.\footnote{ We remind the reader that $\Delta_q(x_1^+)=x_1^+\otimes k_1+ 1\otimes x_1^+$.} Then, again, by a direct computation, one checks that for $j>0$, $w_j$ is annihilated by $\mathcal{X}_{1,-1}^+$ if and only if
	\begin{equation*}
		b/a = q^{k+l-2j+2}.\footnotemark
	\end{equation*}
	\footnotetext{ We also remind the reader that $\Delta_q(\mathcal{X}_{1,-1}^+)=\mathcal{X}_{1,-1}^+\otimes K_1^{-1}+1\otimes\mathcal{X}_{1,-1}^+$.}
	In this case, it follows from
	\begin{equation*}
		[\mathcal{H}_{1,-1},\mathcal{X}_{1,0}^+]=(q+q^{-1})\mathcal{X}_{1,-1}^+
	\end{equation*}
	that $\mathcal{H}_{1,-1}.w_j$ is also annihilated by $\mathcal{X}_{1,0}^+$ and has the same weight as $w_j$. It must therefore be a scalar multiple of $w_j$. Using the relations
	\begin{equation*}
		[\mathcal{H}_{1,l},\mathcal{X}^\pm_{1,k}]=\pm\frac{1}{l}[2l]_q\mathcal{X}^\pm_{1,k+l}
	\end{equation*}
	it follows that $w_j$ is annihilated by all $\mathcal{X}_{1,k}^+$ and an eigenvector of the $\mathcal{H}_{1,k}$. This proves that $V^{(k)}(a)\otimes V^{(l)}(b)$ has a subrepresentation not containing the tensor product of the highest weight vectors in $V^{(k)}(a)$ and $V^{(l)}(b)$ if and only if
	\begin{equation*}
		b/a = q^{k+l-2j+2}
	\end{equation*}
	for some $0<j\leq\min\{m,n\}$. On the other hand, the tensor product has a proper subrepresentation containing the highest weight vectors if and only if its (left) dual
	\begin{equation*}
		(V^{(k)}(a)\otimes V^{(l)}(b))^* \cong V^{(l)}(q^2b)\otimes V^{(k)}(q^2a)
	\end{equation*}
	has a subrepresentation not containing the highest component.\footnote{ We shall say that we used the relation $S_q\circ\operatorname{ev}_{q^2a}=\operatorname{ev}_a\circ S_q$ between the antipodes of $U_q(\mathfrak{sl}_2)$, $U_q(\mathfrak{L}(\mathfrak{sl}_2))$ and the evaluation homomorphism to compute the (left) duals, which is easily checked on the Chevalley generators of $U_q(\mathfrak{L}(\mathfrak{sl}_2))$ in \ref{defprop:U_q(g)_chevalley}.} By the previous argument, this is the case if and only if
	\begin{equation*}
		b/a = q^{-(k+l-2j+2)}
	\end{equation*}
	for some $0<j\leq\min\{m,n\}$.
\end{proof}
In particular, whenever $V^{(k)}(a)\otimes V^{(l)}(b)$ is reducible, the proof implies that it is not completely reducible, nor is $V^{(k)}(a)\otimes V^{(l)}(b)$ isomorphic to $V^{(l)}(b)\otimes V^{(k)}(a)$. Therefore, $U_q(\mathfrak{L}(\mathfrak{g}))$ can't be quasitriangular (or even almost cocommutative). In fact, it implies that the category of finite-dimensional type $\boldsymbol{1}$ representations of $U_q(\tilde{\mathfrak{g}})$ is not quasitensor. Nevertheless, it turns out that $U_q(\tilde{\mathfrak{g}})$ is quasitriangular in a certain 'limiting' sense and that the Grothendieck ring $\operatorname{Rep}(U_q(\tilde{\mathfrak{g}}))$ of the category of finite-dimensional type $\boldsymbol{1}$ representations of $U_q(\tilde{\mathfrak{g}})$ is still a commutative ring. Indeed, Proposition \ref{prop:special_pos_short_ex} proves this for the $\mathfrak{sl}_2$ case, i.e. we have
\begin{cor}\label{cor:comm_of_groth_sl2}
	 The Grothendieck ring $\operatorname{Rep}(U_q(\tilde{\mathfrak{sl}}_2))$ of the category of finite-dimensional type $\boldsymbol{1}$ representations of $U_q(\tilde{\mathfrak{sl}}_2)$ is commutative. $\odot$
\end{cor}
The proof of this corollary is straightforward. For, Proposition \ref{prop:special_pos_short_ex} proves that the product of classes of evaluation representations is commutative whenever their corresponding $q$-segments are in special position. Otherwise the $q$-segments are in general position and the tensor product is irreducible. In this case, the commutativity follows from the second part of Theorem \ref{thm:drinfeld_polynomials}. Then, since every finite-dimensional irreducible type $\boldsymbol{1}$ representation is isomorphic to a tensor product of evaluation representations by Proposition \ref{prop:iso_to_tens_of_ev}, the assertion follows.

We will come back to the discussion of the general case in the next chapter.
\subsection{The trigonometric $R$-matrix of Type $A_l$}\label{subsect:trigon_R_matrix_A_l}
We have seen in the last subsection that (untwisted) quantum affine algebras are not quasitriangular. However, we also mentioned that they still turn out to be quasitriangular in a certain 'limiting' sense. (We also refer the reader to Subsection 12.5B in \cite{CPBook} for the following explanation.) To be more precise, one can show that (untwisted) quantum affine algebras are 'pseudotriangular' (the terminology is due to Drinfeld \cite{D} (1887)), and that, as for quasitriangular Hopf algebras, this property is sufficient to guarantee that there is a solution, the pseudo universal $R$-matrix, of the (quantum) Yang--Baxter equation associated to the tensor product of any pair of finite-dimensional representations. In principle, it can be constructed by the so-called quantum double method which was used to construct the universal $R$-matrix of $U_h(\mathfrak{g})$ (cf. Chapter 8 of \cite{CPBook}), but there are two difficulties caused by differences between the structure of the root system of an untwisted affine Kac--Moody algebra $\tilde{\mathfrak{g}}$ and that of its underlying finite-dimensional Lie algebra $\mathfrak{g}$. The first important difference is that, unlike the situation for $\mathfrak{g}$, not every root of $\tilde{\mathfrak{g}}$ can be obtained by applying an element of the Weyl group of $\tilde{\mathfrak{g}}$ to a simple root. This means that the Weyl group cannot be used to define the non-simple root vectors. The other difference is that $\tilde{\mathfrak{g}}$ has infinitely many positive roots. Thus, if one is to prove a formula like that in Subsection \ref{subsect:QUE_algebras} for the universal $R$-matrix $\tilde{\mathcal{R}}$ of $U_h(\tilde{\mathfrak{g}})$, the coefficients in its $h$-adic expansion will be formal infinite linear combinations of products of root vectors. Hence, there would be no guarantee that $\tilde{\mathcal{R}}$ would act on a tensor product of representations of $U_h(\tilde{\mathfrak{g}})$. This is why we obtain only a 'pseudo-universal' $R$-matrix.

However, despite everything, there is the following result of Drinfeld (1987) \cite{D}. Let $\lambda$ be an indeterminate and let $\tau'_\lambda$ be the automorphism of $U_h(\tilde{\mathfrak{g}})((\lambda))$ given by
\begin{equation*}
	\tau'_\lambda(X_i^\pm) = \lambda^\pm X_i^\pm,\quad\tau'_\lambda(H_i) = H_i,\quad i=0,1,\dots,l.
\end{equation*}
\begin{thm}
	There exists an element $\tilde{\mathcal{R}}(\lambda)\in(U_h(\tilde{\mathfrak{g}})\otimes U_h(\tilde{\mathfrak{g}}))((\lambda))$ such that, if $\rho:U_h(\tilde{\mathfrak{g}})\to \End(V)$ is an indecomposable representation of $U_h(\tilde{\mathfrak{g}})$ on a free $\mathbb{C}[[h]]$-module $V$ of finite rank, then $\tilde{R}^\rho(\lambda)=(\rho\otimes\rho)((\tau'_\lambda\otimes\id)(\tilde{\mathcal{R}}))$ is a well defined element of $\End(V\otimes V)((\lambda))$ which satisfies the (quantum) Yang--Baxter equation in the form
	\begin{equation}
		\tilde{R}_{12}^\rho(\lambda/\mu)\tilde{R}_{13}^\rho(\lambda/\nu)\tilde{R}_{23}^\rho(\mu/\nu)=\tilde{R}_{23}^\rho(\mu/\nu)\tilde{R}_{13}^\rho(\lambda/\nu)\tilde{R}_{12}^\rho(\lambda/\mu).\label{eqn:QYBE_mult}
	\end{equation}
	Up tp a scalar factor, $\sigma\circ \tilde{R}^\rho(\lambda)$ is a rational matrix-valued function of $\lambda$ and is an intertwining operator between the representations $(\rho\otimes\rho)(\tau'_\lambda\otimes\id)$ and $(\rho\otimes\rho)(\id\otimes\tau'_\lambda)$. $\odot$
\end{thm}
\begin{rems}\hspace{1em}
	\begin{enumerate}
		\item The Yang--Baxter equation (\ref{eqn:QYBE_mult}) (with multiplicative parameters) can be brought into difference type (i.e. additive parameters) simply by making the change of variable $\lambda=e^u$, $\mu = e^v$, $\nu = e^w$. As a consequence, $\tilde{R}^\rho$ becomes a 'trigonometric' function.
		\item Khoroshkin and Tolstoy (1992) \cite{KT2} give a formula for $\tilde{\mathcal{R}}(\lambda)$ for arbitrary (untwisted) quantum affine algebras. They define root vectors of $U_h(\tilde{\mathfrak{g}})$ without using the braid group action, using instead the notion of a \textbf{normal ordering} of the positive roots of $\tilde{\mathfrak{g}}$. $\odot$
	\end{enumerate}
\end{rems}
For $U_q(\tilde{\mathfrak{sl}_2})$ and $U_q(\tilde{\mathfrak{sl}_3})$ the Khoroshkin and Tolstoy construction has been explicitly carried out in the paper \cite{BGKNR} for the first fundamental representation and the result for general $U_q(\tilde{\mathfrak{sl}_n})$ is conjectured based on results from direct calculations of the universal $R$-matrix as the (up to normalisations) unique intertwiner $I=\sigma\circ R$. Precisely, they even use a slightly more general automorphism $\tilde{\tau}_\lambda$ defined on the Chevalley generators by
\begin{equation*}
	\begin{split}
		\tilde{\tau}_\lambda(\mathcal{H}_i) = \mathcal{H}_i,\quad \tilde{\tau}_\lambda(\mathcal{X}_i^\pm) = \lambda^{\pm s_i}\mathcal{X}_i^\pm,
	\end{split}
\end{equation*}
where the $s_i$, $i=0,\dots,l$, are some integers. Obviously, $\tilde{\tau}_\lambda$ generalises $\tau_\lambda$ ($s_0=1$ and $s_i=0$ for $i=1,\dots,l$) and $\tau'_\lambda$ ($s_i=1$ for $i=0,\dots,l$). However, it is a priori clear that the rescaling of the finite part, i.e. $i\neq 0$ doesn't have an interpretation in terms of the loop character of $U_q(\tilde{\mathfrak{sl}_n})$ and one can assume that $s_i=0$ for $i=1,\dots,l$ simply by rescaling the (Chevalley) basis elements $\mathcal{X}_{i,0}$ of the second Drinfeld realization \ref{thm:second_drin_quantum_aff}. On the other side, one can clearly take $s_0=1$ up to redefining the indeterminate $\lambda$. Therefore, let us write down the result for $\tau$, i.e. $s_0=1$ and $s_i=0$, $i=1,\dots,l$.
\begin{conj}[the (fundamental) trigonometric $R$-matrix of type $A_l$]
	Let $\xi_l$ be the function defined by
	\begin{equation*}
		\xi_l(\lambda) = \sum_{m=1}^{\infty}\frac{1}{[l+1]_{q^m}}\frac{\lambda^m}{m}
	\end{equation*}
	and let $(V_f,\rho_f)$ be a fundamental representation of $U_q(\tilde{\mathfrak{sl}}_{l+1})$.\footnote{ Where we have set $\rho_f\equiv \rho_{V_f}$ here by abuse of notation. } Then, the (pseudo-)universal $R$-matrix acting on the tensor product of the two fundamental representations $V_f(\lambda) = \tau_\lambda^*(V_f)$ and $V_f$, i.e. $R(\lambda)=(\rho_f\otimes\rho_f)((\tau_\lambda\otimes\id)(\tilde{\mathcal{R}}))$, is given by
	\begin{equation*}
		\begin{split}
			R(\lambda) &= q^{\frac{l}{l+1}}\frac{e^{\xi_l(q^2\lambda)}}{e^{\xi_l(q^{-2}\lambda)}}\left[\sum_aE_{aa}\otimes E_{aa}+q^{-1}\frac{1-\lambda}{1- q^{-2}\lambda}\sum_{a\neq b}E_{aa}\otimes E_{bb}\right.\\
			&\left.+\frac{1-q^{-2}}{1- q^{-2}\lambda}\sum_{a<b}E_{ab}\otimes E_{ba}+\frac{1-q^2}{1-q^2\lambda^{-1}}\sum_{a>b}E_{ab}\otimes E_{ba}\right].\footnotemark\quad \odot
		\end{split}
	\end{equation*}
	\footnotetext{ Where we shall make the definition $q = e^h$.}
\end{conj}
We note that the fundamental representation $V_f$ in the paper \cite{BGKNR} is defined as $(\operatorname{ev}^{-q^2})^*(V_q(\omega_1))$. However, it is clear that any other fundamental representation leads to the same result since the $R$-matrix only depends on the fraction of the spectral parameters of the representations. More importantly, compared to other methods of computing the $R$-matrix, one naturally obtains a proper normalisation implied by the fact that $\tilde{\mathcal{R}}(\lambda)$ is the (pseudo-)universal $R$-matrix of $U_h(\tilde{\mathfrak{g}})((\lambda))$, i.e. $\tilde{\mathcal{R}}(\lambda)$ has the properties (cf. \cite{CPBook} Theorem 12.5.1)
\begin{equation*}
	\begin{split}
		(\Delta_h\otimes\id) (\tilde{\mathcal{R}}(\lambda))& = \tilde{\mathcal{R}}_{13}(\lambda)\tilde{\mathcal{R}}_{23}(\lambda),\quad (\id\otimes\Delta_h) (\tilde{\mathcal{R}}(\lambda)) = \tilde{\mathcal{R}}_{13}(\lambda)\tilde{\mathcal{R}}_{12}(\lambda),\\
		((\tau_\lambda\otimes\id)&(\Delta_h^{\operatorname{op}}(a)))\tilde{\mathcal{R}}(\lambda) = \tilde{\mathcal{R}}(\lambda)((\tau_\lambda\otimes\id)(\Delta_h(a)))
	\end{split}
\end{equation*}
for all $a\in U_h(\tilde{\mathfrak{g}})$. Moreover, $\tilde{\mathcal{R}}(\lambda)$ is a triangular solution of the (quantum) Yang--Baxter equation with multiplicative parameters, i.e.
\begin{equation*}
	\begin{split}
		\tilde{\mathcal{R}}_{12}(\lambda_{12})\tilde{\mathcal{R}}_{13}(\lambda_{13})\tilde{\mathcal{R}}_{23}(\lambda_{23}) &= \tilde{\mathcal{R}}_{23}(\lambda_{23})\tilde{\mathcal{R}}_{13}(\lambda_{13})\tilde{\mathcal{R}}_{12}(\lambda_{12})\text{ and}\\
		\tilde{\mathcal{R}}_{12}(\lambda)^{-1} &= \tilde{\mathcal{R}}_{21}(\lambda^{-1}).\footnotemark
	\end{split}
\end{equation*}
\subsection{Graphical notation}
\label{subsect:graphical_notation}
Let $V_f$ be a fundamental representation of $U_q(\tilde{\mathfrak{g}})$.\footnotetext{ We simplified the notation here by making the definition $\lambda_{ij}\coloneq \lambda_i/\lambda_j$.}\footnote{We usually take $V_f=V_{\omega_1}(1)$ the fundamental representation defined by its Drinfeld polynomials $P_1 = 1-u$ and $P_i=1$, $i\neq1$. However, when $\mathfrak{g}=\mathfrak{sl}_n$, we can also take $V_f=(\operatorname{ev}^{-q^2})^*(V_q(\omega_1))$ (or any other fundamental representation defined through one of the two evaluation homomorphisms), as above.} We depict $V_f$ graphically by an oriented line. An operator $O \in \End(V)$ is associated with a symbol, for instance a point, on the line. Let $\ket{v}\in V_f$ and $A, B \in \End(V_f)$, then $A\cdot B \ket{v}$ is depicted by
\begin{figure}[H]
	\centering
	\begin{tikzpicture}[scale=1.4]
		\draw[thick,-stealth] (0,0)node[anchor=east]{} -- (3,0) node[anchor=west]{$|v\rangle$};
		\filldraw[black] (1,0) circle (1.2pt) node[anchor=south]{$A$};
		\filldraw[black] (2,0) circle (1.2pt) node[anchor=south]{$B$};
		\node at (3.75,-0.1) {.} ;
	\end{tikzpicture}
\end{figure}
Using the automorphism $\tau_\lambda$ (see Equation (\ref{eqn:tau_a_q-deformed})) we define a family of (fundamental) representations with spectral parameter $\lambda\in\mathbb{C}^\times$ as usual, i.e. $V_f(\lambda)\coloneq\tau_\lambda^*(V_f)$ the pull-back of $V_f$ by $\tau_\lambda$.\footnote{This can be done for any $\mathfrak{g}$, but in general we don't have an evaluation homomorphism. A representation $V$ of $\mathfrak{g}$ can still be lifted, but the ambiguity of defining an origin (such as the pull-back $\operatorname{ev}_0^*(V)\eqcolon V(0)$ of $V$ at $0\in\mathbb{C}$) remains (see \cite{CPBook} Theorem 12.5.3).} Thus, we may identify the composition of the flip map $\sigma$ with the trigonometric $R$-matrix $(\rho_{V_f(\lambda)}\otimes\rho_{V_f(\mu)})(\tilde{\mathcal{R}}) \coloneq R(\lambda/\mu)\in \End(V_f(\lambda)\otimes V_f(\mu))$ with any positively oriented vertex between such oriented lines with spectral parameters $\lambda$ and $\mu$, respectively, where $\tilde{\mathcal{R}}$ is the (pseudo-)universal $R$-matrix of the untwisted quantum affine algebra $U_q(\tilde{\mathfrak{g}})$. Note that the orientation of the vertex is naturally induced by the orientation of the lines, i.e. we have
\begin{figure}[H]
	\centering
	\begin{tikzpicture}[scale=1.2]
		\draw[thick,-stealth] (3,0)node[anchor=east]{$\mu$} -- (3,2) node[anchor=west]{};
		\draw[thick,-stealth] (2,1)node[anchor=south]{$\lambda$} -- (4,1) node[anchor=west]{};
		\node at (4.2,0) {.} ;
		\node at (0,1) {$\sigma\circ R(\lambda/\mu)$} ;
		\node at (1.25,1) {$=$} ;
	\end{tikzpicture}
\end{figure}
The (quantum) Yang--Baxter equation in the form (\ref{eqn:QYBE_mult}) is then depicted by
\begin{figure}[H]
	\centering
	\begin{tikzpicture}[scale=1.2]
		\draw[thick,-stealth] (0,1)node[anchor=south]{$\lambda_2$} -- (3,3) node[anchor=west]{};
		\draw[thick,-stealth] (0,2)node[anchor=south]{$\lambda_1$} -- (3,0) node[anchor=west]{};
		\draw[thick,-stealth] (2.25,-.25)node[anchor=east]{$\lambda_3$} -- (2.25,3.25) node[anchor=west]{};
		\node at (8.5,-.25) {.} ;
		\node at (4,1.5) {$=$} ;
		\draw[thick,-stealth] (5,3)node[anchor=south]{$\lambda_1$} -- (8,1) node[anchor=west]{};
		\draw[thick,-stealth] (5,0)node[anchor=south]{$\lambda_2$} -- (8,2) node[anchor=west]{};
		\draw[thick,-stealth] (5.75,-.25)node[anchor=east]{$\lambda_3$} -- (5.75,3.25) node[anchor=west]{};
	\end{tikzpicture}
\end{figure}
As for string diagrams, we can now visualise any braiding of tensor products between (fundamental) representations.
We shall also note that this is a possible (mathematical) starting point for the definition of the $\mathfrak{sl}_{n}$ trigonometric vertex model. When $l=1$, the latter is the famous six-vertex model, a model which describes ice in the two-dimensional plane. However, we will stay with the basic definitions made up until now for this thesis and refer the reader to my dissertation \cite{HJue} for instance, where a more detailed explanation can be found. The general idea is that the behaviour of two-dimensional vertex models build from an $R$-matrix is roughly said completely determined\footnote{ At least in the limit of large lattice size.} by the spectrum of a one parameter family of transfer matrices which, due to the Yang--Baxter equation, constitute a commuting family. It turns out that the latter can be interpreted as the generating function of the local integrals of motion of certain quantum spin chains, the Heisenberg XXZ spin chain, i.e. $l=1$, being probably the most prominent example thereof. Let us briefly discuss this in the next subsection.
\subsection{A physics point of view}\label{subsect:A_physics_pov}
In retrospect of some hints we gave in the last section, let us briefly formulate a physics point of view. It is based on a few lectures about vertex models and spin chains by Frank Göhmann and Herman Boos. In fact, we may also refer the reader to Section 7.5 in the book \cite{CP1991}, where similar things are explained.\\

We define a \textbf{vertex model} on a $\boldsymbol{2D}$ \textbf{square lattice} $\boldsymbol{\mathcal{L}}$ through its \textbf{local configurations} at each vertex.
The possible (Boltzmann-)weights at a \textbf{vertex} $v_{ij}$ are collected in the so-called $\boldsymbol{R}$\textbf{-matrix} (of the model), and depend on variables $\epsilon_i,\,\epsilon_i',\,\epsilon_j,\,\epsilon_j'$ associated to the adjacent edges.
\begin{equation*}
	\begin{tikzpicture}[scale=0.6]
		\usepgflibrary{arrows}
		\draw [thick, -stealth] (1,0) node[anchor=east]{$\epsilon_j$} -- (1,4) node[anchor=east]{$\epsilon_j'$};
		\draw (1.6, 2.4) node{$v_{i\,j}$};
		\draw [thick, -stealth] (-1,2) node[anchor=south]{$\epsilon_i$} -- (3,2) node[anchor=south]{$\epsilon_i'$};
		\draw (6, 2.1) node{$=: R^{\epsilon_i'\epsilon_j'}_{\epsilon_i\epsilon_j}$};
	\end{tikzpicture}
	\vspace{-1em}
\end{equation*}
The \textbf{weight} of a configuration on the lattice is therefore given by the \textbf{product} of the \textbf{local weights}, $W(C) := \prod_{v_{ij}\in \mathcal{L}}R^{\epsilon_n'(C)\epsilon_l'(C)}_{\epsilon_n(C)\epsilon_l(C)}.$
The \textbf{partition function} of the model is then defined by the sum over all possible configurations $Z = \sum_{C}W(C)$.\\

We state the \textbf{integrability of the model} as follows. Assume we have a family of $R$-matrices $R:\mathbb{C}^2\to \operatorname{End}(\mathbb{C}^d\otimes \mathbb{C}^d)$ parametrised by \textbf{two spectral parameters}
\begin{equation*}
	\begin{tikzpicture}[scale=0.6]
		\usepgflibrary{arrows}
		\draw [thick, -stealth] (1,0) node[anchor=east]{$\mu$} -- (1,4);
		\draw (1.6, 2.4);
		\draw [thick, -stealth] (-1,2) node[anchor=south]{$\lambda$} -- (3,2);
		\draw (6, 2.1) node{$=: R(\lambda,\mu)$};
	\end{tikzpicture}
	\vspace{-1em}
\end{equation*}
such that $R(\lambda,\lambda)=P$ is the \textbf{permutation operator} (\textbf{regularity}).
The local \textbf{integrability condition} is then given by the (quantum) \textbf{Yang Baxter equation}
\begin{equation*}
	\begin{tikzpicture}[scale=0.8]
		\usepgflibrary{arrows}
		\draw [thick, -stealth] (3,0.25) node[anchor=east]{$\nu$} -- (3,6.25) ;
		\draw [thick, -stealth] (-1,2.5) node[anchor=south]{$\mu$} -- (4,5.5);
		\draw [thick, -stealth] (-1,4) node[anchor=south]{$\lambda$} -- (4,1);
		\draw (5, 3.25) node{$=$};
		\draw [thick, -stealth] (7,0.25) node[anchor=east]{$\nu$} -- (7,6.25) ;
		\draw [thick, -stealth] (6,1) node[anchor=south]{$\mu$} -- (11,4);
		\draw [thick, -stealth] (6,5.5) node[anchor=south]{$\lambda$} -- (11,2.5);
		\node at (12,0) {,};
	\end{tikzpicture}
	\vspace{-1.25em}
\end{equation*}
where the \textbf{graphical notation} is defined by summing over all configurations of the \textbf{inner edges}.\footnote{ Note that this corresponds exactly to the coordinate formulation of the graphical notation introduced in Subsection \ref{subsect:graphical_notation} after choosing a basis.}

To give the reader a feeling why this is a priori the right condition for the integrability of the model, we may assume that this equation reflects the general scattering of three particles. If we pick one direction as the 'time' direction, then the equation implies that the order of the scattering doesn't matter. Therefore, any $n$-particle scattering can be decomposed into a product of two particle interactions and the order doesn't matter. In fact, from classical mechanics we know that only the two particle problem is classically integrable. Moreover, we may interpret the fact that the order of the scattering doesn't matter in terms of the commutativity of the Grothendieck ring of some Hopf algebra. In fact, it can be elucidated that particles in quantum field theory should indeed correspond to irreducible representations of some Hopf algebra. However, this is as far as we will go with a particle physics motivation of the quantum Yang Baxter equation as an integrability condition. The direct explanation in statistical mechanics is as follows.\\

We define familys of \textbf{inhomogeneous transfer matrices}
\begin{equation*}
	\begin{tikzpicture}[scale=0.5]
		\usepgflibrary{arrows}
		\draw (-14, 1) node{$t_L(\lambda|\mu_1,...,\mu_L)\; :=$};
		\draw [thick, -stealth] (-4,-.5) node[anchor=north]{$\mu_L$} -- (-4,2.5);
		\draw [thick, -stealth] (1,-.5) node[anchor=north]{$\mu_3$} -- (1,2.5);
		\draw [thick, -stealth] (3,-.5) node[anchor=north]{$\mu_2$} -- (3,2.5);
		\draw [thick, -stealth] (5,-.5) node[anchor=north]{$\mu_1$} -- (5,2.5);
		\draw [right hook-stealth, thick] (-6,1) node[anchor=east]{$\lambda$} -- (7,1);
		\draw (-1.5, 0) node{$\cdots$};
		\draw (7.5, -1.5) node{};
	\end{tikzpicture}
	\vspace{-.75em}
\end{equation*}
\begin{equation*}
	\begin{tikzpicture}[scale=0.5]
		\usepgflibrary{arrows}
		\draw (-14, 1) node{$\bar{t}_L(\lambda|\mu_1,...,\mu_L)\; :=$};
		\draw [thick, -stealth] (-4,-.5) node[anchor=north]{$\mu_L$} -- (-4,2.5);
		\draw [thick, -stealth] (1,-.5) node[anchor=north]{$\mu_3$} -- (1,2.5);
		\draw [thick, -stealth] (3,-.5) node[anchor=north]{$\mu_2$} -- (3,2.5);
		\draw [thick, -stealth] (5,-.5) node[anchor=north]{$\mu_1$} -- (5,2.5);
		\draw [stealth-right hook, thick] (-6,1) node[anchor=east]{$\lambda$} -- (7,1);
		\draw (-1.5, 0) node{$\cdots$};
		\draw (7.5, -1.5) node{,};
	\end{tikzpicture}
	\vspace{-.75em}
\end{equation*}
where the \textbf{hook} is used to indicate that a \textbf{line is closed}, i.e. taking the \textbf{trace}. We therefore obtain a \textbf{commutative family} of transfer-matrices using the Yang Baxter equation
\begin{equation*}
	\begin{tikzpicture}[scale=0.4]
		\usepgflibrary{arrows}
		\draw [thick, -stealth] (-21,-.5)node[anchor=north]{$\mu_L$} -- (-21,4.5) ;
		\draw [thick, -stealth] (-16,-.5)node[anchor=north]{$\mu_2$} -- (-16,4.5) ;
		\draw [thick, -stealth] (-14,-.5)node[anchor=north]{$\mu_1$} -- (-14,4.5) ;
		\draw [right hook-stealth, thick] (-23,3) node[anchor=east]{$\lambda'\hspace{-.25em}$} -- (-12,3);
		\draw [right hook-stealth, thick] (-23,1) node[anchor=east]{$\lambda$} -- (-12,1);
		\draw (-18.5, 0) node{$\cdots$};
		\draw (-10, 1.9) node{$=$};
		\draw [thick, -stealth] (-5,-.5)node[anchor=north]{$\mu_L$} -- (-5,4.5) ;
		\draw [thick, -stealth] (0,-.5)node[anchor=north]{$\mu_2$} -- (0,4.5) ;
		\draw [thick, -stealth] (2,-.5)node[anchor=north]{$\mu_1$} -- (2,4.5) ;
		\draw [right hook-, thick] (-7,3) node[anchor=east]{$\lambda'\hspace{-.25em}$} -- (3,3);
		\draw [right hook-, thick] (-7,1) node[anchor=east]{$\lambda$} -- (3,1);
		\draw (-2.5, 0) node{$\cdots$};
		\draw [-, thick] (3, 3) .. controls (3.5,3) .. (4.5,2);
		\draw [-, thick] (3, 1) .. controls (3.5,1) .. (4.5,2);
		
		\draw [-, thick] (4.5, 2) .. controls (5.5,1) .. (6,1);
		\draw [-, thick] (4.5, 2) .. controls (5.5,3) .. (6,3);

		\draw [-, thick] (6, 3) .. controls (6.5,3) .. (7.5,2);
		\draw [-, thick] (6, 1) .. controls (6.5,1) .. (7.5,2);

		\draw [-stealth, thick] (7.5, 2) .. controls (8.5,1) .. (9,1);
		\draw [-stealth, thick] (7.5, 2) .. controls (8.5,3) .. (9,3);
		\draw (11, 1.9) node{$=$};
		\draw [thick, -stealth] (2,-8.5)node[anchor=north]{$\mu_L$} -- (2,-3.5) ;
		\draw [thick, -stealth] (7,-8.5)node[anchor=north]{$\mu_2$} -- (7,-3.5) ;
		\draw [thick, -stealth] (9,-8.5)node[anchor=north]{$\mu_1$} -- (9,-3.5) ;
		\draw [right hook-stealth, thick] (0,-5) node[anchor=east]{$\lambda$} -- (11,-5);
		\draw [right hook-stealth, thick] (0,-7) node[anchor=east]{$\lambda'\hspace{-.25em}$} -- (11,-7);
		\draw (4.5, -8) node{$\cdots$};

		\draw (-5, -6.1) node{$= ... =$};
		\draw [thick, -stealth] (-21,-8.5)node[anchor=north]{$\mu_L$} -- (-21,-3.5) ;
		\draw [thick, -stealth] (-16,-8.5)node[anchor=north]{$\mu_2$} -- (-16,-3.5) ;
		\draw [thick, -stealth] (-12,-8.5)node[anchor=north]{$\mu_1$} -- (-12,-3.5) ;
		\draw [right hook-, thick] (-23,-5) node[anchor=east]{$\lambda'\hspace{-.25em}$} -- (-15,-5);
		\draw [right hook-, thick] (-23,-7) node[anchor=east]{$\lambda$} -- (-15,-7);
		\draw [-, thick] (-15, -5) .. controls (-14.5,-5) .. (-13.5,-6);
		\draw [-, thick] (-15, -7) .. controls (-14.5,-7) .. (-13.5,-6);

		\draw [-, thick] (-13.5, -6) .. controls (-12.5,-7) .. (-12,-7);
		\draw [-, thick] (-13.5, -6) .. controls (-12.5,-5) .. (-12,-5);

		\draw (-18.5, -8) node{$\cdots$};
		\draw [-, thick] (-12, -5) .. controls (-11.5,-5) .. (-10.5,-6);
		\draw [-, thick] (-12, -7) .. controls (-11.5,-7) .. (-10.5,-6);

		\draw [-stealth, thick] (-10.5, -6) .. controls (-9.5,-7) .. (-9,-7);
		\draw [-stealth, thick] (-10.5, -6) .. controls (-9.5,-5) .. (-9,-5);
		\draw (-26, -6.1) node{$=$};
		\draw (11.5, -9.5) node{,};
	\end{tikzpicture}
	\vspace{-1em}
\end{equation*}
where the $\mu_{1},\dots,\mu_{L}$ are called \textbf{inhomogeneity parameters}. Clearly, the transfer matrix description reduces the problem of calculating the \textbf{partition function} and \textbf{correlation functions} on a torus $\mathcal{L}=\mathbb{T}_{L,M}=\mathbb{Z}/L\times \mathbb{Z}/M=:\mathcal{L}$ to the problem of \textbf{diagonalizing} the \textbf{transfer matrix}.\\
	
Now, if $R(\lambda,\mu)$ is \textbf{differentiable} in a \textbf{vicinity of the origin}, we can use the \textbf{homogeneous transfer matrices} $t_L(\lambda)$ and $\bar{t}_L(\lambda)$ to define a \textbf{Hamiltonian}
\begin{align*}
	H_L =& t_L'(\lambda)|_{\lambda=0}\cdot t_L^{-1}(0) = -t_L(0)\cdot\bar{t}_L'(\lambda)|_{\lambda=0}\\
	=& \sum_{j=1}^{L}\partial_\lambda (PR)_{j-1,j}(\lambda,0)|_{\lambda=0}, \quad (PR)_{0,1}:=(PR)_{L,1}
\end{align*}
on the tensor product $V^{\otimes L}$ of the \textbf{local Hilbert spaces} $V=\mathbb{C}^d$. It is not hard to conclude that the \textbf{Trotter formula}
\begin{align*}
	\exp\left(-\beta H_L\right) = \lim_{N\to\infty}\left(t_L\left(-\frac{\beta}{2N}\right)\bar{t}_L\left(\frac{\beta}{2N}\right)\right)^N
\end{align*}
is satisfied (in operator norm) on the space $V^{\otimes L}$.
Then, setting $T=\frac{1}{\beta}$, we can calculate the \textbf{free energy per lattice site} in the \textbf{thermodynamic limit} via
\begin{align*}
	f(T) = -T\lim_{L\to\infty}\frac{1}{L}\ln\{\operatorname{tr}_L(\exp\left(- H_L/T\right))\}.
\end{align*}
Moreover, we can calculate \textbf{correlation functions} by defining the \textbf{(reduced) density matrix}
\begin{equation*}
	\begin{tikzpicture}[scale=0.6]
		\usepgflibrary{arrows}
		\draw [right hook-stealth, thick] (-2,2) -- (-2,14);
		\draw [right hook-stealth, thick] (-4,2) -- (-4,14);
		
		\draw [right hook-stealth, thick] (0,2) -- (0,6.75) node[anchor=south]{\scriptsize{$\epsilon_{m}'$}};
		\draw [right hook-stealth, thick] (4,2) -- (4,6.75) node[anchor=south]{\scriptsize{$\epsilon_{1}'$}};
		
		\draw [-stealth, thick] (0,9.25) node[anchor=north]{\scriptsize{$\epsilon_{m}$}} -- (0,14);
		\draw [-stealth, thick] (4,9.25) node[anchor=north]{\scriptsize{$\epsilon_{1}$}} -- (4,14);
		
		\draw [right hook-stealth, thick] (6,2) -- (6,14);
		\draw [right hook-stealth, thick] (8,2) -- (8,14);
		
		\draw [stealth-right hook, thick] (-5,4) -- (9,4);
		\draw [right hook-stealth, thick] (-5,6) -- (9,6);
		\draw [stealth-right hook, thick] (-5,10) -- (9,10);
		
		\draw (2, 13) node{$\cdots$};
		
		\draw [right hook-stealth, thick] (-5,12) -- (9,12);
		
		\draw (2, 3) node{$\cdots$};
		
		\draw (-10, 8) node{$(D_{1,\dots,m})^{\epsilon_1',\dots,\epsilon_m'}_{\epsilon_1,\dots,\epsilon_m}$ $=$ $\frac{1}{Z}\times$ };
	\end{tikzpicture}
\end{equation*}
Then $D_{1,\dots,m}\in \operatorname{End}(V^{\otimes m})$ and the \textbf{expectation value} of an \textbf{observable} $\boldsymbol{\mathcal{O}}\in\operatorname{End}(V^{\otimes m})$ is defined by the \textbf{trace} $\operatorname{tr}_m(\mathcal{O}D_{1,\dots,m})$.	We shall also note that the term 'reduced' refers to the specific properties of $D_m\coloneq D_{1,\dots,m}$ which basically state that it doesn't matter which $D_m$ is used for the calculation of correlations, as long as the Observable $\boldsymbol{\mathcal{O}}$ is supported in (or an element of) a space $\operatorname{End}(V^{\otimes n})\lhook\joinrel\rightarrow\operatorname{End}(V^{\otimes m})$, $n\leq m$ (see \cite{HJue} for more details).\footnote{ The embedding of elements $\operatorname{End}(V^{\otimes n})\lhook\joinrel\rightarrow\operatorname{End}(V^{\otimes m})$ is (as in Subsection \ref{subsect_quasitriang_Hopf}) by tensoring $m-n$ identities from the right or left. Moreover, it is easy to see that the expectation value doesn't depend on how many of the $m-n$ identities are tensored from the left (resp. from the right).} Precisely, the $D_n$, $n\in\mathbb{N}$ define a morphism on the direct limit $\varinjlim\End(V^{\otimes n})$, call it $D\coloneq \varinjlim D_n$ (see \cite{HJue} for the details).
\newpage
\chapter{The finite-dimensional representation theory of $\boldsymbol{U_q(\tilde{\mathfrak{g}})}$}\label{ch:fin_dim_rep_quantum_aff}
Let us now turn to the general discussion of the finite-dimensional representation theory of $U_q(\tilde{\mathfrak{g}})$.

Section \ref{sect:fin-dim_rep_quantum_aff_sl2} begins with an analysis of the simplest case where $\mathfrak{g}=\mathfrak{sl}_2$, drawing a direct connection to the Drinfeld polynomial and loop-weight description provided in Subsections \ref{subsect:drinfeld_poly} and \ref{subsect:Jimbos_hom_und_ev_reps}.

In Section \ref{sect:gen_untw_case}, we shift our attention to the general case where $\mathfrak{g}$ is an arbitrary finite-dimensional simple Lie algebra over $\mathbb{C}$. We simplify the notation while maintaining the connection to the Drinfeld polynomial description in Subsection \ref{subsect:drinfeld_poly}. Moreover, we discuss a few of the more general assertions which have been made earlier, such as the loop-weight description in Theorem \ref{thm:drinfeld-l_weights} and the commutativity of the Grothendieck ring $\operatorname{Rep}U_q(\tilde{\mathfrak{g}})\coloneq\operatorname{Gr}(\mathrm{C})$ of the category $\mathrm{C}$ of finite-dimensional type $\boldsymbol{1}$ representations of $U_q(\tilde{\mathfrak{g}})$.

Then, in Subsection \ref{subsect:q-char_and_snake_mod}, we elucidate that the analysis of the category $\mathrm{C}$ can be reduced to the consideration of a subcategory $\mathrm{C}_{\mathbb{Z}}$, meaning that one can basically go 'modulo $\mathbb{C}^\times$' (up to special positions), i.e. dividing out the action of $\mathbb{C}^\times$ on $\mathrm{C}$ through $\tau$ (see Equation (\ref{eqn:tau_a_q-deformed})). We introduce the so-called (prime) snake modules and assert that these are the real (and prime) building blocks which describe all finite-dimensional representations in $\mathrm{C}$.\footnote{ Up to a precise description of short exact sequences.} We discuss the so-called 'extended T-system' for prime snakes, a (arguably complete) set of relations in the Grothendieck ring $\operatorname{Gr}(\mathrm{C}_{\mathbb{Z}})$, which has been proven in the paper \cite{MY} of Mukhin and Young in 2012. Moreover, we expect that our assertion can be proven with it. However, the extended T-system is just one (complete) set of relations in the Grothendieck ring and it is a priori far from obvious how to find an explicit description of all possible (tensor) products in $\operatorname{Gr}(\mathrm{C}_{\mathbb{Z}})$ (resp. $\mathrm{C}_{\mathbb{Z}}$). In fact, a proof of our assertion is already implied by the results in the paper \cite{BJY}, though, using a different set of relations called the $S$-system. We give the reference and a brief explanation. The proof of the equivalence of the $S$-system remains an open question.

Knowing about the importance of prime snake modules as the prime irreducible elements in $\mathrm{C}$, we are interested in their $q$-characters. These are discussed in in Subsection \ref{subsect:the_path_formula}. Fortunately, the so-called 'path formula' for $q$-characters (Theorem 6.1 in \cite{MY2}) was proven by Mukhin and Young shortly after introducing the snake modules. It allows an explicit calculation of the $q$-characters of snake modules. We conclude the discussion of the $q$-characters and the basic classification of the irreducible objects in $\mathrm{C}$ by suspecting an elegant solution to the problem of computing projectors onto prime snake modules.

In Subsection \ref{subsect:form_in_terms_of_clust_alg}, we return to a more detailed explanation of the results stated in the paper \cite{BJY}. In fact, the authors prove an entirely different description of the Grothendieck ring $\operatorname{Gr}(\mathrm{C}_{\mathbb{Z}})$ in terms of an infinite cluster algebra, lets call it $\mathcal{A}_\infty$, which can be defined in terms of the symmetrised Cartan matrix $B=DA$ of $\mathfrak{g}$. Precisely, the authors prove that $\mathrm{C}_{\mathbb{Z}}$ is a monoidal categorification of $\mathcal{A}_\infty$, i.e., the Grothendieck ring $\operatorname{Gr}(\mathrm{C}_{\mathbb{Z}})$ is isomorphic to $\mathcal{A}_\infty$ and cluster monomials respectively cluster variables correspond to classes of real simple respectively real prime simple objects. As explained earlier, this is done by proving a different set of relations for the prime snake modules called the $S$-systems. These, in fact, correspond to the possible mutations of cluster variables in the cluster algebra $\mathcal{A}_\infty$, and therefore the real and prime simple objects are exactly the prime snake modules, proving the assertion in Subsection \ref{subsect:q-char_and_snake_mod}.
\section{The finite-dimensional representations of $U_q(\mathfrak{L}(\mathfrak{sl}_2))$}\label{sect:fin-dim_rep_quantum_aff_sl2}
Let $\mathrm{C}$ be the category of finite-dimensional type $\boldsymbol{1}$ representations of $U_q(\tilde{\mathfrak{sl}}_2)$.
We collect the facts in Section \ref{sect:finite-dim_reps} for the case $\mathfrak{g}=\mathfrak{sl}_2$.

Due to Jimbos evaluation homomorphism $\text{ev}_a:U_q(\tilde{\mathfrak{sl}_2})\to U_q(\mathfrak{sl}_2)$ given by Proposition \ref{prop:Jimbos_hom} we can get irreducible type $\boldsymbol{1}$ spin $r/2$ evaluation representations $V^{(r)}(a)$. Moreover, we have calculated their \textbf{Drinfeld polynomials}
in Example \ref{xmpl:spin_r/2_evaluation_rep} as
\begin{equation*}
	P_{r,a}(u) = \prod_{k=1}^{r}(1-q^{r-2k+1}au).
\end{equation*}
To abbreviate the notation, it is convenient to refer to the polynomial $P_a(u) = 1-au$ by an indeterminate $Y_a$, $a\in\mathbb{C}^\times$, which we call a \textbf{fundamental loop-weight}. Thus, the polynomials $\mathcal{P}$ with constant term $1$ correspond to monomials in the fundamental loop-weights $Y_a$, $a\in\mathbb{C}^\times$, which are in one to one correspondence with the irreducible representations of type $\boldsymbol{1}$, reformulating the Drinfeld polynomial description in Theorem \ref{thm:drinfeld_polynomials}. This explains why we call the $Y_a$, $a\in\mathbb{C}^\times$, fundamental loop-weights. Moreover, the Drinfeld polynomials of the spin $r/2$ \textbf{evaluation representations} $V^{(r)}(a)$  (see Example \ref{xmpl:spin_r/2_evaluation_rep}) correspond to the monomial
\begin{equation*}
	S_r(a)\coloneq Y_{aq^{r-1}}Y_{aq^{r-3}}\cdots Y_{aq^{-r+1}},
\end{equation*}
which we may call a $\boldsymbol{q}$\textbf{-string} of length $r$ and centre $a$ in correspondence to the notion of a $q$-segment introduced earlier in Definition \ref{def:q_segment}. In particular, for $r=1$, the spin $1/2$ evaluation representation has the Drinfeld polynomial $P_{1,a}(u) = 1-au = P_a(u)$. Thus, it is the irreducible \textbf{fundamental module} of type $\boldsymbol{1}$ corresponding to the fundamental loop weight $Y_{a}$.

By Proposition \ref{prop:iso_to_tens_of_ev}, we know that every finite-dimensional irreducible type $\boldsymbol{1}$ representation of $U_q(\tilde{\mathfrak{sl}}_2)$ is isomorphic to a tensor product of evaluation representations. Therefore, the structure of the finite-dimensional type $\boldsymbol{1}$ representations and the product structure of the Grothendieck ring is completely described by Proposition \ref{prop:special_pos_short_ex} together with Corollary \ref{cor:comm_of_groth_sl2}.

In fact, the $q$-character of $V^{(r)}(a)$ is given by Corollary \ref{cor:loop_weights_ev_reps_sl2}. Indeed, we have seen in the proof that every loop-weight space is one-dimensional. Therefore the $\boldsymbol{q}$\textbf{-character} of $V^{(r)}(a)$ is given in terms of the fundamental loop-weights by
\begin{equation*}\label{eqn:q-char_sl2_loop_weights}
	\begin{split}
		\chi_q(V^{(r)}(a)) &= \sum_{k=0}^{r} \left(\prod_{m=k+1}^{r}Y_{aq^{r-2m+1}}\right)\left(\prod_{m=1}^kY_{aq^{r-2m+3}}\right)^{-1} \\
		&=\prod_{m=1}^r Y_{aq^{r-2m+1}}\left(\sum_{k=0}^{r}\prod_{j=1}^k 	Y_{aq^{r-2j+1}}^{-1}Y_{aq^{r-2j+3}}^{-1}\right).
	\end{split}
\end{equation*}
Note that, by mapping all the $Y_a^{\pm1}$, $a\in\mathbb{C}^\times$, to the fundamental weights $e^{\pm \omega_1}$ of the group algebra $\mathbb{Z}[P]$, we clearly recover the character formula for the spin $r/2$ representation of $U_q(\mathfrak{sl}_2)$. Therefore, it is natural to define $A_a\coloneq Y_{aq}Y_{aq^{-1}}$, $a\in\mathbb{C}^\times$, as the \textbf{affine roots}.\footnote{ Note that $A_a$ is mapped onto $\alpha_1$.} Moreover, as in the non-affine case, the affine roots define a \textbf{partial ordering} on the loop-weights,\footnote{ By Definition \ref{def:q-character_drinfeld} these correspond to monomials in the $Y_a^{\pm1}$, $a\in\mathbb{C}^\times$.} i.e., for two loop-weights $m'$ and $m$ we have $m'\geq m$, if $m'm^{-1}$ is a monomial in the affine roots $A_a$, $a\in\mathbb{C}^\times$. From here, it is clear that we can formulate a '\textbf{theorem of the highest loop-weight}' in full analogy to the theorem of the highest weight for finite-dimensional representations of $\mathfrak{sl}_2$, but let us leave this for the next subsection where we discuss this for general $\mathfrak{g}$.\footnote{ As usual, $\mathfrak{g}$ refers to a finite-dimensional simple Lie algebra over $\mathbb{C}$.}

Denote by $\mathcal{Y}\coloneq\mathbb{Z}[Y_{a}^{\pm 1}]_{a\in\mathbb{C}^\times}$ the group algebra in the fundamental loop-weights and by $\operatorname{Rep}U_q(\tilde{\mathfrak{sl}}_2)$ the Grothendieck ring of $\mathrm{C}$, which is commutative due to Corollary \ref{cor:comm_of_groth_sl2}. Then, by collecting the information in Proposition \ref{prop:iso_to_tens_of_ev}, it is straightforward to show that $\chi_q:\operatorname{Rep}U_q(\tilde{\mathfrak{sl}}_2)\to \mathcal{Y}$ is an \textbf{injective ring homomorphism} which reduces to the usual character homomorphism on $U_q(\mathfrak{sl}_2) \lhook\joinrel\rightarrow U_q(\tilde{\mathfrak{sl}}_2)$. Hence, $\operatorname{Rep}U_q(\tilde{\mathfrak{sl}}_2)$ is isomorphic to $\mathbb{Z}[t_a]_{a\in\mathbb{C}^\times}$, where $t_a$ is the class of $V_{\omega_1}(a)=V^{(1)}(a)$.

In view of the general discussion in the next section let us reformulate everything as follows. We define the \textbf{Kirillov--Reshetikhin} (KR) \textbf{module} $W^{(r)}(a)$, $r\in \mathbb{N}$, as the irreducible type $\boldsymbol{1}$ representation corresponding to the $q$-string
\begin{equation*}
	S_r(aq^{r-1}) = Y_a Y_{aq^2} \cdots Y_{aq^{2r-2}}.
\end{equation*}
Obviously, this is just the spin $r/2$ evaluation representation $V^{(r)}(aq^{r-1})$, but for the later discussion it is more convenient to work only with representations in the subcategory $\mathrm{C}_{\mathbb{Z}}$, which have loop-weights $Y_{aq^{2k}}$, $k\in\mathbb{Z}$, for some fixed $a\in\mathbb{C}^\times$. In fact, if we go 'modulo' general positions, it is clear by Proposition \ref{prop:special_pos_short_ex} that it is sufficient to only consider representations which correspond to such loop-weights, since otherwise tensor products of irreducible modules are always irreducible, i.e. they 'don't see each other'. On the other side, starting with the representations in the subcategory $\mathrm{C}_{\mathbb{Z}}$ for some fixed $a\in\mathbb{C}^\times$, we can describe all other type $\boldsymbol{1}$ representations in $\mathbb{C}$ by considering the action of $\mathbb{C}^\times$ through $\tau$ and taking tensor products.\footnote{ The Hopf algebra automorphism $\tau$ is defined by Equation (\ref{eqn:tau_a_q-deformed}).} In particular, Proposition \ref{prop:special_pos_short_ex} implies that the classes $[W^{(r)}(a)]$ satisfy the equation
\begin{equation}\label{eqn:Tsys_sl2}
	[W^{(r)}(a)][W^{(r)}(aq^2)]=[W^{(r+1)}(a)][W^{(r-1)}(aq^2)]+1,
\end{equation}
which is usually called the $\boldsymbol{\operatorname{T}}$\textbf{-system} for $\mathfrak{sl}_2$. Conversely, it is possible to recover all other relations in the Grothendieck ring from the $\operatorname{T}$-system. Indeed, the $\operatorname{T}$-system implies that $[W^{(r)}(a)]$ is given by the $r\times r$ determinant
\begin{equation*}
	[W^{(r)}(a)]=
	\begin{vmatrix}
	[W^{(1)}(a)]&1& 0 & \ldots & 0 \\
		1&[W^{(1)}(aq^2)]&1  & \ddots &\vdots  \\
		0&  1&[W^{(1)}(aq^4)]&\ddots  &\vdots  \\
		\vdots& \ddots &  \ddots&\ddots &\vdots  \\
		0&  \ldots &  0&  1&[W^{(1)}(aq^{2r-2})]
	\end{vmatrix},
\end{equation*}
which is easily proven by induction on $r$ (cf. \cite{HL}). With this we conclude the discussion of the $\mathfrak{sl}_2$ case.
\section{The general (untwisted) case}\label{sect:gen_untw_case}
Finally, let us discuss the general case. So let $\mathrm{C}$ be the category of finite-dimensional type $\boldsymbol{1}$ representations of $U_q(\tilde{\mathfrak{g}})$ and denote by $\operatorname{Rep}U_q(\tilde{\mathfrak{g}})$ its Grothendieck ring.
As usual we denote by $A=(a_{ij})$the Cartan matrix of $\mathfrak{g}$ and let $d_i$, $i=1,\dots,l$, the relatively prime integers such that $B=(b_{ij})=(d_i a_{ij}) = DA$ is symmetric.

Denote by $\mathcal{P}^+$ the set of all $I$-tuples of polynomials with constant term $1$ as in Definition \ref{def:affine_positive_weight_lattice}. Then, as extensively discussed in Subsection \ref{subsect:drinfeld_poly}, there is a one to one correspondence between elements in $\mathcal{P}^+$ and irreducible objects in $\mathrm{C}$. Furthermore, for any given representation $V$ in $\mathrm{C}$, we call its corresponding $I$-tuple of polynomials $(P_{i,V})_{i\in I}\in\mathcal{P}^+$ the Drinfeld polynomials of $V$. Equivalently, we can of course also consider indeterminates $u_i$, $i\in I$, and denote the $I$-tuple $(P_{i,V})_{i\in I}$ by the single polynomial $\prod_{i\in I} P_{i,V}(u_i)$. Then, by abuse of notation, we call $P_V\coloneq\prod_{i\in I} P_{i,V}(u_i) \in \mathcal{P}^+$ the Drinfeld polynomial of the representation $V$. Further, denote by $Y_{i,a}\coloneq 1-au_i$, $i\in I$, the Drinfeld polynomial of the fundamental module $V_{\omega_i}(a)$ (cf. Definition \ref{def:fund_modules_quant_aff}), then $\mathcal{P}^+$ is the set of all monomials in the variables $Y_{i,a}\coloneq 1-au_i$, $i\in I$. It is clear from Theorem \ref{thm:drinfeld_polynomials} and the discussion of the representation theory of $U_q(\mathfrak{g})$ at the begin of Section \ref{sect:finite-dim_reps} that, if we map the $Y_{i,a}$ onto the basis elements $e^{\omega_i}$ of the group algebra $\mathbb{Z}[P]$, the Drinfeld polynomial $P_V$ of an irreducible representation $V$ of $\mathrm{C}$ is mapped onto the basis element $e^{\lambda}$ such that $\lambda\in P^+$ is the highest weight of $V$ considered as a representation $U_q(\mathfrak{g})\lhook\joinrel\rightarrow U_q(\tilde{\mathfrak{g}})$. Therefore, it is natural to interpret the Drinfeld polynomials as affine (integral) highest weights. This is of course already clear in the $\mathfrak{sl}_2$ case discussed in the last section, and we will see how everything generalises in the following. 

After the discussion of the $\mathfrak{sl}_2$ case in Subsection \ref{subsect:Jimbos_hom_und_ev_reps} and Section \ref{sect:fin-dim_rep_quantum_aff_sl2} and in order to justify the Definition \ref{def:q-character_drinfeld} of the $q$-character, let us finally give the proof of Theorem \ref{thm:drinfeld-l_weights} (cf. \cite{FR}).
\begin{proof}[Proof of Theorem \ref{thm:drinfeld-l_weights}]
	Since the theorem was stated in Subsection \ref{subsect:drinfeld_poly}, let us write it down again.
	We proof the following assertion.
	\begin{quote}
		The eigenvalues $\Psi_i^\pm(u)$ of $\boldsymbol{\varPhi}_i^\pm(u)$ on any finite-dimensional type $\boldsymbol{1}$ representation of $U_q(\tilde{\mathfrak{g}})$ have the form
		\begin{equation*}
			\Psi_i^\pm(u) = q_i^{\deg(Q_{i})-deg(R_{i})}\frac{Q_{i,V}(q_i^{-2}u)R_{i,V}(u)}{Q_{i,V}(u)R_{i,V}(q_i^{-2}u)},
		\end{equation*}
		as elements of $\mathbb{C}[[u^{\pm 1}]]$, respectively, where $Q_i(u)$ and $R_i(u)$ are polynomials with constant term 1, i.e. $Q_i(u),R_i(u)\in\mathcal{P}^+$. $\odot$
	\end{quote}
	So let $U_q(\tilde{\mathfrak{g}})_{\{i\}}$ be the subalgebra of $U_q(\tilde{\mathfrak{g}})$ generated by $\mathcal{K}_i^{\pm1},\mathcal{H}_{i,r},\mathcal{X}_{i,s}$, $r\in\mathbb{Z}\backslash \{0\}$, $s\in\mathbb{Z}$ isomorphic to $U_{q_i}(\tilde{\mathfrak{sl}}_2)$. Then, the eigenvalues of the $\boldsymbol{\varPhi}_i^\pm(u)$ on a representation $V$ in $\mathrm{C}$ coincide with the eigenvalues of $\boldsymbol{\varPhi}_i^\pm(u)$ on the restriction of $V$ to $U_q(\tilde{\mathfrak{g}})_{\{i\}}$. Therefore, the proof is reduced to the case when $\mathfrak{g}=\mathfrak{sl}_2$ where everything is clear from the discussion in the last Section. Let us repeat it nevertheless.
	
	For an evaluation representation $V^{(r)}(a)$ the assertion is proven in Corollary \ref{cor:loop_weights_ev_reps_sl2}. Then, by Proposition \ref{prop:iso_to_tens_of_ev} every finite-dimensional irreducible type $\boldsymbol{1}$ representation of $U_q(\tilde{\mathfrak{sl}}_2)$ is isomorphic to a tensor product of evaluation representations. Hence, the assertion is proven if we can show that the eigenvalues of $\boldsymbol{\varPhi}_i^\pm(u)$ on the tensor product $V\otimes W$ are equal to the products of the eigenvalues on $V$ and $W$. This, however, is clear from the partial description of the comultiplication in Proposition \ref{prop:comult_partial} and Remark \ref{rem:comult_partial_rem}.
\end{proof}
Theorem \ref{thm:drinfeld-l_weights} motivates
\begin{definition}[the affine weight lattice]
	We define the \textbf{affine weight lattice} $\mathcal{P}$ as the set of all Laurent monomials in the fundamental loop-weights, i.e. a monomial in the $Y_{i,a}^{\pm1}$, $i\in I, a\in\mathbb{C}$. $\odot$
\end{definition}
Then, every loop-weight $\boldsymbol{QR^{-1}}$ is an element of $\mathcal{P}$.
Let us make the following
\begin{rems}\label{rems:affine_simple_roots}\hspace{1em}
	\begin{enumerate}
		\item Using the relations in Theorem \ref{thm:second_drin_quantum_aff} one shows that the generalised eigenvalues of the $\mathcal{H}_{i,m}$, $m>0$, in a loop-weight space $W$ of $V\in\mathrm{C}$ of loop-weight
		\begin{equation*}
			m(W)=\prod_{i\in I}\left(\prod_{r=1}^{k_i}Y_{i,a_{ir}}\prod_{s=1}^{l_i}Y_{i,b_{is}}^{-1}\right)\in\mathcal{P},
		\end{equation*}
		are always of the form
		\begin{equation*}
			\frac{[m]_{q_i}}{m}\left(\sum_{r=1}^{k_i}(a_{ir})^m-\sum_{s=1}^{l_i}(b_{is})^m\right)\qquad a_{ir},\, b_{ir}\in\mathbb{C}^\times,
		\end{equation*}
		and they completely determine the eigenvalues of the $\mathcal{H}_{i,m}$, $(m<0)$, and $\mathcal{K}_i$ for $i\in I$. Then, the $\boldsymbol{q}$\textbf{-character} given by Definition \ref{def:q-character_drinfeld} is the Laurent polynomial $\chi_q(V)\in\mathbb{Z_+}[\mathcal{P}]$ given by
		\begin{equation*}
			\chi_q(V) = \sum_{W}\dim(W)m(W),
		\end{equation*}
		where the sum goes over all loop-weight spaces $W$ of $V$.
		\item Since $U_q(\tilde{\mathfrak{g}})$ has the natural subalgebra $U_q(\mathfrak{g})\lhook\joinrel\rightarrow U_q(\tilde{\mathfrak{g}})$, every $V\in \mathrm{C}$ can be regarded a $U_q(\mathfrak{g})$-module by restriction. Then, as explained in Remark \ref{rem:loop_weight_spaces}, the loop-weight space decomposition of $V$ is a refinement of its decomposition as a direct sum of $U_q(\mathfrak{g})$-weight spaces. Therefore, the map $\operatorname{wt}:\mathcal{P}\to P$, $Y_{i,a}\mapsto \omega_i$ is a homomorphism of abelian groups. Moreover, it follows that its unique $\mathbb{Z}$-linear extension, denoted also by $\operatorname{wt}$, is a homomorphism of group algebras $\mathbb{Z}[\mathcal{P}]\to \mathbb{Z}[P]$ which maps the $q$-character onto the ordinary character of the underlying $U_q(\mathfrak{g})$-module, i.e. $\operatorname{wt}(\chi_q(V)) = \chi(V)$. In particular, the partial order on $P$ induces a partial order on $\mathcal{P}$.\label{enum:wt_remark}
		\item\label{enum:affine_simple_roots} Frenkel and Reshetikhin \cite{FR} give a different but equivalent definition of the fundamental loop-weights $Y_{i,a}$ as elements in an extension of $U_q(\tilde{\mathfrak{g}})$ as follows.
		
		Introduce new variables $\tilde{\mathcal{K}}_i$, $i\in I$, such that
		\begin{equation*}
			\mathcal{K}_j = \prod_{i\in I}\tilde{\mathcal{K}}_i^{a_{ij}}.
		\end{equation*}
		Then, the extended algebra is defined by replacing the generators $\mathcal{K}_i$, $i\in I$, with $\tilde{\mathcal{K}}_i$, $i\in I$.
		Note that, while $\mathcal{K}_i$ corresponds to the simple root $\alpha_i$, $\tilde{\mathcal{K}}_i$ corresponds to the fundamental weight $\omega_i$.\footnote{ In fact, the element $q^x$ in Part \ref{enum:almost_ribbon} of the Remarks \ref{rems:quant_aff} is given by $\prod_{i\in I}\tilde{\mathcal{K}}_i^2$.}
		Define $A(q)= (a_{ij}(q))\coloneq ((q_i+q_i^{-1})\delta_{ij}+(1-\delta_{ij})[a_{ij}]_q)$. Then, the appropriate definition is
		\begin{equation*}
			Y_{i,a} = 
			\tilde{\mathcal{K}}_i^{-1}\exp\left(-(q-q^{-1})\sum_{s=1}^{\infty}\tilde{\mathcal{H}}_{i,-s}u^sa^s\right),
		\end{equation*}
		where 
		$\tilde{\mathcal{H}}_{i,-m} \coloneq \sum_{j\in I} (A(q^m)^{-1})_{ji} \mathcal{H}_{j,-m}$. In their paper, the $q$-character is then defined by taking the quantum trace of the $L$-operator and then applying an analogue of the Harish--Chandra homomorphism.\footnote{ This is also explained in the paper \cite{FM}, where some typos which were confusing for us at the first view are corrected.} In fact, for representations in $\mathrm{C}$ it retains our Definition \ref{def:q-character_drinfeld}. However, we do not intend to review the entire paper. Instead, we point out that they motivate the generalisation of the \textbf{affine simple roots} by defining
		\begin{equation*}
			A_{i,a} = 
			\mathcal{K}_i^{-1}\exp\left(-(q-q^{-1})\sum_{s=1}^{\infty}\mathcal{H}_{i,-s}u^sa^s\right) =\boldsymbol{\varPhi}_i^-(u^{-1}a^{-1}),\quad a\in\mathbb{C}^\times.
		\end{equation*}
		Consequently, $A_{i,a}$ is given in terms of the fundamental loop-weights  by
		\begin{equation*}
			A_{i,a} = Y_{i,aq_i}Y_{i,aq_i^{-1}}\prod_{a_{ji}=-1}Y_{j,a}^{-1}\prod_{a_{ji}=-2}Y_{j,aq}^{-1}Y_{j,aq^{-1}}^{-1}\prod_{a_{ji}=-3}Y_{j,aq^2}^{-1}Y_{j,a}^{-1}Y_{j,aq^{-2}}^{-1},
		\end{equation*}
		and we have $\operatorname{wt}(A_{i,a}) = \alpha_i$ (see also \cite{FM}). $\odot$
	\end{enumerate}
\end{rems}
In view of Part \ref{enum:affine_simple_roots} of the preceding remarks we make the
\begin{definition}[affine simple roots]
	We define the \textbf{affine simple roots} $A_{i,a}$, $i\in I$, $a\in\mathbb{C}^\times$, by
	\begin{equation*}
		A_{i,a} = Y_{i,aq_i}Y_{i,aq_i^{-1}}\prod_{a_{ji}=-1}Y_{j,a}^{-1}\prod_{a_{ji}=-2}Y_{j,aq}^{-1}Y_{j,aq^{-1}}^{-1}\prod_{a_{ji}=-3}Y_{j,aq^2}^{-1}Y_{j,a}^{-1}Y_{j,aq^{-2}}^{-1}.\quad\odot
	\end{equation*}
\end{definition}
Then, in analogy to Equation (\ref{eqn:q-char_sl2_loop_weights}) for $\mathfrak{sl}_2$, we have (cf. \cite{FR} and \cite{FM})
\begin{thm}\label{thm:highest_loop-weight}
	Let $V$ be an irreducible representation in $\mathrm{C}$, then
	\begin{equation*}
		\chi_q(V) = P_{V}(1+\sum_{p} M_p),
	\end{equation*}
	where each $M_p$ is a monomial in the $A_{i,a}^{-1}$, $i\in I$, $a\in\mathbb{C}^\times$. $\odot$
\end{thm}
In order to explain the full analogy to the representation theory of $\mathfrak{g}$ let us make the following definitions.
\begin{definition}[dominant loop-weights]
	We say that a \textbf{loop-weight} $m \in \mathcal{P}$ is \textbf{dominant} (resp. anti-dominant), if $m\in \mathcal{P}^+$ (resp. $m\in \mathcal{P}^-$), i.e. $m$ is a monomial in the (anti-)fundamental loop-weights $Y_{i,a}$ (resp. $Y_{i,a}^{-1}$), $i\in I$, $a\in \mathbb{C}^\times$. $\odot$
\end{definition}
By Theorem \ref{thm:drinfeld_polynomials}, every dominant loop-weight corresponds to an irreducible element in $\mathrm{C}$.
\begin{definition}[affine root lattice]
	Define the \textbf{affine root lattice} $\mathcal{Q}$ as the subgroup of $\mathcal{P}$ generated by the affine simple roots $A_{i,a}$, $i\in I$, $a\in \mathbb{C}^\times$. Moreover, we define the positive (resp. negative) root lattice $Q^+$ (resp. $\mathcal{Q}^-$) by semigroup of monomials in the affine simple roots $A_{i,a}$ (resp. $A_{i,a}^-$), $i\in I$, $a\in \mathbb{C}^\times$. $\odot$
\end{definition}
\begin{definition}[partial ordering of loop-weights]
	We define a \textbf{partial ordering} on the set $\mathcal{P}$ \textbf{of} all \textbf{loop-weights} by saying that $m'\in\mathcal{P}$ is higher than $m\in\mathcal{P}$, denoted $m'\geq m$, if $m' m^{-1}\in \mathcal{Q}^+$. $\odot$
\end{definition}
Consequently, we have
\begin{cor}\hspace{1em}\label{cor:compatibility}
	\begin{enumerate}
		\item The partial ordering on $\mathcal{P}$ is compatible with the partial order on $P$ in the sense $m\leq m' \, \Rightarrow \operatorname{wt}(m) \leq \operatorname{wt}(m')$.
		\item For all $m_+\in\mathcal{P}^+$ we have $\mathcal{L}(V(m_+))\subset m_+\mathcal{Q}^-$, where $\mathcal{L}(V)$ denotes the set of all loop-weights of $V\in\mathrm{C}$, and $V(m_+)$ is the irreducible representation with Drinfeld polynomial $m_+\in\mathcal{P}^+$.\footnote{ I.e. for $V\in\mathrm{C}$, $\mathcal{L}(V)= \{m\in\mathcal{P}\,|\,m\text{ is a loop-weight of }V\}$. The notation $V(m_+)$ is as in Corollary \ref{cor:notation_V(P)}.} $\odot$
	\end{enumerate}
\end{cor}
We can now formulate the
\begin{thm}[theorem of the highest loop-weight]\label{thm:thm_of_the_highest_loop-weight}
	Let $U_q(\tilde{\mathfrak{g}})$ be the (untwisted) quantum affine algebra of the finite-dimensional simple Lie algebra $\mathfrak{g}$ over $\mathbb{C}$. Then,
	\begin{enumerate}
		\item every finite-dimensional irreducible type $\boldsymbol{1}$ representation has a highest loop-weight $m_+\in \mathcal{P}^+$,
		\item the highest loop-weight is always dominant,
		\item two irreducible representations with the same highest loop-weight are isomorphic,
		\item every dominant loop-weight $m_+\in\mathcal{P}^+$ is the highest loop-weight of an irreducible representation $V(m_+)$. $\odot$
	\end{enumerate}
\end{thm}
Clearly, Theorem \ref{thm:thm_of_the_highest_loop-weight} is just a reformulation of the Theorems \ref{thm:drinfeld_polynomials}, \ref{thm:drinfeld-l_weights} and \ref{thm:highest_loop-weight}, where $m_+\in\mathcal{P}^+$ is just the Drinfeld polynomial of $V(m_+)$ as in Corollary \ref{cor:compatibility}. Finally, let us collect the properties of the $q$-character.
\begin{thm}[properties of $\chi_q$]\hspace{1em}
	\label{thm:properties_of_chi_q}
	\begin{enumerate}
		\item The $q$-character $\chi_q:\mathrm{C}\to \mathbb{Z}[\mathcal{P}]$ depends only on the class of $V\in\mathrm{C}$ in its Grothendieck ring $\operatorname{Rep}U_q(\tilde{\mathfrak{g}})$. Moreover, the induced map $\operatorname{Rep}U_q(\tilde{\mathfrak{g}}) \to \mathbb{Z}[\mathcal{P}]$, denoted also by $\chi_q$, defines an injective ring-homomorphism.
		\item For any finite-dimensional representation $V$ of $U_q(\tilde{\mathfrak{g}})$ we have
		$\chi_q(V)\in\mathbb{Z}_+[\mathcal{P}]$.
		\item Let $\chi : \operatorname{Rep}(U_q(\mathfrak{g}))\to \mathbb{Z}[P]$ be the $U_q(\mathfrak{g})$-character homomorphism, let $\operatorname{wt}:\mathbb{Z}[\mathcal{P}]\to\mathbb{Z}[P]$ be the homomorphism defined by $Y^{\pm1}_{i,a}\mapsto e^{\pm\omega_i}$ as in Part \ref{enum:wt_remark} of the Remarks \ref{rems:affine_simple_roots} and let $\operatorname{res}:\operatorname{Rep}U_q(\tilde{\mathfrak{g}})\to \operatorname{Rep}U_q(\mathfrak{g})$ be the restriction homomorphism. Then the diagram
		\begin{equation*}
			\begin{tikzcd}[row sep=large, column sep=large]
				\operatorname{Rep}U_q(\tilde{\mathfrak{g}}) \arrow{r}{\chi_q} \arrow[swap]{d}{\operatorname{res}} & \mathbb{Z}[\mathcal{P}] \arrow{d}{\operatorname{wt}} \\
				\operatorname{Rep}U_q(\mathfrak{g}) \arrow{r}{\chi} & \mathbb{Z}[P]
			\end{tikzcd}
		\end{equation*}
		commutes, i.e. $\chi_q(V)$ reduces to $\chi(V)$ on the subalgebra $U_q(\mathfrak{g}) \leq U_q(\tilde{\mathfrak{g}})$.
		\item\label{enum:q_char_four} $\operatorname{Rep}U_q(\tilde{\mathfrak{g}})$ is a commutative ring which is isomorphic to $\mathbb{Z}[t_{i,a}]_{i\in I;a\in\mathbb{C}^\times}$, where $t_{i,a}$ is the class of $V_{\omega_i}(a)$. $\odot$
	\end{enumerate}
\end{thm}
In fact, we have already proven almost all the assertions in Theorem \ref{thm:properties_of_chi_q}. Indeed, the injectivity of $\chi_q$ follows from Corollary \ref{cor:tensor_product_of_fund}, and then only Part \ref{enum:q_char_four} is not completely obvious. It mainly relies on the fact that $V(\boldsymbol{P})^*$ is isomorphic to $V(\boldsymbol{P}^*)$, where $P_i^*(u) = P_{\bar{i}}(uq^{d^\vee h^\vee})$ and $\bar{i}$ is defined by $\alpha_{\bar{i}} = -w_0(\alpha_i)$ with $w_0$ the longest element of the Weyl group of $\mathfrak{g}$. In fact, this is clear by Remark \ref{rem:shifted_dual} in the case $\mathfrak{g}=\mathfrak{sl}_{l+1}$. In general, it is proven in the paper \cite{CP1996} by Proposition 5.1(b). Then, one applies the identity $(V\otimes W)^* = W^*\otimes V^*$ to compare the composition factors in the tensor products $V\otimes W$ and $W\otimes W$. Alternatively, one can use the fact that the $R$-matrix defines an intertwiner $PR_{V,W}(u):V(u)\otimes W \to W\otimes V(u)$ as explained in Remark \ref{rem:left_mult_intertwiner_R-matrix}.\footnote{ Where $P(a\otimes b) = b\otimes a$.} Then the assertion follows since $R_{V,W}(u)= (\rho_{V(u)}\otimes\rho_W)(\tilde{\mathcal{R}})\in \End(V\otimes W)[[u]]$ is an expansion of a meromorphic function in $u$, which is an isomorphism for generic $u\in\mathbb{C}$.

\begin{rem}\label{rem:tau_equivariant_action}
	Although we didn't explicitly state it, one easily verifies that (by definition) the $q$-character homomorphism $\chi_q:\mathrm{C}\to \mathbb{Z}[\mathcal{P}]$ is $\mathbb{C}^\times$-equivariant. The action of $a\in\mathbb{C}^\times$ on $\mathrm{C}$ is defined through the pullback by $\tau_a$ (see Equation (\ref{eqn:tau_a_q-deformed})), whereas the action of $a\in\mathbb{C}^\times$ is defined on the fundamental loop-weights by $Y_{i,b}\mapsto Y_{i,ab}$, $i\in I$, $b\in\mathbb{C}^\times$. In fact, this is just another reformulation of Corollary \ref{cor:pull-back_by_tau_a}.\footnote{ Combined with Theorem \ref{thm:drinfeld-l_weights}, to be a little more precise.} $\odot$
\end{rem}
\subsection{$q$-characters and snake modules}\label{subsect:q-char_and_snake_mod}
In contrast to the representation theory of finite-dimensional simple Lie algebras $\mathfrak{g}$ over $\mathbb{C}$ discussed in Subsection \ref{subsect:Cartan_data}, $U_q(\tilde{\mathfrak{g}})$ is not simple. This has the consequence that the (loop-)weight space decomposition in Remark \ref{rem:loop_weight_spaces} is in general only a decomposition into Jordan blocks. Hence, the category $\mathrm{C}$ of finite-dimensional type $\boldsymbol{1}$ representations is not semisimple, whereas the category $\mathrm{C}_{\mathfrak{g}}$ of finite-dimensional representations in $\mathfrak{g}$ is. While every short exact sequence in $\mathrm{C}_{\mathfrak{g}}$ splits, short exact sequences in $\mathrm{C}$ in general don't. Therefore, elements in $\mathrm{C}$ are not necessarily completely reducible. Moreover, it can happen that the tensor product of two nontrivial irreducible representations is irreducible. Indeed, we have seen in the last section that this is the case for almost all possible tensor products, i.e. tensor products of irreducible representations of $\mathrm{C}$ are reducible only when their corresponding Drinfeld polynomials are in special position. Precisely, considering only tensor products with irreducible representations that cannot be written as a tensor product of two other nontrivial representations, there is always only countably many such representations in special position to a given representation $V\in\mathrm{C}$. By Corollary \ref{cor:pull-back_by_tau_a}, this follows from the fact that, for two given representations $V$ and $W$, there is always only a finite set of points $a\in\mathbb{C}^\times$ such that $V(a)\otimes W$ is reducible. We may refer to these points as the 'special positions'. However, our discussion certainly motivates
\begin{definition}[special, thin, prime, real]\label{def:special_thin_prime_real}
	A representation $V\in\mathrm{C}$ is called
	\begin{enumerate}
		\item \textbf{special} (resp. anti-special), if $\chi_q(V)$ has exactly one dominant (resp. anti-dominant) loop-weight.
		\item \textbf{thin}, if no loop-weight space of $V$ has dimension greater than one.
		\item \textbf{prime}, if $V$ is not isomorphic to a tensor product of two nontrivial representations in $\mathrm{C}$.
		\item \textbf{real}, if $V\otimes V$ is irreducible. $\odot$
	\end{enumerate}
\end{definition}
We have seen above (Theorem \ref{thm:properties_of_chi_q} Part \ref{enum:q_char_four}) that the Grothendieck ring of $\mathrm{C}$ is the polynomial ring over $\mathbb{Z}$ in the classes $[V_{\omega_i}(a)]\;(i\in I,\, a\in \mathbb{C}^\times)$ of fundamental modules. Moreover, we can generalise Kirillov--Reshetikhin modules (KR) for arbitrary $\mathfrak{g}$ as follows.
\begin{definition}[Kirillov--Reshetikhin module]\label{def:KR_module}
	Let $i\in I$ and $a\in \mathbb{C}^\times$. The Kirillov--Reshetikhin module $W_{i}^{(r)}(a)$ is defined in terms of Theorem \ref{thm:thm_of_the_highest_loop-weight} by
	\begin{equation*}
		W_{i}^{(r)}(a):=V(S_r^i(aq^{r-1})),
	\end{equation*}
	where the $q$-string $S_r^i(a)$ is defined as $S_r^i(a)\coloneq Y_{i,aq^{r-1}}Y_{i,aq^{r-3}}\cdots Y_{i,aq^{-r+1}}$.\footnote{It is obtained from the $q$ string $S_r(a)$ above by mapping $Y_{aq^l}\mapsto Y_{i,aq^l}$.} $\odot$
\end{definition}
Note that $\text{wt}(W_{i}^{(k)}(a)) = k\omega_i$.
In particular, $W^{(1)}_{i}(a)$ coincides with the fundamental module $V_{\omega_i}(a)$. Let us also make the following
\begin{rem}\label{rem:T-system}
	One can show that KR modules are special (and anti-special), thin, prime and real, and the classes $[W_{i}^{(k)}(a)]$ in $\mathbb{Z}[\mathcal{P}]$ satisfy the $\boldsymbol{\operatorname{T}}$\textbf{-system}
	\begin{align}
		[W_{i}^{(k)}(a)][W_{i}^{(k)}(aq^2)]=[W_{i}^{(k+1)}(a)][W_{i}^{(k-1)}(aq^2)]+\prod_{a_{ij}=-1}[W_{j}^{(k)}(aq)]
	\end{align}
	for simply laced $\mathfrak{g}$ \cite{H}.\footnote{ I.e. when $\mathfrak{g}$ is of type $ADE$.}
	Note that it generalises the $\boldsymbol{\operatorname{T}}$\textbf{-system} (\ref{eqn:Tsys_sl2}) \cite{KNS} \cite{N} \cite{H}. As in the $\mathfrak{sl}_2$ case, it can be used to calculate the class $[W^{(k)}_i(a)]$ inductively as a polynomial in the classes of fundamental modules $[V_{\omega_i}(a)]$, $i\in I,\, a\in \mathbb{C}^\times$ (cf. \cite{HL}). Moreover, it can be generalised for arbitrary $\mathfrak{g}$ (see \cite{H}). In physics, this system of equations is usually written in terms of transfer matrices with corresponding auxiliary spaces (cf. \cite{KNS}).\footnote{"The product in $\text{Rep}(U_q(\tilde{\mathfrak{g}}))$ describes traces of tensor products."} However, the Kirillov--Reshetikhin modules and \textit{the $T$-system} only cover a small part of the prime irreducible modules for rank $l\geq 2$. In type $A$, they cover exactly the evaluation representations of representations of $U_q(\mathfrak{gl}_{l+1})$ which have a rectangular Young diagram. $\odot$
\end{rem}
We are interested in a general description of the tensor product in $\mathrm{C}$. To answer this question, it is helpful to keep in mind that simple objects are essentially monomials in the fundamental loop-weights $Y_{i,a}=(1-au_i)$, $i\in I$, $a\in\mathbb{C}^\times$, and, up to short exact sequences,\footnote{ Corresponding to Jordan blocks, as explained above.} objects in $\mathrm{C}$ are finite sums of those. Taking into account Remark \ref{rem:tau_equivariant_action} and the preceding discussion, this question reduces to finding all the special positions, as in the $\mathfrak{sl}_2$-case. Indeed, it is enough to understand all the prime objects in $\mathrm{C}$ since the tensor product can always be factorised into a product of prime objects. Since representations in general position with respect to each other don't lead to anything new in this sense, the idea is to go 'modulo general positions'. Luckily, the set of special positions is described for arbitrary simple objects of $\mathrm{C}$ in the paper \cite{C2002} by Chari. Moreover, going modulo general positions for single laced $\mathfrak{g}$ is described in 3.6 and 3.7 in the paper \cite{HL} of Hernandez and Leclerc. For type $A$ and $B$, it is explained in the papers \cite{MY} and \cite{MY2} by Mukhin and Young, where they introduce a class of irreducible representations, called 'snake modules'. In fact, it turns out that these cover all the prime simple objects of $\mathrm{C}$.
Unfortunately, in view of the scope of this thesis, we will not be able to go through many of the nice proofs in these papers. Furthermore, we will focus on the description of type $A$ and refer the reader to the original papers for type $B$.

\begin{definition}[the subcategory $\mathrm{C}_{\mathbb{Z}}$]\label{def:subcat_C_Z}
	Define subsets $\mathcal{X}\coloneq\{(i,k)\in I\times\mathbb{Z}:i-k=0\mod 2\}$ and $\mathcal{W}\coloneq\{(i,k):(i,k-1)\in \mathcal{X}\}\subset I\times \mathbb{Z}$. Fix $a\in \mathbb{C}^\times$, then we define the subcategory $\mathrm{C}_{\mathbb{Z}}$ of $\mathrm{C}$ by considering only representations whose $q$-characters lie in the subring $\mathbb{Z}[Y_{i,aq^k}^{\pm1}]_{(i,k)\in\mathcal{X}}\subset \mathbb{Z}[\mathcal{P}]$. $\odot$
\end{definition}
In fact, it is clear that the subcategory $\mathrm{C}_{\mathbb{Z}}$ of $\mathrm{C}$ is closed under taking tensor products. Moreover, the definition of $\mathrm{C}_{\mathbb{Z}}$ doesn't depend on $a$ since one can be mapped onto every other by the Hopf algebra automorphism $\tau$, stating that $\mathrm{C}_{\mathbb{Z}}$ is a subcategory of $\mathrm{C}$ 'modulo $\tau$'. Now, the crucial part is to proof the converse statement, that 'modulo $\tau$' $\mathrm{C}$ is also a subcategory of $\mathrm{C}_{\mathbb{Z}}$. Clearly, this is equivalent to saying that elements in $\mathrm{C}_{\mathbb{Z}}$ are always in general position with respect to elements in $\mathrm{C}\backslash\mathrm{C}_{\mathbb{Z}}$, which is precisely one of the main statements in the paper \cite{C2002} by Chari. We therefore have (see \cite{HL} 3.6 and 3.7 and \cite{C2002})
\begin{cor}\label{cor:relation_between_C_and_C_Z}
	Every simple object $S$ in $\mathrm{C}$ can be written as a tensor product $S_1(a_1)\otimes\cdots\otimes S_k(a_k)$ for some simple objects $S_1,\dots,S_k \in \mathcal{C}_{\mathbb{Z}}$ and $\frac{a_i}{a_j}\in \mathbb{C}\backslash q^{2\mathbb{Z}}$.\footnote{ Here $S(a)$ is the pull-back of $S$ by $\tau_a \in \operatorname{Aut}U_q(\tilde{\mathfrak{g}})$ as usual.} $\odot$
\end{cor}  
Hence, the description objects in $\mathrm{C}$ essentially reduces to the description of the objects in $\mathrm{C}_{\mathbb{Z}}$.
By abuse of notation, we set $Y_{i,aq^k}\eqcolon Y_{i,k}$, $A_{i,aq^k}\eqcolon A_{i,k}$, $\mathbb{Z}[Y_{i,k}^{\pm1}]_{(i,k)\in\mathcal{X}} = \mathcal{Y}_{\mathbb{Z}}$ and denote by $\mathcal{P}_{\mathbb{Z}} = 1[Y_{i,k}^{\pm1}]_{(i,k)\in\mathcal{X}}$ and $\mathcal{Q}_{\mathbb{Z}}= 1[A_{i,k}^{\pm1}]_{(i,k)\in\mathcal{W}}$ the restriction of $\mathcal{P}$ and $\mathcal{Q}$ to $\mathrm{C}_{\mathbb{Z}}$, respectively.\footnote{ Note that this is well defined by the $\mathbb{C}^\times$-equivariance of the $q$-character. We also thought that it is natural to denote the monoid of all monomials in indeterminates $x_j$, $j\in J$, by $1[x_j]_{j\in J}$.}

The snake modules in type $A$ are defined as follows. 
\begin{definition}[snake position, snakes and snake modules]
	Let $(i,k)\in\mathcal{X}$.\\
	We say that
	\begin{enumerate}
		\item a point $(i',k')\in\mathcal{X}$ is in \textbf{snake position} with respect to $(i,k)$, iff $k'-k \geq |i'-i|+2$.
		\begin{enumerate}
			\item the point $(i',k')$ is in \textbf{minimal snake position} to $(i,k)$, iff $k'-k$ is equal to the lower bound.
			\item the point $(i',k')\in\mathcal{X}$ is in \textbf{prime snake position} with respect to $(i,k)$, iff $\min\{i'+i,2l+2-i-i'\}\geq k'-k \geq |i'-i|+2$.
		\end{enumerate}
		\item a finite sequence $(i_t,k_t)_{1\leq t\leq M \in \mathbb{N}}$ of points in $\mathcal{X}$ is a \textbf{snake}, iff for all $2\leq t\leq M$, $(i_t,k_t)$ is in \textit{snake position} with respect to $(i_{t-1},k_{t-1})$.
		\begin{enumerate}
			\item $(i_t,k_t)_{1\leq t\leq M \in \mathbb{N}}$ is a \textbf{minimal} (resp. \textbf{prime}) \textbf{snake}, iff any two successive points are in minimal (resp. prime) snake position to each other.
		\end{enumerate}
		\item the simple module $V(m_+)$ is a (\textbf{minimal/prime}) \textbf{snake module}, iff $m_+ = \prod_{t=1}^{M}Y_{i_t,k_t} \in \mathcal{P}^+$ for some (minimal/prime) snake $(i_t,k_t)_{1\leq t\leq M}$. $\odot$
	\end{enumerate}
\end{definition}
In this definition, the term "snake" is meant to be suggestive of the pattern formed by the zeroes of the Drinfeld polynomials of such modules.

One can now proof the following properties (cf. \cite{MY2},\cite{MY},\cite{CH},\cite{BJY}).
\begin{thm}[snake modules]\hspace{1em}
	\label{thm:snake_modules}
	\begin{enumerate}
		\item Snake modules are special, anti-special and thin.
		\item A snake module is prime iff its snake is prime.
		\item Prime snake modules are real.
		\item If a snake module is not prime, then it is isomorphic to a tensor product of prime snake modules defined uniquely up to permutation. $\odot$
	\end{enumerate}
\end{thm}
Thus, prime snake modules are prime, special, anti-special, thin, and real and any snake module decomposes into a tensor product of prime snakes.

The claim is now that prime snake modules are exactly the prime irreducible elements of $\mathrm{C}_{\mathbb{Z}}$. In fact, the main result of Mukhin and Young in their paper \cite{MY}, proofs the existence of a short exact sequence called the extended $\operatorname{T}$-system, which can be used to proof this assertion. To write it down, it is helpful to add the following Definition.
\begin{definition}[neighbouring points and neighbouring snakes]\hspace{1em}
	\begin{enumerate}
			\item For any two successive points $(i,k)$ and $(i',k')$ define the \textbf{neighbouring points} by
			\begin{align*}
					&\mathbb{X}_{i,k}^{i',k'}\coloneq
					\begin{cases}
							((\frac{1}{2}(i+k+i'-k'),\frac{1}{2}(i+k-i'+k')))  & {\scriptstyle k+i>k'-i'}\\
							\emptyset & {\scriptstyle k+i=k'-i'}
						\end{cases}\\
					&\mathbb{Y}_{i,k}^{i',k'}\coloneq
					\begin{cases}
							((\frac{1}{2}(i'+k'+i-k),\frac{1}{2}(i'+k'-i+k)))  & {\scriptstyle k+l+1-i>k'-l-1+i'}\\
							\emptyset & {\scriptstyle k+l+1-i=k'-l-1+i'.}
						\end{cases}
				\end{align*}
			\item For any prime snake $(i_t,k_t)_{1\leq t\leq M}$ we define its \textbf{neighbouring snakes}\\
			$\mathbb{X}\coloneq\mathbb{X}_{(i_t,k_t)_{1\leq t\leq M}}$ and $\mathbb{Y}\coloneq\mathbb{Y}_{(i_t,k_t)_{1\leq t\leq M}}$ by concatenating its neighbouring points. $\odot$
		\end{enumerate}
\end{definition}
Then, the \textbf{extended $\boldsymbol{\operatorname{T}}$-system} is stated in (cf. \cite{MY} Theorem 4.1)
\begin{thm}[the extended $\operatorname{T}$-system]
	\label{thm:the_extended_T-system}
	Let $(i_t,k_t)_{1\leq t\leq M}\subset\mathcal{X}$ be a prime snake of length $M\geq 2$. Let $\mathbb{X}$ and $\mathbb{Y}$ be its neighbouring snakes. Then we have the following relation in the Grothendieck ring $\operatorname{Rep}U_q(\tilde{\mathfrak{g}})$.
	\begin{align}
			\notag \left[V\left(\prod_{t=1}^{M-1}Y_{i_t,k_t}\right)\right] \left[V\left(\prod_{t=2}^{M}Y_{i_t,k_t}\right)\right] = \left[V\left(\prod_{t=2}^{M-1}Y_{i_t,k_t}\right)\right] \left[V\left(\prod_{t=1}^{M}Y_{i_t,k_t}\right)\right]\\ +\left[V\left(\prod_{(i,k)\in\mathbb{X}}Y_{i_t,k_t}\right)\right] \left[V\left(\prod_{(i,k)\in\mathbb{Y}}Y_{i_t,k_t}\right)\right],
		\end{align}
	where the summands on the right hand side are classes of irreducible modules, i.e.
	\begin{align*}
			&V\left(\prod_{t=2}^{M-1}Y_{i_t,k_t}\prod_{t=1}^{M}Y_{i_t,k_t}\right)\cong V\left(\prod_{t=2}^{M-1}Y_{i_t,k_t}\right)\otimes V\left(\prod_{t=1}^{M}Y_{i_t,k_t}\right)\\
			&V\left(\prod_{(i,k)\in\mathbb{X}}Y_{i_t,k_t}\prod_{(i,k)\in\mathbb{Y}}Y_{i_t,k_t}\right)\cong V\left(\prod_{(i,k)\in\mathbb{X}}Y_{i_t,k_t}\right)\otimes V\left(\prod_{(i,k)\in\mathbb{Y}}Y_{i_t,k_t}\right).\quad\odot
		\end{align*}
\end{thm}
To demonstrate it, a simple example of the \textbf{extended $\boldsymbol{\operatorname{T}}$-system} for a minimal snake of length $5$ in $A_4$ is depicted in Figure \ref{fig:ext_t-system_in_A_4} below (cf. \cite{HJue}). Note that in the case of KR modules, i.e., when the minimal snake is a straight line, the theorem reduces to the standard \textbf{$\boldsymbol{\operatorname{T}}$-system} in Remark \ref{rem:T-system}.
	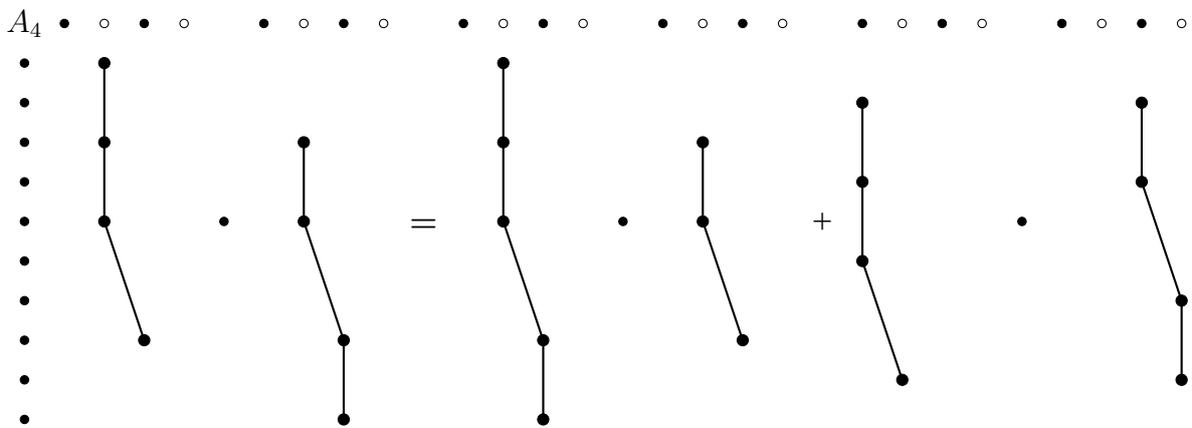
\begin{figure}[H]
	\centering
	\begin{tikzpicture}[scale=.525]
	\filldraw[black] (0,0) circle (3pt) node[anchor=north]{};
	\filldraw[black] (0,1) circle (3pt) node[anchor=north]{};
	\filldraw[black] (0,2) circle (3pt) node[anchor=north]{};
	\filldraw[black] (0,3) circle (3pt) node[anchor=north]{};
	\filldraw[black] (0,4) circle (3pt) node[anchor=north]{};
	\filldraw[black] (0,5) circle (3pt) node[anchor=north]{};
	\filldraw[black] (0,6) circle (3pt) node[anchor=north]{};
	\filldraw[black] (0,7) circle (3pt) node[anchor=north]{};
	\filldraw[black] (0,8) circle (3pt) node[anchor=north]{};
	\filldraw[black] (0,9) circle (3pt) node[anchor=north]{};
	\node at (0,10) {$A_4$};
	
	\filldraw[black] (1,10) circle (3pt) node[anchor=north]{};
	\draw (2,10) circle (3pt) node[anchor=north]{};
	\filldraw[black] (3,10) circle (3pt) node[anchor=north]{};
	\draw (4,10) circle (3pt) node[anchor=north]{};
	
	\draw[thick] (3,2) -- (2,5) -- (2,7) -- (2,9);
	\filldraw[black] (3,2) circle (4pt) (2,5) circle (4pt) (2,7) circle (4pt) (2,9) circle (4pt);
	
	\filldraw[black] (5,5) circle (3pt) node[anchor=north]{};
	
	\filldraw[black] (6,10) circle (3pt) node[anchor=north]{};
	\draw (7,10) circle (3pt) node[anchor=north]{};
	\filldraw[black] (8,10) circle (3pt) node[anchor=north]{};
	\draw (9,10) circle (3pt) node[anchor=north]{};
	
	\draw[thick] (8,0) -- (8,2) -- (7,5) -- (7,7);
	\filldraw[black] (8,0) circle (4pt) (8,2) circle (4pt) (7,5) circle (4pt) (7,7) circle (4pt);
	
	\node at (10,4.9) {$\boldsymbol{=}$};
	
	\filldraw[black] (11,10) circle (3pt) node[anchor=north]{};
	\draw (12,10) circle (3pt) node[anchor=north]{};
	\filldraw[black] (13,10) circle (3pt) node[anchor=north]{};
	\draw (14,10) circle (3pt) node[anchor=north]{};
	
	\draw[thick] (13,0) -- (13,2) -- (12,5) -- (12,7) -- (12,9);
	\filldraw[black] (13,0) circle (4pt) (13,2) circle (4pt) (12,5) circle (4pt) (12,7) circle (4pt) (12,9) circle (4pt);
	
	\filldraw[black] (15,5) circle (3pt) node[anchor=north]{};
	
	\filldraw[black] (16,10) circle (3pt) node[anchor=north]{};
	\draw (17,10) circle (3pt) node[anchor=north]{};
	\filldraw[black] (18,10) circle (3pt) node[anchor=north]{};
	\draw (19,10) circle (3pt) node[anchor=north]{};
	
	\draw[thick] (18,2) -- (17,5) -- (17,7);
	\filldraw[black] (18,2) circle (4pt) (17,5) circle (4pt) (17,7) circle (4pt);
	
	\node at (20,5) {$\scriptstyle\boldsymbol{+}$};
	
	\filldraw[black] (21,10) circle (3pt) node[anchor=north]{};
	\draw (22,10) circle (3pt) node[anchor=north]{};
	\filldraw[black] (23,10) circle (3pt) node[anchor=north]{};
	\draw (24,10) circle (3pt) node[anchor=north]{};
	
	\draw[thick] (22,1) -- (21,4) -- (21,6) -- (21,8);
	\filldraw[black] (22,1) circle (4pt) (21,4) circle (4pt) (21,6) circle (4pt) (21,8) circle (4pt);
	
	\filldraw[black] (25,5) circle (3pt) node[anchor=north]{};
	
	\filldraw[black] (26,10) circle (3pt) node[anchor=north]{};
	\draw (27,10) circle (3pt) node[anchor=north]{};
	\filldraw[black] (28,10) circle (3pt) node[anchor=north]{};
	\draw (29,10) circle (3pt) node[anchor=north]{};
	
	\draw[thick] (29,1) -- (29,3) -- (28,6) -- (28,8);
	\filldraw[black] (29,1) circle (4pt) (29,3) circle (4pt) (28,6) circle (4pt) (28,8) circle (4pt);
	\end{tikzpicture}
\caption{A simple example of the extended $\operatorname{T}$-system for a minimal snake $(i_t,k_t)_{1\leq t\leq 5}$ in $A_4$. It is a short exact sequence of snake modules defined by the minimal snake, its neighbouring snakes and the snakes obtained from it by omitting the first or the last element.}
\label{fig:ext_t-system_in_A_4}
\end{figure}
However, the proof of the fact that prime snake modules are prime irreducible elements of $\mathrm{C}_{\mathbb{Z}}$ is done in a slightly different way in \cite{BJY} for type $A$ and $B$ by showing a system of equations they call the $S$-system. In fact, they contribute to proving the Hernandez Leclerc conjecture, which states that $\mathrm{C}_{\mathbb{Z}}$ is a monoidal categorification of an infinite cluster algebra of type $A$ (respectively $B$), let's call them $\mathcal{A}_\infty$ (respectively $\mathcal{B}^\infty$), by proving that prime snake modules correspond to cluster variables in the corresponding cluster algebra. Then, the set of exchange relations of the mutations in the corresponding cluster algebras correspond precisely to the $S$-systems. In fact, it is reasonable to assume that the $S$-systems and the extended $\operatorname{T}$-systems are equivalent. Certainly, the extended $\operatorname{T}$-systems give rise to nontrivial explicit cluster relations. We will come back to this in Subsection \ref{subsect:form_in_terms_of_clust_alg}.
\subsection{The path formula for $q$-characters}\label{subsect:the_path_formula}
Since prime snake modules are the prime irreducible elements in $\mathrm{C}$, it seems natural to ask for something like Weyls character formula for (prime) snake modules. Fortunately, Mukhin and Young proof the so-called path formula right after introducing the snake modules in their paper \cite{MY2}. It allows an explicit calculation of the $q$-characters of arbitrary snake modules. Moreover, it is used to prove that snake modules are thin and special (and anti-special) in Theorem \ref{thm:snake_modules}. Let us explain how it works (cf. \cite{MY2}, \cite{BJY}).

\begin{definition}[path]
	We define a \textbf{path} as a finite sequence of points in the plane $\mathbb{R}^2$. We write $(a,b)\in \mathtt{p}$, if $(a,b)$ is an element of the sequence $\mathtt{p}$. $\odot$
\end{definition}
\begin{definition}[paths of type $A_l$]
	For each $(i,k)\in\mathcal{X}$ we define the set $\mathtt{P}_{i,k}$ of \textbf{paths} (of type $A_n$) by
	\begin{equation*}
		\begin{split}
			\mathtt{P}_{i,k}\coloneq \{&((0,y_0),(1,y_1),\dots,(l+1,y_{l+1}))|\\
			&y_0=i+k,\, y_{l+1} = l+1-i+k,\, y_{i+1}-y_i=\pm1 \text{ for all }i=0,1,\dots,l\}.
		\end{split}
	\end{equation*}
	Then, the sets $\mathtt{C}_{\mathtt{p}}^\pm$ of upper and lower corners of a path $\mathtt{p} = ((j,y_j))_{0\leq j \leq l+1}\in \mathtt{P}_{i,k}$ (of type $A_l$) are defined by
	\begin{equation*}
		\mathtt{C}_{\mathtt{p}}^\pm = \{(i,y_i)\in \mathtt{p}\, |\, i\in I,\, y_{i-1} = y_i\pm 1 = y_{i+1}\}.\quad \odot
	\end{equation*}
\end{definition}
\begin{definition}[paths to loop-weights]
	We also define a mapping
	\begin{equation*}
		\begin{split}
			\mathtt{m}:&\coprod_{(i,k)\in \mathcal{X}}\mathtt{P}_{i,k} \to \mathbb{Z}[\mathcal{P}_{\mathbb{Z}}],\\
			&\mathtt{p}\mapsto\mathtt{m}(\mathtt{p})\coloneq \prod_{(a,b)\in \mathtt{C}_p^+}Y_{a,b}\prod_{(a,b)\in \mathtt{C}_p^-}Y_{a,b}^{-1},
		\end{split}
	\end{equation*}
	which sends path to loop-weights. $\odot$
\end{definition}
Mukhin and Young also define an action of the affine simple roots on the set of corners of a path by saying that $\mathtt{p}\in\mathtt{P}_{i,k}$ can be lowered (resp. raised) at a point $(a,b)\in\mathcal{W}$, if and only if $(a,b-d_j)\in \mathtt{C}_p^+$ (resp. $\notin\mathtt{C}_p^-$) and $(a,b+d_j)\notin \mathtt{C}_{\mathtt{p}}^+$ (resp. $\in \mathtt{C}_{\mathtt{p}}^-$). If so, raising and lowering moves in type $A$ are defined as follows.
\begin{definition}[raising and lowering moves at corners]
	For each upper corner $(a,b-1)\in \mathtt{C}_{\mathtt{p}}^+$ (resp. lower corner $(a,b+1)\in \mathtt{C}_{\mathtt{p}}^-$) of a path $\mathtt{p}$, we define a lowering (resp. raising move) by
	\begin{equation*}
		\mathtt{p}\mathtt{A}_{a,b}^{\mp 1} \coloneq \left((0,y_0),\dots,(a-1,y_{a-1}),(a,y_a\pm 2),(a+1,y_{a+1}),\dots,(l+1,y_{l+1})\right)
	\end{equation*}
	such that $\mathtt{p}\mathtt{A}_{a,b}^{\mp 1}$ is again a path in $\mathtt{P}_{i,k}$. Graphically, raising and lowering moves are depicted as
	\begin{figure}[H]
		\centering
		\begin{tikzpicture}[scale=.75]
			\filldraw[black] (1,0) circle (3pt) node[anchor=north]{};
			\filldraw[black] (0,1) circle (3pt) node[anchor=north]{};
			\filldraw[black] (2,1) circle (3pt) node[anchor=north]{};
			\draw (0,3) circle (3pt) node at (0,3.4){$\scriptstyle a-1$};
			\draw (1,3) circle (3pt) node at (1,3.375){$\scriptstyle a$};
			\draw (2,3) circle (3pt) node at (2,3.4){$\scriptstyle a+1$};
			\draw[thick] (0,1) -- (1,0)  (1,0) -- (2,1);
			\draw[line width=1.5pt, line cap=round, dash pattern=on 0pt off 2\pgflinewidth] (0,1) -- (1,2)  (1,2) -- (2,1);
			\draw[dashed,thick,-stealth] (1,1.75)node[anchor=east]{} -- (1,0.25) node[anchor=west]{};
			
			\filldraw[black] (6,2) circle (3pt) node[anchor=north]{};
			\filldraw[black] (5,1) circle (3pt) node[anchor=north]{};
			\filldraw[black] (7,1) circle (3pt) node[anchor=north]{};
			\draw (5,3) circle (3pt) node at (5,3.4){$\scriptstyle a-1$};
			\draw (6,3) circle (3pt) node at (6,3.375){$\scriptstyle a$};
			\draw (7,3) circle (3pt) node at (7,3.4){$\scriptstyle a+1$};
			\draw[thick] (5,1) -- (6,2)  (6,2) -- (7,1);
			\draw[line width=1.5pt, line cap=round, dash pattern=on 0pt off 2\pgflinewidth] (5,1) -- (6,0)  (6,0) -- (7,1);
			\draw[dashed,thick,-stealth] (6,.25)node[anchor=east]{} -- (6,1.75) node[anchor=west]{};
			\node at (8,0) {$\odot$};
%
%
%
		\end{tikzpicture}
	\end{figure}
\end{definition}
\begin{definition}[non-overlapping paths]\hspace{1em}
	\begin{enumerate}
		\item We say that an $M$-tuple of paths $(\mathtt{p}_1,\dots,\mathtt{p}_M)$ is \textbf{non-overlapping}, if and only if $\mathtt{p}_s$ is strictly above $\mathtt{p}_t$, for all $s<t$, i.e. if and only if for all $(x,y)\in\mathtt{p}_s$ and $(x,z)\in \mathtt{p}_t$ $\Longrightarrow$ $y<z$.
		\item For any snake $(i_t,k_t)_{1\leq t\leq M}\subset \mathcal{X}$ of length $M\in \mathbb{Z}_{\geq 1}$, we define the corresponding set of non overlapping paths as
		\begin{equation*}
			\begin{split}
				\overline{\mathcal{P}}_{(i_t,k_t)_{1\leq t\leq M}}\coloneq \{(&\mathtt{p}_1,\dots,\mathtt{p}_M)\,|\\
				&\mathtt{p}_t\in \mathtt{P}_{i_t,k_t},\, 1\leq t\leq M, \, (\mathtt{p}_1,\dots,\mathtt{p}_M)\text{ is non-overlapping}\}.\quad \odot
			\end{split}
		\end{equation*}
	\end{enumerate}
\end{definition}
As an obvious consequence of these definitions, one then has
\begin{lem}
	Let $(i_t,k_t)_{1\leq t\leq M}\subset \mathcal{X}$ be a snake of length $M\in \mathbb{Z}_{\geq 1}$ and $(\mathtt{p}_1,\dots,\mathtt{p}_M)\in \overline{\mathcal{P}}_{(i_t,k_t)_{1\leq t\leq M}}$. If $(\mathtt{p}'_1,\dots,\mathtt{p}'_M)= (\mathtt{p}_1,\dots,\mathtt{p}_M)\mathtt{A}_{a,b}^{\pm 1}$, where $(a,b)\in\mathcal{W}$ is any point at which $(\mathtt{p}_1,\dots,\mathtt{p}_M)$ can be raised/lowered, then $\prod_{t=1}^M \mathtt{m}(\mathtt{p}_t') = A_{a,b}^{\pm 1}\prod_{t=1}^M \mathtt{m}(\mathtt{p}_t)$. $\odot$
\end{lem}
In fact, it is not hard to see that the map
\begin{equation*}
	\overline{\mathcal{P}}_{(i_t,k_t)_{1\leq t\leq M}}\to \mathbb{Z}[\mathcal{P}_{\mathbb{Z}}],\quad(\mathtt{p}_1,\dots,\mathtt{p}_M)\mapsto \prod_{t=1}^{M}\mathtt{m}(\mathtt{p}_t)
\end{equation*}
is injective. Then, after proving a few further consequences and by applying their first main result about the structure of $q$-characters (i.e.Theorem 3.4 in \cite{MY2}), Mukhin and Young finally conclude the path formula
\begin{thm}[the path formula]\label{thm:path_formula}
	Let $(i_t,k_t)_{1\leq t\leq M}$ be a snake in $\mathcal{X}$. Then
	\begin{equation*}
		\chi_q\left(V\left(\prod_{t=1}^{M}Y_{i_t,k_t}\right)\right) = \sum_{(\mathtt{p}_1,\dots,\mathtt{p}_M)\in\overline{\mathcal{P}}_{(i_t,k_t)_{1\leq t\leq M}}}\prod_{t=1}^M m(\mathtt{p}_t).
	\end{equation*}
	The module $V\left(\prod_{t=1}^{M}Y_{i_t,k_t}\right)$ is \textbf{thin} and \textbf{special} (and anti-special). $\odot$
\end{thm}
Indeed, this formula allows an explicit calculation of the $q$-character of any snake module. Moreover, by applying the mapping $\operatorname{wt}$, one easily obtains the $\mathfrak{g}$-character of the snake module. Hence, it is possible to understand its decomposition into a direct sum of $\mathfrak{g}$-irreducible subspaces.\footnote{ Irreducible type $\boldsymbol{1}$ $U_q(\mathfrak{g})$-modules, to be precise.} We further suspect that there might be a more elegant way of computing the projectors onto prime snake modules (see for instance \cite{Zab}).

\subsection{Formulation in terms of cluster algebras}\label{subsect:form_in_terms_of_clust_alg}
Cluster algebras were invented by Fomin an Zelevinsky in \cite{FZ}. To state the results in the paper \cite{BJY} let us introduce cluster algebras as follows (cf. \cite{HL2}).\footnote{ We note that the explanation here will mostly also include type $B$. This makes some parts a little bulky for the purpose of discussing only type $A$. We apologize for not finding the time to either cut all the unnecessary parts, or write a comprehensive explanation for both types.}
\begin{definition}[seed]
	Let $\mathcal{F}\coloneq \mathbb{Q}(\{x_j\}_{j\in J})$ be the field of rational functions in the variables $x_j$, $j\in J$. A \textbf{seed} in $\mathcal{F}$ is a pair $\Sigma = (\boldsymbol{y},\mathrm{Q})$, where $\boldsymbol{y} = \{y_j\}_{j\in J}$ is a free generating set of $\mathcal{F}$, and $\mathrm{Q}$ is a quiver with no loops, nor $2$-cycles, and vertices labelled by $j\in J$ such that the number of incident arrows is finite for each vertex. $\odot$
\end{definition}
\begin{definition}[mutation of a seed]
	For $k\in J$ we define the \textbf{mutation} $\mu_k$ \textbf{of a seed} $(\boldsymbol{y},\mathrm{Q})$ at the vertex $k$ by $\mu_k(\boldsymbol{y},\mathrm{Q})= (\boldsymbol{y}',\mathrm{Q}')$ such that $\boldsymbol{y}'=\{y_j'\}_{j\in J}$, $y_j' =y_j$ for $i\neq k$, and
	\begin{equation*}
		y_k' = \frac{\prod_{i\to k}y_i+\prod_{k\to j}y_j}{y_k},
	\end{equation*}
	where the first (respectively, second) product in the right-hand side is over all arrows of $\mathrm{Q}$ with target (respectively, source) $k$, and $\mathrm{Q}'$ is obtained from $\mathrm{Q}$ by
	\begin{enumerate}[label=\normalfont(\roman*)]
		\item adding a new arrow $i\to j$ for every existing pair of arrows $i\to k$ and $k\to j$;\label{enum:cluster_mutation_op_i}
		\item reversing the orientation of every arrow with target or source equal to $k$;
		\item erasing every pair of opposite arrows created by
		\ref{enum:cluster_mutation_op_i}. $\odot$
	\end{enumerate}
\end{definition}
\begin{definition}[mutation class, cluster, cluster variables and cluster monomials]
	The \textbf{mutation class} $\mathcal{C}(\Sigma)$ is the set of all seeds obtained from an initial seed $\Sigma$ by a finite sequence of mutations. If $\Sigma'=(\boldsymbol{y}',\mathrm{Q}')$ is a seed in $\mathcal{C}(\Sigma)$, then the subset $\{y_j'\}_{j\in J}$ is called a \textbf{cluster}, and its elements are called \textbf{cluster variables}. Moreover, \textbf{cluster monomials} are monomials in the cluster variables supported on a single cluster. $\odot$
\end{definition}
\begin{definition}[cluster algebra]
	The \textbf{cluster algebra} $\mathcal{A}_{\Sigma}$ corresponding to an initial seed $\Sigma$ is the subring of $\mathcal{F}$ generated by all cluster variables. $\odot$
\end{definition}
It is worth noting that $\mu_k$ is an involution for each $k\in J$. In this sense, we call two quivers $\mathrm{Q}_1$ and $\mathrm{Q}_2$ \textbf{mutation equivalent}, if one is obtained from the other by a finite sequence of mutations, and write $\mathrm{Q}_1\sim \mathrm{Q}_2$. This motivates
\begin{definition}[mutation sequence]
	For a (not necessary finite) sequence $S = (k_1,k_2,\dots)$ we define the mutation sequence $\mu_S$ by mutating first at the vertex $k_1$, then at the vertex $k_2$, etc. Alternatively, we write $\mu_S = \cdots \circ \mu_{k_2}\circ \mu_{k_1}$. $\odot$
\end{definition}
Now, the quiver which is important for the definition of $\mathcal{A}_\infty$ (and $\mathcal{B}_\infty$) is introduced in
\begin{defprop}[the quiver $G^-$]\label{defprop:Quiver_G}
	Let $B=DA= (b_{ij})$ be the symmetrized Cartan matrix of $\mathfrak{g}$ ($B=A$ and $D=\operatorname{diag}(1)$ for type $A_n$). We define a quiver $\tilde{\Gamma}$ with vertex set $V=I\times\mathbb{Z}$ determined by
	\begin{equation*}
		((i,r)\rightarrow (j,s)) \Longleftrightarrow (b_{ij}\neq 0\quad\text{and}\quad s=r+b_{ij}).
	\end{equation*}
	The quiver $\tilde{\Gamma}$ has two isomorphic connected components.
	We pick one of the components of $\tilde{\Gamma}$ and call it $\Gamma$.
	Further, we define a second labelling $\mathcal{X}$ of the vertices of $\Gamma$ by means of the transformation $(i,r)\mapsto(i,r+d_i)$ and denote by $G$ the same quiver as $\Gamma$ but with vertices labelled by $\mathcal{X}$.
	
	Let $\mathcal{X}^-:= \mathcal{X}\cap (I\times \mathbb{Z}_{\leq 0})$, then \textbf{the} (infinite) \textbf{quiver} $\boldsymbol{G^-}$ is defined as the full subquiver of $G$ with vertex set $\mathcal{X}^-$.
	Moreover, for $i\in I$, we call the subset $\mathcal{X}^-_i:= \mathcal{X}\cap (\{i\}\times \mathbb{Z}_{\leq 0})$ the $i$th column of $G^-$. $\odot$
\end{defprop}
Indeed, in type $A_l$, $\mathcal{X}$ is as in the Definition \ref{def:subcat_C_Z} of the subcategory $\mathrm{C}_{\mathbb{Z}}$. With these definitions, the (infinite) cluster algebra $\mathcal{A}_\infty$ is given by
\begin{definition}[the cluster algebra $\mathcal{A}_\infty$]
	Consider the (infinite) set of indeterminates $\boldsymbol{z}^- =\{z_{i,r}\,|\,(i,r)\in \mathcal{X}^-\}$ over $\mathbb{Q}$, where $\mathcal{X}^-$ is given by $\mathcal{X}\cap (I\times \mathbb{Z}_{\leq 0})$. We define \textbf{the} (infinite) \textbf{cluster algebra} $\boldsymbol{\mathcal{A}_\infty}$ of type $A$ by the initial seed $\Sigma_\infty = (\boldsymbol{z}^-,G^-)$, i.e. we have $\mathcal{A}_\infty = A_{\Sigma_\infty}$ with the above definitions. $\odot$
\end{definition}
The first crucial observation, stated in part $(a)$ of Theorem 3.1 in \cite{HL2}, is now that one can realise a (down) shift  on $\mathcal{A}_\infty$ in terms of an infinite mutation sequence $\mu_{\mathcal{S}}$, meaning that $\mu_{\mathcal{S}}$ doesn't change the quiver $G^-$. The infinite sequence $\mathcal{S}$ is defined as follows (cf. \cite{HL2} Subsection 2.2.3).
\begin{definition}[the infinite sequence $\mathcal{S}$]
	We attach to each column of $G^-$ an initial label given by the index of the vertex $(i,r)$, for which $r$ is maximal among the vertices of the column. We form a sequence of $d^\vee l$ columns by induction as follows. At each step we pick a column whose label $(i,r)$ has maximal $r$ among labels of all columns. After picking a column with label $(i,r)$, we change its label to $(i,r-b_{ii})$. Then, reading column after column in this ordering, from top to bottom, we obtain the \textbf{infinite sequence} $\boldsymbol{\mathcal{S}}$ of vertices of $G^-$. $\odot$
\end{definition}
We observe
\begin{cor}
	The quiver of $\mu_{\mathcal{S}}(\Sigma_\infty)$ is equal to the quiver of $\Sigma_\infty$, that is, to $G^-$. $\odot$
\end{cor}
The second crucial observation, stated in part $(b)$ of Theorem 3.1 in \cite{HL2}, is based on the following change of variables. For $(i,r)\in\mathcal{X}^-$, we set
\begin{equation}
	z_{i,r} = \prod_{k\geq 0, r+kb_{ii}\leq 0}Y_{i,r+kb_{ii}}.\label{eqn:change_of_var}
\end{equation}
In fact, this is just the highest loop-weight of the Kirillov--Reshetikhin module $W_i^{(k_{i,r})}(aq^r)$ as introduced in Definition \ref{def:KR_module}, where $k_{i,r}\in\mathbb{Z}$ is uniquely determined by the condition
\begin{equation}
	0< k_{i,r}b_{ii}-|r|\leq b_{ii}.\label{eqn:k_ir}
\end{equation}
With this change of variables, the statement is formulated in
\begin{thm}\label{thm:KR_cluster_seed}
	Let $\Sigma_{m,\infty} = \mu_{\mathcal{S}}^m(\Sigma_\infty)$ be the seed obtained from $\Sigma_\infty$ after $m$ repetitions of the mutation sequence $\mu_{\mathcal{S}}$. Let $z_{i,r}^{(m)}$ be the cluster variable of $\Sigma_{m,\infty}$ sitting at vertex $(i,r)\in \mathcal{X}^-$.\footnote{ This is a Laurent monomial in the initial variables $z_{j,s}$, $(j,s)\in \mathcal{X}^-$.} Let further $y_{i,r}^{(m)}\in \mathbb{Z}[\mathcal{P}]$ be the Laurent polynomial obtained from $z_{i,r}^{(m)}$ after performing the change of variables given by Equation (\ref{eqn:change_of_var}) and suppose that $m\geq h^\vee/2$. Then the $y_{i,r}^{(m)}$ are the $q$-characters of Kirillov--Reshetikhin modules. Precisely, for $m\geq h^\vee/2$,
	\begin{equation*}
		y_{i,r}^{(m)} = \chi_q(W_i^{(k)}(aq^{r-2d^\vee m})),
	\end{equation*}
	where $k = k_{i,r}$ is determined by Equation (\ref{eqn:k_ir}). $\odot$
\end{thm}
Hence, when $m\geq h^\vee/2$, the action of $\mu_{\mathcal{S}}$ on the seed $\Sigma_{m,\infty}$ is indeed given by shifting the index $r$ by $-2d^\vee$. Moreover, since $\mu_{\mathcal{S}}$ and $\tau$ are clearly automorphisms of their corresponding categories, we have identified the initial seed $\Sigma_\infty$ of $\mathcal{A}_\infty$ with the set of Kirillov Reshetikhin modules. In a similar fashion, mutation sequences for prime snake modules are defined in the paper \cite{BJY} and prime snake modules are identified with cluster variables of $\mathcal{A}_\infty$.\footnote{ We omit the details in view of the scope of this thesis for now. However, we have the impression that some details can be reformulated in a slightly more elegant way, at least from the point of view of a mathematical physicist.} Together with the fact that prime snake modules are prime and real (cf. \cite{BJY}), this proves 
\begin{thm}[the Hernandez--Leclerc conjecture for prime snake modules]
	\textbf{Prime snake modules} in type $A_n$ (and $B_n$) correspond to \textbf{cluster variables} in their corresponding cluster algebra $\mathcal{A}_\infty$ (resp $\mathcal{B}_\infty$). $\odot$
\end{thm}
Hence, proving the assertion in Subsection \ref{subsect:q-char_and_snake_mod} in agreement with the Hernandez--Leclerc conjecture\footnote{That $\mathcal{A}_\infty$ is a monoidal categorification of $\mathrm{C}_{\mathbb{Z}}$. I.e., that there is a one to one correspondence between cluster monomials and irreducible elements, and that cluster monomials correspond to prime irreducible elements.} (see \cite{HL} Definition 2.1). Moreover, we conclude that the extended $\operatorname{T}$-systems stated in Subsection \ref{subsect:q-char_and_snake_mod} provide non-trivial explicit relations between cluster variables.
%
\newpage
\begin{appendices}
\chapter{Appendix}
\label{App:A}
Let us list some useful identities for $q$-numbers.
\begin{lem}\label{lem:propert_q-bin_coeff}
	Let $q$ be an indeterminate. Then
	\begin{enumerate}[label=\normalfont(\roman*)]
		\item $\recbinom{r}{k}_q=q^{-k}\recbinom{r-1}{k}_q+q^{r-k}\recbinom{r-1}{k-1}_q\text{ if }r\geq k\geq 0,$
		\item $\sum_{k=0}^{r}(-1)^k\recbinom{r}{k}_qq^{-(r-1)k}=0\text{ if }r\geq0$. $\odot$
	\end{enumerate}
\end{lem}
\begin{cor}
	$\recbinom{r}{k}_q\in \mathbb{Z}[q,q^{-1}]$ if $r\geq k\geq 0$. $\odot$
\end{cor}
Since we are juggling around a lot with polynomials in this thesis, let us also give a fun reference \cite{Carroll_sep_poly}, which provides a nice and elementary proof of the assertion that a polynomial in each variable separately is a polynomial.
\end{appendices}
\cleardoublepage
\phantomsection
\addcontentsline{toc}{chapter}{References}{}
\bibliographystyle{amsplain}
\bibliography{literatur}
\newpage
\phantomsection
\addcontentsline{toc}{chapter}{Acknowledgement}{}
\section*{Acknowledgement}

Henrik Jürgens acknowledges the
support of the grant FAR UNIMORE project CUP-E93C23002040005.\\

I thank Prof. Dr. Matthias Wendt for agreeing to be my first assessor and especially for the ideas and almost weekly discussions over more than a year when I was still applying the representation theory in my PhD project in 2022. Furthermore, I am grateful for the support and discussions during the last months when I finally got to write down everything for my thesis.\\

I thank Prof. Dr. Markus Reineke for agreeing to be my second assessor. Especially, I am grateful that he did the very motivating and well structured courses on linear algebra 1 and 2 here in Wuppertal when I began studying in 2013.\\

I thank my whole physics working group and all my friends in research who have supported me over the years. Especially, I thank my doctor father Hermann Boos, my friends and colleagues Frank Göhmann, Andreas Klümper, Khazret Nirov, and Alec Cooper for all the scientific input and ideas over many years that led me to where I am now. Of course, there are many more, and I am thankful for all the nice and helpful conversations with so many other researchers I have met and friends I have made on countless conferences or have been researching together with in my working group for a long time. Thank you!\\

%
%
%
%

I thank my family for their support over the years and for enabling me to fully focus on my studies and research. Specifically, I thank my mother Doris, my father Heiko and my sister Nadine.\\

Last but not least I thank my friends, especially Jens, for their support, helpful advices and the help with the correction of my thesis.\\
\newpage
\end{document}